\documentclass[a4paper]{amsart}
\pdfoutput=1
\usepackage{style_kan}
\title{Free fibrations, lax colimits and Kan extensions for $(\infty,2)$-categories}
\author{Fernando Abell\'an}
\author{Rune Haugseng}
\author{Louis Martini}
\begin{document}
\begin{abstract} 
  In the first part of this paper we study fibrations
  of $(\infty,2)$-categories. We give a simple characterization of
  such fibrations in terms of a certain square being a pullback, and
  apply this to show that in some cases \itcats{} of functors and
  partially (op)lax transformations preserve fibrations. We also
  describe free fibrations of $(\infty,2)$-categories, including in
  the case where we only ask for (co)cartesian lifts of 
  specified 1- and 2-morphisms in the base, and describe the right
  adjoint to pullback from fibrations to such partial fibrations along
  an arbitrary functor. In the second part of the paper we apply these
  results to study colimits and Kan extensions of \itcats{}. Most
  notably, we give a fibrational description of both partially (op)lax
  and weighted (co)limits of $(\infty,2)$-categories and construct
  partially lax Kan extensions. Among other results, we also include a
  model-independent version of cofinality for $(\infty,2)$-categories
  and briefly consider presentable $(\infty,2)$-categories,
  characterizing them as accessible localizations of presheaves of
  $\infty$-categories.
\end{abstract}

\maketitle

\tableofcontents

\section{Introduction}\label{sec:intro}

As our understanding of homotopy-coherent structures in mathematics
has grown, it has become increasingly evident that in order to
adequately capture some of the phenomena encountered in practice, it
is necessary to go beyond the by now well-studied realm of \icats{}
into even higher categorical spheres. The next level of complexity, to
which we will restrict our attention in the present paper, is given by
\emph{\itcats{}}, which are the homotopy-coherent analogue of
classical 2-categories or bicategories.

One way to approach to the theory of \itcats{} is to regard them as
\icats{} enriched in \icats{} \cite{GHenriched,HaugsengRect}. The
theory of enriched \icats{} has recently been extensively developed by
Hinich~\cite{HinichYoneda} and Heine \cite{HeineEquiv,heineweighted},
so this makes available to us tools such as the Yoneda lemma, weighted
(co)limits, and Kan extensions. However, there are important aspects
of \itcat{} theory that are not visible from this perspective: most
notably, as is well-known from the classical theory of 2-categories
we can use 2-morphisms to define \emph{lax} versions of many
constructions, such as the Gray tensor product and lax
transformations \cite{CMGray,GHLGray}. The theory of \emph{fibrations}
and their straightening to functors, which has turned out to play an
extremely significant role in the development of the theory of
\icats{} and many of its applications, also does not make sense for
general enrichments, but has an analogue for \itcats{}
\cite{GHLFib,ASI23,Nuiten}; straightening of fibrations furthermore
gives us access to certain lax phenomena \cite{TwoVariable,AGH24}.

Our goal in this paper is to further develop the theory of fibrations
of \itcats{}, and then apply these fibrational results to obtain a
relatively self-contained development of colimits and Kan extensions of
\itcats{}. In particular, we give a thorough treatment of \emph{free}
fibrations and the pushforward of fibrations along an arbitrary
functor (including \emph{cofree} fibrations); this allows us to access
partially lax versions of both (co)limits and Kan extensions, which we
relate to their analogues from the enriched theory.

\subsubsection*{Fibrations of $(\infty,2)$-categories}
A cocartesian fibration of \itcats{}, which we will refer
to as a \emph{$(0,1)$-fibration} in order to have succinct terminology
for all four variances of $(\infty,2)$-fibrations, is a functor
$p \colon \tE \to \tB$ where $\tE$ has $p$-cocartesian lifts of all
1-morphisms in $\tB$ and $p$-cartesian lifts of all 2-morphisms (see
\S\ref{sec:review fib} for a more precise definition). Our first main
result gives a new\footnote{Note, however, that (as we discuss in
  \S\ref{sec:e-eqce}) this statement can be rephrased as
  characterizing partial $(0,1)$-fibrations by an orthogonality
  condition, which has also previously been considered by Loubaton in
  the setting of \iicats{} \cite[Theorem 3.2.2.24]{loubaton}.}
characterization of such fibrations in terms of a certain square of
\itcats{} being a pullback. To state this, we need to consider an
\itcat{} where some 1- and 2-morphisms have been singled out; to
avoid confusion with pre-existing terminology we refer to such objects
as \emph{decorated} \itcats{}, reserving the term \emph{marked} to
refer to the case where we only label 1-morphisms.

\begin{theoremint}
  Suppose $p \colon \tE \to \tB$ is a functor of \itcats{} and $\tEd$
  is a decoration of $\tE$ by both 1- and 2-morphisms. Then $p$ is
  a $(0,1)$-fibration, with these decorations giving the $p$-cocartesian
  1-morphisms and $p$-cartesian 2-morphisms in $\tE$, \IFF{} the
  commutative square
  \[
    \begin{tikzcd}
     \DARopl(\tEd) \ar[r, "\ARopl(p)"] \ar[d, "\ev_{1}"'] & \ARopl(\tB)  \ar[d, "\ev_{1}"] \\
     \tE \ar[r, "p"'] & \tB
    \end{tikzcd}
  \]
  is a pullback of \itcats{}.
\end{theoremint}

Here $\DARopl(\tEd)$ is the locally full sub-\itcat{} of the oplax
arrow \itcat{} $\ARopl(\tE)$ whose
\begin{itemize}
\item objects are the decorated 1-morphisms in $\tEd$,
\item morphisms are the oplax squares
  \[
    \begin{tikzcd}
        \bullet \ar[r] \ar[d, mid vert] & \bullet  \ar[d, mid vert] \\
        \bullet \ar[r] \ar[Rightarrow, ur, mid vert] & \bullet
      \end{tikzcd}
    \]
    containing a decorated 2-morphism.
  \end{itemize}

We prove this theorem in \S\ref{sec:pfib char}; in fact, we prove more generally
such a characterization of \emph{partial $\ve$-fibrations} for all
four variances $\ve \in \{0,1\}^{\times 2}$, meaning functors
$p \colon \tE \to \tB$ where $\tE$ only has (co)cartesian lifts of
certain 1- and 2-morphisms specified by a decoration of $\tB$.  Let us
next note two useful results we prove as easy consequences of this
characterization:

\begin{corint}
  Suppose $p \colon \tE \to \tB$ is a $(0,1)$-fibration. Then so is
  the functor
  \[ p_{*} \colon \FUN(\tK,\tE)^{\lax} \to \FUN(\tK,\tB)^{\lax}, \]
  given by composition with $p$, for
  any \itcat{} $\tK$. Its cocartesian morphisms are those lax
  transformations whose component at each object of $\tK$ is
  $p$-cocartesian and whose lax naturality squares all contain a
  $p$-cartesian 2-morphism, and its cartesian 2-morphisms are those
  whose component at each object is $p$-cartesian.
\end{corint}

\begin{corint}
  For any \itcat{} $\tA$, the target functor
  \[ \ev_{1} \colon \ARopl(\tA) \to \tA\]
  is a $(0,1)$-fibration. Its cocartesian morphisms are the
  commutative squares whose source is an equivalence in $\tA$, and its
  cartesian 2-morphisms are those whose source is an equivalence.
\end{corint}

We prove the first statement, and its partial and decorated
generalizations, in \S\ref{sec:fib Funlax}, using decorated versions
of the Gray tensor product and lax transformations we set up in
\S\ref{sec:dec gray}. The second statement, whose generalized version
is proved in \S\ref{sec:ARoplax fibs}, has previously been proved
using scaled simplicial sets by Gagna, Harpaz, and
Lanari~\cite[Theorem 3.0.7]{GHLFib}; for us, it is the starting point for
the construction of free fibrations:

\begin{theoremint}\label{thm:free partial fib}
  For any \itcat{} $\tB$, the forgetful functor
  \[ \Fib^{(0,1)}_{/\tB} \to \CatITsl{\tB} \]
  has a left adjoint, which sends a functor $F \colon \tA \to \tB$ to
  \[ \tA \times_{\tB} \ARopl(\tB) \to \tB,\]
  where the pullback is taken via $F$ and $\ev_{0}$ and the functor to
  $\tB$ is given by $\ev_{1}$.
\end{theoremint}
Here $\CatITsl{\tB}$ is the \icat{} of \itcats{} with a map to $\tB$
and $\Fib^{(0,1)}_{/\tB}$ is its subcategory that contains the
$(0,1)$-fibrations and the maps that preserve cocartesian 1-morphisms
and cartesian 2-morphisms. A version of this result (or more precisely
its dual) based on scaled simplicial sets was previously proved by the
first author and Stern~\cite[Theorem 3.17]{AS23}.

We prove \cref{thm:free partial fib} and its partial variants in
\S\ref{sec:free fib}, and also consider some further variations of
free fibrations in \S\S\ref{sec:free dec}--\ref{sec:free fib on dec}.
After this we move on to considering right adjoints to pullback for
\itcats{}: In \S\ref{sec:smooth} we upgrade the existence of right
adjoints to pullbacks along fibrations, due to
Abell\'an--Stern~\cite{ASI23}, to the decorated setting, and we then
apply this in \S\ref{sec:cofree} to show that there is a right adjoint
to pulling back fibrations to partial fibrations along an arbitrary
functor. As a particularly notable special case, we have:
\begin{theoremint}\label{thm:pushfwd fib intro}
  For any functor $F \colon \tA \to \tB$ and any collection $I$ of
  1-morphisms in $\tA$, the functor
  \[ F^{*} \colon \Fib^{(0,1)}_{/\tB} \to \Fib^{(0,1)}_{/(\tA,I)}, \]
  given by pullback along $F$,  has a right adjoint.
\end{theoremint}
Here $\Fib^{(0,1)}_{/(\tA,I)}$ denotes the subcategory of
$\CatITsl{\tA}$ whose objects are the $(0,1)$-fibrations and whose
morphisms are required to preserve all cartesian 2-morphisms, but only
those cocartesian morphisms that lie over $I$. Another notable special
case is the existence of \emph{cofree} fibrations --- in other words,
the forgetful functor
\[ \Fib^{(0,1)}_{/\tB} \to \CatITsl{\tB} \]
has a right adjoint.

We also apply these results in \S\ref{sec:loc fib} to prove that if we
have a functor to decorated \itcats{}, then the $(0,1)$-fibration for
the functor obtained by inverting the decorations is given by a
localization of the fibration for the original functor, generalizing a
result of Hinich~\cite{HinichDK} to \itcats{}.

\subsubsection*{(Co)limits and Kan extensions}
In the last part of the paper we apply our results on fibrations to
study (co)limits and Kan extensions for \itcats{}. Given a functor $F
\colon \tA \to \tC$ and a collection $I$ of 1-morphisms in $\tA$, we
can define the \emph{$I$-(op)lax limit} of $F$ as an object of $\tC$ that
represents the presheaf of \icats{}
\[ c \mapsto \Nat^{I\dplax}_{\tA,\tC}(\underline{c}, F), \] where the
right-hand side denotes the \icat{} of (op)lax natural transformations
from the constant functor $\underline{c}$ to $F$ whose naturality
squares at morphisms in $I$ commute. This definition and its dual
gives one notion of \itcatl{} (co)limits, and in \S\ref{sec:laxlim} we
show that it is equivalent to the definition of partially (op)lax
(co)limits previously studied by the first author in \cite{AbMarked}
and by Gagna, Harpaz and Lanari in \cite{GagnaHarpazLanariLaxLim} in
the context of scaled simplicial sets.

For a functor $F$ as above and a copresheaf of \icats{} $W$ on $\tA$, we
can also define the \emph{$W$-weighted limit} of $F$ as an object of
$\tC$ that represents the presheaf
\[ c \mapsto \Nat_{\tA,\CATI}(W,
  \tC(c,F)),\] where the right-hand side denotes the \icat{} of natural
transformations from $W$ to the \icat{} of maps from $c$ into
$F(\blank)$ in $\tC$. In \S\ref{sec:wtlim} we show that weighted (co)limits can
be expressed as partially lax and oplax (co)limits over the fibrations
for the weight, and conversely that partially (op)lax colimits can be
computed as weighted colimits. This was also previously proved by Gagna,
Harpaz and Lanari \cite{GagnaHarpazLanariLaxLim}, but our work on free
fibrations allows us to give a more explicit description of the
weights that show up.

We then turn to our main new result on (co)limits in \S\ref{sec:lax
  colim 2cat}, where we compute partially (op)lax (co)limits in the \itcat{}
$\CATIT$ of small \itcats{} in terms of fibrations:
\begin{theoremint}
  Suppose $p \colon \tE \to \tA$ is the $(0,1)$-fibration for a
  functor $F \colon \tA \to \CATIT$. For a collection $I$ of
  1-morphisms in $\tA$, we have:
  \begin{enumerate}
  \item The $I$-lax colimit of $F$ is obtained from $\tE$ by inverting all cartesian
    2-morphisms and those cocartesian 1-morphisms that lie over $I$.
  \item The $I$-lax limit of $F$ is given by the \itcat{} of sections
    $\tA \to \tE$ of $p$ that send all 2-morphisms to $p$-cartesian
    2-morphisms and the 1-morphisms in $I$ to $p$-cocartesian
    morphisms.
  \end{enumerate}
\end{theoremint}
Other variances of fibrations can be used to get the $I$-oplax
(co)limits of $F$. Note that a similar description of partially
(op)lax (co)limits for \iicats{} appears in work of
Loubaton~\cite[Examples 4.2.3.12--13]{loubaton}.

We next prove some results on the existence and preservation of
(co)limits in \S\ref{sec:stuff}, including the existence of certain
(co)limits in \itcats{} of functors and partially lax transformations
(\cref{prop:colimfunctorcat}). After this we see in \S\ref{sec:cofinal
  new} that our framework of decorated \itcats{} and partial
fibrations makes it very easy to understand cofinality for functors of
marked \itcats{}, and we obtain model-independent proofs of most of
the results on cofinality from
\cite{AS23,AbMarked,GagnaHarpazLanariLaxLim} with very little effort.

If we combine \cref{thm:pushfwd fib intro} with the straightening for
partially lax transformations from \cite{AGH24}, we get that for a
functor of \itcats{} $F \colon \tA \to \tB$ and a collection $I$ of
1-morphisms in $\tA$, the
restriction functor
\[ F^{*} \colon \FUN(\tB, \CATIT) \to \FUN(\tA, \CATIT)^{I\dlax}\]
has a right adjoint. In \S\ref{sec:kanext} we combine this with our
work on (co)limits and the Yoneda embedding to show that such adjoints
exist more generally:
\begin{theoremint}
  Suppose $F$ is as above and $\tC$ is an \itcat{} that for every $b
  \in \tB$ admits $I_{b}$-lax limits over $\tA_{b \upslash}$, then the
  restriction functor
  \[ F^{*} \colon \FUN(\tB, \tC) \to \FUN(\tA, \tC)^{I\dlax}\]
  has a right adjoint $F^{I\dlax}_{*}$, the \emph{$I$-lax right Kan
    extension} functor along $F$, given at $\phi \colon \tA \to \tC$
  and $b \in \tB$ by
  \[ (F^{I\dlax}_{*}\phi)(b) \simeq \lim^{I_{b}\dlax}_{\tA_{b
        \upslash}} F.\]
\end{theoremint}
Here $\tA_{b \upslash}$ denotes the pullback
$\tA \times_{\tB} \tB_{b \upslash}$ where $\tB_{b \upslash}$ is the
fibre of $\ev_{0} \colon \ARopl(\tB) \to \tB$ at $b$, and $I_{b}$
denotes the maps therein that project to $I$. More
generally, we obtain both left and right partially (op)lax Kan
extensions, extending to a general target results proved using scaled
simplicial sets in \cite{Abellan2023}. Taking the marking $I$ to be
maximal, we obtain ordinary Kan extensions for \itcats{}; these can
also be obtained using enriched \icat{} theory \cite{heineweighted},
and our (co)limit formula specializes to the expected one in terms of
weighted (co)limits (\cref{cor:strong kan ext}).

We apply our results on Kan extensions to prove a Bousfield--Kan
formula for weighted (co)limits in \S\ref{sec:bousfieldkan}; this
gives a different proof of the \itcatl{} case of a result of
Heine~\cite[Theorem 3.44]{heineweighted} for general
enrichments. After this we use our description of pushforward for
fibrations to obtain a functorial version of the Yoneda lemma in
\S\ref{sec:free cocomp}, which we further apply to prove that taking
the \itcat{}
\[ \tPSh(\tC) := \FUN(\tC^{\op}, \CATI)\]
of presheaves of \icats{} gives the free cocompletion of a small
\itcat{} $\tC$.

Finally, we briefly study \emph{presentable} \itcats{} in
\S\ref{sec:pres}. We define these rather na\"ively as cocomplete
\itcats{} whose underlying \icat{} is presentable, and as our main
result obtain the following equivalent characterizations of this
notion:
\begin{theoremint}\label{thm:pres char intro}
  The following are equivalent for an \itcat{} $\tC$:
  \begin{enumerate}[(1)]
  \item $\tC$ is presentable.
  \item $\tC$ is cocomplete and locally small, and there is a regular
    cardinal $\kappa$ and a small full sub-\itcat{} $\tC_{0} \subseteq
    \tC$ consisting of 2-$\kappa$-compact objects, such that every
    object in $\tC$ is the conical colimit of a diagram in $\tC$
    indexed by a $\kappa$-filtered \icat{}.
  \item There is a small \itcat{} $\tJ$ and a fully faithful functor 
    $\tC \hookrightarrow \tPSh(\tJ)$ that admits a left adjoint and whose image is closed under
    conical colimits over $\kappa$-filtered \icats{} for some regular
    cardinal $\kappa$.
  \item There is a small \itcat{} $\tJ$, a small set $S$ of 1-morphisms in $\tPSh(\tJ)$ and an
    equivalence $\tC \simeq \Loc_{S}(\tPSh(\tJ))$ with the full
    sub-\itcat{} of objects that are 2-local with respect to $S$.
  \end{enumerate}
\end{theoremint}
We refer the reader to \S\ref{sec:pres} for the precise definitions of
the terms that appear here. Note that similar results on
presentability for general enriched \icats{} also appear in the work of
Heine, in particular \cite[Theorem 5.10]{heineweighted}.

\subsubsection*{Acknowledgments}
We thank the Max Planck Institute for Mathematics in Bonn, where part of this work was carried out, for its hospitality and support.

\section{Preliminaries}
\label{sec:decorated}

In this section we first recall some terminology and basic results on
\itcats{} in \S\ref{sec:twocat}. We then review marked \itcats{} in
\S\ref{sec:marked} and their marked Gray tensor product in
\S\ref{sec:marked gray}. Next, we review fibrations of \itcats{} and
their straightening in \S\ref{sec:review fib} and lax slices and cones
in \S\ref{sec:lax slice}. After this we are ready to introduce
\emph{decorated} \itcats{}, by which we mean \itcats{} equipped with
collections of special 1- and 2-morphisms, in \S\ref{sec:dec defn};
this will provide a very convenient setting for our work on fibrations
in the following sections. Finally, we introduce a decorated version
of the Gray tensor product in \S\ref{sec:dec gray}.

\subsection{$(\infty,2)$-categories and Gray tensor products}
\label{sec:twocat}

In this section we introduce our basic terminology and notation and
recall some fundamental results on \itcats{} and their Gray tensor products.
As we hardly ever need to refer to any particular model
of \itcats{}, we will not review any of them here.

\begin{notation}\ 
  \begin{itemize}
  \item Generic \itcats{} are denoted $\tA, \tB, \tC, \dots$, and
    generic \icats{} $\oA, \oB, \oC, \dots$.
  \item We write $* = [0] = C_{0}$ for the 0-cell or point,
    $[1] = C_{1}$ for the 1-cell or arrow, and $C_{2}$ for the 2-cell
    or free 2-morphism. We note also the following notation for
    2-categories built from these:
    \[ \partial [1] := \{0\} \amalg \{1\}, \quad \partial C_{2} :=
      [1] \amalg_{\partial [1]} [1], \quad \partial C_{3} := C_{2}
      \amalg_{\partial C_{2}} C_{2}.\]
  \item If $\tC$ is an \itcat{} with objects $x,y$, we will generally
    denote the \icat{} of morphisms from $x$ to $y$ in $\tC$ by
    $\tC(x,y)$, while if $\oC$ is an
    \icat{} we will write $\oC(x,y)$ or $\Map(x,y)$ for the \igpd{} of
    morphisms.
  \item We write $\CatI$ and $\CatIT$ for the \icats{} of (small)
    \icats{} and \itcats{}, and $\CATI$ and $\CATIT$ for the \itcats{}
    thereof, respectively; similarly, we write $\GpdI$ for the \icat{}
    of (small) \igpds{} or spaces.
  \item If $\tC$ and $\tD$ are \itcats{}, we write $\FUN(\tC,\tD)$ for
    the \itcat{} of functors between them (the internal Hom in
    $\CatIT$), and $\Fun(\tC,\tD)$ for its underlying \icat{}. For
    functors $F,G \colon \tC \to \tD$ we denote the \icat{} of maps
    from $F$ to $G$ in $\FUN(\tC,\tD)$ by $\Nat_{\tC,\tD}(F,G)$.
  \item For the arrow $[1]$, we write $\AR(\tC) := \FUN([1],\tC)$ for
    the \itcat{} of arrows in an \itcat{} $\tC$ and $\Ar(\tC)$ for its
    underlying \icat{}; the functor to $\tC$ given by evaluation at
    $i \in [1]$ is denoted $\ev_{i}$ ($i = 0,1$).
  \item We write $\tPSh(\tC) := \FUN(\tC^{\op},\CATI)$ for the
    \itcat{} of presheaves of \icats{} on an \itcat{} $\tC$, and
    $h_{\tC} \colon \tC \to \tPSh(\tC)$ for the Yoneda embedding
    \cite{HinichYoneda}.
  \item We say an \itcat{} $\tC$ \emph{admits tensors} by an \icat{}
    $\oK$ if the copresheaf $\Map(\oK, \tC(c,\blank))$ is corepresentable by an
    object $\oK \boxtimes c$ for every $c \in
    \tC$; dually, $\tC$ \emph{admits cotensors} by $\oK$ if 
    the presheaves
    $\Map(\oK, \tC(\blank, c))$ are representable by objects
    $c^{\oK}$. If $\tC$ admits (co)tensors by all small \icats{}, we
    say that $\tC$ is \emph{(co)tensored}.
  \item If $\tC$ is an \itcat{}, we write $\tC^{\simeq}$ or $\tC^{\leq
    0}$ for its
    underlying (or core) \igpd{}, and $\tC^{\leq 1}$ for its
    underlying \icat{}; these are right adjoint to the inclusions of
    $\GpdI$ and $\CatI$ in $\CatIT$, respectively.
  \item These inclusions also have left adjoints, which we denote
    \[ \|\blank\| = \Lgpd \colon \CatIT \to \GpdI,\]
    \[ \Lcat \colon \CatIT \to \CatI,\]
    respectively.
  \end{itemize}
\end{notation}

\begin{defn}\label{defn:orientalsprelim}
  For $n \geq 0$ we define $\tO^{n}$ to be the following strict 2-category:
  \begin{itemize}
    \item The objects are the elements of $[n]$.
    \item For $i,j \in [n]$, the mapping category $\tO^n(i,j)$ is the poset of subsets $S \subseteq [n]$ such that $\min(S)=i$ and $\max(S)=j$,
    ordered by inclusion.
    \item For a triple $i,j,k \in [n]$, the composition maps are given by taking unions.
  \end{itemize}
  The 2-categories $\tO^{n}$ correspond to the (2-truncated)
  \emph{orientals} or oriented simplices defined
  by Street~\cite{StreetOriental}, and $\tO^{n}$ should be thought of
  as an $n$-simplex with compatibly oriented 2-morphisms inserted in
  all of its 2-dimensional faces.
\end{defn}

\begin{defn}
  Suppose $\tC$ and $\tD$ are \itcats{} and that $\tC$ admits tensors by an
  \icat{} $\oK$. A functor $F \colon \tC \to \tD$ then induces for
  $c \in \tC$ a functor
  \[ \tC(\oK \boxtimes c, \oK \boxtimes c)^{\simeq} \simeq \Map(\oK,
    \tC(c,\oK \boxtimes c)) \to \Map(\oK, \tD(Fc,F(\oK \boxtimes
    c))),\] so that we get a canonical functor
  $\oK \to \tD(Fc,F(\oK \boxtimes c))$ from $\id_{\oK \boxtimes c}$. We say that $F$
  \emph{preserves tensors by $\oK$} if this exhibits the copresheaf
  $\Map(\oK, \tD(Fc, \blank))$ as corepresented by
  $F(\oK \boxtimes c)$.
\end{defn}

\itcats{} can be defined as \icats{} \emph{enriched} in $\CatI$ (see
\cite{HaugsengRect}). The general theory of enriched \icats{} has
recently been extensively developed by Hinich~\cite{HinichYoneda} and
Heine~\cite{HeineEquiv,heineweighted}. In particular, it follows from
Heine's work that an \itcat{} $\tC$ is tensored \IFF{} its enrichment
arises from a $\CatI$-module structure on the underlying \icat{}
$\tC^{\leq 1}$, and this correspondence is part of an equivalence of
\icats{}. From this we get the following useful procedure to construct
\itcats{} and functors among them:

\begin{observation}\label{propn:upgrade functor}
  Suppose $\phi \colon \CatI \to \oC$ is a product-preserving
  functor. This makes $\oC$ an algebra over $\CatI$, and so a module
  with the action given by $(\oK, c) \mapsto \phi(\oK) \times c$. If
  the functor $\oK \mapsto \phi(\oK) \times c$ has a right adjoint
  $\tC(c,\blank)$ for all $c \in \oC$, then we can upgrade $\oC$ to an
  \itcat{} $\tC$, where the mapping \icats{} satisfy
  \[ \Map(\oK, \tC(c,c')) \simeq \oC(\phi(\oK) \times c, c').\]
  Moreover, if we have a commutative triangle of \icats{}
  \[
    \begin{tikzcd}
      {} & \CatI \ar[dl, "\phi"'] \ar[dr, "\psi"] \\
      \oC \ar[rr, "F"] & & \oD
    \end{tikzcd}
  \]
  where $\phi$ and $\psi$ preserve products and there are right
  adjoints as above, and also the functor $F$ preserves products, then
  $F$ upgrades to a functor of \itcats{} for the $\CatI$-enrichments
  of $\oC$ and $\oD$ induced by the functors $\phi$ and
  $\psi$. Similarly, if the same properties hold with $\CatI$ replaced
  by $\CatIT$, we can upgrade (functors of) \icats{} to (functors of)
  $(\infty,3)$-categories.
\end{observation}

We next note the following useful criterion for upgrading adjunctions;
this is also a special case of \cite[Proposition
4.4.1]{StefanichPres} (where the presentability assumption is not used in
the proof) or \cite[Lemma A.2.14]{secondaryK}.
\begin{propn}\label{propn:upgrade adjoint}
  Suppose $F \colon \tC \to \tD$ is a functor of \itcats{} such that
  $\tC$ admits tensors by $[1]$ and $F$ preserves these. If the
  functor $F^{\leq 1}$ is a left adjoint, then so is $F$.
\end{propn}
Note that since $(\blank)^{\leq 1}$ is a 2-functor (\eg{} by
\cref{propn:upgrade functor}), it preserves adjunctions, and so in the
situtation above the underlying functor of \icats{} of the right
adjoint of $F$ is necessarily the right adjoint of $F^{\leq 1}$.
\begin{proof}
  By the Yoneda Lemma for \itcats{} it suffices to show that the
  presheaf of \icats{} $\tD(F(\blank), d)$ is representable for all
  $d \in \tD$. Since $F^{\leq 1}$ has a right adjoint $g$, we have a
  counit map $\epsilon \colon F(g(d)) \to d$ so that the composite
  \[ \tC(c, g(d))^{\simeq} \to \tC(F(c), F(g(d)))^{\simeq}
    \to \tC(F(c), d)^{\simeq} \]
  is an equivalence for every $c$. We claim that we also have an equivalence without
  taking cores; to see this it is enough to check that the composite
  \[ \Map([1], \tC(c, g(d))) \to \Map([1], \tC(F(c), F(g(d))))
    \to \Map([1], \tC(F(c), d)) \]
  is an equivalence, and since $F$ preserves tensors by $[1]$ we may
  identify this as the first map with $c$ replaced by $[1] \boxtimes
  c$.
\end{proof}

\begin{cor}\label{propn:upgrade adjn enr}
  Suppose $F \colon \oC \to \oD$ is a functor of \icats{} as in
  \cref{propn:upgrade functor}.
  \begin{enumerate}[(i)]
  \item If $F$ has a right adjoint $R$, then the adjunction
    upgrades canonically to an adjunction $F \dashv R$ of \itcats{}.
  \item If $F$ has a left adjoint $L$, the tensoring of
    $\oC$ and $\oD$ by $[1]$ has a right adjoint, and the 
    canonical map
    \[ L([1] \boxtimes d) \to [1] \boxtimes Ld \]
    is an equivalence for all $d$, then the adjunction
    upgrades canonically to an adjunction $L \dashv F$ of \itcats{}.
  \end{enumerate}
\end{cor}
\begin{proof}
  The first part is immediate from \cref{propn:upgrade adjoint}, while
  the second part follows from this together with the observation that
  $F$ preserves all cotensors by $[1]$ \IFF{} its left adjoint
  preserves all tensors.
\end{proof}

We also recall some useful terminology for subobjects of \icats{} and \itcats{}:
\begin{defn}
    A functor $F \colon \oC \to \oD$ of \icats{} is:
  \begin{itemize}
  \item \emph{faithful} if $\oC(c,c') \to \oD(Fc,Fc')$ is a
    monomorphism of \igpds{} for all $c,c' \in \oC$;
  \item an \emph{inclusion} if it is faithful 
    and  $\oC^{\simeq} \to \oD^{\simeq}$ is a monomorphism --- we then
    also say that $\oC$ is a \emph{subcategory}
    of $\oD$.
  \end{itemize}
\end{defn}

\begin{defn}\label{defn:subcategories}
  A functor $F \colon \tC \to \tD$ of \itcats{} is called
  \begin{itemize}
  \item \emph{fully faithful} if
  $\tC(c,c') \to \tD(Fc,Fc')$ is an equivalence of \icats{} for all
  objects $c,c' \in \tC$ --- we then also say that $\tC$ is a \emph{full
    sub-\itcat{}} of $\tD$;
  \item \emph{locally fully faithful} if $\tC(c,c') \to \tD(Fc,Fc')$
    is fully faithful for all $c,c' \in \tC$;
  \item a \emph{locally full inclusion} if $F$ is locally fully faithful and
    $\tC^{\simeq} \to \tD^{\simeq}$ is a monomorphism --- we then also say
    that $\tC$ is a \emph{locally full sub-\itcat{}} of $\tD$;
  \item \emph{locally faithful} if $\tC(c,c') \to \tD(Fc,Fc')$ is
    faithful for all $c,c' \in \tC$;
  \item \emph{locally an inclusion} if $\tC(c,c') \to \tD(Fc,Fc')$ is
    a subcategory inclusion for all $c,c' \in \tC$;
  \item an \emph{inclusion} if it is locally an
    inclusion and $\tC^{\simeq} \to \tD^{\simeq}$ is a monomorphism
    --- we then also say that $\tC$ is a \emph{sub-\itcat{}} of $\tD$.
  \end{itemize}
  We also say that a sub-\itcat{} $\tC$ of $\tD$ is \emph{wide} if
  $\tC^{\simeq} \to \tD^{\simeq}$ is an equivalence. As discussed in
  \cite[\S 2.5]{AGH24}, all of these properties can be characterized
  by $F$ being right orthogonal to some finite set of maps in $\CatIT$.
\end{defn}

\begin{notation}\label{not:gray}
  The \emph{Gray tensor product} of \itcats{} will play an important
  role in this paper, but as this is by now a standard construction we
  will not review the definition here (see for instance \cite{GHLGray}
  for a definition using scaled simplicial sets or \cite{CMGray} for a version based on $\Theta$-spaces).
  \begin{itemize}
  \item The Gray tensor product of two \itcats{} $\tC$ and $\tD$ is
    denoted $\tC \otimes \tD$.
  \item We write $\Fun^{\plax}(\blank,\blank)$ for the adjoints of the
    Gray tensor in each of the two variables, so that we have natural
    equivalences
    \[ \Map(\tC \otimes \tD, \tE) \simeq \Map(\tD,
      \FUN(\tC,\tE)^{\oplax}) \simeq \Map(\tC, \FUN(\tD,
      \tE)^{\lax}).\]
    Here $\Fun^{\plax}(\tC,\tD)$ is an \itcat{} whose objects are
    functors from $\tC$ to $\tD$, with morphisms given by \emph{(op)lax
      natural transformations} among these.
  \item For $[1]$, we write
    \[ \ARplax(\tC) := \FUN([1], \tC)^{\plax}.\]
  \item For functors $F,G \colon \tC \to \tD$, we denote the \icat{} 
    of maps from $F$ to $G$ in $\Fun^{\plax}(\tC,\tD)$ by $\Nat^{\plax}_{\tC,\tD}(F,G)$.
  \end{itemize}
\end{notation}

\begin{observation}\label{obs:gray square pushout}
  There is a pushout square   \[
    \begin{tikzcd}
     \partial C_{2} \ar[r] \ar[d] & C_{2}  \ar[d] \\
     \partial ([1] \otimes [1]) \ar[r] & {[1] \otimes [1]}
    \end{tikzcd}
  \]
  where $\partial ([1] \otimes [1])$ denotes the boundary
  \[(\partial [1] \times [1]) \amalg_{\partial [1] \times \partial
      [1]} ([1] \times \partial [1]) \simeq [2] \amalg_{\{0,2\}}
    [2],\] with the left vertical map given by gluing two copies of
  $d_{1} \colon [1] \to [2]$. For example, using the model of scaled
  simplicial sets it is easy to see that there is a pushout
 \[
    \begin{tikzcd}
     \partial \mathbb{O}^{2} \ar[r] \ar[d] & \mathbb{O}^{2}  \ar[d] \\
     \partial ([1] \otimes [1]) \ar[r] & {[1] \otimes [1]},
    \end{tikzcd}
  \]
  from which the square above follows using \cite[Proposition 2.3.6]{AGH24}.
\end{observation}

We emphasize that the Gray tensor product is not compatible with
cartesian products, and so does not give a functor of
\itcats{}. However, it is possible to upgrade
$\FUN^{\plax}(\tA,\blank)$ for a fixed $\tA$ to a 2-functor; we will
not prove this here, but instead we note the useful consequence that
this functor preserves adjunctions:

\begin{construction}
  Suppose $R \colon \CatIT \to \CatIT$ is a product-preserving functor
  and $\alpha \colon \id \to R$ is a natural transformation. We then
  have a natural map
  \[ \FUN(\tA,\tB) \to \FUN(R\tA, R\tB)\]
  adjoint to the composite
  \[ \FUN(\tA, \tB) \times R\tA \xto{\alpha \times \id}
    R\FUN(\tA,\tB) \times R\tA \simeq R\left(\FUN(\tA,\tB) \times
      \tA\right) \xto{R(\ev)} R\tB.\]
  In particular, from a natural transformation $\phi \colon \tA \times [1] \to
  \tB$ from $F$ to $G$ we get a natural transformation
  \[ \tilde{R}(\phi) \colon R\tA \times [1] \xto{\id \times \alpha} R\tA \times R[1] \simeq
    R(\tA \times [1]) \xto{R\phi} R(\tB).\]
  It is easy to check that this construction is compatible with both
  horizontal and vertical composition of natural transformations, so
  that we get:
\end{construction}

\begin{propn}\label{propn:prod pres gives 2-fun}
  Let $R \colon \CatIT \to \CatIT$ be a product-preserving functor
  and $\alpha \colon \id \to R$ a natural transformation. Suppose $F
  \colon \tC \to \tD$ is a functor of \itcats{} with right adjoint
  $G$, with unit $\eta$ and counit $\epsilon$. Then $R(F)$ is left
  adjoint to $R(G)$ with unit $\tilde{R}(\eta)$ and counit
  $\tilde{R}(\epsilon)$. \qed
\end{propn}

\begin{remark}
  In fact, in the situation above we should be able to upgrade $R$ to
  a functor of \itcats{} (or even $(\infty,3)$-categories) using the
  transformation $\alpha$: Since $R$ preserves products, it is a
  monoidal functor with respect to the cartesian product, and $\alpha$
  is automatically a monoidal transformation. We can in particular
  regard $R$ as a (lax) linear endofunctor for the canonical action of
  $\CatIT$ on itself. This corresponds to a $\CatIT$-enriched functor
  $R_{*}\CATIT \to \CATIT$ \cite{HinichYoneda,HeineEquiv}, where we
  push forward the enrichment along $R$. Moreover, pushforward of
  enrichment is a 2-functor (for example because it arises from composition in the
  \itcat{} of $\infty$-operads in the model of \cite{GHenriched}), so the
  monoidal transformation $\alpha$ also gives rise to a $\CatIT$-enriched functor
  $\CATIT \to R_{*}\CATIT$. The composite $\CATIT \to R_{*}\CATIT \to
  \CATIT$ should then be the desired enriched lift of $R$, but
  unfortunately we have not found any proof in the literature that the underlying
  functor of \icats{} recovers the original functor $R$, and we will
  not go into this here.
\end{remark}

\begin{example}
  For any \itcat{} $\tA$, the functor \[\FUN(\tA,\blank)^{\plax} \colon
  \CatIT \to \CatIT\] preserves adjunctions, since it preserves
  products (being a right adjoint) and we have a natural
  transformation from the identity given by the constant diagrams
  \[ \tC \to \FUN(\tA, \tC)^{\plax},\]
  or equivalently induced by the natural
  transformation of left adjoints 
  \[ \tA \otimes (\blank) \to [0] \otimes (\blank) \simeq \id.\]
\end{example}

\subsection{Marked \itcats{}}
\label{sec:marked}

In this subsection we review \emph{marked} \itcats{}, which are
\itcats{} equipped with a collection of morphisms, and compare two
descriptions of the \icat{} thereof.

\begin{defn}
  A \emph{marked \itcat{}} $(\tC,E)$ consists of an \itcat{} $\tC$
  together with a collection $E$ of morphisms in $\tC$. We may assume
  that $E$ contains all equivalences in $\tC$ and is moreover closed
  under composition, so that it determines a wide locally full
  sub-\itcat{} $\tC_{E} \hookrightarrow \tC$. This allows us to define
  the \icat{} $\MCatIT$ of marked \itcats{} as the full subcategory of
  $\Ar(\CatIT)$ spanned by the wide locally full subcategory
  inclusions. We write $\Um \colon \MCatIT \to \CatIT$ for the
  forgetful functor $(\tC,E) \mapsto \tC$, given by restricting
  $\ev_{1}$.
\end{defn}

\begin{propn}\label{propn:mono gives faithful from Ar}
  Given an \icat{} $\oC$ equipped with a collection $S$ of
  1-morphisms, let $\Ar_{S}(\oC)$ be the full subcategory of
  $\Ar(\oC)$ spanned by the morphisms in $S$.  If all elements of $S$
  are monomorphisms, then $\ev_{1} \colon \Ar_{S}(\oC) \to \oC$ is a
  faithful functor.
\end{propn}
\begin{proof}
  It suffices to show that for morphisms $f \colon x \to y$, $g \colon a \to b$ in $\oC$, the map
  \[ \Ar(\oC)(f,g) \to \oC(y,b)\]
  induced by evaluation at $1$ is a monomorphism of \igpds{} if $g$ is
  a monomorphism in $\oC$. Here $\Ar(\oC)(f,g)$ fits in a pullback
  square
  \[
    \begin{tikzcd}
     \Ar(\oC)(f,g) \ar[r, "\ev_{0}"] \ar[d, "\ev_{1}"'] & \oC(x,a)  \ar[d, "g_{*}"] \\
     \oC(y,b) \ar[r, "f^{*}"'] & \oC(x,b),
    \end{tikzcd}
  \]
  so this is clear since monomorphisms are closed under pullback.
\end{proof}

\begin{cor}
  The functor $\Um \colon \MCatIT \to \CatIT$ is faithful. \qed
\end{cor}

\begin{observation}\label{obs:adjointDcatIT}
  By \cite[Observations 2.5.6, 2.5.8]{AGH24}, a functor of \itcats{} is a wide locally full subcategory inclusion
  \IFF{} it is right orthogonal to $\emptyset \to *$,
  $\partial C_{2} \to C_{2}$.  It follows from \cite[Proposition
  5.5.5.7]{LurieHTT} that such functors form the
  right class in a factorization system on $\CatIT$, which implies
  that the inclusion $\MCatIT \hookrightarrow \Ar(\CatIT)$ has a left
  adjoint, given by factoring a morphism through this factorization
  system. As the morphisms we consider
  orthogonality for also go between compact \itcats{}, it is moreover clear that
  $\MCatIT$ is closed under filtered colimits in $\Ar(\CatIT)$,
  so that this left adjoint 
  exhibits $\MCatIT$ as an accessible localization. Thus $\MCatIT$ is
  a presentable \icat{}.
\end{observation}

\begin{observation}
  The functor $\ev_{1} \colon \Ar(\CatIT) \to \CatIT$ has the
  following properties:
  \begin{enumerate}[(1)]
  \item $\ev_{1}$ has a left adjoint, which takes an \itcat{} $\tA$ to
    $\emptyset \to \tA$.
  \item $\ev_{1}$ has a right adjoint $\const$, which takes an \itcat{} $\tA$ to
    $\tA \xto{=} \tA$.
  \item The right adjoint $\const$ is itself left adjoint to $\ev_{0}$.
  \end{enumerate}
  From this we can conclude that the restriction $\Um$ of $\ev_{0}$ to 
  has the following corresponding properties:
    \begin{enumerate}[(1')]
  \item $\Um$ has a left adjoint $(\blank)^{\flat}$, which takes an
    \itcat{} $\tA$ to
    \[ \tA^{\flat} \,\,:=\,\, \tA_{\mathrm{eq}} \to \tA,\]
    where $\tA_{\mathrm{eq}}$ is the sub-\itcat{} that contains only
    the invertible 1-morphisms, but all 2-morphisms among these.
  \item $\Um$ has a right adjoint $(\blank)^{\sharp}$, which takes an \itcat{} $\tA$ to
    \[ \tA^{\sharp} \,\,:=\,\, \tA \xto{=} \tA.\]
  \item The right adjoint $(\blank)^{\sharp}$ has a further
    right adjoint, which takes $(\tA,E)$ to $\tA_{E}$.
  \end{enumerate}
\end{observation}

It is often convenient to regard the marked 1-morphisms in a marked
\itcat{} $(\tB,E)$ as being specified by a wide subcategory of
$\tB^{\leq 1}$ instead of a sub-\itcat{} of $\tB$; this is justified by the
following observation, which leads to an alternative description of
$\MCatIT$:
\begin{propn}
  Let $\Ar_{\mathrm{lff}}(\CatIT)$ denote the full subcategory of
  $\Ar(\CatIT)$ spanned by the locally fully faithful functors, and
  let $\Ar_{\mathrm{fa}}(\CatI)$ denote the full subcategory of
  $\Ar(\CatI)$ spanned by the faithful functors. Then the commutative
  square
  \[
    \begin{tikzcd}
     \Ar_{\mathrm{lff}}(\CatIT) \ar[r, "\ev_{1}"] \ar[d,
     "(\blank)^{\leq 1}"'] & \CatIT  \ar[d, "(\blank)^{\leq 1}"] \\
     \Ar_{\mathrm{fa}}(\CatI) \ar[r, "\ev_{1}"'] & \CatI
    \end{tikzcd}
  \]
  is a pullback.
\end{propn}
\begin{proof}
  Suppose $\tB$ is an \itcat{} and $F \colon \oA \to \tB^{\leq 1}$ is
  a faithful functor. Then the composite $\oA \to \tB$ admits a unique
  factorization as $\oA \xto{i} \tA \xto{F'} \tB$ where $F'$ is
  locally fully faithful and $i$ is left orthogonal to locally fully
  faithful functors. By \cite[Theorem 5.3.7]{Soergel} this left
  orthogonal class consists precisely of the functors that are
  essentially surjective on objects and on all mapping \icats{}. It
  follows that $i^{\leq 1} \colon \oA \to \tA^{\leq 1}$ is both fully
  faithful and essentially surjective, \ie{} an equivalence. In
  particular, $F \simeq F'^{\leq 1}$. Conversely, if we start with a
  locally fully faithful functor $G \colon \tA \to \tB$, then the
  composite $\tA^{\leq 1} \xto{G^{\leq 1}} \tB^{\leq 1} \to \tB$ factors as
  $\tA^{\leq 1} \to \tA \xto{G} \tB$ where the first functor is left
  orthogonal to locally fully faithful functors; by uniqueness, this
  means that applying the factorization to $G^{\leq 1}$ recovers
  $G$. This construction therefore gives an inverse to the functor 
  \[ \Ar_{\mathrm{lff}}(\CatIT) \to \Ar_{\mathrm{fa}}(\CatI)
    \times_{\CatI} \CatIT
  \]
  from the commutative square, as required.
\end{proof}

\begin{cor}\label{cor:MCatIT via 1cat}
  The commutative square
  \[
    \begin{tikzcd}
     \MCatIT \ar[r, "\Um"] \ar[d, "(\blank)^{\leq 1}"'] & \CatIT  \ar[d, "(\blank)^{\leq 1}"] \\
     \MCatI \ar[r, "\ev_{1}"'] & \CatI
    \end{tikzcd}
  \]
  is a pullback. \qed
\end{cor}

\begin{lemma}\label{lem:localizationdecorations}
  The functor $(\blank)^{\flat} \colon \CatIT \to \MCatIT$ has a left
  adjoint \[\Lm \colon \MCatIT \to \CatIT.\] This is given by inverting
  the marked 1-morphisms, \ie{} for a marked \itcat{} $(\tA,E)$ we
  have a natural pushout square
\[
  \begin{tikzcd}
   \tA_{E}^{\leq 1} \ar[r] \ar[d] & \|\tA_{E}^{\leq 1}\|  \ar[d] \\
   \tA \ar[r] & \Lm(\tA,E).
  \end{tikzcd}
\]
\end{lemma}
\begin{proof}
  It is clear that this pushout defines a functor $\Lm$. To
  prove that it is a left adjoint as desired, we observe that we have
  natural equivalences
  \[
    \begin{split}
      \Map_{\CatIT}(\Lm(\tA,E), \tB) & \simeq \Map_{\CatIT}(\tA,\tB) \times_{\Map_{\CatIT}(\tA_{E}^{\leq 1}, \tB)} \Map_{\CatIT}(\|\tA_{E}^{\leq 1}\|, \tB) \\
                                     & \simeq \Map_{\CatIT}(\tA, \tB) \times_{\Map_{\CatI}(\tA_{E}^{\leq 1}, \tB^{\leq 1})} \Map_{\CatI}(\tA_{E}^{\leq 1}, \tB^{\simeq}) \\
      & \simeq \Map_{\MCatIT}((\tA,E), \tB^{\flat}),
    \end{split}
  \]
  where the last equivalence follows from the description of $\MCatIT$ as a pullback in \cref{cor:MCatIT via 1cat}.
\end{proof}

\begin{notation}\label{marked int hom}
  For marked \itcats{} $(\tA,I)$ and $(\tB,J)$, we write
  \[\MFUN((\tA,I),(\tB,J))\] for the full subcategory of
  $\FUN(\tA,\tB)$ spanned by the functors that preserve the
  markings. We give this the marking consisting of the natural transformations $\tA \times [1] \to \tB$ whose component at every $a \in \tA$ is marked in $(\tB,J)$, \ie{} this is a marked functor
  \[ (\tA,I) \times [1]^{\sharp} \to (\tB,J).\]
\end{notation}

\begin{observation}
  The \icat{} $\MCatIT$ is cartesian closed, with internal Hom given
  by $\MFUN(\blank,\blank)$. Since $\Um$ is faithful, to prove this it
  suffices to observe that for marked \itcats{} $(\tA,I), (\tB,J), (\tC, K)$, a functor
  $\tA \times \tB \to \tC$
  is a marked functor
  \[ (\tA,I) \times (\tB,J) \to (\tC,K)\] \IFF{} its adjoint
  $\tA \to \FUN(\tB, \tC)$ factors through a marked functor
  $(\tA,I) \to \MFUN((\tB,J), (\tC,K))$. We can thus regard $\MCatIT$ as enriched in itself; we can also transfer this enrichment along $\Um$ and make $\MCatIT$ an $(\infty,3)$-category with mapping \itcats{}
  $\Um \MFUN(\blank,\blank)$, \ie{} the full sub-\itcat{} of
  $\FUN(\blank,\blank)$ spanned by the marked functors; we write
  $\MCATIT$ for the underlying \itcat{}. Note that this then arises
  from the product-preserving functor $(\blank)^{\flat} \colon \CatI
  \to \MCatIT$; since we have commutative triangles
  \[
    \begin{tikzcd}
      {} & \CatI \ar[dl, hookrightarrow] \ar[dr, "(\blank)^{\flat}"] \\
      \CatIT \ar[rr, "(\blank)^{\flat}"] & & \MCatIT,
    \end{tikzcd}
    \quad
    \begin{tikzcd}
      {} & \CatI \ar[dl, "(\blank)^{\flat}"'] \ar[dr, hookrightarrow] \\
      \MCatIT \ar[rr, "\Um"] & & \MCatIT,
    \end{tikzcd}
  \]
  the functors $(\blank)^{\flat}$, $\Um$, and $(\blank)^{\sharp}$
  upgrade to 2-functors with adjunctions of \itcats{}
  \[ (\blank)^{\flat} \dashv \Um \dashv (\blank)^{\sharp}\]
  by 
  \cref{propn:upgrade adjn enr}.
\end{observation}

\subsection{Marked Gray tensors and partially (op)lax transformations}
\label{sec:marked gray}
In this subsection we review the marked version of the Gray tensor
product, where certain 2-morphisms are inverted. This has previously
been studied in a model-categorical setting in \cite{Abellan2023}, as
well as in \cite[\S 4.1]{GagnaHarpazLanariLaxLim} (but without the
marking of the Gray tensor).

\begin{defn}
Given marked
\itcats{} $(\tA,I)$ and $(\tB,J)$, we define $\tA \otimes_{I,J} \tB$
as the \itcat{} obtained by inverting the 2-morphisms in squares
$[1] \otimes [1] \xto{f \otimes g} \tA \otimes \tB$ where either $f$
lies in $I$ or $g$ lies in $J$. In other words, we have a pushout
\[
  \begin{tikzcd}
   \tA_{I} \otimes \tB \amalg \tA \otimes \tB_{J}  \ar[r] \ar[d] & \tA_{I} \times \tB \amalg \tA \times \tB_{J}  \ar[d] \\
   \tA \otimes \tB \ar[r] & \tA \otimes_{I,J} \tB.
  \end{tikzcd}
\]
This has a marking generated by the
image of $\tA_{I} \times \tB^{\simeq} \amalg \tA^{\simeq} \times
\tB_{J}$,
giving a marked \itcat{} $(\tA,I) \otimesm (\tB,J)$.
\end{defn}
\begin{observation}\label{obs:max marked gray is cart} 
  If either $I$ or $J$ is the maximal marking, then $(\tA,I) \otimesm
  (\tB,J)$ is equivalent to the cartesian product $(\tA,I) \times (\tB,J)$.
\end{observation}

\begin{defn}
  Given marked \itcats{} $(\tA,I)$ and $(\tB,J)$, we define
  $\MFUN((\tB,J), (\tC,K))^{\lax}$ to be the locally full sub-\itcat{}
  of $\FUN(\tB,\tC)^{\lax}$ whose
  \begin{itemize}
  \item objects are the marked functors, \ie{} those that take
    morphisms in $\tB_{J}$ into $\tC_{K}$,
  \item morphisms are the lax transformations $[1] \otimes \tB \to
    \tC$ that factor through a marked functor $[1] \otimes_{\flat,I}
    \tB \to \tC$, which unpacks to those where the lax naturality square
    associated to a morphism in $I$ commutes, and the restrictions to
    each object of $[1]$ gives a marked functor.
  \end{itemize}
  We equip this with the marking given by the strong natural
  transformations among marked functors (which are precisely the
  marked functors $[1] \otimes_{\sharp,I} \tB \to \tC$ by
  \cref{obs:max marked gray is cart}). Reversing the order of the Gray
  tensor we similarly define marked \itcats{}
  $\MFUN((\tB,J), (\tC,K))^{\oplax}$.
\end{defn}

\begin{lemma}
  We have natural equivalences
  \[ 
    \begin{split}
      \Map_{\MCatIT}((\tA,I) \otimesm (\tB,J), (\tC,K))
      & \simeq  \Map_{\MCatIT}((\tA,I), \MFUN((\tB,J), (\tC,K))^{\lax}) \\
      & \simeq  \Map_{\MCatIT}((\tB,J), \MFUN((\tA,I), (\tC,K))^{\oplax}), \\    
    \end{split}
  \]
  so that there are adjunctions
  \[ (\tA,I) \otimesm \blank \dashv \MFUN((\tA,I), \blank)^{\oplax},
    \quad
    \blank \otimesm (\tA,I)  \dashv \MFUN((\tA,I), \blank)^{\lax}.
  \]
\end{lemma}
\begin{proof}
  We prove the first equivalence; the second is proved by the same
  argument. Unpacking the definition of the mapping space on the
  left, we see that it can be identified with the space of functors
  $\tA \otimes \tB \to \tC$ such that
  \begin{itemize}
  \item the 2-morphism in the lax square associated to a pair of
    morphisms $f$ from $\tA$ and $g$ from $\tB$ is invertible if
    either $f$ lies in $I$ or $g$ lies in $J$,
  \item the morphism in $\tC$ associated to a morphism in $\tA$ and an
    object in $\tB$ lies in $K$ if the morphism from $\tA$ lies in
    $I$, and similarly with the roles of $\tA$ and $\tB$ reversed.
  \end{itemize}
  These correspond to functors $\tA \to \FUN(\tB,\tC)^{\lax}$ such
  that
  \begin{itemize}
  \item for every object of $\tA$, the associated functor $\tB \to
    \tC$ is marked,
  \item for every morphism in $\tA$, the associated lax natural
    transformation has commuting naturality squares at all morphisms
    in $J$,
  \item for every morphism in $I$, the associated lax natural
    transformation has commuting naturality squares at all morphisms
    in $\tB$, \ie{} it is strong.
  \end{itemize}
  These are precisely the marked functors from $(\tA,I)$ to
  $\MFUN((\tB,J), (\tC,K))^{\lax}$, as required.
\end{proof}

\begin{propn}
  The marked Gray tensor product is associative, \ie{} we have a
  natural equivalence
  \[ ((\tA, I) \otimesm (\tB, J)) \otimesm (\tC, K) \simeq
  (\tA, I) \otimesm ((\tB, J) \otimesm (\tC, K)). \]
\end{propn}
\begin{proof}
  We first prove there is an equivalence on underlying \itcats{}. For
  this, consider the commutative square
  \[
    \begin{tikzcd}
      (\tA_{I} \otimes \tB) \otimes \tC \amalg (\tA \otimes \tB_{J})
     \otimes \tC \amalg (\tA \otimes \tB) \otimes \tC_{K} \ar[r] \ar[d] & \tA_{I} \times (\tB \otimes \tC) \amalg \tB_{J} \times (\tA \otimes \tC) \amalg (\tA \otimes \tB) \times \tC_{K} \ar[d] \\
     (\tA \otimes \tB) \otimes \tC \ar[r] & (\tA \otimes_{I,J} \tB) \otimes_{L,K} \tC,
    \end{tikzcd}
  \]
  where $L$ denotes the marking of $\tA \otimes_{I,J} \tB$; we claim this is a pushout. Indeed, we can factor this horizontally through the square
  \[
    \begin{tikzcd}
      (\tA_{I} \otimes \tB) \otimes \tC \amalg (\tA \otimes \tB_{J})
     \otimes \tC \amalg (\tA \otimes \tB) \otimes \tC_{K} \ar[r] \ar[d] & (\tA_{I} \times \tB) \otimes \tC \amalg (\tA \times \tB_{J})
     \otimes \tC \amalg (\tA \otimes \tB) \otimes \tC_{K}  \ar[d]  \\
     (\tA \otimes \tB) \otimes \tC \ar[r] & (\tA \otimes_{I,J} \tB) \otimes \tC,
    \end{tikzcd}
  \]
  which is a pushout since $\otimes$ preserves colimits in each variable, followed by the square
  \[
    \begin{tikzcd}
 (\tA_{I} \times \tB) \otimes \tC \amalg (\tA \times \tB_{J})
     \otimes \tC \amalg (\tA \otimes \tB) \otimes \tC_{K}  \ar[d] \ar[r] & \tA_{I} \times (\tB \otimes \tC) \amalg \tB_{J} \times (\tA \otimes \tC) \amalg (\tA \otimes \tB) \times \tC_{K} \ar[d] \\
     (\tA \otimes_{I,J} \tB) \otimes \tC \ar[r] & (\tA \otimes_{I,J} \tB) \otimes_{L,K} \tC.
    \end{tikzcd}
  \]
  It suffices to show that this is a pushout, for which we consider the diagram

  \[
    \begin{tikzcd}
      (\tA_{I} \times \tB^{\simeq}) \otimes \tC \amalg (\tA^{\simeq} \times \tB_{J}) \otimes \tC \amalg (\tA \otimes \tB) \otimes \tC_{K} \ar[d] \ar[r] & (\tA_{I} \times \tB^{\simeq}) \times \tC \amalg (\tA^{\simeq} \times \tB_{J}) \times \tC \amalg (\tA \otimes \tB) \times \tC_{K} \ar[d]
            \\
       (\tA_{I} \times \tB) \otimes \tC \amalg (\tA \times \tB_{J})
     \otimes \tC \amalg (\tA \otimes \tB) \otimes \tC_{K}  \ar[d] \ar[r] & \tA_{I} \times (\tB \otimes \tC) \amalg \tB_{J} \times (\tA \otimes \tC) \amalg (\tA \otimes \tB) \times \tC_{K} \ar[d] \\
      (\tA \otimes_{I,J} \tB) \otimes \tC \ar[r] & (\tA \otimes_{I,J} \tB) \otimes_{L,K} \tC.
    \end{tikzcd}
  \]
  Here the top square is a pushout by \cite[Proposition 2.8.1]{AGH24},
  while the composite square is a pushout by the definition of the
  marked Gray tensor. Hence the bottom square is also a pushout. The same argument applied to $\tA \otimes_{I,M} (\tB \otimes_{J,K} \tC)$, where $M$ denotes the marking of $\tB \otimes_{J,K} \tC$, produces a pushout square
  \[
    \begin{tikzcd}
      \tA_{I} \otimes (\tB \otimes \tC) \amalg \tA \otimes (\tB_{J}
     \otimes \tC) \amalg \tA \otimes (\tB \otimes \tC_{K}) \ar[r] \ar[d] & \tA_{I} \times (\tB \otimes \tC) \amalg \tB_{J} \times (\tA \otimes \tC) \amalg (\tA \otimes \tB) \times \tC_{K} \ar[d] \\
     \tA \otimes (\tB \otimes \tC) \ar[r] & \tA \otimes_{I,M} (\tB \otimes_{L,K} \tC),
    \end{tikzcd}
  \]
  which is clearly equivalent to the first pushout via the associativity of the Gray tensor. Moreover, the marking is in both cases given by the image of \[\tA_{I} \times \tB^{\simeq} \times \tC^{\simeq} \amalg \tA^{\simeq} \times \tB_{J} \times \tC^{\simeq} \amalg \tA^{\simeq} \times \tB^{\simeq} \times \tC_{K},\] so we get the required natural equivalence of marked \itcats{}.
\end{proof}

\begin{cor}\label{marked lax hom adj}
  For marked \itcats{} $(\tA,I)$, $(\tB,J)$ and $(\tC,K)$, we
  have natural equivalences of marked \itcats{}
  \[ \MFUN((\tA,I), \MFUN((\tB,J), (\tC,K))^{\lax})^{\oplax} \simeq
    \MFUN((\tB,J), \MFUN((\tA,I), (\tC,K))^{\oplax})^{\lax},
  \]
  \[ \MFUN((\tA,I), \MFUN((\tB,J), (\tC,K))^{\lax})^{\lax} \simeq
    \MFUN((\tA,I) \otimesm (\tB,J), (\tC,K))^{\lax},
  \]
  \[  \MFUN((\tB,J), \MFUN((\tA,I), (\tC,K))^{\oplax})^{\oplax}
    \simeq
    \MFUN((\tA,I) \otimesm (\tB,J), (\tC,K))^{\oplax}.
  \]
\end{cor}
\begin{proof}
  Apply the Yoneda lemma and the associativity of the marked Gray
  tensor. (Cf.~\cite[Lemma 2.2.10]{AGH24} for the unmarked case.)
\end{proof}

\begin{defn}
Given a marked \itcat{} $(\tA,I)$ and an (unmarked) \itcat{} $\tB$, we
define
\[ \FUN(\tA, \tB)^{I\dlax} := \Um \MFUN((\tA,I), \tB^{\sharp})^{\lax},\]
\[ \FUN(\tA, \tB)^{I\doplax} := \Um \MFUN((\tA,I), \tB^{\sharp})^{\oplax},\]
Unpacking the various adjunctions, we see that functors $\tK \to \FUN(\tA,\tB)^{I\dplax}$ correspond to functors
\[ \tK \otimes_{\flat,I} \tA \to \tB, \quad \tA \otimes_{I,\flat} \tK \to \tB\]
in the lax and oplax cases, respectively. Here $\FUN(\tA,
\tB)^{I\dplax}$ can be identified as a wide and locally full
sub-\itcat{} of $\FUN(\tA,\tB)^{\plax}$ whose morphisms are the
\emph{$I$-(op)lax transformations}, meaning those (op)lax
transformations
\[ [1] \otimes \tA \to \tB, \quad \tA \otimes [1] \to \tB \] that
factor through $[1] \otimes_{\flat,I} \tA$ and
$\tA \otimes_{I,\flat} [1]$, respectively, meaning that the (op)lax
naturality squares associated to the morphisms in $I$ actually
commute.
\end{defn}

\begin{notation}
  Given functors of \itcats{} $F,G \colon \tA \to \tB$ and a marking $(\tA,I)$, we write
  $\Nat^{I\dplax}_{\tA,\tB}(F,G)$ for the
  \icat{} of morphisms from $F$ to $G$ in $\FUN(\tA,\tB)^{I\dplax}$.
\end{notation}

\begin{observation}\label{lax oplax reverse marked}
  Applying \cref{marked lax hom adj} with 
  $(\tA,I) = \tA^{\flat}$ and $(\tC,K) = \tC^{\sharp}$, we see that
  there are natural equivalences
  \[ \FUN(\tA, \FUN(\tB, \tC)^{J\dlax})^{\lax} \simeq 
    \FUN(\tA \otimes_{\flat,J} \tB, \tC)^{J'\dlax},\]
  \[ \FUN(\tA, \FUN(\tB, \tC)^{J\doplax})^{\oplax} \simeq \FUN(\tB
    \otimes_{J,\flat} \tA, \tC)^{J''\doplax},\] where $J'$ and $J''$
  are both generated by the image of $\tB_{J} \times \tA^{\simeq}$. This means we have a pullback square
  \[
    \begin{tikzcd}
     \FUN(\tA, \FUN(\tB, \tC)^{J\dlax})^{\lax} \ar[r] \ar[d] & \FUN(\tA \otimes_{\flat,J} \tB, \tC)^{\lax}  \ar[d] \\
     \lim_{\tA^{\simeq}} \FUN(\tB, \tC)^{J\dlax} \ar[r] & \lim_{\tA^{\simeq}} \FUN(\tB, \tC)^{\lax},
    \end{tikzcd}
  \]
  and similarly in the oplax case.  We also see that
  $\FUN(\tA, \FUN(\tB, \tC)^{J\dlax})^{\oplax}$ is equivalent to the
  full sub-\itcat{} of $\FUN(\tB, \FUN(\tA, \tC)^{\oplax})^{J\dlax}$
  spanned by the functors that take morphisms in $J$ to strong natural
  transformations.
\end{observation}

\subsection{Fibrations and the straightening
  equivalence}\label{sec:review fib}

In this section we review the \itcatl{} analogues of (co)cartesian
fibrations, first introduced in \cite{GHLFib}, and their straightening
equivalence with functors $\CATIT$, due to
Abell\'an--Stern~\cite{ASI23} and Nuiten~\cite{Nuiten}. We also recall
the extension to straightening for partially (op)lax transformations
proved in \cite{AGH24}.

\begin{defn}
  Let $p \colon \tE \to \tB$ be a functor of \itcats{}.
  \begin{itemize}
  \item A morphism
  $f \colon x \to y$ in $\tE$ is \emph{$p$-cocartesian} (or \emph{$p$-$0$-cartesian}) if for all objects
  $z \in \tE$ the commutative square of
  \icats{}
  \[
    \begin{tikzcd}
     \tE(y,z) \ar[r, "f^{*}"] \ar[d] & \tE(x,z)  \ar[d] \\
     \tB(p(y),p(z)) \ar[r, "p(f)^{*}"'] & \tB(p(x),p(z))
    \end{tikzcd}
  \]
  is a pullback. Dually, we say $f$ is \emph{$p$-cartesian} (or \emph{$p$-$1$-cartesian}) if
  it is $p^{\op}$-cocartesian when viewed as a morphism in
  $\tE^{\op}$.
\item A 2-morphism $\alpha$ in $\tE$ between 1-morphisms from
  $x$ to $y$ is \emph{weakly $p$-(co)cartesian} (or
  \emph{weakly $p$-$i$-cartesian}) if it is a (co)cartesian morphism for the
  functor
  \[ p_{x,y} \colon \tE(x,y) \to \tB(p(x),p(y)).\]
\item A 2-morphism $\alpha$ in $\tE$ between 1-morphisms from
  $x$ to $y$ is \emph{$p$-(co)cartesian} (or
  \emph{$p$-$i$-cartesian}) if for any morphisms $f \colon x' \to x$
  and $g \colon y \to y'$
  the whiskering $g \circ \alpha \circ f$ is weakly $p$-(co)cartesian.
  \end{itemize}
\end{defn}

\begin{lemma}\label{lem:cart 2mor detect after compose cocart mor}
  Suppose $f \colon x \to y$ is a $p$-cocartesian morphism for a
  functor of \itcats{} $p \colon \tE \to \tB$. Let $u \colon h \to h'$
  be a 2-morphism between morphisms $y \to z$. If $u \circ f$ is
  weakly $p$-(co)cartesian then so is $u$.
\end{lemma}
\begin{proof}
  Since $f$ is $p$-cocartesian, we have a pullback square of \icats{}
  \[
    \begin{tikzcd}
     \tE(y,z) \ar[r, "f^{*}"] \ar[d, "p_{y,z}"'] & \tE(x,z)  \ar[d, "p_{x,z}"] \\
     \tB(py,pz) \ar[r, "p(f)^{*}"'] & \tB(px,pz)
    \end{tikzcd}
  \]
  and by assumption $f^{*}u$ is $p_{x,z}$-(co)cartesian. This implies
  that $u$ is
  $p_{y,z}$-(co)cartesian by \cite[2.4.1.3(2)]{LurieHTT}.
\end{proof}

\begin{observation}\label{obs:cart lifts unique}
  Let $p \colon \tE \to \tB$ be a functor of \itcats{}. Then
  $p$-(co)cartesian lifts are unique when they exist.
  More precisely, if we let $\Map_{i\dcart}([1],
  \tE)$ denote the subspace of $\Map([1], \tE)$ spanned by the
  $i$-cartesian morphisms, then the functor
  \[ \Map_{i\dcart}([1], \tE) \to \Map([1], \tB) \times_{\tB^{\simeq}}
    \tE^{\simeq}, \]
  induced by the commutative square
  \[
    \begin{tikzcd}
     \Map([1], \tE) \ar[r, "\ev_{i}"] \ar[d, "{\Map([1], p)}"'] & \tE^{\simeq}  \ar[d, "p^{\simeq}"] \\
     \Map([1], \tB) \ar[r, "\ev_{i}"'] & \tB^{\simeq},
    \end{tikzcd}
  \]
  is a monomorphism, \ie{} its fibres are either empty or
  contractible. To see this, we can observe that if there exists a
  cartesian lift $\bar{f} \colon x \to y$ of some morphism $f \colon a
  \to b$ in
  $\tB$, then composition with $\bar{f}$ identifies the \igpd{} of
  cartesian lifts of $f$ with target $y$ with that of cocartesian
  lifts of $\id_{b}$ with target $y$; the latter is the \igpd{} of
  equivalences with target $y$, which is always contractible.
\end{observation}

\begin{defn}
  Given a functor of \itcats{} $p \colon \tE \to \tB$, we say that
  $\tE$ \emph{has $p$-(co)cartesian lifts} of a morphism $f \colon b
  \to b'$ in $\tB$ if for all $x \in \tE, p(x) \simeq b$, there exists
  a $p$-cocartesian morphism $\bar{f} \colon x \to y$ with an equivalence
  $p(\bar{f}) \simeq f$ extending that for the source. Similarly, we
  say that $\tE$ \emph{has $p$-(co)cartesian lifts} of a 2-morphism
  if the induced functor on mapping \icats{} has $p$-(co)cartesian
  lifts in the previous sense, and these are preserved under pre- and postcomposition with 1-morphisms.
\end{defn}

\begin{defn}
  A functor $\pi \colon \tE \to \tB$ of \itcats{} is a \emph{$(0,1)$-fibration} if:
  \begin{enumerate}
  \item $\tE$ has $p$-cocartesian lifts of all morphisms in $\tB$.
  \item $\tE$ has $p$-cartesian lifts of all 2-morphisms in $\tB$.
  \end{enumerate}
  If $\pi \colon \tE \to \tB$ and $\pi' \colon \tE' \to \tB'$ are $(0,1)$-fibrations, we say that a commutative square of \itcats{}
  \[
    \begin{tikzcd}
      \tE \arrow{r}{\psi} \arrow{d}[swap]{\pi} & \tE' \arrow{d}{\pi'} \\
      \tB \arrow{r}{\phi} & \tB'
    \end{tikzcd}
  \]
  is a \emph{morphism of $(0,1)$-fibrations} if $\psi$ preserves $\pi$-cocartesian 1-morphisms and $\pi$-cartesian 2-morphisms.
  
  Similarly, for all choices of $i,j \in \{0,1\}$ we have the notion
  of $(i,j)$-fibrations (which have $i$-cartesian 1-morphisms and
  $j$-cartesian 2-morphisms) and their morphisms. We write
  $\FIB^{(i,j)}_{/\tB}$ for the locally full sub-\itcat{}
  (cf. \cref{defn:subcategories}) of $\CATITsl{\tB}$ containing the
  $(i,j)$-fibrations, the morphisms of $(i,j)$-fibrations, and all
  2-morphisms among these.
\end{defn}

\begin{defn}
  An $\ve$-fibration $p \colon \tE \to \tB$ is \emph{1-fibred} if its
  fibres are \icats{}. (See \cite[Proposition 3.2.14]{AGH24} for an
  alternative characterization of these.)
\end{defn}

\begin{observation}\label{obs:fib change variance}
  For a functor $p \colon \tE \to \tB$, the following are equivalent:
  \begin{itemize}
  \item $p$ is a $(1,0)$-fibration.
  \item $p^{\op}$ is a $(0,0)$-fibration.
  \item $p^{\co}$ is a $(1,1)$-fibration.
  \item $p^{\coop}$ is a $(0,1)$-fibration.
  \end{itemize}
\end{observation}

\begin{notation}\label{not:vepsilon}
  To lighten the notation, we will usually write $\vepsilon$-fibration for
  $\vepsilon \in \{0,1\}^{\times 2}$, and we will
  similarly denote $\FIB^{(i,j)}_{/\tB}$ as
  $\FIB^{\vepsilon}_{/\tB}$. Given $\vepsilon=(i,j)$ we will also write
  $\overline{\vepsilon}=(1-i,1-j)$ for the conjugate
  variance. This allows us to use short-hand notation for various
  structures associated with fibrations; in particular, we define
  \[ \tB^{\veop} :=
    \begin{cases}
      \tB, & \ve = (0,1),\\
      \tB^{\op}, & \ve = (1,0), \\
      \tB^{\co}, & \ve = (0,0), \\
      \tB^{\coop}, & \ve = (1,1).
    \end{cases}
  \]
\end{notation}

 \begin{theorem}[Nuiten~\cite{Nuiten}, Abell\'an--Stern~\cite{ASI23}]\label{thm:str}
   There are equivalences of \itcats{}
   \[ \FIB^{\ve}_{/\tB} \simeq \FUN(\tB^{\veop}, \CATIT),\]
   \ie{}   
    \begin{align*}
      \FIB^{(0,1)}_{/\tB} \xrightarrow{\simeq} \FUN(\tB, \CATIT), & & \FIB^{(1,0)}_{/\tB} \xrightarrow{\simeq} \FUN(\tB^\op, \CATIT),
   \end{align*}
    \begin{align*}
      \FIB^{(0,0)}_{/\tB} \xrightarrow{\simeq} \FUN(\tB^{\co}, \CATIT), & & \FIB^{(1,1)}_{/\tB} \xrightarrow{\simeq} \FUN(\tB^\coop, \CATIT),
  \end{align*}
  which are contravariantly natural in $\tB$ with respect to pullback
  on the left and composition on the right. \qed
 \end{theorem}

 \begin{theorem}[Abell\'an--Stern~\cite{ASI23}]\label{thm:fibsexponentiable}
   Let $p \colon \tX \to \tS$ be an $(i,j)$-fibration. Then the pullback functor
  \[
      p^* \colon \CATITsl{\tS} \to \CATITsl{\tX}
   \]
   admits a right adjoint.
\end{theorem}
\begin{proof}
  By \cite[Theorem 3.90]{ASI23}, the functor $p^*$ admits a right
  adjoint at the level of underlying \icats{}. The \itcat{} structure
  on $\CATITsl{\tS}$ arises from the product functor
  \[ (\blank) \times \tS \colon \CatI \to \CatITsl{\tS},\]
  which is preserved by $p^{*}$ as $p^{*}(\oK \times \tS) \simeq \oK
  \times \tX$. It therefore follows from \cref{propn:upgrade adjn enr}
  that the adjunction $p^{*} \dashv p_{*}$ upgrades to an adjunction
  of \itcats{}.
\end{proof}

\begin{defn}\label{def:fibmarked}
  Let $(\tC, E)$ be a marked \itcat{}. For $\ve = (i,j)$ we write $\FIB_{/(\tC,E)}^{\vepsilon}$ for the locally full sub-\itcat{} of $\CATITsl{\tC}$ whose objects are $\ve$-fibrations over $\tC$ and whose morphisms preserve $i$-cartesian morphisms that lie over $E$ as well as all $j$-cartesian 2-morphisms.
\end{defn}

\begin{theorem}[Abell\'an--Gagna--Haugseng, {\cite[Proposition 3.5.7]{AGH24}}]\label{thm:elaxstr}
  For $(\tC,E)$ a marked \itcat{}, we have equivalences of
  \itcats{}
  \[
    \FIB^{(0,1)}_{/(\tC,E)}\simeq \FUN(\tC, \CATIT)^{\elax}, \enspace \enspace \FIB^{(0,0)}_{/(\tC,E)}   \simeq \FUN(\tC^{\co},\CATIT)^{\elax}
  \]
  \[
    \FIB^{(1,1)}_{/(\tC,E)}  \simeq \FUN(\tC^{\coop},\CATIT)^{\eoplax}, \enspace \enspace \FIB^{(1,0)}_{/(\tC,E)}   \simeq \FUN(\tC^{\op},\CATIT)^{\eoplax},
  \]
  given on objects by straightening. These equivalences are all contravariantly
  natural in $(\tC,E)$ with respect to pullback on the left and
  composition on the right. \qed
\end{theorem}

\subsection{Lax slices and cones}
\label{sec:lax slice}

In this subsection we review the lax versions of slices and cones, and
observe that they are related in the expected way. The analogous
constructions in the setting of scaled simplicial sets were previously
studied by Gagna--Harpaz--Lanari in \cite[\S 5.2]{GagnaHarpazLanariLaxLim}.

\begin{definition}\label{def:laxslices}
  Let $\bCC$ be a \itcat{} and let $c \in \bCC$. We define lax versions of the slice construction by means of the following pullback squares
  \[
    \begin{tikzcd}
      \bCC_{c \upslash} \arrow[r] \arrow[d] & \ARplax(\tC) \arrow[d,"\ev_0"] \\
      {[0]} \arrow[r,"c"] & \bCC{,}
    \end{tikzcd} \quad \quad \quad \begin{tikzcd}
      \bCC_{c \downslash} \arrow[r] \arrow[d] & \ARlax(\tC) \arrow[d,"\ev_0"] \\
      {[0]} \arrow[r,"c"] & \bCC{,}
    \end{tikzcd}
  \]
  \[
    \begin{tikzcd}
      \bCC_{ \upslash c} \arrow[r] \arrow[d] & \ARplax(\tC) \arrow[d,"\ev_1"] \\
      {[0]} \arrow[r,"c"] & \bCC{,}
    \end{tikzcd} \quad \quad \quad \begin{tikzcd}
      \bCC_{ \downslash c} \arrow[r] \arrow[d] & \ARlax(\tC) \arrow[d,"\ev_1"] \\
      {[0]} \arrow[r,"c"] & \bCC{.}
    \end{tikzcd}
  \]
  The lax slices come equipped with:
  \begin{itemize}
    \item a functor $\bCC_{c\upslash } \to \bCC$ induced by $\ev_1$ which is a $(0,1)$-fibration, which classifies the corepresentable functor $\bCC(c,-) \colon \bCC \to \CATI$;
    \item a functor $\bCC_{c\downslash } \to \bCC$ induced by $\ev_1$ which is a $(0,0)$-fibration, which classifies the corepresentable functor $\bCC(c,-)^{\op} \colon \bCC^\co \to \CATI$;
    \item a functor $\bCC_{\upslash c } \to \bCC$ induced by $\ev_0$ which is a $(1,0)$-fibration, which classifies the representable functor $\bCC(-,c) \colon \bCC^\op \to \CATI$;
    \item a functor $\bCC_{\downslash c} \to \bCC$ induced by $\ev_0$ which is a $(1,1)$-fibration, which classifies the representable functor $\bCC(-,c)^{\op} \colon \bCC^{\coop} \to \CATI$.
    \end{itemize}
    The first claim is \cite[Proposition 4.1.8]{LurieGoodwillie}, and
    the others follow by reversing 1- and 2-morphisms.
\end{definition}

\begin{remark}
  The notation chosen for the lax slices is designed to give a description of the 1-morphisms. More precisely, an object in $\bCC_{\upslash c}$ is given by a map $u \colon  x \to c$ and a morphism from $u$ to $v \colon y \to c$ can be represented by triangle
  \[\begin{tikzcd}[column sep=small]
  x && y \\
  & c
  \arrow[from=1-1, to=1-3]
  \arrow[""{name=0, anchor=center, inner sep=0}, "u"', from=1-1, to=2-2]
  \arrow["v", from=1-3, to=2-2]
  \arrow[shorten <=8pt, shorten >=16pt, Rightarrow, from=0, to=1-3].
 \end{tikzcd}\]
 which commutes up to a (non-invertible) 2-morphism. Similarly, a 1-morphism in $\bCC_{\downslash c}$ is given by a laxly commuting triangle where the associated 2-cell points in the other direction.
\end{remark}

\begin{defn}
  Given a marked \itcat{} $(\tJ, E)$ and a functor $F \colon \tJ \to \tC$, we define the \itcat{} $\tC_{\downslash F}^{\elax}$ of \emph{$E$-lax cones} on $F$ as the pullback
  \[
    \begin{tikzcd}
     \tC_{\downslash F}^{\elax} \ar[r] \ar[d] & \FUN(\tJ,\tC)^{\elax}_{\downslash F}  \ar[d] \\
     \tC \ar[r, "\const"'] & \FUN(\tJ, \tC)^{\elax};
    \end{tikzcd}
  \]
  an $E$-lax cone on $F$ is thus by definition an $E$-lax natural transformation $\underline{c} \to F$ from a constant functor. Similarly, we define \itcats{} $\tC_{\upslash F}^{\eoplax}$ of \emph{$E$-oplax cones}, $\tC_{\downslash F}^{\elax}$ of \emph{$E$-lax cocones}, and
  $\tC_{\upslash F}^{\eoplax}$ of \emph{$E$-oplax cocones}.
\end{defn}

\begin{observation}\label{obs:straighten lax cones}
  Since pullback of fibrations corresponds to composition under
  straightening, the $(1,1)$-fibration $\tC_{\downslash F}^{\elax} \to
  \tC$ corresponds to the
  functor \[\Nat^{\elax}_{\tJ,\tC}(\underline{\blank}, F)^{\op}\colon
    \tC^{\coop} \to \CATI.\] Similarly,
  \begin{itemize}
  \item $\tC_{\upslash F}^{\eoplax}$ corresponds to $\Nat^{\eoplax}_{\tJ,\tC}(\underline{\blank}, F)$,
  \item $\tC_{\downslash F}^{\elax}$ corresponds to $\Nat^{\elax}_{\tJ,\tC}(F, \underline{\blank})^{\op}$,
  \item $\tC_{\upslash F}^{\eoplax}$ corresponds to $\Nat^{\eoplax}_{\tJ,\tC}(F, \underline{\blank})$.
  \end{itemize}
\end{observation}

We can also describe the \itcats{} of partially lax (co)cones in terms of joins:
\begin{defn}\label{def:elaxcones}
  Let $(\tJ,E)$ be a marked \itcat{}. We define lax versions of the cone construction by means of the following pushout diagrams:
  \[
     \begin{tikzcd}
        \{0\}\times \tJ \arrow[r] \arrow[d] & {[0]} \arrow[d] \\
       {[1]}\otimes_{\flat,E} \tJ \arrow[r] & \tJ_{E\dlax}^{\triangleleft}
     \end{tikzcd}, \enspace \enspace \begin{tikzcd}
       \{1\}\times \tJ
       \arrow[r] \arrow[d] & {[0]} \arrow[d] \\
       {[1]}\otimes_{\flat,E} \tJ \arrow[r] & \tJ_{E\dlax}^{\triangleright}.
     \end{tikzcd}
  \] 
  Similarly, we define $\tJ^{\triangleleft}_{\eoplax}$ and
  $\tJ^{\triangleright}_{\eplax}$ by reversing the order of the Gray
  tensor product above. 
\end{defn}

\begin{lemma}\label{lem:lax slice cone}
  For $(\tJ,E)$ a marked \itcat{} and a functor $F \colon \tJ \to \tC$, there is a pullback square
  \[
    \begin{tikzcd}
      \tC_{\downslash F}^{\elax} \ar[r] \ar[d] &  \FUN(\tJ_{E\dlax}^{\triangleleft}, \tC)^{\lax} \ar[d] \\
      \{F\}\ar[r] & \FUN(\tJ, \tC)^{\lax},
    \end{tikzcd}
  \]
  and similarly for the other three variants.
\end{lemma}
\begin{proof}
  By definition of the lax slice, we have a pullback
  \[
    \begin{tikzcd}
     \tC_{\downslash F}^{\elax} \ar[r] \ar[d] & \ARlax(\FUN(\tJ, \tC)^{\elax})  \ar[d] \\
     \tC \times [0] \ar[r, "\const \times F"'] & (\FUN(\tJ,\tC)^{\elax})^{\times 2}.
    \end{tikzcd}
  \]
  By \cref{lax oplax reverse marked}, we also have a pullback
  \[
    \begin{tikzcd}
      \ARlax(\FUN(\tJ,\tC)^{\elax}) \ar[r] \ar[d] & \FUN([1] \otimes_{\flat,E} \tJ, \tC)^{\lax}  \ar[d] \\
      (\FUN(\tJ,\tC)^{\elax})^{\times 2}\ar[r] & (\FUN(\tJ,\tC)^{\lax})^{\times 2}.
    \end{tikzcd}
  \]
  Combining the two, we get a pullback square
  \[
    \begin{tikzcd}
     \tC_{\downslash F}^{\elax} \ar[r] \ar[d] &   \ar[d] \FUN([1] \otimes_{\flat,E} \tJ, \tC)^{\lax} \\
     \tC \times [0] \ar[r, "\const \times F"'] & (\FUN(\tJ,\tC)^{\lax})^{\times 2}.
    \end{tikzcd}
  \]
  We can factor the bottom horizontal map as
  \[\tC \times [0] \xto{\id \times F} \tC \times \FUN(\tJ,\tC)^{\lax} \xto{\const \times \id} (\FUN(\tJ,\tC)^{\lax})^{\times 2}.\]
  By definition of $\tJ^{\triangleleft}_{\elax}$, this gives a horizontal factorization of our square through the pullback
 \[
    \begin{tikzcd}
     \FUN(\tJ^{\triangleleft}_{\elax}, \tC)^{\elax} \ar[r] \ar[d] &   \ar[d] \FUN([1] \otimes_{\flat,E} \tJ, \tC)^{\lax} \\
     \tC \times [0] \ar[r, "\const \times F"'] & (\FUN(\tJ,\tC)^{\lax})^{\times 2},
    \end{tikzcd}
  \]
  from which the pullback we want to prove is clear.
\end{proof}

\begin{lemma}\label{lem:inclusionconeff}
  Let $p \colon \tA \to [1]$ and consider the functor $\tA \to  \tA^{\prime}=\tA \coprod_{\tA_1} [0]$ over $[1]$, where $\tA_{1}$ denotes the fibre over $1$. Then we have an equivalence of \itcats{}, 
  \[
    \tA_0 \xrightarrow{\simeq} \tA^{\prime}_0.
  \]
\end{lemma}
\begin{proof}
  The map $\tA_0 \to \tA^{\prime}_0$ is clearly essentially
  surjective, so we only need to show that it is fully faithful. To
  see this we will compute $\tA^{\prime}$ in a model of \itcats{}. We
  start by taking $\mathfrak{A} \to \Delta^{1}$ to be a model of
  $\tA \to [1]$ as a fibration in scaled simplicial sets; then the
  fibres $\mathfrak{A}_{i}$ will model $\tA_{i}$ ($i = 0,1$) as these
  are given by homotopy pullbacks. Next, we take $\mathcal{A} :=
  \mathfrak{C}^{\mathrm{sc}}(\mathfrak{A})$ to be the cofibrant model
  for $\tA$ as a category enriched in
  marked simplicial sets (or $\Set_\Delta^{+}$-enriched
  categories) obtained from the left Quillen equivalence
  $\mathfrak{C}^{\mathrm{sc}}$ (see \cite{LurieGoodwillie}). We define
  $\mathcal{A}_{i}$ for $i = 0,1$ similarly; then the resulting
  functors $\mathcal{A}_{i} \to \mathcal{A}$ are cofibrations (since
  $\mathfrak{C}^{\mathrm{sc}}$ is a left Quillen functor), and also
  full subcategory inclusions by the definition of this functor.

  We now define a
  $\Set_\Delta^{+}$-enriched category $\widehat{\mathcal{A}}$ as
  follows:
  \begin{itemize}
    \item The set of objects is given by those of $\mathcal{A}_0$ in addition to a ``cone point'', which we denote by $v$.
    \item The marked simplicial sets of maps are given by
      $\widehat{\mathcal{A}}(a,b)=\mathcal{A}_0(a,b)$, $\widehat{\mathcal{A}}(v,v)=*$,
      $\widehat{\mathcal{A}}(v,b)=\emptyset$, and finally $\widehat{\mathcal{A}}(a,v)$
      is defined to be the (conical)
      $\Set_\Delta^{+}$-enriched colimit of the functor
      $\mathcal{A}(a,\iota(-)) \colon \mathcal{A}_1 \to \Set_\Delta^{+}$
      where $\iota \colon \mathcal{A}_1 \to \mathcal{A}$ denotes the obvious
      inclusion.
  \end{itemize}
  This datum assembles naturally into a
  $\Set_\Delta^{+}$-enriched category that fits  into a commutative diagram 
  \[
    \begin{tikzcd}
      \mathcal{A}_1 \arrow[d] \arrow[r] & {[0]} \arrow[d] \\    
      \mathcal{A} \arrow[r,"p"] & \widehat{\mathcal{A}}.
    \end{tikzcd}
  \]
  To finish the proof we will show that this is a pushout square; as
  the left vertical morphism is a cofibration between cofibrant
  objects and $[0]$ is also cofibrant, this implies it is a homotopy
  pushout and so models a pushout in the \icat{} of \itcats{}. In
  order to prove this, suppose that we have a functor $f \colon \mathcal{A} \to \mathcal{X}$ such that its restriction to $\mathcal{A}_{1}$ is constant on an object $x$. We define a functor $\hat{f} \colon \widehat{\mathcal{A}} \to \mathcal{X}$ such that $f= \hat{f}\circ p$. To define $\hat{f}$ it will be enough to define for every $a \in \mathcal{A}_0$ a map
  \[
    \varphi_a \colon \widehat{\mathcal{A}}(p(a),v) \to \mathcal{X}(f(a),x)
  \]
  and show that the all of these choices assemble into a functor. Note
  that by construction we have a natural transformation of functors
  $\mathcal{A}(a,\iota(-)) \to \mathcal{X}(p(a),x)$ where we view the latter functor
  as being constant. The universal property of the colimit provides us
  with the desired functor $\varphi_a$. It is immediate to verify that
  the choices above assemble to yield a functor $\hat{f}$ as
  desired. It is clear that $\hat{f}$ is the unique such extension, so
  this concludes the proof.
\end{proof}

\begin{propn}\label{prop:inclusionconeff}
  For any marked \itcat{} $(\tI, E)$, the canonical functors
  \[\tI \to \tI^{\colimcone}_{\elaxoplax}, \quad \tI \to
    \tI^{\limcone}_{\elaxoplax}\]
  induced by the inclusions $\{0\},\{1\} \to [1]$, respectively, are
  all fully faithful.
\end{propn}
\begin{proof}
  We will only prove the case $\{0\} \to [1]$; the remaining case is
  completely analogous. We first show that the map $\tI \to [1]
  \otimes_{\flat,E} \tI$ is fully faithful. To see this we observe that the composite
  \[
    \tI \to [1] \otimes_{\flat,E} \tI \xrightarrow{\pi} [1] \times \tI
  \]
  is fully faithful, so it will suffice to show that $\pi$ induces an
  equivalence of mapping \icats{} at the objects in the essential image of the first functor.

  To see this, we can implement this map using scaled simplicial sets
  (\cite{LurieGoodwillie}) and verify the previous claim after applying the rigidification functor $\mathfrak{C}^{\mathrm{sc}}$, cf. \cite[Definition 3.1.10, Theorem
  4.2.2.]{LurieGoodwillie}. Note that this is immediate, as the resulting functor between $\Set_\Delta^{+}$-enriched categories sets is an isomorphism on the underlying simplicially enriched categories and the decorations on the corresponding mapping simplicial sets agree.

  To finish the proof, we apply \cref{lem:inclusionconeff} to the map $[1] \otimes_{\flat,E} \tI \to [1]$.
\end{proof}

\begin{observation}\label{rem:conestraightening}
  Let $* \in \tI^{\colimcone}_{\elaxoplax}$ denote the cone point and consider the following pullback square
  \[
    \begin{tikzcd}
      \tI(*)=\left(\tI^{\colimcone}_{\elaxoplax}\right)_{ \upslash j}\times_{\tI^{\colimcone}_{\elaxoplax}}\tI  \arrow[d] \arrow[r] & \left(\tI^{\colimcone}_{\elaxoplax}\right)_{ \upslash *} \arrow[d] \\
      \tI \arrow[r] &\tI^{\colimcone}_{\elaxoplax}
    \end{tikzcd}
  \]
  We can use the description of the straightening functor \cite[Definition 3.5.1]{LurieGoodwillie}  to see that $\tI(*) \to \tI$ is the 1-fibred $(1,0)$-fibration corresponding to straightening the identity functor on $\tI$, where we equip the source with the marking given by the edges corresponding to $E$. Putting this all together it follows that for every 1-fibred $(1,0)$-fibration $p \colon \tX \to \tI$ we can identify $\FIB^{(1,0)}_{/\tI}(\tI(*),\tX)$ with the full sub-\itcat{} of $\CATITsl{\tI}(\tI,\tX)$ on those functors over $\tI$ which send the 1-morphisms in $E$ to cartesian morphisms in $\tX$.
\end{observation}

\subsection{Decorated \itcats{}}
\label{sec:dec defn}

In this subsection we introduce the notion of \emph{decorated
  \itcats{}}, which will provide a convenient framework for our work
on fibrations in the subsequent sections. By a decorated \itcat{} we
mean an \itcat{} equipped with collections of (``decorated'') 1- and
2-morphisms (which are chosen independently, \ie{} the decorated
2-morphisms \emph{do not} have to go between decorated 1-morphisms).
We can further assume that these
collections are closed under composition and include all equivalences,
so that we can formally define these objects as follows:
\begin{defn}
  A \emph{decorated \itcat{}} $\tA^{\diamond}$ is a span of \itcats{}
  \[ \tA^{\diamond}_{(1)} \xto{i_{1}} \tA \xfrom{i_{2}} \tA^{\diamond}_{(2)}\]
  such that (see \cref{defn:subcategories})
  \begin{itemize}
  \item $i_{1}$ is a wide locally full sub-\itcat{},
  \item $i_{2}$ is a wide and
  locally wide sub-\itcat{} inclusion.
\end{itemize}
We define the \icat{} $\DCatIT$ as the full subcategory of
$\Fun(\Lambda^{2,\op}_{0}, \CatIT)$ spanned by the decorated
\itcats{}, where $\Lambda^{2,\op}_{0}$ denotes the category
\[ 1 \to 0 \from 2.\]
\end{defn}

\begin{observation}\label{obs:adjointDcatIT}
  It follows easily from \cite[Observations 2.5.6, 2.5.8]{AGH24} that a functor of \itcats{} is
  \begin{itemize}
  \item a wide locally full subcategory inclusion \IFF{} it is right
    orthogonal to $\emptyset \to *$, $\partial C_{2} \to C_{2}$,
  \item a wide and locally wide subcategory inclusion
    \IFF{} it is right
    orthogonal to $\emptyset \to *$, $\partial [1] \to [1]$, $\partial
    C_{3} \to C_{2}$.
  \end{itemize}
  In particular, by \cite[Proposition 5.5.5.7]{LurieHTT} both of these
  types of functors form the right class in a factorization system on
  $\CatIT$. From this it is easy to see that the inclusion
  $\DCatIT \hookrightarrow \Fun(\Lambda^{2,\op}_{0}, \CatIT)$ has a
  left adjoint, given by factoring a span of \itcats{} using these
  factorization systems. Moreover, as the morphisms we consider
  orthogonality for all go between compact \itcats{}, it is clear that
  $\DCatIT$ is closed under filtered colimits in
  $\Fun(\Lambda^{2,\op}_{0}, \CatIT)$, so that this left adjoint
  exhibits $\DCatIT$ as an accessible localization. In particular,
  $\DCatIT$ is a presentable \icat{}. Moreover, we can describe
  (co)limits of decorated \itcats{} as follows:
  \begin{itemize}
  \item the limit of a diagram in $\DCatIT$ is computed in
    $\Fun(\Lambda^{2,\op}_{0}, \CatIT)$, and so is given by the limit
    of the underlying \itcats{} equipped with the limits of the
    subcategories of decorations;
  \item the colimit of a diagram in $\DCatIT$ is given by factoring
    the colimit in 
    $\Fun(\Lambda^{2,\op}_{0}, \CatIT)$, and so is given by the colimit
    of the underlying \itcats{} equipped with the decorations
    generated by the images of those in the diagram.
  \end{itemize}
\end{observation}

\begin{observation}
  It is easy to see that the functor $\ev_{0} \colon
  \Fun(\Lambda^{2,\op}_{0}, \CatIT) \to \CatIT$ has the following
  properties:
  \begin{enumerate}[(1)]
  \item $\ev_{0}$ has a left adjoint, which takes an \itcat{} $\tA$ to
    the span
    \[ \emptyset \to \tA \from \emptyset.\]
  \item $\ev_{0}$ has a right adjoint $\const$, which takes an \itcat{} $\tA$ to
    the span
    \[ \tA \xto{=} \tA \xfrom{=} \tA.\]
  \item The right adjoint $\const$ has a further right adjoint, which
    takes a span
    \[ \tA \to \tB \from \tC \]
    to its limit $\tA \times_{\tB} \tC$.
  \item $\ev_{0}$ is a cocartesian fibration; the cocartesian
    transport of a span
    \[ \tA \xto{F} \tB \xfrom{G} \tC \]
    along $\phi \colon \tB \to \tB'$ is given by composition with
    $\phi$, giving
    \[ \tA \xto{\phi F} \tB' \xfrom{\phi G} \tC. \]
  \item $\ev_{0}$ is a cartesian fibration; the cartesian
    transport of a span
    \[ \tA \to \tB \from \tC \]
    along $\psi \colon \tB' \to \tB$ is given by pullback along $\psi$, giving
    \[ \tA \times_{\tB} \tB' \to \tB' \from \tC \times_{\tB} \tB'. \]
  \end{enumerate}
  From this we can conclude that the restriction of $\ev_{0}$ to
  \[ \Ud \colon \DCatIT \to \CatIT \]
  has the following corresponding properties:
    \begin{enumerate}[(1')]
  \item $\Ud$ has a left adjoint $(\blank)^{\flat\flat}$, which takes an \itcat{} $\tA$ to
    the span
    \[ \tA^{\flat\flat} \,\,:=\,\, \tA_{\mathrm{eq}} \to \tA \from \tA^{\leq 1},\]
    where $\tA_{\mathrm{eq}}$ is the subcategory that contains only
    the invertible 1-morphisms, but all 2-morphisms among these.
  \item $\Ud$ has a right adjoint $(\blank)^{\sharp\sharp}$, which takes an \itcat{} $\tA$ to
    the span
    \[ \tA^{\sharp\sharp} \,\,:=\,\, \tA \xto{=} \tA \xfrom{=} \tA.\]
  \item The right adjoint $(\blank)^{\sharp\sharp}$ has a further
    right adjoint $D$, which takes a decorated \itcat{}
    $\tA^{\diamond}$ to the pullback
    \[ D(\tA^{\diamond}) \,\,:=\,\, \tA^{\diamond}_{(1)} \times_{\tA}
      \tA^{\diamond}_{(2)},\]
    \ie{} the sub-\itcat{} of $\tA$ that contains the decorated
    1-morphisms and the decorated 2-morphisms among these.
  \item $\Ud$ is a cocartesian fibration; the cocartesian
    transport of a decorated \itcat{} $\tA^{\diamond}$ along 
    $\phi \colon \tA \to \tA'$ is given by composing with
    $\phi$ and then applying the localization to $\DCatIT$ (\ie{}
    factoring the span
    \[ \tAd_{(1)} \to \tA' \from \tAd_{(2)} \]
    using the appropriate factorization systems).
  \item $\Ud$ is a cartesian fibration; the cartesian transport of a
    decorated \itcat{} $\tA^{\diamond}$ along a functor
    $\psi \colon \tA' \to \tA$ is given by pullback along $\psi$
    (since this preserves the right classes in our factorization
    systems).
  \end{enumerate}
\end{observation}

\begin{observation}\label{var:idecorated}
  Let $\DtCatIT \subseteq \Ar(\CatIT)$ denote the full subcategory
  spanned by functors that are wide and locally wide sub-\itcat{}
  inclusions. Then the equivalence
  \[ \Fun(\Lambda^{2,\op}_{0}, \CatIT) \simeq \Ar(\CatIT)
    \times_{\CatIT} \Ar(\CatIT),\]
  arising from the obvious pushout decomposition of
  $\Lambda^{2,\op}_{0}$, restricts to an equivalence
  \[ \DCatIT \simeq \MCatIT \times_{\CatIT} \DtCatIT. \]
  From \cref{cor:MCatIT via 1cat} we then get an equivalence,
  \begin{equation}
    \label{eq:dcat via mcat}
    \DCatIT \simeq \MCatI \times_{\CatI} \DtCatIT,
  \end{equation}
  so that we can regard a decorated \itcat{} $\tCd$ as a diagram
  \[ \tC^{\diamond,\leq 1}_{(1)} \to \tC \from \tCd_{(2)}\]
  where $\tC^{\diamond,\leq 1}_{(1)}$ is a wide subcategory of
  $\tC^{\leq 1}$ and $\tCd_{(2)}$ is a wide and locally wide
  sub-\itcat{} of $\tC$. We note that the forgetful functor $\Udm
  \colon \DCatIT
  \to \MCatIT$ has
  \begin{itemize}
  \item a right adjoint $(\blank)^{\sharp}$, which takes $(\tC,E) \in
    \MCatIT$ to the decorated \itcat{}
    \[ (\tC,E)^{\sharp} \,:=\, \tC_{E} \to \tC \xfrom{=} \tC,\]
  \item a left adjoint $(\blank)^{\flat}$, which takes $(\tC,E) \in
    \MCatIT$ to the decorated \itcat{}
    \[ (\tC,E)^{\sharp} \,:=\, \tC_{E} \to \tC \from \tC^{\leq
        1},\]
  \end{itemize}
  Moreover, $\Udm(\blank)^{\sharp} \simeq \id \simeq
  \Udm(\blank)^{\flat}$, so both adjoints are fully faithful.
\end{observation}

\begin{notation}
  If $\oK$ is an \icat{}, then it has a unique class of decorated
  2-morphisms, and it is sometimes convenient to write
  \[ \oK^{\sharp} = \oK^{\sharp\sharp} = \oK^{\sharp\flat}, \quad
    \oK^{\flat} = \oK^{\flat\flat} = \oK^{\flat\sharp}.\]
  Similarly, an \igpd{} $X$ has a unique decoration and we may write
  \[ X = X^{\flat\flat} = X^{\sharp\sharp}.\]
\end{notation}

\begin{propn}\label{cor:forgetfaithful}
  $\Ud \colon \DCatIT \to \CatIT$ is a faithful functor.
\end{propn}
\begin{proof}
  This is immediate from \cref{propn:mono gives faithful from Ar}.
\end{proof}

\begin{lemma}\label{lem:localizationdecorations}
  The functor $(\blank)^{\flat\flat}$ has a left adjoint
  $\Ld \colon \DCatIT \to \CatIT$, given by inverting the decorated
  1-morphisms and 2-morphisms.
\end{lemma}
\begin{proof}
  We use the description of $\DCatIT$ from \cref{eq:dcat via mcat},
  and view a decorated \itcat{} $\tAd$ as
  \[ \tA^{\diamond,\leq 1}_{(1)} \to \tA \from \tAd_{(2)};\]
  then we define $\Ld(\tAd)$ as the iterated pushout
  \[ \Ld(\tAd) := \|\tA^{\diamond,\leq 1}_{(1)}\|
    \amalg_{\tA^{\diamond,\leq 1}_{(1)}} \tA \amalg_{\tAd_{(2)}}
    \tau_{(\infty,1)}(\tAd_{(2)}).\]
  It is clear that this defines a functor $\Ld \colon \DCatIT \to \CatIT$. For an \itcat{}
  $\tB$, we then have a natural equivalence
  \[
    \begin{split}
      \Map_{\CatIT}(\Ld(\tAd), \tB) & \simeq \Map(\|\tA^{\diamond,\leq 1}_{(1)}\|, \tB) \times_{\Map(\tA^{\diamond,\leq 1}_{(1)}, \tB)} \Map(\tA, \tB) \times_{\Map(\tAd_{(2)}, \tB)} \Map(\tau_{(\infty,1)}\tAd_{(2)}, \tB) \\
                                  & \simeq \Map(\tA^{\diamond,\leq 1}_{(1)}, \tB^{\simeq}) \times_{\Map(\tA^{\diamond,\leq 1}_{(1)}, \tB)} \Map(\tA, \tB) \times_{\Map(\tAd_{(2)}, \tB)} \Map(\tAd_{(2)}, \tB^{\leq 1}) \\
      & \simeq \Map_{\DCatIT}(\tAd, \tB^{\flat\flat}),
    \end{split}
  \]
  since the decorated \itcat{} $\tB^{\flat\flat}$ is given by the diagram
  \[ \tB^{\simeq} \to \tB \from \tB^{\leq 1}\] in terms of the
  description from \cref{eq:dcat via mcat}. This shows that
  the functor $\Ld$ is left adjoint to $(\blank)^{\flat\flat}$, as
  required.
\end{proof}

\begin{defn}\label{def:dec int hom}
  For decorated \itcats{} $\tAd$ and $\tBd$, we write
  $\DFUN(\tAd,\tBd)$ for the full subcategory of $\FUN(\tA,\tB)$
  spanned by the decorated functors. We equip this with the decoration
  where
  \begin{itemize}
  \item a natural transformation $\tA \times [1] \to \tB$ is decorated
    if its component at every $a \in \tA$ is a decorated 1-morphism in
    $\tB$, \ie{} if it is a decorated functor $\tAd \times
    [1]^{\sharp} \to \tBd$;
  \item a 2-morphism $\tA \times C_{2} \to \tB$ is decorated if its
    component at every $a \in \tA$ is a decorated 2-morphism in
    $\tB$, \ie{} if it is a decorated functor $\tAd \times
    C_{2}^{\flat\sharp} \to \tBd$.
  \end{itemize}
\end{defn}

\begin{propn}
  The \icat{} $\DCatIT$ is cartesian closed, with internal Hom given
  by $\DFUN(\blank,\blank)$.
\end{propn}
\begin{proof}
  Since $\Ud$ is faithful, it suffices to check that for decorated \itcats{} $\tAd,\tBd,\tCd$, a
  functor $\tA \times \tB \to \tC$ is decorated \IFF{} its adjoint
  $\tA \to \FUN(\tB, \tC)$ factors through a decorated functor $\tAd
  \to \DFUN(\tBd,\tCd)$, which is clear from unpacking the two conditions.
\end{proof}

We can thus regard $\DCatIT$ as enriched in itself. More importantly,
we can upgrade it to an $(\infty,2)$- (or $(\infty,3)$-)category:
\begin{construction}\label{const: DCatIT as 2-cat}
  The functor $(\blank)^{\flat\flat} \colon \CatIT \to \DCatIT$
  preserves cartesian products, and for $\tAd \in \CatIT$ the functor
  $(\blank)^{\flat\flat} \times \tAd$ has a right adjoint, namely
  $\Ud \DFUN(\tAd,\blank)$. It follows from \cref{propn:upgrade
    functor} that we can upgrade $\DCatIT$ to an
  $(\infty,3)$-category; we write $\DCATIT$ for its underlying
  \itcat{}. Unpacking the definition, we see that the mapping \icat{}
  between two objects $\tAd$ and $\tBd$ is
  \[ (\Ud\DFUN(\tAd, \tBd))^{\leq 1},\] which is the full subcategory of
  $\Fun(\tA, \tB)$ spanned by the decorated functors:
  its morphisms are the decorated functors
  \[ \tAd \times [1]^{\flat} \to \tBd,\]
  which are \emph{all} the natural
  transformations among decorated functors. We also note that the
  functors $\Ud$ and $(\blank)^{\flat\flat}$  preserve products (being
  right adjoints) and that there is an equivalence $\Ud \circ
  (\blank)^{\flat\flat}$, so using \cref{propn:upgrade adjn enr} we
  get adjunctions of \itcats{}
  \[ (\blank)^{\flat\flat} \dashv \Ud \dashv
    (\blank)^{\sharp\sharp}.\]
  The functor $\Ud$ is given on mapping \icats{} by the fully faithful inclusions
  \[ (\Ud\DFUN(\tAd, \tBd))^{\leq 1} \to \Fun(\tA, \tB)^{\leq 1},\]
  so that this is a locally fully faithful functor.
  Moreover, the functor $(\blank)^{\flat\flat}$ preserves cotensors by
  $[1]$, as in general the canonical map
  \[ \FUN(\tC,\tD)^{\flat\flat} \to
    \DFUN(\tC^{\flat\flat},\tD^{\flat\flat})\]
  is an equivalence, so its left adjoint $\Ld$ also upgrades to an
  adjunction of \itcats{} $\Ld \dashv (\blank)^{\flat\flat}$ by the
  dual of \cref{propn:upgrade adjoint}.
\end{construction}

\begin{warning}
  The \itcat{} $\DCATIT$ is \emph{not} a sub-\itcat{} of
  $\FUN(\Lambda^{2,\op}_{0}, \CATIT)$. This is because we take
  \emph{all} 2-morphisms among decorated functors, \ie{} the decorated
  functors of the form $\tAd \times [1]^{\flat} \to \tBd$, as
  2-morphisms in $\DCATIT$, while only those of the form
  $\tAd \times [1]^{\sharp} \to \tBd$ come from 2-morphisms in
  $\FUN(\Lambda^{2,\op}_{0}, \CATIT)$.
\end{warning}

\begin{observation}
  The forgetful functor $\Udm \colon \DCatIT \to \MCatIT$ fits in a commutative triangle
  \[
    \begin{tikzcd}
      {} & \CatI \ar[dl, "(\blank)^{\flat\flat}"'] \ar[dr, "(\blank)^{\flat}"] \\
      \DCatIT \ar[rr, "\Udm"] & & \MCatIT;
    \end{tikzcd}
  \]
  both this functor and its right adjoint $(\blank)^{\sharp}$
  therefore upgrade to 2-functors (and an adjunction of \itcats{}) by
  \cref{propn:upgrade adjn enr}; the same goes for the left adjoint
  $(\blank)^{\sharp}$, so we get an adjoint triple
  $(\blank)^{\flat} \dashv \Udm \dashv (\blank)^{\sharp}$ also on the
  level of \itcats{}.  Moreover, we have a natural equivalence of
  decorated \itcats{}
  \[ \MFUN(\blank,\blank)^{\sharp} \simeq
    \DFUN((\blank)^{\sharp},(\blank^{\sharp})),\]
  which in particular shows that $(\blank)^{\sharp}$ enhances to a fully faithful
  functor of \itcats{} $\MCATIT \hookrightarrow \DCATIT$.
\end{observation}

\subsection{Decorated Gray tensor products and (op)lax
  transformations}
\label{sec:dec gray}

In this subsection we introduce a decorated version of the Gray tensor
product and show that it has right adjoints in each variable. We warn
the reader that this is given by the ordinary Gray tensor product
(as reviewed in \S\ref{sec:twocat}) equipped with certain decorations,
and should not be confused with the \emph{marked} Gray tensor product
of \S\ref{sec:marked gray}, where certain 2-morphisms are inverted. In
fact, as we will briefly discuss, we can also define a decorated
version of marked Gray tensors.

\begin{defn}\label{defn:dfunlax}
  Given two decorated \itcats{} $\tAd$ and $\tBd$, we define an
  (undecorated) 
  \itcat{} $\DFUN(\tAd,\tBd)^{\plax}$ as the locally full sub-\itcat{}
  of $\FUN(\tA,\tB)^{\plax}$ where
  \begin{itemize}
  \item the objects are decorated functors $\tAd \to \tBd$,
  \item the morphisms are (op)lax natural transformations
    such that for every decorated 1-morphism in $\tA$, the 2-morphism
    in the corresponding (op)lax naturality square in $\tB$ is
    decorated.
  \end{itemize}
  We enchance this to a decorated \itcat{} by declaring that
  \begin{itemize}
  \item an (op)lax transformation is a decorated 1-morphism if all of
    its components are decorated 1-morphisms in $\tB$ and all of the
    2-morphisms in its (op)lax naturality squares are decorated,
  \item a 2-morphism is decorated if its component 2-morphism for
    every object in $\tA$ is decorated.
  \end{itemize}
\end{defn}

\begin{defn}
  Given decorated \itcats{} $\tAd$ and $\tBd$, we define a decorated
  \itcat{} $\tAd \otimesd \tBd$ to be given by the
  decoration of $\tA \otimes \tB$ so that
  \begin{itemize}
  \item the decorated 1-morphisms are generated by those in
    $\tA^{\diamond}_{(1)} \times \tB^{\simeq}$ and $\tA^{\simeq}
    \times \tB^{\diamond}_{(1)}$
  \item the decorated 2-morphisms are generated by those in
    $\tA^{\diamond}_{(2)} \times \tB^{\simeq}$, 
    $ \tA^{\simeq} \times\tB^{\diamond}_{(2)}$,
    $\tA^{\diamond,\leq 1}_{(1)} \otimes \tB^{\leq 1}$ and
    $\tA^{\leq 1} \otimes \tB^{\diamond,\leq 1}_{(1)}$.
  \end{itemize}
  In other words, in addition to decorating the (2-)morphisms
  associated to an object of $\tA$ and a decorated (2-)morphism in $\tB$
  and vice versa, we also decorate the 2-morphism in the lax square
  associated to a pair of morphisms from $\tA$ and $\tB$ if one of
  them is marked in $\tAd$ or $\tBd$.
\end{defn}

\begin{observation}
  Let $\tAd,\tBd,\tCd$ be decorated \itcats{}.
  A functor $\tC \to \FUN(\tA,\tB)^{\oplax}$ factors
  through $\DFUN(\tAd,\tBd)^{\oplax}$ \IFF{} the adjoint functor $\tA \otimes \tC \to \tB$ satisfies:
  \begin{itemize}
  \item for every object $c \in \tC$, the restriction to $\{c\} \times
    \tA \to \tB$ is decorated,
  \item for every morphism $c \to c'$ in $\tC$, the restriction to
    $\tA \otimes [1] \to \tB$ takes any decorated 1-morphism in
    $\tA$ to an oplax square with a decorated 2-morphism.
  \end{itemize}
  Moreover, this functor is decorated if in addition
    \begin{itemize}
  \item for every decorated morphism $c \to c' \in \tC$, the
    restriction to $\tA \otimes [1] \to \tB$ takes every object of
    $\tA$ to a decorated 1-morphism in $\tB$ and every morphism in
    $\tA$ to an oplax square with a decorated 2-morphism,
  \item for every decorated 2-morphism in $\tC$, the restriction to
    $\tA \otimes C_{2} \to \tB$ takes every object of $\tA$ to a
    decorated 2-morphism in $\tB$.
  \end{itemize}
  In total, these conditions say that we have a decorated
  functor $\tAd \otimesd \tCd \to \tBd$, and this is moreover
  equivalent to the other adjoint functor $\tA \to \FUN(\tC,
  \tB)^{\lax}$ factoring through a decorated functor $\tAd \to
  \DFUN(\tCd,\tBd)^{\lax}$. Thus the
  functor \[\blank \otimesd \blank \colon \DCatIT \times
    \DCatIT \to \DCatIT\] has a right adjoint in each variable, given
  by
  \[
    \begin{split}
      \Map_{\DCatIT}(\tCd, \DFUN(\tAd,\tBd)^{\oplax}) & \simeq
      \Map_{\DCatIT}(\tAd \otimesd \tCd, \tBd) \\ & \simeq
      \Map_{\DCatIT}(\tAd, \DFUN(\tCd,\tBd)^{\lax}).
    \end{split}
  \]
\end{observation}

\begin{notation}\label{notation:eps-lax}
  To make our notation more compact, it will be convenient
  to write
  \[ \tAd \otimesde \tBd :=
    \begin{cases}
      \tAd \otimesd \tBd, & \vepsilon = (0,1), (1,0),\\
      \tBd \otimesd \tAd, & \vepsilon = (0,0), (1,1),
    \end{cases}
  \]
  \[ \DFUN(\tAd,\tBd)^{\velax}  :=
    \begin{cases}
      \DFUN(\tAd, \tBd)^{\oplax}, & \vepsilon = (0,1), \quad (1,0), \\
      \DFUN(\tAd, \tBd)^{\lax}, & \vepsilon = (0,0), \quad (1,1). \\
    \end{cases}
  \]
  We then have natural equivalences
  \[ \Map(\tAd, \DFUN(\tBd,\tCd)^{\velax}) \simeq \Map(\tBd \otimesde
    \tAd, \tCd) \simeq \Map(\tBd, \DFUN(\tAd,\tCd)^{\cvelax}).\]
\end{notation}

\begin{warning}
  The convention in \cref{notation:eps-lax} is motivated by the
  characterization of partial $\ve$-fibrations in \cref{thm:pb crit
    for partial fib} and related results, and
  does \emph{not} match the
  variance of straightening to (op)lax transformations in
  \cref{thm:elaxstr}. In fact, a convention that fits both situations
  is impossible, as above we pair $(0,1)$ with $(1,0)$, while in
  \ref{thm:elaxstr} we instead pair $(0,1)$ with $(0,0)$.
\end{warning}

\begin{propn}
  For decorated \itcats{} $\tAd,\tBd,\tCd$ we have a natural equivalence
  \[\tAd \otimesd (\tBd \otimesd \tCd) \simeq (\tAd \otimesd \tBd) \otimesd \tCd.\]
\end{propn}
\begin{proof}
  We need to see that the decorations match up under the associativity
  equivalence for the Gray tensor product, which is immediate from the
  definition.
\end{proof}

\begin{remark}
  We expect that the decorated Gray tensors can in fact be enhanced to
  make $\DCatIT$ a monoidal \icat{}. As we do not need to know
  anything beyond the associativity statement above, we will not
  attempt to prove this, however.
\end{remark}

\begin{cor}\label{cor: dec funlax commute}
  There are natural equivalences of decorated \itcats{}
  \[
    \begin{split}
      \DFUN(\tAd, \DFUN(\tBd, \tCd)^{\lax})^{\oplax}
      & \simeq \DFUN(\tBd,
        \DFUN(\tAd,\tBd)^{\oplax})^{\lax}, \\
      \DFUN(\tAd \otimesd \tBd, \tCd)^{\oplax} &  \simeq \DFUN(\tBd, \DFUN(\tAd,\tCd)^{\oplax})^{\oplax}, \\
      \DFUN(\tAd \otimesd \tBd, \tCd)^{\lax} &  \simeq \DFUN(\tAd, \DFUN(\tBd,\tCd)^{\lax})^{\lax},       
    \end{split}
  \]
  for $\tAd,\tBd,\tCd \in \DCatIT$. \qed
\end{cor}

\begin{warning}
  Our \emph{decorated} Gray tensor product consists of the normal Gray
  tensor of \itcats{} equipped with decorated 1- and 2-morphisms. It
  should not be confused with the \emph{marked} Gray tensor product of
  marked \itcats{}, as reviewed above in \S\ref{sec:marked},
  where certain 2-morphisms in the Gray tensor product determined by
  the marking are \emph{inverted}.  In fact, we can combine
  the two by considering decorated \itcats{} equipped with an
  additional collection of \emph{marked} morphisms:
\end{warning}

\begin{variant}\label{variant: marked Gray}
  Suppose we have a decorated \itcat{} $\tAd$ and in addition a
  marking $(\tA,I)$ on its underlying \itcat{}. For $\tBd$ another
  decorated \itcat{}, we then define $\DFUN(\tAd,\tBd)^{I\dplax}$ to
  be the intersection of $\DFUN(\tAd,\tBd)^{\plax}$ and
  $\FUN(\tA,\tB)^{I\dplax}$ in $\FUN(\tA,\tB)^{\plax}$, with
  decorations inherited from the former. We can then also define
  decorated versions $\blank \otimesd_{\flat,I} \tAd $ and
  $\tAd \otimesd_{I,\flat} \blank$ of these marked Gray tensor
  products, which participate in adjunctions
  \[
    \begin{split}
      \blank \otimesd_{\flat,I} \tAd & \dashv
                                       \DFUN(\tAd,\blank)^{I\dlax}, \\
      \tAd \otimesd_{I,\flat} \blank & \dashv \DFUN(\tAd,\blank)^{I\doplax}.
    \end{split}
  \]
\end{variant}

\begin{lemma}\label{marked dec gray assoc}
  Suppose $\tAd,\tBd,\tCd$ are decorated \itcats{} and we also have a
  marked \itcat{} $(\tB,I)$. Then as an \itcat{},
  \[ \DFUN(\tAd, \DFUN(\tBd,\tCd)^{I\dlax})^{\oplax}\]
  is equivalent to the full sub-\itcat{} of
  \[ \DFUN(\tBd, \DFUN(\tAd,\tCd)^{\oplax})^{I\dlax} \]
  spanned by the functors that take the morphisms in $I$ to strong
  transformations.
\end{lemma}
\begin{proof}
  By definition, $\DFUN(\tAd, \DFUN(\tBd,\tCd)^{I\dlax})^{\oplax}$ is
  a locally full sub-\itcat{} of
  $\FUN(\tA, \DFUN(\tBd,\tCd)^{I\dlax})^{\oplax}$, which in turn is a
  locally full sub-\itcat{}  of
  $\FUN(\tA, \FUN(\tB,\tC)^{I\dlax})^{\oplax}$ by \cite[Theorem 2.6.3]{AGH24}. From \cref{lax oplax
    reverse marked} we know that we can identify this with the full
  sub-\itcat{} of $\FUN(\tB, \FUN(\tA,\tC)^{\oplax})^{I\dlax}$ spanned by the
  functors that take the morphisms in $I$ to strong
  transformations. This \itcat{} in turn contains \[\DFUN(\tBd,
  \DFUN(\tAd,\tCd)^{\oplax})^{I\dlax}\] as a locally full sub-\itcat{},
  so it suffices to observe that these sub-\itcats{} match up.
\end{proof}

\begin{remark}
  More generally, one can also consider decorated upgrades of the
  marked Gray tensor $\tA \otimes_{I,J} \tB$ of two marked \itcats{}
  $(\tA,I)$ and $(\tB,J)$. This gives a decorated marked Gray tensor
  on pairs of decorated marked \itcats{} with corresponding
  adjunctions, but we shall refrain from spelling this out.
\end{remark}

\begin{observation}
  Consider the special case of \cref{variant: marked Gray} where
  $\tAd$ is marked by the collection $I$ consisting of all of its
  1-morphisms. Then for any decorated \itcat{} $\tBd$ both $\tAd
  \otimesd_{I,\flat} \tBd$ and
  $\tBd
  \otimesd_{\flat,I} \tAd$ coincide with the cartesian
  product $\tAd \times \tBd$, and the right adjoints
  $\DFUN(\tAd,\tBd)^{I\dplax}$ both coincide with the internal Hom 
  $\DFUN(\tAd,\tBd)$ from \cref{def:dec int hom}.
\end{observation}

\begin{observation}\label{obs:DFUN Ilax full sub}
  Suppose $\tAd$ and $\tBd$ are decorated \itcats{}, and $\tA$ is
  equipped with a marking $I$ such that all decorated 1-morphisms in
  $\tA$ are also marked. Then $\DFUN(\tAd,\tBd)^{I\dplax}$ is a full
  sub-\itcat{} of $\FUN(\tA,\tB)^{I\dplax}$, since for any $I$-(op)lax
  transformation the (op)lax naturality square for any decorated
  1-morphism in $\tA$ commutes, and so in particular contains a
  decorated 2-morphism.
\end{observation}

\section{Fibrations of decorated \itcats{}}
\label{sec:pfib}

In this section we will use the formalism of decorated \itcats{} to
study \emph{partial fibrations}, by which we mean functors of
\itcats{} that admit cartesian or cocartesian lifts of certain specified 1- and
2-morphisms in the base. We introduce this notion in \S\ref{sec:def
  pfib} and then consider the variant notion of \emph{decorated}
partial fibrations in \S\ref{sec:dec p fib}, where there is also a
well-behaved decoration that contains these (co)cartesian lifts; when
the base is maximally decorated, we prove that decorated fibrations
classify functors to $\DCATIT$. In \S\ref{sec:maps in Aroplax} we
study mapping \icats{} in (op)lax arrow \itcats{}, as a preliminary to
proving a very useful characterization of partial fibrations in
\S\ref{sec:pfib char}, in terms of a certain commutative square being
a pullback. We apply this criterion in \S\ref{sec:fib Funlax} to show
that under certain conditions the functor $\FUN(\tC,\blank)^{\plax}$,
as well as its decorated and partially (op)lax versions, preserves
fibrations. Our pullback criterion also implies a characterization of
partial $\ve$-fibrations as being right orthogonal to certain maps of
decorated \itcats{}, and in \S\ref{sec:e-eqce} we study the
corresponding left orthogonal class, the \emph{$\ve$-cofibrations}, as
well as the morphisms that are local with respect to partial
$\ve$-fibrations on a fixed base, the \emph{$\ve$-equivalences}.

\subsection{Partial fibrations}
\label{sec:def pfib}

In this subsection we introduce the notion of \emph{partial
  $\ve$-fibrations} of \itcats{}. After giving the definition, we relate them to decorated
\itcats{} and mention some examples.

\begin{defn}
  Let $\tBd$ be a decorated \itcat{}. For $\vepsilon = (i,j)$, we say that a morphism of 
  \itcats{} $p \colon \tEd \to \tBd$ is a \emph{partial
    $\vepsilon$-fibration} with respect to the decoration $\tBd$ if
  \begin{itemize}
  \item $\tE$ has $p$-$i$-cartesian lifts of 1-morphisms in $\tBd_{(1)}$,
  \item $\tE$ has $p$-$j$-cartesian lifts of 2-morphisms in $\tBd_{(2)}$.
  \end{itemize}
  We also say that a commutative triangle
  \[
    \begin{tikzcd}
     \tE \ar[rr, "f"] \ar[dr, "p"'] & & \tF \ar[dl, "q"] \\
      & \tB
    \end{tikzcd}
  \]
  is a \emph{morphism of partial $\vepsilon$-fibrations} if $f$
  preserves these $i$-cartesian 1-morphisms and $j$-cartesian
  2-morphisms; we write $\PFibe_{/\tBd}$ for the subcategory of
  $\CatITsl{\tB}$ containing the partial $\vepsilon$-fibrations over
  $\tB$ and the morphisms thereof. Similarly, we define
  $\PFIBe_{/\tBd}$ to be the locally full sub-\itcat{} of
  $\CATITsl{\tB}$ on these objects and morphisms.
\end{defn}

\begin{variant}
  More generally, for a decorated functor $F \colon \tBd \to \tCd$, we
  say that a commutative square
  \[
    \begin{tikzcd}
     \tE \ar[r, "G"] \ar[d, "p"'] & \tF  \ar[d, "q"] \\
     \tB \ar[r, "F"'] & \tC
    \end{tikzcd}
  \]
  is a \emph{morphism of partial $\vepsilon$-fibrations} if $p$ and
  $q$ are partial $\vepsilon$-fibrations with respect to $\tBd$ and
  $\tCd$, respectively, and $G$ preserves $i$-cartesian morphisms and
  $j$-cartesian 2-morphisms over the decorations in $\tBd$.
\end{variant}

\begin{examples}
  Let $p \colon \tE \to \tB$ be a functor of \itcats{} and $\vepsilon=(i,j)$.
  \begin{itemize}
  \item $p$ is always a partial $\vepsilon$-fibration with respect to
    $\tB^{\flat\flat}$.
  \item $p$ is a partial $\vepsilon$-fibration with respect to
    $\tB^{\sharp\flat}$ \IFF{} $\tE$ has all $p$-$i$-cartesian lifts
    of morphisms in $\tB$.
  \item $p$ is a partial $\vepsilon$-fibration with respect to
    $\tB^{\flat\sharp}$ \IFF{} $\tE(x,y) \to \tB(p(x),p(y))$ is a
    $j$-cartesian fibration of \icats{} for all $x,y \in \tE$, and the
    $j$-cartesian morphisms are closed under whiskerings. This is
    equivalent to the composition functors
    \[ \tE(x,y) \times \tE(y,z) \to \tE(x,z)\]
    preserving $j$-cartesian morphisms for all $x,y,z \in \tE$, so
    this condition on $p$ corresponds to being
    \emph{$j$-cartesian-enriched}
    in the terminology of \cite{AGH24}.
  \item Combining the last two examples, we see that $p$ is a partial
    $\vepsilon$-fibration with respect to $\tB^{\sharp\sharp}$ \IFF{}
    $p$ is an $\vepsilon$-fibration as defined above in
    \S\ref{sec:review fib}.
  \end{itemize}
\end{examples}

\begin{observation}\label{obs:partial fib canonical dec}
  For $p \colon \tE \to \tB$ a functor of \itcats{}, it follows from
  the composability of pullback squares that $p$-(co)cartesian
  morphisms are closed under composition. Similarly, weakly
  $p$-(co)cartesian 2-morphisms are closed under vertical
  composition. Since horizontal composition can be decomposed into
  whiskering and vertical composition, it follows that
  $p$-(co)cartesian 2-morphisms are closed under both vertical and
  horizontal composition. If $p$ is a partial $(i,j)$-fibration with
  respect to $\tBd$, this means that
  \begin{itemize}
  \item the $p$-$i$-cartesian morphisms over $\tBd_{(1)}$ determine a
    locally full sub-\itcat{} of $\tE$,
  \item the $p$-$j$-cartesian 2-morphisms over $\tBd_{(2)}$ determine
    a wide and locally wide sub-\itcat{} of $\tE$.
  \end{itemize}
  Thus when $p$ is a partial fibration we obtain a canonical lift of
  the decoration of $\tBd$ to a decoration $\tEn$ of $\tE$, which leads us to the following definition:
\end{observation}

\begin{defn}
  Let $\tBd$ be a decorated \itcat{}. For $\vepsilon = (i,j)$, we say that a morphism of 
  \itcats{} $p \colon \tEd \to \tBd$ is a \emph{partial
    $\vepsilon$-fibration} if
\begin{itemize}
\item $\tE$ has $p$-$i$-cartesian lifts of 1-morphisms in $\tBd_{(1)}$,
  and $\tEd_{(1)}$ is precisely the wide locally full sub-\itcat{} of
  these,
\item $\tE$ has $p$-$j$-cartesian lifts of 2-morphisms in $\tBd_{(2)}$,
  and $\tEd_{(2)}$ is precisely the wide locally wide sub-\itcat{} of these.
\end{itemize}
We can then identify $\PFibe_{/\tBd}$ with the full subcategory of $\DCatITsl{\tBd}$
spanned by the partial $\vepsilon$-fibrations in this sense.
\end{defn}

\begin{example}\label{ex:proj par fib}
  For any decorated \itcat{} $\tBd$ and any \itcat{} $\tK$, the
  projection
  \[ \tK^{\flat\flat} \times \tBd \to \tBd\]
  is a partial $\vepsilon$-fibration for any $\vepsilon$. (Note,
  however, that the projection $\tKd \times \tBd \to \tBd$ for some
  non-minimal decoration of $\tK$ will \emph{not} be a partial
  fibration, since $\tKd \times \tBd$ then has decorated $1$- or
  $2$-morphisms that are not (co)cartesian.)
\end{example}

\begin{example}\label{ex:1-fibred as partial}
  Let $(\tE,I)$ be a marked \itcat{} with a functor
  $p \colon \tE \to \tB$. Then $(\tE,I)^{\sharp} \to
  \tB^{\sharp\sharp}$ is a partial $\ve$-fibration for $\ve = (i,j)$ \IFF{} $p$ is an
  $\ve$-fibration, $I$ is the collection of $p$-$i$-cartesian
  1-morphisms, and all 2-morphisms in $\tE$ are
  $j$-cartesian. Equivalently, $p$ is a \emph{1-fibred}
  $\ve$-fibration with $I$ as $p$-$i$-cartesian 1-morphisms.
\end{example}

\subsection{Decorated partial fibrations}
\label{sec:dec p fib}

In this section we introduce a variant of the definition of partial
fibrations where we allow additional decorations beyond the
(co)cartesian 1- and 2-morphisms; these will be useful at several
points in the next section. We will show that in the case where
the base $\tB^{\sharp\sharp}$ is maximally decorated, these can be
straightened to functors to $\DCATIT$.

\begin{defn}
  A morphism of decorated \itcats{} $p \colon \tEd \to \tBd$ is a
  \emph{decorated partial $\vepsilon$-fibration} for
  $\vepsilon = (i,j)$ if in the commutative diagram (where we use the
  description from \cref{eq:dcat via mcat})
  \[
    \begin{tikzcd}
     \tE_{(1)}^{\dleqo} \ar[r] \ar[d, "p^{\leq 1}_{(1)}"'] & \tE  \ar[d, "p"] &
     \tEd_{(2)} \ar[l] \ar[d, "p_{(2)}"]\\
     \tB^{\dleqo}_{(1)} \ar[r] & \tB & \tBd_{(2)}, \ar[l]
    \end{tikzcd}
  \]
  \begin{itemize}
  \item the functor $p \colon \tE \to \tB$ is a partial
    $\vepsilon$-fibration with respect to $\tBd$,
  \item the functor
    $p^{\leq 1}_{(1)} \colon \tE^{\dleqo}_{(1)} \to
    \tB^{\dleqo}_{(1)}$ is an $i$-fibration of \icats{},
  \item the functor $p_{(2)} \colon \tEd_{(2)} \to \tBd_{(2)}$ is a 
    partial $\vepsilon$-fibration with respect to the
    restriction of the decorations on $\tBd$,
  \item both commutative squares are morphisms of partial
    $\vepsilon$-fibrations, \ie{} they preserve the $i$-cartesian
    1-morphisms and $j$-cartesian 2-morphisms.
  \end{itemize}
\end{defn}

\begin{observation}\label{remark:dec fib explicit}
  Let us spell this out more concretely in the case $\vepsilon = (0,1)$: a functor $p \colon \tEd \to \tBd$ is a decorated partial $(0,1)$-fibration if
  \begin{itemize}
  \item the underlying functor of \itcats{} $p \colon \tE \to \tB$ is a partial $(0,1)$-fibration,
  \item the cocartesian morphisms over $\tBd_{(1)}$ are among the decorated morphisms of $\tEd$, and given a commutative triangle
    \[
      \begin{tikzcd}
       x \ar[rr, "f"] \ar[dr, "h"'] & & y \ar[dl, "g"] \\
        & z
      \end{tikzcd}
    \]
    in $\tE$ whose image in $\tB$ consists of decorated morphisms,
    such that $h$ is decorated and $f$ is cocartesian, then $g$ is
    also decorated.
  \item the cartesian 2-morphisms over $\tBd_{(2)}$ are among the decorated 2-morphisms of $\tEd$, and 
    \begin{itemize}
    \item if $f \colon x \to y$ in $\tE$ is a cocartesian morphism
      over a decorated morphism in $\tB$, then for all $z \in \tE$ a
      2-morphism in $\tE(y,z)$ whose image is decorated in $\tBd$ is decorated in $\tE$ \IFF{} its composition with $f$ in $\tE(x,z)$ is decorated;
    \item given a commutative triangle
    \[
      \begin{tikzcd}
       f \ar[rr, "\alpha"] \ar[dr, "\gamma"'] & & g \ar[dl, "\beta"] \\
        & h
      \end{tikzcd}
    \]
    in $\tE(x,y)$ whose image in $\tB(px,py)$ consists of decorated 2-morphisms,
    such that $\gamma$ is decorated and $\beta$ is cartesian, then $\alpha$ is
    also decorated.
    \end{itemize}
  \end{itemize}
\end{observation}

\begin{warning}
  Note that we do \emph{not} ask for the functor
  $p_{(1)} \colon \tEd_{(1)} \to \tBd_{(1)}$ to be a partial
  $\vepsilon$-fibration of \itcats{}.  This would require decorated
  1-morphisms to be preserved under $j$-cartesian transport along
  decorated 2-morphisms. Since this is generally false for
  $i$-cartesian 1-morphisms in a partial $\vepsilon$-fibration, this
  is far too restrictive a condition to impose. 
\end{warning}

\begin{defn}
  We say that a commutative triangle
  \[
    \begin{tikzcd}
     \tEd \ar[rr, "F"] \ar[dr, "p"'] & & \tFd \ar[dl, "q"] \\
      & \tBd,
    \end{tikzcd}
  \]
  in $\DCatIT$, where $p$ and $q$ are decorated partial
  $\vepsilon$-fibrations, is a \emph{morphism of decorated partial
    $\vepsilon$-fibrations} if the decorated functor $F$ also
  preserves $i$-cartesian 1-morphisms and $j$-cartesian 2-morphisms
  over the decorations in $\tBd$. The \itcat{} $\DPFIBe_{/\tBd}$ of
  decorated partial $\vepsilon$-fibrations over $\tBd$ is then defined
  to be the locally full sub-\itcat{} of $\DCATITsl{\tBd}$ whose
  objects are the decorated partial $\vepsilon$-fibrations and whose
  1-morphisms are the morphisms thereof.
\end{defn}

\begin{observation}
  The forgetful functor $\DCATITsl{\tBd} \to \CATITsl{\tB}$ restricts
  to a locally fully faithful functor
  \[ \DPFIBe_{/\tBd} \to \PFIBe_{/\tBd}. \]
\end{observation}

\begin{variant}
  In the special case where the base is maximally decorated, we refer to decorated partial $\vepsilon$-fibrations over $\tB^{\sharp\sharp}$ as \emph{decorated $\vepsilon$-fibrations} and write
  \[ \DFIBe_{/\tB} := \DPFIBe_{/\tB^{\sharp\sharp}}.\]
\end{variant}

\begin{examples}\label{ex:canonicalexamplesdecfib}\ 
  \begin{itemize}
  \item Let $\tBd$ be a decorated \itcat{} and suppose
    $p \colon \tE \to \tB$ is a partial $\vepsilon$-fibration with
    respect to $\tBd$. Then $p \colon \tEn \to \tBd$ is a
    decorated partial $\vepsilon$-fibration using the canonical
    decoration of $\tE$ (\cref{obs:partial fib canonical dec}); this
    follows from \cref{lem:cart 2mor detect after compose cocart mor}.
  \item For any decorated \itcats{} $\tKd$ and $\tBd$, the projection
    $\tKd \times
    \tBd \to \tBd$ is a decorated partial $\vepsilon$-fibration for
    any $\vepsilon$.
  \item Suppose $p \colon \tE \to \tB$ is an $\ve$-fibration. Then
    $\tE^{\sharp\sharp} \to \tB^{\sharp\sharp}$ is a decorated
    $\ve$-fibration.
  \end{itemize}
\end{examples}

We now turn to straightening for decorated $\vepsilon$-fibrations:
\begin{thm}\label{thm: dec straighten}
  Let $\tB$ be an \itcat{}. The straightening equivalence for $\ve$-fibrations extends to
  decorated fibrations to give a natural commutative square
  \[
    \begin{tikzcd}
      \DFIB^{\ve}_{/\tB} \ar[r, "\sim"] \ar[d] & \FUN(\tB^{\veop},
      \DCATIT)  \ar[d] \\
      \FIB^{\ve}_{/\tB} \ar[r, "\sim"'] & \FUN(\tB^{\veop}, \CATIT).
    \end{tikzcd}
  \]
\end{thm}
We will prove this in the case $\vepsilon=(0,1)$, leaving the 3 other
variants implicit. Since the vertical functors here are locally fully
faithful (by \cite[Theorem 2.6.3]{AGH24} in the right-hand case), the
following observation shows it suffices to prove that we get an
equivalence on underlying \icats{}:
\begin{observation}\label{obs: loc ff can check eqce on 1-cats}
  Suppose we have locally fully faithful functors $\tA, \tB \to \tC$
  and an equivalence $\tA^{\leq 1} \simeq \tB^{\leq 1}$ over
  $\tC^{\leq 1}$. Then we can lift this to an equivalence $\tA \simeq
  \tB$ over $\tC$. Indeed, we have a commutative square
  \[
    \begin{tikzcd}
     \tA^{\leq 1} \ar[r] \ar[d] & \tB  \ar[d] \\
     \tA \ar[r] \ar[dotted, ur] & \tC
    \end{tikzcd}
  \]
  where the left vertical map is left orthogonal to the right vertical
  map, so that there is a unique dotted lift $\tA \to \tB$. Similarly,
  the top horizontal map is left orthogonal to the bottom horizontal
  map, and by uniqueness the composites of these lifts are identities.
\end{observation}

We first consider the case where the base is an \icat{}:
\begin{propn}\label{propn: dec str over 1-cat}
  For $\oB$ an \icat{} there is a natural commutative square of \icats{}
    \[
    \begin{tikzcd}
      \DFib^{(0,1)}_{/\oB} \ar[r, "\sim"] \ar[d] & \Fun(\oB,
      \DCatIT)  \ar[d] \\
      \Fib^{(0,1)}_{/\oB} \ar[r, "\sim"'] & \Fun(\oB, \CatIT).
    \end{tikzcd}
  \]
\end{propn}
\begin{proof}
  Since $\oB$ is an \icat{}, we can think of a decorated $(0,1)$-fibration over $\oB$ as a diagram
    \[
    \begin{tikzcd}
     \tEd_{(1)} \ar[r] \ar[d, "p_{(1)}"'] & \tE  \ar[d, "p"] &
     \tEd_{(2)} \ar[l] \ar[d, "p_{(2)}"]\\
     \oB \ar[r, "="] & \oB & \oB, \ar[l, "="']
    \end{tikzcd}
  \]
  where $p$, $p_{(1)}$, and $p_{(2)}$ are all $0$-fibrations, and the top
  horizontal maps are morphisms of $0$-fibrations. Similarly,
  morphisms in $\DFib^{(0,1)}_{/\oB}$
  correspond to maps of such cospans over $\oB$ where all components
  preserve cocartesian morphisms. We can thus
  identify $\DFib^{(0,1)}_{/\oB}$ with a full subcategory of
  $\Fun(\Lambda^{2,\op}_{0}, \Fib^{(0,1)}_{/\oB})$. Under
  straightening, the latter is equivalent to
  $\Fun(\oB, \Fun(\Lambda^{2,\op}_{0}, \Fun(\oB, \CatIT))$. As it is
  easy to see that a morphism of fibrations over $\oB$ is a wide
  locally full inclusion or a wide and locally wide inclusion \IFF{}
  it is so fibrewise, it follows that the full
  subcategory corresponding to $\DFib^{(0,1)}_{/\oB}$ can be
  identified as $\Fun(\oB,
  \DCatIT)$, as required.
\end{proof}

\begin{propn}\label{propn:DFib pb from 1-cat}
  For $\tB$ an \itcat{}, the natural commutative square of \icats{}
  \[
    \begin{tikzcd}
      \DFib^{(0,1)}_{/\tB}
      \ar[r] \ar[d] & \Fib^{(0,1)}_{/\tB}  \ar[d] \\
      \DFib^{(0,1)}_{/\tB^{\leq 1}} \ar[r] &
      \Fib^{(0,1)}_{/\tB^{\leq 1}}
    \end{tikzcd}
  \]
  is a pullback.
\end{propn}

For the proof, we need the following observation:
\begin{lemma}\label{lem:reconstruct dec fib from 1-cat}
  Suppose $p \colon \tE \to \tB$ is a $(0,1)$-fibration and set $\tE' := \tE \times_{\tB} \tB^{\leq 1}$. If $\tE'^{\diamond} \to \tB^{\leq 1,\sharp}$ is a decorated $(0,1)$-fibration, then
  \begin{enumerate}[(i)]
  \item the 2-morphisms in $\tE$ that factor as a vertical composite
    of a decorated 2-morphism in $\tE'^{\diamond}_{(2)}$ followed by a
    cartesian 2-morphism are closed under both horizontal and vertical
    composition, and so form a wide and locally wide subcategory
    $\tEd_{(2)}$ of $\tE$, which is the minimal such containing both
    $\tE'^{\diamond}_{(2)}$ and $\tEn_{(2)}$;
  \item the restriction of $p$ to $\tEd_{(2)} \to \tB$ is a
    $(0,1)$-fibration, and the inclusion $\tEd_{(2)} \to \tE$ is a
    morphism of $(0,1)$-fibrations;
  \item if $\tEd$ is the decoration of $\tE$ given by the 1-morphisms in $\tE'^{\diamond}_{(1)}$ and the 2-morphisms in $\tEd_{(2)}$, then $\tEd \to \tBd$ is a decorated $(0,1)$-fibration.
  \end{enumerate}
\end{lemma}
\begin{proof}
  In (i), it is clear that the specified 2-morphisms are closed under  horizontal composition, since this is true for both the decorated 2-morphisms in $\tE'^{\diamond}$ and the $p$-cartesian 2-morphisms in $\tE$. It therefore suffices to consider a vertical composition
  \[
    \begin{tikzcd}[column sep=large]
      x \ar[r, bend left=75, "f", ""'{name=A}] \ar[r, "g"{near start}, ""{name=B},
      ""'{name=C}] \ar[r, bend right=75, ""{name=D}, "h"']& y
      \arrow[from=A, to=B, Rightarrow, "\alpha"]
      \arrow[from=C, to=D, Rightarrow, "\beta"{near start}]
    \end{tikzcd}
  \]
  with $\alpha$ cartesian and $\beta$ decorated over an
  equivalence (\ie{} in the ``wrong'' order). Then $\beta$  factors as
  \[
    \begin{tikzcd}[column sep=large]
      x \ar[r, "\gamma"] & x' \ar[r, bend left=30, "g'", ""'{name=A}]  \ar[r, bend right=30, ""{name=D}, "h'"']& y
      \arrow[from=A, to=D, Rightarrow, "\beta'"]
    \end{tikzcd}
  \]
  with $\gamma$ cocartesian over $p(g) \simeq p(h)$, while
  $\alpha$ factors as
  \[
    \begin{tikzcd}[column sep=large]
       x \ar[r, bend left=30, "f'", ""'{name=A}]  \ar[r, bend
       right=30, ""{name=D}, "\gamma"']& x' \ar[r, "g'"]
      \arrow[from=A, to=D, Rightarrow, "\alpha'"] & y
    \end{tikzcd}
  \]
  with $\alpha'$ cartesian. Thus our vertical composition is
  equivalent to the horizontal composition
  \[
    \begin{tikzcd}[column sep=large]
       x \ar[r, bend left=30, "f'", ""'{name=A}]  \ar[r, bend
       right=30, ""{name=B}, "\gamma"']& x' \ar[r, bend left=30, "g'", ""'{name=C}]  \ar[r, bend right=30, ""{name=D}, "h'"']& y,
      \arrow[from=C, to=D, Rightarrow, "\beta'"]
      \arrow[from=A, to=B, Rightarrow, "\alpha'"]
    \end{tikzcd}
  \]
  which we can in turn rewrite as a vertical composition in the
  ``right'' order. This proves (i).

  For (ii), we note firstly that cocartesian morphisms in $\tE$ are
  also cocartesian in $\tEd_{(2)}$, since both the morphisms in
  $\tE'^{\diamond}_{(2)}$ and $\tEn_{(2)}$ are detected after
  whiskering with a cocartesian morphism. Moreover, it is clear from the definition of $\tEd_{(2)}$ that for $x,y \in \tE$,
  the functor
  \[ \tEd_{(2)}(x,y) \to \tB(px,py)\] is a cartesian fibration with the cartesian morphisms inherited from $\tEd(x,y)$, as required. 

  Finally, (iii) is immediate from (ii) and the definition of a decorated $(0,1)$-fibration.
\end{proof}

\begin{proof}[Proof of \cref{propn:DFib pb from 1-cat}]
  Using \cref{lem:reconstruct dec fib from 1-cat}, we can define a functor from 
  \[ \Fib^{(0,1)}_{/\tB} \times_{\Fib^{(0,1)}_{/\tB^{\leq 1}}} \DFib^{(0,1)}_{/\tB^{\leq 1}} \to \DFib^{(0,1)}_{/\tB}\]
  by taking a pair $(p \colon \tE \to \tB, p' \colon \tE'^{\diamond} \to \tB^{\leq 1})$ (where $\tE' := \tE \times_{\tB} \tB^{\leq 1}$)
  to the localization to $\DCatIT$  of the diagram
  \[ \tE'^{\leq 1}_{(1)} \to \tE \from \tEn_{(2)} \amalg
    \tE'^{\diamond}_{(2)}\] over $\tB^{\sharp\sharp}$, which is
  clearly functorial in $(p,p')$.
  In this case, $\tEd_{(2)} \times_{\tB} \tB^{\leq 1}$ recovers $\tE'^{\diamond}_{(2)}$; on the other hand, for $\tEd \to \tB^{\sharp\sharp}$ in 
  $\DFib^{(0,1)}_{/\tB}$ we see that $\tEd_{(2)}$ consists precisely of those 2-morphisms that factor as a decorated 2-morphism over an equivalence followed by a cartesian 2-morphism. Thus the composites in both directions are equivalences, which completes the proof.
\end{proof}

\begin{proof}[Proof of \cref{thm: dec straighten}]
  From \cref{obs: loc ff can check eqce on 1-cats} it suffices to show
  that we have an equivalence on underlying \icats{}.
  By orthogonality we also have a natural pullback square of \icats{}
  \[
    \begin{tikzcd}
      \Fun(\tB, \DCATIT) \ar[r] \ar[d] & \Fun(\tB, \CATIT)  \ar[d] \\
      \Fun(\tB^{\leq 1}, \DCATIT) \ar[r] & \Fun(\tB^{\leq 1},
      \CATIT).
    \end{tikzcd}
  \]
  Combined with the pullback square from \cref{propn:DFib pb from
    1-cat} this means that it suffices to construct a natural cospan
  of equivalences
  \[
    \begin{tikzcd}
     \Fun(\tB^{\leq 1}, \DCatIT) \ar[r] \ar[d, "\sim"'] &
     \Fun(\tB^{\leq 1}, \CatIT)  \ar[d, "\sim"] & \Fun(\tB, \CATIT)
     \ar[l], \ar[d, "\sim"] \\
     \DFib^{(0,1)}_{/\tB^{\leq 1}} \ar[r] &
     \Fib^{(0,1)}_{/\tB^{\leq 1}} & \Fib^{(0,1)}_{/\tB}, \ar[l]
    \end{tikzcd}
  \]
  which we get from \cref{propn: dec str over 1-cat} and the
  naturality of straightening.
\end{proof}

Passing to left adjoints, we also get:
\begin{cor}\label{cor: dec str comp flat}
  There is a natural commutative square of \itcats{}
  \[
    \begin{tikzcd}
      \FIB^{\ve}_{/\tB} \ar[r, "\sim"] \ar[d, hookrightarrow, "(\blank)^{\natural}"] & \FUN(\tB^{\veop},
      \CATIT)  \ar[d, hookrightarrow, "(\blank)^{\flat\flat}_{*}"] \\
      \DFIB^{\ve}_{/\tB} \ar[r, "\sim"'] & \FUN(\tB^{\veop}, \DCATIT)..
    \end{tikzcd}
  \]
  Furthermore, the left adjoint to $(\blank)^{\natural}$ corresponds
  under straightening to composition with the localization functor
  $\Ld \colon \DCATIT \to \CATIT$. \qed
\end{cor}

We expect that it is also possible to extend \cref{thm:elaxstr} to describe the
\itcats{} $\FUN(\tB, \DCATIT)^{\eplax}$ in terms of decorated
fibrations. However, as this seems somewhat annoying to prove and is
not required in this paper, we content ourselves with the following
statement in this case:
\begin{propn}\label{propn:lax tr DCATIT fib}
  Let $(\tB, E)$ be a marked \itcat{}. For functors
  $F,G \colon \tB \to \DCATIT$ with corresponding decorated
  $\ve$-fibrations $\tPd_{\ve},\tQd_{\ve} \to \tB^{\sharp\sharp}$, there is a
  natural equivalence
  \[ \Nat^{\elax}_{\tB,\DCATIT}(F,G) \simeq
    \DFun^{E\dcoc}_{/(\tB,E)^{\sharp}}(\tPd_{(0,1)}|_{E},
    \tQd_{(0,1)}|_{E}),\] where $\tPd_{(0,1)}|_{E}$ denotes the
  pullback $\tPd_{(0,1)} \times_{\tBss} (\tB,E)^{\sharp}$ (given by
  restricting the decorated 1-morphisms of $\tPd_{(0,1)}$ to those that lie over $E$)
  and
  \[\DFun^{E\dcoc}_{/(\tB,E)^{\sharp}}(\tPd_{(0,1)}|_{E},
    \tQd_{(0,1)}|_{E}) \subseteq
    \DFun_{/(\tB,E)^{\sharp}}(\tPd_{(0,1)}|_{E}, \tQd_{(0,1)}|_{E})\]
  is the full subcategory of functors that preserve cocartesian
  morphisms over $E$ and all cartesian 2-morphisms. Similarly, we have
  natural equivalences
  \[
    \begin{split}
      \Nat^{\eoplax}_{\tB,\DCATIT}(F,G)
      & \simeq  \DFun^{E\dcart}_{/(\tB^{\op},E)^{\sharp}}(\tPd_{(1,0)}|_{E},
                                          \tQd_{(1,0)}|_{E}) \\
      & \simeq \DFun^{E\dcart}_{/(\tB^{\coop},E)^{\sharp}}(\tPd_{(1,1)}|_{E},
        \tQd_{(1,1)}|_{E}), \\
\Nat^{\elax}_{\tB,\DCATIT}(F,G) & \simeq
    \DFun^{E\dcoc}_{/(\tB^{\co},E)^{\sharp}}(\tPd_{(0,0)}|_{E}, \tQd_{(0,0)}|_{E}).
    \end{split}
  \]
\end{propn}

\begin{lemma}\label{lem:lff on Funlax}
  Suppose $F \colon \tC \to \tD$ is a locally fully faithful
  functor. Then so is $F_{*} \colon \FUN(\tA,\tC)^{E\dplax} \to
  \FUN(\tA,\tD)^{E\dplax}$ for any marked \itcat{} $(\tA,E)$, and we
  can identify
  \[ \Nat^{E\dplax}_{\tA,\tC}(\alpha,\beta) \subseteq
    \Nat^{E\dplax}_{\tA,\tD}(F\alpha, F\beta)\]
  as the full subcategory of those $E$-(op)lax transformations $\eta
  \colon F\alpha \to F\beta$ whose component $\eta_{x} \in
  \tD(F\alpha(x), F\beta(x))$ lies in $\tC(\alpha(x),\alpha(y))$ for
  every object $x \in \tA$.
\end{lemma}
\begin{proof}
  It follows from \cite[Corollary 2.7.13]{AGH24} that $F_{*}$ is
  locally fully faithful, so we are left with identifying its image on
  mapping \icats{}. The stated condition is clearly necessary, so we
  need to prove that it suffices. We consider the oplax case; using
  the adjunction between $\FUN(\tA,\blank)^{E\dplax}$ and $\tA
  \otimes_{E,\natural} \blank$ we can rephrase this as: $F$ is right
  orthogonal to the map \[\tA^{\simeq} \times [1] \amalg_{\tA^{\simeq}
      \times \partial [1]} \tA^{\leq 1} \times \partial [1] \to \tA
    \otimes_{E,\natural} [1].\]
  Since $F$ is locally fully faithful \IFF{} it is right orthogonal to
  the map $\partial C_{2} \to C_{2}$, it suffices to show that this
  map is in the saturated class generated by the latter map, which we know
  from \cite[Theorem 5.3.7]{Soergel} consists of the functors that are
  essentially surjective on objects and 1-morphisms.
  The canonical description of $\tA$ as a colimit of objects in
  $\Theta_{2}$ is preserved by $(\blank)^{\leq 1}$ and
  $(\blank)^{\simeq}$, so arguing as in the
  proof of \cite[Corollary 2.5.14]{AGH24} we can reduce to the case
  where $(\tA,E)$ is $[0], [1]^{\flat}, [1]^{\sharp}$ and
  $C_{2}^{\flat}$. In these cases it is clear that the given map is
  surjective on objects and 1-morphisms, which completes the proof.
\end{proof}

\begin{proof}[Proof of \cref{propn:lax tr DCATIT fib}]
  We prove the lax case. By \cref{lem:lff on Funlax} the locally fully faithful forgetful
  functor $\Ud \colon \DCATIT \to \CATIT$ identifies
  \[ \Nat^{E\dlax}_{\tB,\DCATIT}(F,G) \subseteq
    \Nat^{E\dlax}_{\tB,\CATIT}(\Ud F, \Ud G)\]
  as the full subcategory of $E$-lax transformations $\phi \colon \Ud
  F \to \Ud G$ such that each component $\phi_{b} \colon \Ud F(b) \to
  \Ud G(b)$ is a decorated functor $F(b) \to G(b)$ for all $b \in
  \tB$. Now \cref{thm:elaxstr} gives a natural equivalence
  \[ \Nat^{E\dlax}_{\tB,\CATIT}(\Ud F, \Ud G) \simeq
    \FIB^{(0,1)}_{/(\tB,E)}(\tP, \tQ),
  \]
  where the right-hand side is the full subcategory of
  $\Fun_{/\tB}(\tP,\tQ)$ on functors that preserve cocartesian
  morphisms over $E$ and all cartesian 2-morphisms. We can therefore
  identify $\Nat^{E\dlax}_{\tB,\DCATIT}(F,G)$ with the full
  subcategory of the latter spanned by the functors $r \colon \tP \to \tQ$ over
  $\tB$ for which the map on fibres is a decorated functor $r_{b}
  \colon \tPd_{b}
  \to \tQd_{b}$ for all $b \in \tB$. A decorated morphism in $\tPd$
  factors as a cocartesian morphism followed by a decorated morphism
  in a fibre, so this condition on 1-morphisms is equivalent to $r$
  preserving all decorated morphisms over $E$; similarly, the
  condition on 2-morphisms is equivalent to $r$ preserving \emph{all}
  decorated 2-morphisms. The \icat{} in question is therefore
  $\DFun^{E\dcoc}_{/(\tB,E)^{\sharp}}(\tPd_{(0,1)}|_{E},
  \tQd_{(0,1)}|_{E})$, as required.
\end{proof}

\subsection{Maps in (op)lax arrows}
\label{sec:maps in Aroplax}

Our goal in this subsection is to identify the mapping \icats{} in
$\AR^{\plax}(\tC)$, which is the key input needed for the
characterization of partial fibrations in the following subsection.

\begin{propn}\label{ARlax mor pb}
  For an \itcat{} $\tC$ and morphisms $f \colon x \to y$ and $g \colon
  a \to b$ in $\tC$, there are pullback squares of \icats{}
  \[
    \begin{tikzcd}
      \AR^{\lax}(\tC)(f,g) \ar[r] \ar[d] &   \ar[d, "{\ev_{0},\ev_{1}}"]  \Ar(\tC(x,b)) \\
     \tC(x,a) \times \tC(y,b)  \ar[r, "g_{*} \times f^{*}"'] & \tC(x,b) \times \tC(x,b),
   \end{tikzcd}
   \quad
   \begin{tikzcd}
      \AR^{\oplax}(\tC)(f,g) \ar[r] \ar[d] &   \ar[d, "{\ev_{1},\ev_{0}}"]  \Ar(\tC(x,b)) \\
     \tC(x,a) \times \tC(y,b)  \ar[r, "g_{*} \times f^{*}"'] & \tC(x,b) \times \tC(x,b).
    \end{tikzcd}
  \]
\end{propn}

\begin{remark}
  In the notation of \cite{2Topoi}, \cref{ARlax mor pb} says that the \icats{}
  $\AR^{\plax}(\tC)(f,g)$ are \emph{oriented pullbacks} or \emph{comma
    objects}:
  \[
    \begin{split}
      \AR^{\lax}(\tC)(f,g) & \simeq \tC(x,a) \orientedtimeslr_{\tC(x,b)}
                              \tC(y,b), \\
      \AR^{\oplax}(\tC)(f,g) & \simeq \tC(x,a) \orientedtimesrl_{\tC(x,b)}
                              \tC(y,b). \\
    \end{split}
  \]
\end{remark}

\begin{proof}[Proof of \cref{ARlax mor pb}]
  We prove the oplax case; the lax case is proved similarly. We can
  identify the mapping \icat{} $\AR^{\oplax}(\tC)(f,g)$ as the
  pullback
  \[
    \begin{tikzcd}
     \AR^{\oplax}(\tC)(f,g) \ar[r] \ar[d] & \AR^{\oplax}(\AR^{\oplax}(\tC))   \ar[d, "{\ev_{0},\ev_{1}}"] \\
     * \ar[r, "{(f,g)}"'] &  \AR^{\oplax}(\tC) \times \AR^{\oplax}(\tC).
    \end{tikzcd}
  \]
  Here $\AR^{\oplax}(\AR^{\oplax}(\tC))$ is equivalent to
  $\FUN([1] \otimes [1], \tC)^{\oplax}$.
  
  The pushout square from \cref{obs:gray square pushout} gives a commutative cube
\[
  \begin{tikzcd}[row sep=small,column sep=small]
   \FUN([1] \otimes [1], \tC)^{\oplax} \ar[rr] \ar[dr] \ar[dd] & &
   \FUN(C_{2}, \tC)^{\oplax} \ar[dr] \ar[dd] \\
    & \FUN(\partial ([1] \otimes [1]), \tC) \ar[rr,crossing over]  & &
    \FUN(\partial C_{2}, \tC)^{\oplax} \ar[dd] \\
   \AR^{\oplax}(\tC)^{\times 2} \ar[rr] \ar[dr, "="] & & \tC^{\times
     2} \ar[dr, "="] \\
    & \AR^{\oplax}(\tC)^{\times 2} \ar[rr] \ar[uu,leftarrow,crossing
    over]  & & \tC^{\times 2}
  \end{tikzcd}
\]
where both the top and bottom squares are pullbacks, and we recognize
$\AR^{\oplax}(\tC)(f,g)$ as the fibre at $(f,g)$ in the top left
corner. Our desired pullback square will be the induced pullback on
fibres from this cube.

To identify the fibre at the bottom left we consider the pushout
decomposition
\[
  \begin{tikzcd}
    {[0]^{\amalg 4}} \ar[r] \ar[d] & {[1] \amalg [1]}  \ar[d] \\
    {[1] \amalg [1]}\ar[r] & \partial ([1] \otimes [1]),
  \end{tikzcd}
\]
which induces a pullback square
\[
  \begin{tikzcd}
   \FUN(\partial ([1] \otimes [1]), \tC)^{\oplax} \ar[r] \ar[d] & \AR^{\oplax}(\tC)^{\times 2}  \ar[d] \\
   \AR^{\oplax}(\tC)^{\times 2} \ar[r] & \tC^{\times 4};
  \end{tikzcd}
\]
from this we see that the fibre at $(f \colon x \to y,g \colon a \to
b)$ on the left is equivalent to $\tC(x,a) \times \tC(y,b)$.

To describe the fibres on the right side of the cube we consider the
pushouts
\[
  \begin{tikzcd}
   {[1] \amalg [1]} \ar[r] \ar[d] & \partial ([1] \otimes
   [1])  \ar[d] \ar[r] & {[1] \otimes [1]} \ar[d] \\
   {[0] \amalg [0]} \ar[r] & \partial C_{2} \ar[r] & C_{2},
  \end{tikzcd}
\]
which give pullbacks
\[
  \begin{tikzcd}
   \FUN(C_{2}, \tC)^{\oplax} \ar[r] \ar[d] & \FUN([1] \otimes [1], \tC)^{\oplax}  \ar[d]  \\
   \FUN(\partial C_{2}, \tC)^{\oplax} \ar[r] \ar[d] &
   \FUN(\partial([1] \otimes [1]), \tC)^{\oplax} \ar[d] \\
   \tC^{\times 2} \ar[r] & \AR^{\oplax}(\tC)^{\times 2}.
  \end{tikzcd}
\]
Here the fibre at $(x,b)$ at the top right is identified with $\Ar
\tC(x,b)$ by applying $\AR^{\oplax}(\blank)$ to the pullback
\[
  \begin{tikzcd}
   \tC(x,b) \ar[r] \ar[d] & \AR^{\oplax}(\tC)  \ar[d] \\
   \{(x,b)\} \ar[r] & \tC^{\times 2},
  \end{tikzcd}
\]
while the pushout for $\partial ([1] \otimes [1])$ above again
identifies the fibre at the middle right as $\tC(x,b)^{\times 2}$,
with the map from $\Ar \tC(x,b)$ given by taking source and target.

We thus have the desired pullback square
\[
  \begin{tikzcd}
   \AR^{\oplax}(\tC)(f,g) \ar[r] \ar[d] & \Ar \tC(x,b)  \ar[d] \\
   \tC(x,a) \times \tC(y,b) \ar[r] & \tC(x,b)^{\times 2}
  \end{tikzcd}
\]
where the bottom horizontal map can be identified as given by
composition with $f$ and $g$ since this ultimately came from the map
$\partial C_{2} \to \partial ([1] \otimes [1])$ given by two copies of
$d_{1}$.
\end{proof}

\begin{observation}\label{partial C2 cons on 1-mor}
  Let $j \colon \partial C_{2} \to C_{2}$ be the boundary
  inclusion. For later use, we note that in the course of this proof
  we saw that in the commutative triangle
  \[
    \begin{tikzcd}
     \FUN(C_{2},\tC)^{\oplax} \ar[rr, "j^{*}"] \ar[dr] & &
     \FUN(\partial C_{2},\tC)^{\oplax} \ar[dl] \\
      & \tC^{\times 2}
    \end{tikzcd}
  \]
  the map on fibres over $x,y$ was the evaluation map
  \[ \Ar \tC(x,y) \to \tC(x,y)^{\times 2}.\]
  Hence the fibre of $j^{*}$ at the
  object corresponding to $f,g \colon x \to y$ is the \igpd{}
  $\tC(x,y)(f,g)$. This shows in particular that $j^{*}$ is
  conservative on 1-morphisms.
\end{observation}

We can extend this description to decorated \itcats{}, for which we
introduce the following notation:
\begin{defn}
  For a decorated \itcat{} $\tCd$, let
  \[\DAR^{\plax}(\tCd) := \DFUN([1]^{\sharp}, \tCd)^{\plax}.\]
  In other words, this is the locally full sub-\itcat{} of
  $\AR^{\plax}(\tC)$ where
  \begin{itemize}
  \item the objects are the morphisms in $\tCd_{(1)}$,
  \item the morphisms are the (op)lax squares where the 2-morphisms
    lie in $\tCd_{(2)}$.
  \end{itemize}
  It is convenient to write
  \[ \DARelax(\tCd) :=
    \begin{cases}
      \DAR^{\oplax}(\tCd), & \vepsilon = (0,1), \, (1,0), \\
      \DAR^{\lax}(\tCd), & \vepsilon = (0,0), \, (1,1). \\
    \end{cases}
  \]
\end{defn}

Viewing $\DAR^{\plax}(\tCd)$ for a decorated \itcat{} $\tCd$ as a
locally full subcategory of $\AR^{\plax}(\tC)$, we get:
\begin{cor}\label{DARlax mor pb}
  For a decorated \itcat{} $\tCd$ and decorated morphisms $f \colon x \to y$ and $g \colon
  a \to b$ in $\tC$, there are pullback squares of \icats{}
  \[
    \begin{tikzcd}
      \DAR^{\lax}(\tCd)(f,g) \ar[r] \ar[d] &   \ar[d, "{s,t}"]  \DAr(\tC(x,b)) \\
     \tC(x,a) \times \tC(y,b)  \ar[r, "g_{*} \times f^{*}"'] & \tC(x,b) \times \tC(x,b),
   \end{tikzcd}
   \quad
   \begin{tikzcd}
      \DAR^{\oplax}(\tCd)(f,g) \ar[r] \ar[d] &   \ar[d, "{t,s}"]  \DAr(\tC(x,b)) \\
     \tC(x,a) \times \tC(y,b)  \ar[r, "g_{*} \times f^{*}"'] & \tC(x,b) \times \tC(x,b),
    \end{tikzcd}
  \]
  where $\DAr(\tC(x,b))$ denotes the full subcategory of
  $\Ar(\tC(x,b))$ spanned by the decorated 2-morphisms. \qed
\end{cor}

\begin{observation}\label{obs:DAR map 1-dec}
  For a marked \itcat{} $(\tC,I)$,
  \cref{DARlax mor pb} says that 
  we get a pullback
  \[
    \begin{tikzcd}
      \DAR((\tC,I)^{\flat})(f,g) \ar[r] \ar[d] &   \ar[d,
      "g_{*}"] \tC(x,a)   \\
      \tC(y,b)  \ar[r, "f^{*}"'] & \tC(x,b).
    \end{tikzcd}
  \]
  This applies in particular to $\AR(\tC) = \DAR(\tC^{\sharp\flat})$.
\end{observation}

\begin{observation}\label{obs: maps in ARlax ids}
  Specializing \cref{DARlax mor pb} to identity morphisms, we get pullback squares
  \[
    \begin{tikzcd}
      \DAR^{\lax}(\tCd)(\id_{x},\id_{y}) \ar[r] \ar[d] &   \ar[d, "{s,t}"]  \DAr(\tC(x,y)) \\
     \tC(x,y) \times \tC(x,y)  \ar[r, "="'] & \tC(x,y) \times \tC(x,y),
   \end{tikzcd}
   \quad
    \begin{tikzcd}
      \DAR^{\oplax}(\tCd)(\id_{x},\id_{y}) \ar[r] \ar[d] &   \ar[d, "{t,s}"]  \DAr(\tC(x,y)) \\
     \tC(x,y) \times \tC(x,y)  \ar[r, "="'] & \tC(x,y) \times \tC(x,y),
   \end{tikzcd}   
  \]
  so that we have equivalences $\DAR^{\plax}(\tCd)(\id_{x},\id_{y})
  \simeq \DAr(\tC(x,y))$, appropriately compatible with the evaluation
  maps to $\tC(x,y)$.
\end{observation}

Combining the last two observations, we see that the functor
$s_{0}^{*} \colon \tC \to
\AR(\tC)$, given by restriction along the degeneracy $s_{0} \colon [1]
\to [0]$, is fully faithful. This has the following useful
consequence:
\begin{lemma}\label{lem:ess surj restr cons}
  Suppose $f \colon \tC' \to \tC$ is essentially surjective on
  objects. Then the functor $f^{*} \colon \FUN(\tC, \tD) \to
  \FUN(\tC', \tD)$ is conservative on 1-morphisms for any \itcat{} $\tD$.
\end{lemma}
\begin{proof}
  A functor is conservative on morphisms if it is right orthogonal to
  $s_{0} \colon [1] \to [0]$, so we want to show that the commutative
  square
  \[
    \begin{tikzcd}
     \Map([0], \FUN(\tC, \tD)) \ar[r, "(f^{*})_{*}"] \ar[d, "s_{0}^{*}"'] & \Map([0],
     \FUN(\tC', \tD))  \ar[d, "s_{0}^{*}"] \\
     \Map([1], \FUN(\tC, \tD)) \ar[r, "(f^{*})_{*}"'] & \Map([1] \FUN(\tC', \tD))
    \end{tikzcd}
  \]
  is a pullback. This is equivalent to the square
  \[
    \begin{tikzcd}
     \Map(\tC, \tD) \ar[r, "f^{*}"] \ar[d, "s_{0}^{*}"'] & \Map(\tC', \tD)  \ar[d, "(s_{0}^{*})_{*}"] \\
     \Map(\tC, \AR(\tD)) \ar[r, "f^{*}"'] & \Map(\tC', \AR(\tD)),
    \end{tikzcd}
  \]
  which is a pullback since $s_{0}^{*} \colon \tD \to \AR(\tD)$ is
  fully faithful, $f$ is essentially surjective, and the essentially
  surjective functors are precisely those that are left orthogonal to
  the fully faithful ones (\eg{} by \cite[Theorem 5.3.7]{Soergel}).
\end{proof}

\subsection{Characterizing partial fibrations}
\label{sec:pfib char}

Our goal in this subsection is to prove the following useful
characterization of partial fibrations\footnote{A version of this
  result can also be extracted from Loubaton's work on
  $(\infty,\infty)$-categories (see \cref{rmk:loubaton fib}).}:
\begin{thm}\label{thm:pb crit for partial fib}
  A morphism of decorated \itcats{} $p \colon \tEd \to \tBd$ is a partial
  $\vepsilon$-fibration for $\vepsilon=(i,j)$ \IFF{} the commutative square
  \begin{equation}
    \label{eq:DAR fib pb}
    \begin{tikzcd}
      \DARelax(\tEd) \ar[r, "\ev_{i}"] \ar[d, "\DARelax(p)"'] & \tE  \ar[d, "p"] \\
      \DARelax(\tBd) \ar[r, "\ev_{i}"'] & \tB
    \end{tikzcd}
  \end{equation}
  is a pullback of \itcats{}.
\end{thm}

\begin{warning}
  If $p \colon \tEd \to \tBd$ is a partial fibration, the square
  \cref{eq:DAR fib pb} is generally \emph{not} a pullback of \emph{decorated}
  \itcats{} when $\DAR^{\velax}(\tEd)$ is equipped with its standard
  decoration from \cref{defn:dfunlax}: For $\vepsilon = (0,1)$, a morphism in $\DARelax(\tEd)$
  is a diagram
  \[
    \begin{tikzcd}
     \bullet \ar[r] \ar[d, mid vert] & \bullet  \ar[d, mid vert] \\
     \bullet \ar[r] \ar[Rightarrow, ur, mid vert] & \bullet
    \end{tikzcd}
  \]
  in $\tE$, where the vertical morphisms and diagonal 2-morphism are
  decorated.  In the standard decoration this is decorated if also the
  horizontal morphisms are decorated. On the other hand, for the
  pullback of the decorations in \cref{eq:DAR fib pb}, a morphism is
  decorated whenever the \emph{top} horizontal morphism is decorated
  in $\tEd$ and the image of the bottom horizontal morphism is
  decorated in $\tBd$.  For these decorations to agree we would need
  the cartesian transport of a cocartesian 1-morphism to again be
  cocartesian, which is usually \emph{not} the case. However, we will
  see below in \cref{lem:dec pb for dec fib} that we \emph{do} get a pullback of decorated
  \itcats{} if we modify the decorations on $\DAR^{\velax}(\blank)$
  appropriately. 
\end{warning}

\begin{remark}
  In \cref{thm:pb crit for partial fib} the assumption that we start
  with a decoration on $\tE$ can be weakened slightly: since 
  (co)cartesian 1-morphisms are automatically closed under
  composition, we do not strictly need the assumption that the
  decorated morphisms in $\tE$ are closed under
  composition. Similarly, it is enough to assume the decorated
  2-morphisms are closed under whiskering.
\end{remark}

The starting point for the proof of \cref{thm:pb crit for partial fib}
is the following description of (co)cartesian morphisms in terms of arrow \itcats{}:
\begin{propn}\label{propn:cocart mor via AR homs}
  Given a functor $p \colon \tE \to \tB$ of \itcats{}, the following
  are equivalent for a morphism $\bar{f} \colon \bar{x} \to \bar{y}$
  in $\tE$ over $f \colon x \to y$ in $\tB$:
  \begin{enumerate}[(i)]
  \item $\bar{f}$ is $p$-$i$-cartesian.
  \item For every morphism $\bar{g} \colon \bar{a} \to \bar{b}$ in
    $\tE$ over $g \colon a \to b$ in $\tB$, we have
    \begin{itemize}
    \item for $i = 0$, the commutative square
    of \icats{}
    \[
      \begin{tikzcd}
        \AR(\tE)(\bar{f},\bar{g}) \ar[r] \ar[d] &
        \tE(\bar{x}, \bar{a})  \ar[d] \\
       \AR(\tB)(f,g) \ar[r] & \tB(x,a)
      \end{tikzcd}
    \]
    is a pullback.
  \item for $i = 1$, the commutative square
    of \icats{}
    \[
      \begin{tikzcd}
        \AR(\tE)(\bar{g},\bar{f}) \ar[r] \ar[d] &
        \tE(\bar{a}, \bar{x})  \ar[d] \\
       \AR(\tB)(g,f) \ar[r] & \tB(a,x)
      \end{tikzcd}
    \]
    is a pullback.
    \end{itemize}
  \item The previous condition holds when $\bar{g} = \id_{\bar{a}}$
    for all $\bar{a}$ in $\tE$.
  \end{enumerate}
\end{propn}
\begin{proof}
  We prove the cocartesian case. Then the commutative square in (ii)
  fits in a commutative cube
  \[
    \begin{tikzcd}[row sep=small,column sep=small]
     \AR(\tE)(\bar{f},\bar{g}) \ar[rr] \ar[dr] \ar[dd] & & \tE(\bar{x},\bar{a}) \ar[dr] \ar[dd] \\
      & \AR(\tB)(f,g) \ar[rr,crossing over]  & & \tB(x,a) \ar[dd] \\
     \tE(\bar{y},\bar{b}) \ar[rr] \ar[dr] & & \tE(\bar{x},\bar{b}) \ar[dr] \\
      & \tB(y,b) \ar[rr] \ar[uu,leftarrow,crossing over]  & & \tB(x,b)
    \end{tikzcd}
  \]
  where the back and front faces are pullbacks by \cref{obs:DAR map
    1-dec}. If $\bar{f}$ is $p$-cocartesian then the bottom square is
  also a pullback, hence so is the top square; thus (i) implies
  (ii). On the other hand, if $\bar{g}$ is an equivalence, then the
  vertical maps in the right square are invertible, and hence so are
  the two other vertical maps, since they are pullbacks of these. The
  top square is therefore a pullback \IFF{} the bottom square is
  one. Thus if the top square is a pullback for $\bar{g} =
  \id_{\bar{a}}$ for all $\bar{a}$, then $\bar{f}$ is cocartesian, so
  that (iii) implies (i).
\end{proof}

From this we immediately obtain the special case of  \cref{thm:pb crit
  for partial fib} where only 1-morphisms are decorated:
\begin{cor}\label{cor:pb for cocart 1-mor}
  Suppose
  $p \colon (\tE,I) \to (\tB,J)$ is a
  functor of marked \itcats{}. Then $p^{\flat}$ is a partial
  $\vepsilon$-fibration for $\vepsilon = (i,j)$, \ie{} $\tE$ has
  $p$-$i$-cartesian lifts of marked 1-morphisms in $\tB$, \IFF{}
  the commutative square
  \begin{equation}
    \label{eq:DAR fib pb 1-dec}
    \begin{tikzcd}
      \DAR((\tE,I)^{\flat}) \ar[r, "\ev_{i}"] \ar[d] & \tE  \ar[d] \\
      \DAR((\tB,J)^{\flat}) \ar[r, "\ev_{i}"'] & \tB
    \end{tikzcd}
  \end{equation}
  is a pullback, where $\DAR((\tE,I)^{\flat})$ denotes the
  full subcategory of $\AR(\tE)$ on the decorated 1-morphisms.
\end{cor}
\begin{proof}
  We consider the case $\vepsilon = (0,1)$ and first suppose that $p$
  is a partial $(0,1)$-fibration.
  From \cref{obs:DAR map 1-dec} and \cref{propn:cocart mor via AR
    homs} we then know that in \cref{eq:DAR fib pb 1-dec} we have pullbacks
  on 
  all mapping \icats{}. It thus only remains to show we have a
  pullback of \igpds{}
  \[
    \begin{tikzcd}
     \Map^{\diamond}([1], \tE) \ar[r] \ar[d] & \tE^{\simeq}  \ar[d] \\
     \Map^{\diamond}([1], \tB) \ar[r] & \tB^{\simeq}
    \end{tikzcd}
  \]
  on cores, where $\Map^{\diamond}([1], \blank)$ denotes the
  sub-\igpd{} of $\Map([1], \blank)$ on decorated morphisms. This
  follows from the uniqueness of cocartesian lifts as in
  \cref{obs:cart lifts unique}, since $\Map^{\diamond}([1], \tE)$ is
  precisely the \igpd{} of cocartesian lifts of the morphisms in
  $\tBd_{(1)}$, and these all exist by assumption.

  Conversely, if the square is a pullback, then \cref{propn:cocart mor
    via AR homs} implies that the decorated morphisms in $\tE$ are all
  $p$-cocartesian, while the pullback on cores shows that there is a
  (unique) decorated lift of every decorated morphism in $\tB$ with
  any given source. Thus $\tE$ has all $p$-cocartesian lifts of the
  decorated morphisms in $\tB$, and these are precisely its decorated
  morphisms, as required.
\end{proof}

\begin{proof}[Proof of \cref{thm:pb crit for partial fib}]
  We consider the case of partial $(0,1)$-fibrations; the other 3 cases
  are proved in the same way, or follow by reversing 1- and/or
  2-morphisms.  Let us first assume that $p$ is a partial
  $(0,1)$-fibration. For decorated 1-morphisms
  $\bar{f} \colon \bar{x} \to \bar{y}$ and
  $\bar{g} \colon \bar{a} \to \bar{b}$ in $\tE$ over
  $f \colon x \to y$ and $g \colon a \to b$, we then have a
  commutative cube
  \[
    \begin{tikzcd}[row sep=small,column sep=small]
     \DAR^{\oplax}(\tEd)(\bar{f},\bar{g}) \ar[rr] \ar[dr] \ar[dd] & & \DAR^{\oplax}(\tBd)(f,g) \ar[dr] \ar[dd] \\
      & \DAr(\tE(\bar{x},\bar{b})) \ar[rr,crossing over]  & & \DAr(\tB(x,b)) \ar[dd] \\
     \tE(\bar{x},\bar{a}) \times \tE(\bar{y},\bar{b}) \ar[rr] \ar[dr] & & \tB(x,a) \times \tB(y,b) \ar[dr] \\
      & \tE(\bar{x},\bar{b})^{\times 2} \ar[rr] \ar[uu,leftarrow,crossing over]  & & \tB(x,b)^{\times 2}
    \end{tikzcd}
  \]
  where the left and right faces are pullbacks by \cref{DARlax mor
    pb}. Here we can factor the bottom square as
  \[
    \begin{tikzcd}
      \tE(\bar{x},\bar{a}) \times \tE(\bar{y},\bar{b}) \ar[r] \ar[d] &
      \tE(\bar{x},\bar{a}) \times \tB(y,b)
     \ar[d] \ar[r] & \tB(x,a) \times \tB(y,b) \ar[d] \\
     \tE(\bar{x},\bar{b}) \times \tE(\bar{x},\bar{b}) \ar[r] &
     \tE(\bar{x},\bar{b}) \times \tB(x,b) \ar[r]
     & \tB(x,b) \times \tB(x,b),
    \end{tikzcd}
  \]
  where the left square is a pullback since $\bar{f}$ is a cocartesian
  1-morphism. Using this we can reorganize our cube as
    \[
    \begin{tikzcd}[row sep=small,column sep=small]
     \DAR^{\oplax}(\tEd)(\bar{f},\bar{g}) \ar[rr] \ar[dr] \ar[dd] & & \DAR^{\oplax}(\tBd)(f,g) \ar[dr] \ar[dd] \\
      & \DAr(\tE(\bar{x},\bar{b})) \ar[rr,crossing over]  & & \DAr(\tB(x,b)) \ar[dd] \\
     \tE(\bar{x},\bar{a}) \times \tB(y,b) \ar[rr] \ar[dr] & & \tB(x,a) \times \tB(y,b) \ar[dr] \\
     & \tE(\bar{x},\bar{b}) \times \tB(x,b) \ar[rr] \ar[uu,leftarrow,crossing over]  & & \tB(x,b)^{\times 2},
    \end{tikzcd}
  \]
  where we know that the left and right faces are pullbacks. Moreover,
  the front face here is a pullback by \cref{cor:pb for cocart 1-mor}, since $\tE(\bar{x},\bar{b}) \to
  \tB(x,b)$ has cartesian lifts of the decorated morphisms in
  $\tB(x,b)$. It follows that the back face is also a pullback, which
  implies that the commutative square of mapping \icats{}
  \[
    \begin{tikzcd}
     \DAR^{\oplax}(\tEd)(\bar{f},\bar{g}) \ar[r] \ar[d] & \DAR^{\oplax}(\tBd)(f,g)  \ar[d] \\
     \tE(\bar{x},\bar{a}) \ar[r] & \tB(x,a)
    \end{tikzcd}
  \]
  is a pullback. To see that \cref{eq:DAR fib pb} is a pullback, it
  then only remains to prove that it gives a pullback square on
  cores. This again follows from the uniqueness of cocartesian lifts
  (\cref{obs:cart lifts unique}), just as in the proof of \cref{cor:pb for cocart 1-mor}.

  We now prove the converse: Since equivalences are always decorated,
  for objects $\bar{x}$ and $\bar{y}$ in $\tE$ over $x,y \in \tB$ we
  get a pullback square
  \[
    \begin{tikzcd}
     \DAR^{\oplax}(\tEd)(\id_{\bar{x}}, \id_{\bar{y}}) \ar[r] \ar[d] & \DAR^{\oplax}(\tBd)(\id_{x},\id_{y})  \ar[d] \\
     \tE(\bar{x},\bar{y}) \ar[r] & \tB(x,y).
    \end{tikzcd}
  \]
  Using \cref{obs: maps in ARlax ids} we can identify this with the
  square 
  \[
    \begin{tikzcd}
     \DAr(\tE(\bar{x},\bar{y})) \ar[r] \ar[d] & \DAr(\tB(x,y))  \ar[d] \\
     \tE(\bar{x},\bar{y}) \ar[r] & \tB(x,y).
    \end{tikzcd}
  \]
  Since this is a pullback, \cref{cor:pb for cocart 1-mor} implies that
  $\tE(\bar{x},\bar{y})$ has cartesian lifts of the decorated
  2-morphisms in $\tB(x,y)$, and these are precisely the decorated
  2-morphisms in $\tE(x,y)$. It follows that a decorated 2-morphism in
  $\tEd$ is invertible precisely if its image in $\tB$ is invertible,
  which means that we have a pullback square
  \[
    \begin{tikzcd}
     \DAR((\Udm \tEd)^{\flat}) \ar[r] \ar[d] & \DAR^{\oplax}(\tEd)  \ar[d] \\
     \DAR((\Udm \tBd)^{\flat}) \ar[r] & \DAR^{\oplax}(\tBd).
    \end{tikzcd}
  \]
  Applying \cref{cor:pb for cocart 1-mor} to the combination of this
  with the pullback square \cref{eq:DAR fib pb}, we can conclude that
  $\tE$ has cocartesian lifts of the decorated 1-morphisms in $\tB$,
  and these are precisely the decorated 1-morphisms in $\tE$, as
  required.
\end{proof}

\subsection{Partial fibrations on functors}
\label{sec:fib Funlax}

Suppose $p \colon \oE \to \oB$ is a (co)cartesian fibration of
\icats{}; then so is $p_{*} \colon \Fun(\oK,\oE) \to \Fun(\oK, \oB)$
for any \icat{} $\oK$, and the $p_{*}$-(co)cartesian morphisms are
precisely the natural transformations that are componentwise
(co)cartesian.  Our goal in this subsection is to prove an \itcatl{}
version of this statement and its generalization to partially (op)lax
transformations. We will do this in the framework of decorated
\itcats{}, and so we will more precisely
give conditions under which the functors
$\DFUN(\tKd,\blank)^{I\dplax}$ preserve partial fibrations.

Combining our discussion of decorated Gray tensors with the criterion
of \cref{thm:pb crit for partial fib}, we immediately get the
following useful case:
\begin{cor}\label{cor:DFUNlax fib no dec}
  Suppose $p \colon \tEd \to \tBd$ is a partial $\vepsilon$-fibration. Then for
  any \itcat{} $\tK$, the induced functor of decorated \itcats{}
  \[ p_{*} \colon \DFUN(\tK^{\flat\flat}, \tEd)^{\checkvepsilon\dlax} \to
    \DFUN(\tK^{\flat\flat}, \tBd)^{\checkvepsilon\dlax} \]
  is again a partial $\vepsilon$-fibration.
\end{cor}

\begin{remark}
  More explicitly, restricting to the case $\ve = (0,1)$ for
  concreteness, \cref{cor:DFUNlax fib no dec} says that the functor
  \[ p_{*} \colon \FUN(\tK, \tE)^{\lax} \to \FUN(\tK, \tB)^{\lax}\]
  is a partial $(0,1)$-fibration with respect to the decoration of the
  target where
  \begin{itemize}
  \item the decorated 1-morphisms are the lax transformations whose
    component at every object of $\tK$ is a decorated 1-morphism in
    $\tB$ and
    whose lax naturality squares contain decorated 2-morphisms,
  \item the decorated 2-morphisms are those whose component at every
    object of $\tK$ is a decorated 2-morphism in $\tB$.
  \end{itemize}
  Moreover, a 1-morphism in the source over such a decorated
  1-morphism is cocartesian precisely when its component at every
  object of $\tK$ is $p$-cocartesian and its lax naturality squares
  all contain $p$-cartesian 2-morphisms. A 2-morphism is similarly
  cartesian if its component at every object of $\tK$ is a
  $p$-cartesian 2-morphism. In particular, if $p$ is a 
  $(0,1)$-fibration, then so is $p_{*}$, with these as its
  (co)cartesian 1- and 2-morphisms.
\end{remark}

\begin{proof}[Proof of \cref{cor:DFUNlax fib no dec}]
  We prove the case of partial $(0,1)$-fibrations. Applying
  \cref{thm:pb crit for partial fib}, we need to show that the
  commutative square
  \[
    \begin{tikzcd}
     \DAR^{\oplax}(\DFUN(\tK^{\flat\flat},\tEd)^{\lax}) \ar[r] \ar[d] & \DAR^{\oplax}(\DFUN(\tK^{\flat\flat},\tBd)^{\lax})  \ar[d] \\
     \DFUN(\tK^{\flat\flat},\tE)^{\lax} \ar[r] & \DFUN(\tK^{\flat\flat},\tB)^{\lax}
    \end{tikzcd}
  \]
  is a pullback of \itcats{}. But using \cref{cor: dec funlax commute}
  we can rewrite this as the square
  \[
    \begin{tikzcd}
     \FUN(\tK,\DAR^{\oplax}(\tEd))^{\lax} \ar[r] \ar[d] &
     \FUN(\tK,\DAR^{\oplax}(\tBd))^{\lax}   \ar[d] \\
     \FUN(\tK,\tE)^{\lax} \ar[r] & \FUN(\tK,\tB)^{\lax},
    \end{tikzcd}
  \]
  which is a pullback since $\FUN(\tK,\blank)^{\lax}$ preserves
  limits, being a right adjoint.
\end{proof}

We can generalize this result using a decorated upgrade of
\cref{thm:pb crit for partial fib} (which we prove in the more general
case of decorated partial fibrations for later use).
\begin{notation}
  Suppose $\tBd$ is a decorated \itcat{}. We write
  \[\DARelax(\tBd)^{\diafib}\] for the decoration of $\DARelax(\tBd)$
  where
  \begin{itemize}
  \item an (op)lax square is a decorated 1-morphism
    \IFF{} it commutes
    and its image under both $\ev_{0}$ and $\ev_{1}$ are decorated in
    $\tBd$,
  \item a 2-morphism is decorated \IFF{} its image under $\ev_{0}$ and
    $\ev_{1}$ are both decorated in $\tBd$.
  \end{itemize}
  (Note that this differs from the decoration considered in
  \cref{defn:dfunlax}, which is adjoint to the decorated Gray tensor
  product: the difference is that we ask for decorated 1-morphisms to
  be \emph{commutative} squares rather than contain a decorated
  2-morphism.)
  If $\tEd \to \tBd$ is a decorated partial fibration, and $\tEn$ as
  usual denotes the (co)cartesian decoration of $\tE$, we also
  slightly abusively write $\DARelax(\tEn)^{\diafib}$ for the
  decoration of $\DARelax(\tEn)$ induced by viewing it as a full
  sub-\itcat{} of $\DARelax(\tEd)^{\diafib}$.
\end{notation}

\begin{propn}\label{lem:dec pb for dec fib}
  Suppose $p \colon \tEd \to \tBd$ is a decorated partial
  $\vepsilon$-fibration. Then the commutative square
  \[
    \begin{tikzcd}
     \DARelax(\tEn)^{\diafib} \ar[r] \ar[d] & \DARelax(\tBd)^{\diafib}  \ar[d] \\
     \tEd \ar[r] & \tBd
    \end{tikzcd}
  \]
  is a pullback of decorated \itcats{}. 
\end{propn}
\begin{proof}
  Since the underlying functor $\tE \to \tB$ is a partial
  $\vepsilon$-fibration, we know that the underlying square of
  \itcats{} is a pullback by \cref{thm:pb crit for partial fib}. It
  therefore suffices to show that the decorations match; we check
  this in the case $\vepsilon = (0,1)$. A 1-morphism in
  $\DARoplax(\tEn)$ is then an oplax square
    \[
    \begin{tikzcd}
     \bullet \ar[r] \ar[d, mid vert, "\natural"'] & \bullet  \ar[d, mid
     vert, "\natural"] \\
     \bullet \ar[r] \ar[Rightarrow, ur, mid vert, "\natural"] & \bullet
    \end{tikzcd}
  \]
  in $\tE$, whose vertical morphisms and diagonal 2-morphism are
  decorated in $\tEn$, which means they are cocartesian and cartesian,
  respectively. This corresponds to a decorated 1-morphism in the
  pullback $\DARoplax(\tBd)^{\diafib} \times_{\tBd} \tEd$ \IFF{}
  \begin{itemize}
  \item the top horizontal morphism is decorated in
    $\tEd$,
  \item and the image in $\DARoplax(\tBd)^{\diafib}$ is decorated, \ie{}
    it commutes and its horizontal morphisms are decorated.
  \end{itemize}
  It follows that the 2-morphism in our oplax square in $\tE$ is cartesian over an
  equivalence, so it must be invertible and we get a commutative
  square; the bottom horizontal morphism is then also decorated since
  $p$ was a decorated $(0,1)$-fibration, as noted in \cref{remark:dec fib explicit}.
  The case of 2-morphisms follows similarly from the two properties of
  decorated 2-morphisms in a decorated partial fibrations stated in \ref{remark:dec fib explicit}.
\end{proof}

Using this, we can strengthen \cref{cor:DFUNlax fib no dec} as follows:
\begin{cor}\label{cor:dec DFUNlax fib}
  Suppose $p \colon \tEd \to \tBd$ is a partial $\vepsilon$-fibration
  and $\tKd$ is a decorated \itcat{} equipped with a marking $E$ such
  that all decorated 1-morphisms are marked. Then 
  the induced functor of decorated \itcats{}
  \[ p_{*} \colon \DFUN(\tKd, \tEd)^{E-\checkvepsilon\dlax} \to
    \DFUN(\tKd, \tBd)^{E-\checkvepsilon\dlax} \]
  is again a partial $\vepsilon$-fibration. 
\end{cor}
\begin{proof}
  Applying $\DFUN(\tKd, \blank)^{E-\checkvepsilon\dlax}$ to the
  pullback square of decorated \itcats{} from
  \cref{lem:dec pb for dec fib}, we get the pullback square
  \[
    \begin{tikzcd}
     \DFUN(\tKd, \DARelax(\tEd)^{\diafib})^{E-\checkvepsilon\dlax} \ar[r] \ar[d] &  \DFUN(\tKd, \DARelax(\tBd)^{\diafib})^{E-\checkvepsilon\dlax} \ar[d] \\
     \DFUN(\tKd, \tEd)^{E-\checkvepsilon\dlax} \ar[r] & \DFUN(\tKd,
     \tBd)^{E-\checkvepsilon\dlax}.
    \end{tikzcd}
  \]
  We also know from \cref{marked dec gray assoc} that we can identify
  \[ \DARelax(\DFUN(\tKd, \tBd)^{E-\checkvepsilon\dlax})\] with the
  full sub-\itcat{}  of
  $\DFUN(\tKd, \DARelax(\tBd))^{E-\checkvepsilon\dlax})$ on functors
  that take morphisms in $E$ to commuting squares (where
  $\DARelax(\tBd)$ has the standard decoration). On the other hand,
  $\DFUN(\tKd, \DARelax(\tBd)^{\diafib})^{E-\checkvepsilon\dlax}$ is
  the full sub-\itcat{} on functors that take the decorated
  1-morphisms to commuting squares. Since the decorated morphisms
  are all marked, this means that $\DARelax(\DFUN(\tKd, \tBd)^{E-\checkvepsilon\dlax})$
  is a full
  sub-\itcat{} of
  \[\DFUN(\tKd,
  \DARelax(\tBd)^{\diafib})^{E-\checkvepsilon\dlax}.\] Moreover, we
  have a pullback square
  \[
    \begin{tikzcd}
     \DARelax(\DFUN(\tKd, \tEd)^{E-\checkvepsilon\dlax}) \ar[r] \ar[d] & \DARelax(\DFUN(\tKd, \tBd)^{E-\checkvepsilon\dlax}) \ar[d] \\
     \DFUN(\tKd,
  \DARelax(\tEd)^{\diafib})^{E-\checkvepsilon\dlax} \ar[r] & \DFUN(\tKd,
  \DARelax(\tBd)^{\diafib})^{E-\checkvepsilon\dlax},
    \end{tikzcd}
  \]
  since the objects of $\DARelax(\tEd)$ are squares that
  contain a $j$-cartesian 2-morphism, which therefore commute \IFF{}
  their images in $\DARelax(\tBd)$ also commute. Putting these squares
  together we conclude that $p_{*}$ is a partial $\ve$-fibration by
  the criterion of \cref{thm:pb crit for partial fib}.
\end{proof}

In particular, we have the extreme case where all morphisms are marked:
\begin{cor}
  Suppose $p \colon \tEd \to \tBd$ is a partial $\vepsilon$-fibration
  and $\tKd$ is a decorated \itcat{}. Then 
  the induced functor of decorated \itcats{}
  \[ p_{*} \colon \DFUN(\tKd, \tEd) \to
    \DFUN(\tKd, \tBd) \]
  is again a partial $\vepsilon$-fibration. \qed
\end{cor}

\subsection{\texorpdfstring{$\vepsilon$}{Epsilon}-equivalences and \texorpdfstring{$\vepsilon$}{Epsilon}-cofibrations}
\label{sec:e-eqce}

We can interpret \cref{thm:pb crit for partial fib} as characterizing
partial $\vepsilon$-fibrations by a right orthogonality property. In
this section we will spell this out and then study the corresponding
left orthogonal class, the \emph{$\ve$-cofibrations}, as well as the
morphisms that are local equivalences with respect to partial
$\ve$-fibrations over a fixed base, which we call \emph{$\ve$-equivalences}.

Using \cref{notation:eps-lax}, we have the following
reinterpretation of \cref{thm:pb crit for partial fib}:
\begin{cor}\label{cor:partial fib right orth}
  A morphism of decorated \itcats{} $\tEd \to \tBd$ is a partial
  $\vepsilon$-fibration for $\vepsilon = (i,j)$ \IFF{} it is right
  orthogonal to
\[\{i\} \otimesde C_{k}^{\flat\flat} \to
  [1]^{\sharp}
  \otimesde C_{k}^{\flat\flat}\] for $k = 0,1,2$.\qed
\end{cor}

\begin{defn}\label{def:epsiloncof}
  A morphism in $\DCatIT$ is an \emph{$\vepsilon$-cofibration} for
  $\vepsilon = (i,j)$ if it lies in the saturated class generated by
  $\{i\} \otimesde C_{k}^{\flat\flat} \to [1]^{\sharp} \otimesde
  C_{k}^{\flat\flat}$ for $k = 0,1,2$.
\end{defn}

\begin{remark}\label{rmk:loubaton fib}
  Loubaton~\cite{loubaton} defines (two cases of) fibrations of
  decorated $(\infty,\infty)$-categories as morphisms right orthogonal
  to generalizations of our $\ve$-cofibrations. His Theorem~3.2.2.24
  then characterizes these fibrations in terms of the existence of
  (co)cartesian lifts of decorated $i$-morphisms for all $i$, and so
  gives by adjunction a version of our criterion from \cref{thm:pb
    crit for partial fib} for $(\infty,\infty)$-categories.
\end{remark}

\begin{observation}\label{obs:epscofabs}
  We saw in \cref{obs:adjointDcatIT} that $\DCatIT$ is a presentable
  \icat{}. It therefore follows from \cite[Proposition
  5.5.5.7]{LurieHTT} that partial $\vepsilon$-fibrations form the
  right class in a factorization system on $\DCatIT$ whose left class
  consists of the $\vepsilon$-cofibrations. In particular, partial
  $\vepsilon$-fibrations are closed under base change, retracts, and
  limits in $\Ar(\DCatIT)$, and satisfy a cancellation property: if we
  have a commutative triangle of decorated \itcats{}
  \[
    \begin{tikzcd}
     \tEd \ar[rr, "f"] \ar[dr, "p"'] & & \tPd \ar[dl, "q"] \\
      & \tBd
    \end{tikzcd}
  \]
  such that $q$ is a partial $\vepsilon$-fibration, then $p$ is a
  partial $\vepsilon$-fibration \IFF{} $f$ is so. It follows, for
  example, that in this situation a morphism in $\tE$ is
  $p$-(co)cartesian over a decorated morphism in $\tBd$ \IFF{} it is $f$-(co)cartesian over a
  $q$-(co)cartesian morphism in $\tP$.
\end{observation}

\begin{observation}\label{obs:epscofoverB}
  It also follows from \cref{cor:partial fib right orth} that the full
  subcategory $\PFibe_{/\tBd} \subseteq \DCatITsl{\tBd}$ consists
  precisely of the objects that are local with respect to all
  morphisms of the form
    \begin{equation}
      \label{eq:gen eps equiv}
      \begin{tikzcd}
        \{i\} \otimesde C_{k}^{\flat\flat} \ar[rr] \ar[dr] & & {[1]}^{\sharp}
        \otimesde C_{k}^{\flat\flat} \ar[dl] \\
        & \tBd.
      \end{tikzcd}
    \end{equation}
  Since there is a set of equivalence classes of such morphisms and
  $\DCatIT$ is presentable, this implies that there exists a
  localization functor \[L^{\vepsilon}_{\tBd} \colon \DCatITsl{\tBd}
  \to \PFibe_{/\tBd},\] left adjoint to the inclusion.
\end{observation}

\begin{defn}
  We say that a morphism in $\DCatITsl{\tBd}$ is an
  \emph{$\vepsilon$-equivalence} over $\tBd$ if it is taken to an
  equivalence by $L^{\vepsilon}_{\tBd}$, or equivalently if it is
  in the strongly saturated class generated by the morphisms \cref{eq:gen eps equiv}.
\end{defn}

\begin{observation}\label{obs:cofib vs equiv}
  We make some elementary observations on $\vepsilon$-cofibrations and
  $\vepsilon$-equivalences:
  \begin{itemize}
  \item A morphism $f \colon \tAd \to \tBd$ is an
    $\vepsilon$-cofibration \IFF{} it is an $\vepsilon$-equivalence
    when viewed as a map $f \to \id_{\tBd}$ in $\DCatITsl{\tBd}$.
  \item For any decorated functor $p \colon \tAd \to \tBd$, the
    functor
    \[ p_{!} \colon \DCatITsl{\tAd} \to \DCatITsl{\tBd}\]
    given by composition with $p$ preserves $\vepsilon$-equivalences,
    since its right adjoint $p^{*}$ preserves partial
    $\vepsilon$-fibrations.
  \item If $p$ is itself a partial $\vepsilon$-fibration, then $p_{!}$
    furthermore reflects $\vepsilon$-equivalences: In this case
    $p_{!}$ restricts to a functor $\PFibe_{/\tAd} \to \PFibe_{/\tBd}$ left
    adjoint to pullback, so that we get an equivalence of left
    adjoints $L^{\vepsilon}_{\tBd}p_{!} \simeq
    p_{!}L^{\vepsilon}_{\tAd}$. Since equivalences in $\PFibe_{/\tBd}$
    are detected in $\CatIT$, it follows that for a morphism $f$ in
    $\DCatITsl{\tAd}$ we have that $L^{\vepsilon}_{\tAd}(f)$ is an
    equivalence \IFF{} $p_{!}L^{\vepsilon}_{\tAd}(f) \simeq
    L^{\vepsilon}_{\tBd}(p_{!}f)$ is an equivalence. Note that if
    $p$ factors as
    $\tAd \xto{q} \tXd \xto{r} \tBd$, then this implies
    that $q_{!}$ also detects $\ve$-equivalences.
  \item As a special case, this means that if a morphism
    \[
      \begin{tikzcd}
       \tAd \ar[rr, "f"] \ar[dr, "p"'] & & \tBd \ar[dl, "q"] \\
        & \tCd
      \end{tikzcd}
    \]
    in $\DCatITsl{\tCd}$ is an $\vepsilon$-equivalence and $q$ is a partial $\vepsilon$-fibration,
    then $f$ is an $\vepsilon$-cofibration.
  \end{itemize}
\end{observation}

\begin{propn}
  $[1]^{\sharp} \to [0]$ and $C_{2}^{\flat\sharp} \to [1]^{\flat}$ are
  $\vepsilon$-cofibrations for all $\vepsilon$.
\end{propn}
\begin{proof}
  Let $\vepsilon = (i,j)$. For the first map, we observe that the
  composition
  \[ \{i\} \to [1]^{\sharp} \to [0],\]
  is the identity, from which it follows by
  cancellation that $[1]^{\sharp} \to [0]$ must be an
  $\vepsilon$-cofibration. Similarly, from the composition
  \[ \{i\} \times [1]^{\flat} \to [1]^{\sharp} \otimesde [1]^{\flat} \to [1]^{\flat}\]
  we see that $[1]^{\sharp} \otimesde [1]^{\flat} \to [1]^{\flat}$ is an
  $\vepsilon$-cofibration. Since we also have a pushout
  \[
    \begin{tikzcd}
     {[1]^{\sharp} \amalg [1]^{\sharp}} \ar[r] \ar[d]
     &{[0] \amalg [0]}  \ar[d] \\
     {[1]^{\sharp} \otimesde [1]^{\flat}} \ar[r] & C_{2}^{\flat\sharp},
    \end{tikzcd}
  \]
  it again follows by cancellation that $C_{2}^{\flat\sharp} \to
  [1]^{\flat}$
  is an $\vepsilon$-cofibration.
\end{proof}

Since $\vepsilon$-cofibrations are closed under cobase change and colimits, we get:
\begin{cor}\label{obs:loc e-eqce}
  For any decorated \itcat{} $\tCd$, the morphism $\tCd \to
  \Ld(\tCd)^{\flat\flat}$ is an $\ve$-equivalence for any $\ve$, as is
  any pushout thereof, \ie{} any morphism obtained by inverting a
  collection of decorated 1- and 2-morphisms. \qed
\end{cor}

\begin{observation}\label{obs:lax homotopy eps-equiv}
  The projection $[1]^{\sharp} \otimesd \tDd \to \tDd$ is a
  localization at decorated 1- and 2-morphisms, and so is an
  $\ve$-cofibration by \cref{obs:loc e-eqce}. Hence the triangle
  \[
    \begin{tikzcd}
     {[1]^{\sharp}} \otimesd \tDd  \ar[rr, "\text{proj}"] \ar[dr, "p
     \circ \text{proj}"'] & &
     \tDd \ar[dl, "p"] \\
      & \tBd
    \end{tikzcd}
  \]
  is an $\ve$-equivalence over $\tBd$. By the 2-out-of-3 property, the
  two inclusions $\tDd \to [1]^{\sharp} \otimesd \tDd$ are also
  $\ve$-equivalences, and moreover become equivalent  after applying
  $L^{\ve}_{\tBd}$, since they both give inverses of the equivalence
  $L^{\ve}_{\tBd}(\text{proj})$. Hence any lax transformation
  \[ \alpha \colon [1]^{\sharp} \otimesd \tDd \to \tQd \]
  over $\tBd$ from $\alpha_{0}$ to $\alpha_{1}$ gives a natural
  equivalence between $L^{\ve}_{\tBd}(\alpha_{0})$ and
  $L^{\ve}_{\tBd}(\alpha_{1})$. We can therefore use $[1]^{\sharp}
  \otimesd (\blank)$ to
  define ``homotopies'' and ``homotopy equivalences'' that are in
  particular always $\ve$-equivalences. The following special case of
  this will be useful later:
\end{observation}

\begin{lemma}\label{lem: lax tr eps-equiv}
  Suppose $F \colon \tCd \to \tDd$ is a decorated functor over $\tBd$
  such that there exists $G \colon \tDd \to \tCd$ over $\tBd$ such
  that $GF \simeq \id_{\tCd}$ as well as a commutative triangle
  \[
    \begin{tikzcd}
     {[1]^{\sharp}} \otimesd \tDd  \ar[rr, "\rho"] \ar[dr] & & \tDd \ar[dl] \\
      & \tBd
    \end{tikzcd}
  \]
  where $\rho$ is a lax natural transformation between $\id_{\tDd}$
  and $FG$. Then for any decorated functor $f \colon \tAd \to \tBd$,
  the pullback $f^{*}F \colon f^{*}\tCd \to f^{*}\tDd$ is an
  $\ve$-equivalence over $\tAd$.
\end{lemma}
\begin{proof}
  Applying \cref{obs:lax homotopy eps-equiv} to $\rho$, we see that
  $L^{\ve}_{\tBd}(G)$ becomes an inverse of $L^{\ve}_{\tBd}(F)$, so
  that this is indeed an $\ve$-equivalence. Pulling back $\rho$ along
  $f$, we moreover get an analogous natural transformation for
  $f^{*}F$ as the composite
  \[ [1]^{\sharp} \otimesd f^{*}\tDd \to f^{*}([1]^{\sharp} \otimesd
    \tDd) \xto{f^{*}\rho} f^{*}\tDd,\]
  so the same argument shows that $f^{*}F$ is also a $\ve$-equivalence.
\end{proof}

As a consequence of \cref{cor:dec DFUNlax fib} we also get the
following class of $\ve$-cofibrations:
\begin{cor}
  Suppose $f \colon \tAd \to \tBd$ is an $\vepsilon$-cofibration and
  $\tKd$ is a decorated \itcat{} equipped with a marking $E$ such that
  all decorated 1-morphisms are marked. Then
  $f \otimesde_{\flat,E} \tKd$ is again an
  $\vepsilon$-cofibration. In particular, $f \times \tKd$ is an
  $\vepsilon$-cofibration for any $\tKd$. \qed
\end{cor}

To apply certain results proved in the setting of scaled simplicial
sets in our context, it will be useful to find an alternative set of
generating $\ve$-cofibrations in terms of 
orientals (see \cref{defn:orientalsprelim}), which are easier to work with in that model.
\begin{defn}\label{def:orientalhorns}
  Let $\mathbb{O}^{n}$ denote the 2-truncated $n$-dimensional
  \emph{oriental}. We
  define:
  \begin{itemize}
     \item A decorated \itcat{} $\mathbb{O}^{n,\diamond}$ by decorating the 1-morphism $(n-1) \to n$ and every 2-morphism determined by a subset inclusion of the form $\{i,n\} \subset \{i,n-1,n\}$.
     \item A decorated \itcat{} $\Lambda^n_i \mathbb{O}^{\diamond}$ for $0\leq i\leq n$ whose underlying \itcat{} is given by the colimit $\colim_{I \subset [n]}\mathbb{O}^{I}$ where $I$ ranges over the collection of non-empty subsets with the property that $[n]\setminus \{i\} \not\subset I$ and where a $k$-morphism is decorated if and only if its image under the canonical map $\Lambda^{n}_i\mathbb{O} \to \mathbb{O}^n$ is.
     \item A decorated \itcat{} $\mathbb{O}_+^{n,\dagger}$, whose underlying \itcat{} $\mathbb{O}^{n}_+$ is obtained by inverting every 2-morphism in $\mathbb{O}^n$ of the form $\{0,j\} \to \{0,1,j\}$, and where we decorate the 1-morphism $0 \to 1$.
     \item A decorated 2-category $\Lambda^n_0\mathbb{O}^{\dagger}_+$ whose underlying 2-category is obtained from $\Lambda^n_0 \mathbb{O}$ by collapsing the same 2-morphisms as above and where the decorations are induced by the functor $\Lambda^n_0\mathbb{O}_+ \to \mathbb{O}^n_+$. 
   \end{itemize} 
\end{defn}

\begin{defn}
  Let $[1]^{\flat,\otimes^{k}}$ denote the $k$-fold decorated Gray tensor product of $[1]^{\flat}$.  We define decorated 2-categories $\sqsubset^{\diamond}_k$ for $k=1,2,3$ as follows:
  \begin{itemize}
    \item For $k=1$ we denote $\sqsubset^{\diamond}_1=[0]$.
    \item For $k=2$ we set $\sqsubset^{\diamond}_2= [1]^{\flat} \amalg_{\partial [1]}([1]^{\sharp}\amalg [1]^{\sharp})$ where the map $\partial [1] \to [1]^{\sharp}\amalg [1]^{\sharp}$ is given by the disjoint union of the maps selecting the initial vertex.
    \item For $k=3$, we denote $\partial ([1]^{\flat,\otimes^2})=[2]^{\flat}\amalg_{0<2}[2]^{\flat}$ and define \[\sqsubset^{\diamond}_3= [1]^{\flat,\otimes^2} \amalg_{\partial([1]^{\flat,\otimes^2})} [1]^{\sharp}\otimes \partial([1]^{\flat,\otimes^2})\] where the attaching map $\partial([1]^{\flat,\otimes^2}) \to [1]^{\sharp}\otimesde \partial([1]^{\flat,\otimes^2})$ is induced by the map $\{0\} \to [1]^{\sharp}$.  
  \end{itemize}
  We note that we have natural decorated functors $\sqsubset^{\diamond}_k \to [1]^{\sharp}\otimesde [1]^{\flat,\otimes^{k-1}}$ for $k=2,3$.
\end{defn}

\begin{remark}\label{rem:posets}
  In several of the arguments below we will need to construct maps
  between orientals and iterated Gray products of $[1]$. For this
  purpose, it is useful to recall that such morphisms may be
  described, in the model of scaled simplicial sets, as maps of posets
  compatible with the relevant decorations, or scalings. Concretely,
  the $n$th oriental is modeled by (the nerve of) the poset $[n]$
  equipped with the minimal scaling, while $[1]^{\flat,\otimes^{k}}$
  is given by the $k$-fold cartesian product of $[1]$, endowed
  with its Gray tensor product scaling (cf. \cite{GHLGray}).
\end{remark}

\begin{lemma}\label{rem:box}
  The maps $\sqsubset^{\diamond}_k \to [1]^{\sharp}\otimes
  [1]^{\flat,\otimes^{k-1}}$ are $(0,1)$-cofibrations for
  $k=1,2,3$. Moreover, their saturation is the class of $(0,1)$-cofibrations.
\end{lemma}
\begin{proof}
  Let us observe that the saturated class of the morphisms
  \[
    \{0\}\times [1]^{\flat,\otimes^{k-1}} \xrightarrow{} [1]^{\sharp}\otimesde [1]^{\flat,\otimes^{k-1}}, \enspace k=1,2,3
  \]
  is precisely given by the $(0,1)$-cofibrations. Moreover, for each
  $k$ we have a factorization
  \[
    \{0\}\times [1]^{\flat,\otimes k-1} \to \sqsubset_{k}^{\diamond} \to [1]^{\sharp}\otimesde [1]^{\flat,\otimes k-1}.
  \]
  Here it is easy to check that the first map is in the saturated
  class of the morphisms
  $\sqsubset^{\diamond}_\ell \to [1]^{\sharp}\otimesde
  [1]^{\flat,\otimes^{\ell-1}}$ for $\ell<k$, so we see that this is a
  $(0,1)$-cofibration by working inductively on $k$. It then follows from
  cancellation that the second map must also be a
  $(0,1)$-cofibration. This argument also shows that the maps
  $\{0\}\times [1]^{\flat,\otimes^{k-1}} \xrightarrow{}
  [1]^{\sharp}\otimesde [1]^{\flat,\otimes^{k-1}}$ lie in the
  saturated class generated by these maps, which proves the converse.
\end{proof}

\begin{lemma}\label{lem:innerretract}
  Let $0<i<n$, then the maps $\Lambda^n_i \mathbb{O}^{\diamond} \to \mathbb{O}^{n,\diamond}$ are $(0,1)$-cofibrations.
\end{lemma}
\begin{proof}
  Let $[1]^{\times n}$ denote the $n$-fold cartesian product. There
  exists a functor \[r^{\prime}_n \colon [1]^{\times n} \to [n]\] given
  by $v=\{v_i\}_{i=1}^{n} \mapsto n - (\alpha_v -1)$ where
  $\alpha_v=\min\{i \enspace | \enspace v_i=1\}$ where we make the
  convention that $\alpha_v=0$ if $v_i=0$ for $i=1,2,\dots,n$. This
  map admits a section $i^{\prime}_n \colon [n] \to [1]^{\times n}$
  which sends $j$ to to the element $\{v_i\}_{i=1}$ with $v_i=0$ for
  $i \leq n-j$ and $v_i=1$ otherwise. These maps of partially ordered
  sets induce functors (see \cref{rem:posets}) of decorated 2-categories $i_n \colon
  \mathbb{O}^{n,\diamond} \to [1]^{\sharp}\otimesde [1]^{\otimesde
    n-1,\flat}$, $r_n \colon [1]^{\sharp}\otimesde [1]^{\otimes
    n-1,\flat} \to \mathbb{O}^{n,\diamond}$  such that $r_n \circ
  i_n=\id$.

  We conclude that we have a retract diagram
  \[\begin{tikzcd}
  {\mathbb{O}^{n-1,\flat}} & \{0\}\otimesde{[1]^{\otimes n-1,\flat}} & {\mathbb{O}^{n-1,\flat}} \\
  {\mathbb{O}^{n,\diamond}} & {[1]^{\sharp}\otimesde [1]^{\otimes n-1,\flat}} & {\mathbb{O}^{n,\diamond}}{.}
  \arrow[from=1-1, to=1-2]
  \arrow[from=1-1, to=2-1]
  \arrow[from=1-2, to=1-3]
  \arrow[from=1-2, to=2-2]
  \arrow[from=1-3, to=2-3]
  \arrow[from=2-1, to=2-2]
  \arrow[from=2-2, to=2-3]
\end{tikzcd}\]
  By the previous discussion we know that the maps $\mathbb{O}^{n-1,\flat} \to \mathbb{O}^{n,\diamond}$ are $(0,1)$-cofibrations. We consider the factorization $\mathbb{O}^{n-1,\flat} \to \Lambda^n_i\mathbb{O^{\diamond}} \to \mathbb{O}^{n,\diamond}$ and observe that the first map is a $(0,1)$-cofibration by an easy inductive argument on $n$. The result follows by cancellation.
\end{proof}

\begin{lemma}\label{lem:outeretract}
The morphisms $\Lambda^n_0\mathbb{O}^{\dagger} \to \mathbb{O}^{n,\dagger}_+$ for $n=1,2,3$ are $(0,1)$-cofibrations.
\end{lemma}
\begin{proof}
It suffices to show that there exists a retract diagram,
\[\begin{tikzcd}
 {\Lambda^{n}_{0}\mathbb{O}^{\dagger}_{+}} & {\sqsubset^{\diamond}_k} & {\Lambda^{n}_{0}\mathbb{O}^{\dagger}_{+}} \\
 {\mathbb{O}^{n,\dagger}_{+}} & {[1]^{\sharp}\otimes[1]^{\flat,\otimes^{n-1}}} & {\mathbb{O}^{n,\dagger}_{+}}
 \arrow[from=1-1, to=1-2]
 \arrow[from=1-1, to=2-1]
 \arrow[from=1-2, to=1-3]
 \arrow[from=1-2, to=2-2]
 \arrow[from=1-3, to=2-3]
 \arrow["\iota_n", from=2-1, to=2-2]
 \arrow["r_k", from=2-2, to=2-3]
\end{tikzcd}\]
for $n=1,2,3$. We only deal with the case $n=3$ as the remaining cases are similar and easier. The map $\iota_3$ is induced by the functor of posets $i^{\prime}_3 \colon [3] \to [1]\times [1]\times [1]$ (see \cref{rem:posets}) given by
\[
  (0,0,0) \to (1,0,0) \to (1,0,1) \to (1,1,1).
\]
We consider a map $T \colon [1] \times [1]\times [1] \to [1] \times [1]\times [1]$ depicted graphically as,
\[\begin{tikzcd}
  & {(1,0,0)} && {(1,1,1)} \\
  {(0,0,0)} && {(1,0,0)} \\
  & {(1,1,1)} && {(1,1,1)} \\
  {(1,0,1)} && {(1,0,1)}
  \arrow[from=1-2, to=1-4]
  \arrow[from=1-2, to=3-2]
  \arrow[from=1-4, to=3-4]
  \arrow[from=2-1, to=1-2]
  \arrow[from=2-1, to=2-3]
  \arrow[from=2-1, to=4-1]
  \arrow[from=2-3, to=1-4]
  \arrow[from=2-3, to=4-3]
  \arrow[from=3-2, to=3-4]
  \arrow[from=4-1, to=3-2]
  \arrow[from=4-1, to=4-3]
  \arrow[from=4-3, to=3-4]
\end{tikzcd}\]
and define $r^{\prime}_3$ as its factorization through $[3]$. The map $r^{\prime}_3$ induces the desired functor $r_3$ in our statement. 
\end{proof}

\begin{propn}\label{prop:01equivorientals}
  The following collection of maps generates
  the $(0,1)$-cofibrations as a saturated class:
  \begin{enumerate}[(i)]
    \item  $\Lambda^n_i \mathbb{O}^{\diamond} \to \mathbb{O}^{n,\diamond}$ for $n=2,3$ and $0<i<n$. 
    \item $\Lambda^n_0\mathbb{O}^{\dagger} \to \mathbb{O}^{n,\dagger}_+$ for $n=1,2,3$.
  \end{enumerate}
\end{propn}
\begin{proof}
  By \cref{rem:box} it will be enough to show that the maps in the statement are $(0,1)$-cofibrations and that the morphisms 
  \begin{equation}\label{eq:orientals}
    \sqsubset^{\diamond}_k \to [1]^{\sharp}\otimesde [1]^{\otimes k-1}, \enspace k=1,2,3
  \end{equation}
  lie in their saturated hull. The first claim comes from \cref{lem:innerretract} and \cref{lem:outeretract}. 

  Now we address the second claim. The case $k=1$ is trivial. For the case $k=2$, we can produce a filtration which we diagrammatically depict below,

   \[
    \begin{tikzcd}
  \bullet \\
  \bullet
  \arrow[from=1-1, to=2-1]
\end{tikzcd}
\hookrightarrow
\begin{tikzcd}
  \bullet & \bullet \\
  \bullet & \bullet
  \arrow[mid vert, from=1-1, to=1-2]
  \arrow[from=1-1, to=2-1]
  \arrow[mid vert, from=2-1, to=2-2]
\end{tikzcd} 
\hookrightarrow
\begin{tikzcd}
  \bullet & \bullet \\
  \bullet & \bullet
  \arrow[mid vert, from=1-1, to=1-2]
  \arrow[from=1-1, to=2-1]
  \arrow[""{name=0, anchor=center, inner sep=0}, from=1-1, to=2-2]
  \arrow[mid vert, from=2-1, to=2-2]
  \arrow[ Rightarrow, from=0, to=2-1]
\end{tikzcd} 
\hookrightarrow
\begin{tikzcd}
  \bullet & \bullet \\
  \bullet & \bullet
  \arrow[from=1-1, to=1-2, mid vert]
  \arrow[from=1-1, to=2-1]
  \arrow[""{name=0, anchor=center, inner sep=0}, from=1-1, to=2-2]
  \arrow[from=1-2, to=2-2]
  \arrow[from=2-1, to=2-2, mid vert]
  \arrow[Rightarrow, from=0, to=2-1]
\end{tikzcd}
\]
where the first step is obtained by taking a pushout along a morphism
of type (ii), the second along a morphism of type (i) and the final
step is obtained by taking a pushout along a morphism of type
(ii). (Note the use of barred arrows to denote decorated
1-morphisms.)

For the case $k=3$, we consider the same filtration as above but after taking the decorated Gray tensor product with $[1]^{\flat}$ on the right. This reduces our claim to showing that the maps
\[
     \Lambda^2_1 \mathbb{O}^{\diamond}\otimesde [1]^{\flat} \to \mathbb{O}^{2,\diamond}\otimesde [1]^{\flat}, \quad \quad
     \Lambda^2_0\mathbb{O}^{\dagger}_+\otimesde [1]^{\flat}  \to \mathbb{O}^{2,\dagger}_+\otimesde [1]^{\flat}.
  \]
 can be expressed as an iterated pushouts of morphisms of type (i) and (ii). The corresponding filtrations can be directly adapted from \cite[Proposition 2.16]{GHLGray} after recalling that in the model structure of scaled simplicial sets the minimally scaled $n$-simplex corresponds to the $n$-th oriental. 
\end{proof}

\section{Free fibrations and pushforwards for $(\infty,2)$-categories}

When setting up the theory of \icats{}, it turns out to be
surprisingly useful to know that the forgetful functor from
(co)cartesian fibrations on a fixed base $\oB$ to \icats{} over $\oB$
has a left adjoint, the \emph{free} (co)cartesian fibration functor,
which has a concrete description as a pullback of the arrow \icat{} of
$\oB$. Our first main goal in this section is to prove the \itcatl{}
version of this result, in the general context of partial
fibrations. We start by showing that the evaluation functors from
$\ARplax(\tB)$ to $\tB$ are fibrations in \S\ref{sec:ARoplax fibs},
and then use these to construct free partial fibrations in
\S\ref{sec:free fib}. In \S\ref{sec:free dec} we extend this result
slightly to describe free \emph{decorated} fibrations, which we use in
\S\ref{sec:free fib on dec} to describe the left adjoint to the
forgetful functor from fibrations over $\tB$ to decorated \itcats{}
over $\tBss$. We then turn to the second main goal of this section,
which is to describe right adjoints to pullback. In \S\ref{sec:smooth}
we show that pullback along a decorated fibration $\tA^{\diamond} \to
\tD^{\sharp \sharp}$ on slices of decorated \itcats{} has a right
adjoint, and this preserves fibrations of the opposite variance. We
then use this in \S\ref{sec:cofree} to identify the right adjoint to
pullback on partial fibrations along an arbitrary such functor; this
includes the case of \emph{cofree} fibrations. Finally, in
\S\ref{sec:loc fib} we apply these results to identify the fibration
for the localization of a functor to decorated \itcats{}.

\subsection{Fibrations from (op)lax arrows}
\label{sec:ARoplax fibs}
If $\oC$ is an \icat{}, then $\ev_{i} \colon \Ar(\oC) \to \oC$ is a
cocartesian fibration for $i = 1$ and a cartesian fibration for $i =
0$, with the (co)cartesian morphisms precisely those that go to
equivalences under the opposite evaluation. Our goal in this section
is to generalize this statement to decorated \itcats{}.

\begin{notation}
  For $\vepsilon = (i,j)$ and $\tCd$ a decorated \itcat{}, let \[\DARelax(\tCd)^{\flat[\vepsilon]}\] denote the decoration of
  $\DARelax(\tCd)$ where
  \begin{itemize}
  \item an (op)lax square is a decorated 1-morphism
    \IFF{} it commutes,
    its image under $\ev_{i}$ is an equivalence, and its image under
    $\ev_{1-i}$ is decorated;
  \item a 2-morphism is decorated \IFF{} its image under $\ev_{i}$
    is an equivalence and its image under $\ev_{1-i}$ is decorated.
  \end{itemize}
\end{notation}

\begin{thm}\label{thm:ev is partial fib} 
  For any decorated \itcat{} $\tCd$ and $\vepsilon = (i,j)$, the
  functor
  \[ \ev_{1-i} \colon \DARelax(\tCd)^{\flat[\vepsilon]} \to \tCd \]
  is a partial $\vepsilon$-fibration.
\end{thm}

\begin{remark}
  The fully decorated case of this theorem has already been proved by
  Gagna, Harpaz, and Lanari as \cite[Theorem 3.0.7]{GHLFib}.
\end{remark}

\begin{remark}
  Less compactly, this means that
  \begin{itemize}
  \item $\ev_{1} \colon \DARoplax(\tCd) \to \tC$ is a partial
    $(0,1)$-fibration,
  \item $\ev_{0} \colon \DARoplax(\tCd) \to \tC$ is a partial
    $(1,0)$-fibration,
  \item $\ev_{1} \colon \DARlax(\tCd) \to \tC$ is a partial
    $(0,0)$-fibration,
  \item $\ev_{0} \colon \DARlax(\tCd) \to \tC$ is a partial
    $(1,1)$-fibration,
  \end{itemize}
  all with respect to the given decoration $\tCd$ on $\tC$. We also
  have that
    \begin{itemize}
  \item an (op)lax square is an $i$-cartesian morphism \IFF{} it commutes
    and its image under $\ev_{i}$ is an equivalence,
  \item a 2-morphism is $j$-cartesian \IFF{} its image under $\ev_{i}$
    is an equivalence.
  \end{itemize}
\end{remark}

We will prove this using the criterion of \cref{thm:pb crit for
  partial fib}, for which we will need the following result on
localizations of \itcats{}:
\begin{propn}\label{propn:locn loc full incl}
  Suppose $F \colon \tC \to \tC'$ is a localization at certain 1- and
  2-morphisms. Then $F^{*} \colon \FUN(\tC', \tD)^{\plax} \to
  \FUN(\tC, \tD)^{\plax}$ is a locally full inclusion for any \itcat{} $\tD$.
\end{propn}

For this, we in turn need some preliminary results on conservative
functors:

\begin{observation}\label{obs:conservative pbs}
 Given a functor of \itcats{} $F \colon \tA \to \tB$, consider the
 commutative square
 \[
   \begin{tikzcd}
    \tA^{\leq i} \ar[r] \ar[d, "F^{\leq i}"'] & \tA^{\leq j}  \ar[d,
    "F^{\leq j}"] \\
    \tB^{\leq i} \ar[r] & \tB^{\leq j}
   \end{tikzcd}
 \]
 for $0 \leq i < j \leq 2$ (with $(\blank)^{\leq 2}= \id$). We have:
  \begin{enumerate}[(i)]
  \item $F$ is conservative on 2-morphisms \IFF{} the square is a
    pullback for $i = 1, j = 2$.
  \item $F$ is conservative on 1-morphisms \IFF{} the square is a
    pullback for $i = 0, j = 1$.
  \item $F$ is conservative on both 1- and 2-morphisms \IFF{} the
    square is a pullback for $i = 0,1$ and $j = 2$.
  \end{enumerate}
\end{observation}

\begin{propn}\label{propn:ess surj gives 2-cons}\ 
  \begin{enumerate}[(i)]
  \item The commutative square
    \[
      \begin{tikzcd}
       \partial ([1] \otimes [1]) \otimes ([1] \otimes [1]) \ar[r] \ar[d] &
       {[1]^{\otimes 4}}  \ar[d] \\
       \partial ([1] \otimes [1]) \otimes [1] \ar[r] & ([1] \otimes [1]) \otimes [1]
      \end{tikzcd}
    \]
    is a pushout.
  \item For any \itcat{} $\tC$, the functor
    \[\FUN([1] \otimes [1], \tC)^{\plax} \to \FUN(\partial ([1] \otimes
    [1]), \tC)^{\plax}\]
    is conservative on 2-morphisms.
  \item Suppose $F \colon \tA \to \tB$ is an essentially surjective
    functor of \itcats{}. Then
    $\FUN(\tB, \tC)^{\plax} \to \FUN(\tA, \tC)^{\plax}$ is
    conservative on 2-morphisms for all $\tC$.
  \end{enumerate}
\end{propn}
\begin{proof}
  A functor is conservative on 2-morphisms \IFF{} it is right
  orthogonal to $C_{2} \to [1]$. This is
  equivalent to being right orthogonal to $[1] \otimes [1] \to [1]
  \times [1]$, as each is a cobase change of the other.
  From \cite[2.4.4]{AGH24} we know that $\AR^{\oplax}(\tC) \to \tC
  \times \tC$ is conservative on 2-morphisms for any $\tC$, which is then
  equivalent to the commutative square
  \begin{equation}
    \label{eq:po for Aroplax 2-cons}
    \begin{tikzcd}
      {[1] \otimes [1]} \times \partial [1] \ar[r] \ar[d] &
      {[1]^{\otimes 3}}  \ar[d] \\
      {[1] \times [1]} \times \partial [1] \ar[r] & {([1] \times [1])} \otimes [1]
    \end{tikzcd}
  \end{equation}
  being a pushout.

  To prove that the square in (i) is a pushout, we first consider
  \[
    \begin{tikzcd}
      {([1] \otimes [1])^{\amalg 4}} \ar[r] \ar[d] & {[1]^{\otimes 2}}
      \otimes \partial ([1]^{\otimes 2)}
       \ar[d] \ar[r] & {[1]^{\otimes 4}} \ar[d] \\
      {([1] \times [1])^{\amalg 4}} \ar[r] &  {[1]^{\times 2}}
      \otimes \partial ([1]^{\otimes 2}) 
     \ar[r] & {[1]^{\times 2}} \otimes [1]^{\otimes 2},
    \end{tikzcd}
  \]
  where the left square is obtained by gluing 4 copies of \cref{eq:po
    for Aroplax 2-cons} and so is a pushout. To see that the
  right-hand square is a pushout it therefore suffices to show that
  the outer composite square is a pushout. Next, we factor this square
  instead as
  \[
    \begin{tikzcd}
      {([1] \otimes [1])^{\amalg 4}} \ar[r] \ar[d] & {[1]^{\otimes 3}}
      \times \partial [1]     
       \ar[d] \ar[r] & {[1]^{\otimes 4}} \ar[d] \\
      {([1] \times [1])^{\amalg 4}} \ar[r] &  {[1]^{\times 2}} \otimes
      [1] \times \partial [1]
     \ar[r] & {[1]^{\times 2}} \otimes [1]^{\otimes 2}.
    \end{tikzcd}
  \]  
  Here the left square is a coproduct of two copies of \cref{eq:po
    for Aroplax 2-cons} and the right square is \cref{eq:po for
    Aroplax 2-cons} tensored with $[1]$, so that both are
  pushouts.
  
  The orthogonality condition in (ii) is now immediate from the square
  in (i) being a pushout. To prove (iii), first recall (\eg{} from \cite[Theorem 5.3.7]{Soergel}) that
  essentially surjective and fully faithful functors form a
  factorization system on $\CatIT$, and that a functor is fully
  faithful \IFF{} it is right orthogonal to $\partial C_{i} \to C_{i}$
  for $i = 1,2$ by \cite[Proposition 2.5.10]{AGH24}. Hence the essentially surjective functors are the
  smallest saturated class containing these two maps. On the other
  hand, the class of maps for which the desired conclusion holds is
  clearly saturated, so it suffices to prove it in these two
  cases. The map $\partial C_{2} \to C_{2}$ generates the same
  saturated class as $\partial [1]^{\otimes 2} \to [1]^{\otimes 2}$ by
  \cite[2.5.13]{AGH24}, so this case follows from part (ii), while the
  case $i = 1$ is part of \cite[2.4.4]{AGH24}.
\end{proof}

\begin{proof}[Proof of \cref{propn:locn loc full incl}]
  We prove the lax case; the oplax case follows similarly, or by using
  equivalences of the form
  $\FUN(\tA, \tB)^{\oplax} \simeq \FUN(\tA^{\op},\tB^{\op})^{\lax,
    \op}$.  By definition the functor $F$ is obtained by inverting
  certain 1- and 2-morphisms in $\tC$. Since the class of functors
  for which the desired conclusion is true is clearly closed under
  cobase change, it suffices to show that it holds for the functors
  $\ell_{i} \colon \tC \to \tau_{\leq i}\tC$ for $i = 0,1$ given by localizing
  $\tC$ to an $(\infty,i)$-category. The condition that
  $\ell_{i}^{*} \colon \FUN(\tau_{\leq i}\tC, \tD)^{\lax} \to \FUN(\tC,
  \tD)^{\lax}$ is right orthogonal to some functor
  $\phi \colon \tA \to \tB$ is equivalent to the commutative square
  \[
    \begin{tikzcd}
     \Map(\tC, \FUN(\tB, \tD)^{\oplax,\leq i}) \ar[r] \ar[d] & \Map(\tC, \FUN(\tB, \tD)^{\oplax})  \ar[d] \\
     \Map(\tC, \FUN(\tA, \tD)^{\oplax,\leq i}) \ar[r] & \Map(\tC, \FUN(\tA, \tD)^{\oplax})
    \end{tikzcd}
  \]
  being a pullback, which is true for all $\tC$ \IFF{}
  \[
    \begin{tikzcd}
     \FUN(\tB,\tD)^{\oplax,\leq i} \ar[r] \ar[d] & \FUN(\tB,\tD)^{\oplax}  \ar[d] \\
     \FUN(\tA,\tD)^{\oplax,\leq i} \ar[r] & \FUN(\tA,\tD)^{\oplax}
    \end{tikzcd}
  \]
  is a pullback of \itcats{}. By \cref{obs:conservative pbs}, this
  holds for both $i=0,1$ \IFF{}
  $\phi^{*} \colon \FUN(\tB, \tD)^{\oplax} \to \FUN(\tA,\tD)^{\oplax}$
  is conservative on 1- and 2-morphisms.

  By \cite[2.5.8]{AGH24}, a functor is a locally full inclusion \IFF{}
  it is right orthogonal to $\partial [1]\to [0]$ and
  $\partial C_{2} \to C_{2}$. Applying the preceding discussion to
  these maps, we see that for the statement we want it is enough to
  prove that the functors
  \[ \tD \to \tD^{\times 2}, \quad \FUN(C_{2}, \tD)^{\oplax} \to
    \FUN(\partial C_{2}, \tD)^{\oplax}\]
  are conservative on 1- and 2-morphisms. For the first map this holds
  since the inclusions $\tD^{\leq i} \to \tD$ are monomorphisms, so
  that we have pullbacks
  \[
    \begin{tikzcd}
     \tD^{\leq i} \ar[r] \ar[d] & \tD  \ar[d] \\
     \tD^{\leq i, \times 2} \ar[r] & \tD^{\times 2}.
    \end{tikzcd}
  \]
  For the second we saw in \cref{partial C2 cons on 1-mor} that it is
  conservative on 1-morphisms, and we just proved in \cref{propn:ess
    surj gives 2-cons} that it is so on 2-morphisms.
\end{proof}

\begin{observation}\label{obs:cons on lax transfos}
  The proof of \cref{propn:locn loc full incl} shows that for any
  \itcat{} $\tC$, the functor
  \[ F^{*} \colon \FUN(\tB, \tC)^{\plax} \to \FUN(\tA, \tC)^{\plax}\]
  is conservative on 1- and 2-morphisms provided $F$ is in the
  saturated class generated by $\partial C_{2} \to C_{2}$ and
  $\partial [1] \to [0]$. By \cite[Theorem 5.3.7]{Soergel}, the
  saturated class generated by the first map consists of those
  functors $F$ that are surjective on objects with the further
  property that $\tA(a,a') \to \tB(Fa,Fa')$ is surjective on objects
  for all $a,a' \in \tA$, so all such maps have this conservativity
  property. It would be interesting to identify also the larger class
  obtained by adding our second generator, \ie{} the class of maps
  that are left orthogonal to locally full inclusions rather than all
  locally fully faithful functors.
\end{observation}

\begin{propn}\label{propn:it arelax pb}
  The commutative square
    \[
    \begin{tikzcd}
     \DARelax(\DARelax(\tCd)^{\flat[\vepsilon]}) \ar[r] \ar[d] & \DARelax(\tCd)  \ar[d, "\ev_{1}"] \\
     \DARelax(\tCd) \ar[r, "\ev_{0}"'] & \tC
    \end{tikzcd}
  \]
  is a pullback of \itcats{}, giving an equivalence
  \[ \DARelax(\DARelax(\tCd)^{\flat[\vepsilon]}) \simeq \DFUN([2]^{\sharp},
    \tCd)^{\velax}\]
  of \itcats{}.
\end{propn}
\begin{proof}
  We prove the case $\vepsilon = (0,1)$. Since
  $\DFUN(\blank, \tCd)^{\oplax}$ preserves limits, we can identify the
  pullback in the square as $\DFUN([2]^{\sharp}, \tCd)^{\oplax}$. On
  the other hand, from \cref{lax oplax reverse marked} we know that
  $\DAR^{\oplax}(\DAR^{\oplax}(\tCd)^{\flat[(0,1)]})$ is a locally
  full subcategory of
  $\AR^{\oplax}(\AR^{\oplax}(\tC)) \simeq \FUN([1] \otimes [1],
  \tC)^{\oplax}$, whose
  \begin{itemize}
  \item objects are oplax squares that commute and whose top edge is
    invertible, which we can identify with functors $[2] \to \tC$,
  \item morphisms can similarly be identified with functors $[2]
    \otimes [1] \to \tC$.
  \end{itemize}
  It therefore suffices to show that the functor
  $p^{*} \colon \FUN([2], \tC)^{\oplax} \to \FUN([1] \otimes [1], \tC)^{\oplax}$,
  given by composition with the functor
  $p \colon [1] \otimes [1] \to [2]$ that inverts the 2-morphism and
  the top edge, is a locally full subcategory inclusion, which follows
  from \cref{propn:locn loc full incl}.
\end{proof}

\begin{proof}[Proof of \cref{thm:ev is partial fib}]
  Combine \cref{thm:pb crit for partial
    fib} with \cref{propn:it arelax pb}.
\end{proof}

We also have the following decorated variant of the theorem:

\begin{cor}\label{propn:ev dec fib}
  For any decorated \itcat{} $\tCd$ and $\vepsilon = (i,j)$, the
  functor
  \[ \ev_{1-i} \colon \DARelax(\tCd)^{\diafib} \to \tCd \]
  is a decorated partial $\vepsilon$-fibration.
\end{cor}
\begin{proof}
  Unpacking the definitions, we see that
  \[ \DARelax(\tCd)^{\diafib}_{(2)} \simeq
    \DARelax(\tCd_{(2)}),\]
  where we think of $\tCd_{(2)}$ as a decorated \itcat{} with the
  decorations inherited from $\tCd$, and
  \[ \DARelax(\tCd)^{\diafib, \leq 1}_{(1)} \simeq
    \Ar(\tC^{\diamond,\leq 1}_{(1)}).\]
  Thus the restricted functor
  \[ \ev_{1-i} \colon \DARelax(\tCd)^{\diafib}_{(2)} \to
    \tCd_{(2)} \]
  is also a partial $\vepsilon$-fibration by \cref{thm:ev is partial
    fib}, with its $i$-cartesian morphisms and $j$-cartesian
  2-morphisms agreeing with those in $\DARelax(\tCd)$,
  while  
  \[ \ev_{1-i} \colon  \DARelax(\tCd)^{\diafib, \leq 1}_{(1)} \to
    \tC^{\diamond,\leq 1}_{(1)}\]
  is an $i$-fibration of \icats{} whose $i$-cartesian morphisms agree
  with those in $\DARelax(\tCd)$.
\end{proof}

\subsection{Free partial fibrations}
\label{sec:free fib}

For a functor of \icats{} $f \colon \oC \to \oB$, we can identify the
free cocartesian fibration on $f$ as the functor $\oC \times_{\oB}
\Ar(\oB) \to \oB$ given by evaluation at $1$, where the pullback is
formed using $f$ and evaluation at $0$. This was first proved in
\cite{GHN}, with improved versions of the proof later given in
\cite{AyalaMazelGeeRozenblyumEnr,Shah}. In this section we will prove
the analogue of this statement for \itcats{}, and more generally
describe free \emph{partial} fibrations as follows:

\begin{thm}\label{thm:free partial fib}
  Let $\tBd$ be a decorated \itcat{}. The forgetful functor
  \[\UeB \colon \PFibe_{/\tBd} \to \CatITsl{\tB}\] has a left adjoint
  $\FeB$, given for $p \colon \tE \to \tB$ by
  \[ \FeB(p) \,\,:=\,\, \tE^{\flat\flat}
    \times_{\tB^{\flat\flat}} \DARelax(\tBd)^{\flat[\vepsilon]} \to \tBd, \]
  where the pullback is via $\ev_{i}$  and the map to $\tBd$ is
  given by $\ev_{1-i}$.
\end{thm}

\begin{remark}
  The case of the theorem where $\tB$ is fully decorated and
  $\ve = (1,0)$ has previously been proved by Abell\'an and Stern as
  \cite[Theorem 3.17]{AS23}.
\end{remark}

\begin{construction}\label{cons:free}
  More precisely, we define the functor $\FeB$ as
  the composite
  \[ \CatITsl{\tB} \simeq \PFibe_{/\tB^{\flat\flat}} \xto{\ev_{i}^{*}}
    \PFibe_{/\DARelax(\tBd)^{\flat[\vepsilon]}} \xto{\ev_{1-i,!}}
    \PFibe_{/\tBd}\] where we have used that partial
  $\vepsilon$-fibrations are closed under base change along any
  functor, such as $\ev_{i}$, and under composition with a partial
  $\vepsilon$-fibration, such as $\ev_{1-i}$ (\cref{thm:ev is partial
    fib}).
\end{construction}

The unit of the adjunction is easy to define:
\begin{construction}
  The degeneracy $[1] \to [0]$ induces a commutative triangle
  \[
    \begin{tikzcd}
     \tB \ar[rr, "s_{0}^{*}"] \ar[dr] & & \DARelax(\tBd) \ar[dl, "\ev_{1-i}"] \\
      & \tB
    \end{tikzcd}
  \]
  of \itcats{}. By pulling this back we obtain a natural transformation
  \[ \etaeB \colon \id \to \UeB
    \FeB.\]
  That is, for $p \colon \tE \to \tB$, the map $\eta^{\vepsilon}(p)$
  is the commutative triangle
  \[
    \begin{tikzcd}
     \tE \ar[rr, "\etaeB(p)"] \ar[dr, "p"'] & & \tE \times_{\tB} \DARelax(\tBd) \ar[dl] \\
      & \tB
    \end{tikzcd}
  \]
  with the horizontal map given by composing with $s_{0}$ in the
  second factor.
\end{construction}

When $p$ is a partial $\vepsilon$-fibration over $\tBd$, we can use
\cref{thm:pb crit for partial fib} to give an alternative
identification of $\FeB(p)$; this allows us to define the counit of
the adjunction and prove one of the triangle identities:
\begin{propn}\label{propn:free fib unit counit constr}
Suppose $p \colon \tE \to \tB$ is a partial $\vepsilon$-fibration over
$\tBd$. 
\begin{enumerate}[(i)]
\item \cref{thm:pb crit for partial fib} extends to a natural equivalence
  \[\FeB(p) \simeq \DARelax(\tEn)^{\flat[\vepsilon]}
    \xto{p \circ \ev_{1-i}} \tBd \]
  of decorated \itcats{}.
\item Under this equivalence the map
  $\etaeB \colon p \to \UeB
  \FeB(p)$ corresponds to the degeneracy map
  \[ s_{0}^{*} \colon \tE \to \DARelax(\tEn).\]
\item $\ev_{1-i}$ gives a natural morphism of decorated \itcats{}
  \[\DARelax(\tEn)^{\flat[\vepsilon]} \to \tEn\]
  over $\tBd$, and so a natural transformation
  \[\epseB
\colon \FeB\UeB \to
\id.\]
\item The composite
  \[ \tE \xto{\etaeB} \FeB(p)
    \xto{\epseB} \tE\]
  is the identity, \ie{} the composite natural transformation
  \[ \UeB \xto{\etaeB\UeB} \UeB\FeB\UeB \xto{\UeB\epseB} \UeB\]
  is the identity of $\UeB$.
\end{enumerate}
\end{propn}
\begin{proof}
  We prove the case $\vepsilon = (0,1)$ to simplify the notation.
  Then \cref{lem:dec pb for dec fib} gives an equivalence of decorated
  \itcats{}
  \[ \DARoplax(\tEn)^{\diafib} \isoto \DARoplax(\tBd)^{\diafib}
    \times_{\tBd} \tEn. \]
  Pulling this back along $\tE^{\flat\flat} \to \tEn$ then gives the
  desired equivalence
  \[ \DARoplax(\tEn)^{\flat[(0,1)]} \isoto
    \FFree^{(0,1)}_{\tBd}(p)\]
  in (i).

  Part (ii) is immediate from the commutative diagram
  \[
    \begin{tikzcd}
     \tE \ar[r, "s_{0}^{*}"'] \ar[d, "p"'] \ar[rr, bend left, "="] & \DARoplax(\tEn)  \ar[d,
     "\DARoplax(p)"] \ar[r, "\ev_{0}"'] & \tE \ar[d, "p"] \\
     \tB \ar[r, "s_{0}^{*}"] \ar[rr, bend right, "="']  & \DARoplax(\tBd) \ar[r, "\ev_{0}"] & \tB,
    \end{tikzcd}
  \]
  where the right-hand square is a pullback.

  For (iii), we observe that for a decorated 1-morphism in
  $\DARoplax(\tEn)^{\flat[(0,1)]}$, which is of the form
  \[
    \begin{tikzcd}
     \bullet \ar[r, equals] \ar[d, mid vert] & \bullet  \ar[d, mid vert] \\
     \bullet \ar[r]  & \bullet,
    \end{tikzcd}
  \]
  the cancellation property of cocartesian morphisms implies that the
  bottom horizontal morphism must also be cocartesian, so that
  $\ev_{1}$ preserves decorated 1-morphisms. The cancellation property
  of cartesian 2-morphisms similarly implies that $\ev_{1}$ also
  preserves decorated 2-morphisms, which completes the proof.

  Part (iv) now follows from the fact that under the identication of
  (ii), the composite $\tE \to \FFree^{(0,1)}_{\tBd}(p) \to \tE$ is identified with
  \[ \tE \xto{s_{0}^{*}} \DARoplax(\tEn) \xto{\ev_{1}} \tE,\]
  which is manifestly the identity.
\end{proof}

It remains to check the other triangle identity, for which we need to
identify the composite $\FeB\UeB\FeB$ more explicitly:
\begin{observation}
  Suppose $\ve = (i,j)$ and set
  \[ i' =
    \begin{cases}
      0, & i = 0,\\
      2, & i = 1.
    \end{cases}
  \]
  We define $\DFUN([2]^{\sharp}, \tBd)^{\velax,\flat[\vepsilon]}$ to
  be equipped with the decorations where
  \begin{itemize}
  \item a 1-morphism is decorated if it is a strong natural
    transformation such that its images under $\ev_{i'}$ and $\ev_{1}$
    are invertible and its image under $\ev_{2-i'}$ is decorated,
  \item a 2-morphism is decorated if its images under $\ev_{i'}$ and $\ev_{1}$
    are invertible and its image under $\ev_{2-i'}$ is decorated.
  \end{itemize}
  For $p \colon \tE \to \tB$, iterating the definition of
  $\FeB(\blank)$ as a pullback then gives an equivalence
  \begin{equation}
    \label{eq:free of free}
    \begin{split}
      \FeB\UeB\FeB(p)
      & \simeq (\tE \times_{\tB} \DARelax(\tBd))^{\flat\flat}
        \times_{\tB^{\flat\flat}} \DARelax(\tBd)^{\flat[\vepsilon]} \\
      & \simeq
        \tE^{\flat\flat}\times_{\tB^{\flat\flat}} \DFUN([2]^{\sharp},
        \tBd)^{\velax,\flat[\vepsilon]},
    \end{split}
  \end{equation}
  where the pullback is taken via $\ev_{i'}$ and the map to $\tB$ is
  given by $\ev_{2-i'}$. On the other hand, since $\FeB(p)$ is an
  $\vepsilon$-fibration, \cref{propn:free fib unit counit constr}
  implies that we have an equivalence
  \[ \FeB\UeB\FeB(p) \simeq \DARelax(\FeB(p))^{\flat[\vepsilon]}.\]
  Unpacking the definition of $\FeB(p)$, we see that the right-hand side is
  equivalent to the pullback
  \[ \DARelax(\tE^{\flat\flat})^{\flat[\vepsilon]}
    \times_{\DARelax(\tB^{\flat\flat})^{\flat[\vepsilon]}}
    \DARelax(\DARelax(\tBd)^{\flat[\vepsilon]})^{\flat[\vepsilon]},\]
  which we can identify with
  \[ \tE^{\flat\flat} \times_{\tB^{\flat\flat}} \DFUN([2]^{\sharp},
    \tBd)^{\velax,\flat[\vepsilon]}\]
  using \cref{propn:it arelax pb}. We claim that the resulting
  equivalence
  \begin{equation}
    \label{eq:free of free v2}
    \begin{split}
      \FeB\UeB\FeB(p) & \simeq \DARelax(\FeB(p))^{\flat[\vepsilon]} \\
                      & \simeq \tE^{\flat\flat} \times_{\tB^{\flat\flat}} \DFUN([2]^{\sharp},
    \tBd)^{\velax,\flat[\vepsilon]},
    \end{split}
  \end{equation}
  is the same as \cref{eq:free of free}. Indeed, this second
  equivalence arises from the commutative
  diagram
  \[
    \begin{tikzcd}
     \DARelax(\FeB(p))^{\flat[\vepsilon]} \ar[r] \ar[d] &
     \DARelax(\DARelax(\tBd)^{\flat[\vepsilon]})^{\flat[\vepsilon]}
     \ar[d] \ar[r] & \DARelax(\tBd)^{\flat[\vepsilon]} \ar[d] \\
     \FeB(p)^{\flat\flat} \ar[r] \ar[d] & \DARelax(\tBd)^{\flat\flat}
     \ar[r] \ar[d] & \tB^{\flat\flat} \\
     \tE^{\flat\flat} \ar[r] & \tB^{\flat\flat},
    \end{tikzcd}
  \]
  where the outer square in the top row is a pullback by
  \cref{propn:free fib unit counit constr}, the top right square by 
  \cref{propn:it arelax pb}, and the bottom square by the definition
  of $\FeB(p)$. Our equivalence \cref{eq:free of free v2} then arises
  from the composite square in the left column being a pullback. But
  this composite square can also be factored as
  \[
    \begin{tikzcd}
     \DARelax(\FeB(p))^{\flat[\vepsilon]} \ar[r] \ar[d] &
     \DARelax(\DARelax(\tBd)^{\flat[\vepsilon]})^{\flat[\vepsilon]}
     \ar[d]  \\
     \DARelax(\tE^{\flat\flat})^{\flat[\vepsilon]} \ar[r] \ar[d, "\sim"] & \DARelax(\tB^{\flat\flat})^{\flat[\vepsilon]}
      \ar[d, "\sim"] \\
     \tE^{\flat\flat} \ar[r] & \tB^{\flat\flat},
    \end{tikzcd}
  \]
  where the lower vertical maps are both equivalences, and this is the
  square that gives the first equivalence \cref{eq:free of free}.
\end{observation}

\begin{propn}\label{propn: FUF ids}
  Under the equivalence \cref{eq:free of free}, we have:
  \begin{enumerate}[(i)]
  \item the map
  \[\epseB\FeB \colon 
    \FeB\UeB\FeB(p)
    \to \FeB(p)\]
  corresponds to the pullback of the composition functor
  \[ d_{1}^{*} \colon \DFUN([2]^{\sharp},
    \tBd)^{\velax,\flat[\vepsilon]} \to
    \DARelax(\tBd)^{\flat[\vepsilon]},\]
\item and the map
  \[\FeB\etaeB \colon
    \FeB(p) \to 
    \FeB\UeB\FeB(p)
  \]
  corresponds to the pullback of the degeneracy functor
  \[ s_{i}^{*} \colon \DARelax(\tBd) \to \DFUN([2]^{\sharp},
    \tBd)^{\velax}.\]
\item The composite
  \[\FeB(p) \xto{\FeB\etaeB} 
    \FeB\UeB\FeB(p) \xto{\epseB\FeB} \FeB(p) \]
  is the identity.
\end{enumerate}
\end{propn}

\begin{proof}
  Assume $\vepsilon = (0,1)$.
  Unpacking the equivalence \cref{eq:free of free v2} we see that \[\ev_{1} \colon
  \DARelax(\FeB(p))^{\flat[\vepsilon]} \to \FeB(p)\] is given by the
  identity on $\tE^{\flat\flat}$ and $\tB^{\flat\flat}$ and by $d_{1}$
  on $\DFUN([2]^{\sharp}, \tBd)$, as required for (i).
  On the other hand, unpacking \cref{eq:free of free} we get (ii).
  Putting these together, (iii) is then immediate.
\end{proof}

\begin{proof}[Proof of \cref{thm:free partial fib}]
  It follows from \cref{propn: FUF ids} and \cref{propn:free fib unit
    counit constr} that the triangle identities hold, so that the
  natural transformations $\etaeB$ and $\epseB$ are the unit and
  counit of an adjunction $\FeB \dashv \UeB$.
\end{proof}

\begin{observation}\label{obs:upgrade free fib adj}
  The functor $\UeB \colon \PFibe_{/\tBd} \to \CatITsl{\tB}$ fits in a
  commutative triangle
    \[
    \begin{tikzcd}
      {} & \CatI \ar[dl, "(\blank)^{\flat\flat} \times \tBd"'] \ar[dr,
      "(\blank) \times \tB"] \\
      \PFibe_{/\tBd} \ar[rr, "\UeB"] & & \CatITsl{\tB}.
    \end{tikzcd}
  \]
  Moreover, the canonical map
  $\FeB(\tK \times p) \to \tK^{\flat\flat} \times \FeB(p)$ is an
  equivalence for any \itcat{} $\tK$. It therefore follows from \cref{propn:upgrade
    adjn enr} that the adjunction $\FeB \dashv \UeB$ upgrades to an
  adjunction of \itcats{}. In fact, replacing $\CatI$ by $\CatIT$ in
  the diagram above, this same argument shows that we can even get an
  adjunction of $(\infty,3)$-categories --- in particular, for any
  functor $p \colon \tC \to \tB$ and any partial $\vepsilon$-fibration
  $q \colon \tEd \to \tBd$, the adjunction induces equivalences of
  \itcats{}
  \[ \FUN_{/\tB}(\tC, \tE) \simeq \DFUN_{/\tBd}(\FeB(p),
    \tEd).\]
\end{observation}

As a first application of our description of free fibrations, we
obtain the following characterization of representable $(1,0)$-fibrations:
\begin{propn}\label{propn:rep cart fib cond}
  Let $p \colon \tE \to \tB$ be a 1-fibred $(1,0)$-fibration. Then the
  following are equivalent for an object $e \in \tE$ over $b = p(e)$
  in $\tB$:
  \begin{enumerate}[(1)]
  \item There exists an equivalence
    \[
      \begin{tikzcd}
       \tB_{\upslash b} \ar[rr, "\sim"] \ar[dr] & & \tE \ar[dl, "p"] \\
        & \tB
      \end{tikzcd}
    \]
    that takes $\id_{b}$ to $e$.
  \item The morphism $\{e\}^{\flat\flat} \to \tEn$ is a
    $(1,0)$-equivalence over $\tB^{\sharp\sharp}$.
  \item The morphism $\{e\}^{\flat\flat} \to \tEn$ is a
    $(1,0)$-equivalence over $\tE^{\sharp\sharp}$.
  \item For every object $x \in \tE$ there exists a cartesian morphism
    $x \to e$, and every cartesian morphism $x \to e$ is an initial
    object of $\tE(x,e)$.
  \end{enumerate}
  If these conditions hold, we say that $e$ exhibits $p$ as a
  representable $(0,1)$-fibration, which is represented by $b \in \tB$.
\end{propn}
\begin{proof}
  By \cref{thm:free partial fib}, the commutative triangle
  \[
    \begin{tikzcd}
     * \ar[rr, "e"] \ar[dr, "b"'] & & \tE \ar[dl, "p"] \\
      & \tB
    \end{tikzcd}
  \]
  extends uniquely to a morphism of $(1,0)$-fibrations out of the free
  fibration on $* \xto{b} \tB$, \ie{} to
  \[
    \begin{tikzcd}
     \tB_{\upslash b} \ar[rr, "F"] \ar[dr] & & \tE \ar[dl, "p"] \\
      & \tB.
    \end{tikzcd}
  \]
  If the first condition holds, the equivalence in question must
  therefore be this functor $F$, while $F$ is an equivalence precisely
  if $\{e\}^{\flat\flat} \to \tEn$ is a $(1,0)$-equivalence over
  $\tB^{\sharp\sharp}$, since $\tB_{\upslash b} \to \tB$ is the free
  $(1,0)$-fibration on $\{e\}$. This shows the first two conditions
  are equivalent, while the second and third are equivalent since
  \cref{obs:cofib vs equiv} implies firstly that (2) is equivalent to
  having a $(1,0)$-equivalence over $\tEn$, and secondly
  that these are detected over $\tE^{\sharp\sharp}$ since the
  fibration factors through this.

  To obtain the final condition, note that the functor $F$ takes an
  object $f \colon b' \to b$ to the source $f^{*}e$ of the cartesian
  morphism over $f$ with target $e$, so that $F$ is essentially
  surjective \IFF{} every object of $\tE$ is the source of a cartesian
  morphism to $e$. If this holds, $F$ is an equivalence \IFF{} it is
  fully faithful, which means that for $f \colon b' \to b$ and
  $g \colon b'' \to b$ in $\tB$ the horizontal morphism in the
  commutative triangle (where both downward maps are left fibrations)
  \[
    \begin{tikzcd}
      \tB_{\upslash b}(f,g) \ar[rr] \ar[dr, ] & & \tE(f^{*}e,g^{*}e) \ar[dl] \\
      & \tB(b',b'')
    \end{tikzcd}
  \]
  is an equivalence. We first consider the case where $g =
  \id_{b}$. Then we can use \cref{ARlax mor pb} to identify $\tB_{\upslash
    b}(f,\id_{b})$ as $\tB(b',b)_{f/}$, so we get an equivalence
  \IFF{} the initial object $f$ is mapped to an initial object in
  $\tE(f^{*}e,e)$; here $f$ is sent to the cartesian morphism $\bar{f}
   \colon f^{*}e
  \to e$, so this precisely corresponds to the condition that
  $\bar{f}$ is an initial object. It remains to see that this
  gives an equivalence also for general targets $g$, for which we
  observe that $g$ gives a cartesian morphism $g \to \id_{b}$, and
  composition with this and its image in $\tE$ gives a commutative cube
  \[
    \begin{tikzcd}[row sep=small,column sep=small]
     \tB_{\upslash b}(f,g) \ar[rr] \ar[dr] \ar[dd] & & \tE(f^{*}e,g^{*}e) \ar[dr] \ar[dd] \\
      & \tB_{\upslash b}(f,\id_{b}) \ar[rr,crossing over, "\sim"{near start}]  & & \tE(f^{*}e,e) \ar[dd] \\
     \tB(b',b'') \ar[rr,"=" near end] \ar[dr] & & \tB(b',b'') \ar[dr] \\
      & \tB(b',b) \ar[rr,"="] \ar[uu,leftarrow,crossing over]  & & \tB(b',b)
    \end{tikzcd}
  \]
  Here the bottom, left and right faces are pullbacks, hence so is the
  top face, and so the top horizontal map is indeed an equivalence, as
  required.
\end{proof}

\begin{observation}\label{rep fib reformulate}
  Under straightening, if $F \colon \tB \to \CATI$ is the functor
  corresponding to the 1-fibred $(1,0)$-fibration $p \colon \tE \to
  \tB$, the conditions of \cref{propn:rep cart fib cond} for $e \in
  \tE$ over $b \in \tB$ are
  equivalent to the existence of a natural equivalence
  \[ \tB(\blank,b) \simeq F \] under which $\id_{b}$ corresponds to
  $e \in F(b) \simeq \tE_{b}$.
\end{observation}

\begin{remark}
  Condition (2) in \cref{propn:rep cart fib cond} is the
  characterization of representability studied in \cite[\S
  4.3]{GagnaHarpazLanariLaxLim}, while condition (3) says precisely
  that the functor of marked \itcats{} $\{e\}^{\flat} \to (\tE,C)$ is
  \emph{$(1,0)$-cofinal} in the sense of \S\ref{sec:cofinal new}, by
  \cref{propn:cofinal charn}.
\end{remark}

\subsection{Free decorated partial fibrations}
\label{sec:free dec}
In this subsection we consider a variant of \cref{thm:free partial
  fib}, which will describe free \emph{decorated}
partial fibrations. We also show that the unit maps of the resulting
adjunction have the property that any pullbacks thereof are
$\ve$-equivalences, which will be a crucial input to our discussion of
cofree fibrations below.

\begin{propn}\label{propn:free dec partial fib}
  Let $\tBd$ be a decorated \itcat{}. The forgetful functor
  \[\DUeB \colon \DPFIBe_{/\tBd} \to \DCATITsl{\tBd}\] has a left adjoint
  $\DFeB$, given for $p \colon \tCd \to \tBd$ by
  \[ \DFeB(p) \,\,:=\,\, \tCd
    \times_{\tBd} \DARelax(\tBd)^{\diafib} \to \tBd, \]
  where the pullback is via $\ev_{i}$  and the map to $\tBd$ is
  given by $\ev_{1-i}$.
\end{propn}

\begin{observation}
  With this notation, we can interpret \cref{lem:dec pb for dec fib}
  as giving a natural equivalence
  \[ \DFeB\DUeB(p) \simeq \DARelax(\tEn)^{\diafib}\]
  for any decorated partial $\ve$-fibration $p \colon \tEd \to \tBd$.
\end{observation}

\begin{proof}[Proof of \cref{propn:free dec partial fib}]
  We'll show that we get an adjunction on underlying \icats{}; this
  upgrades to \itcats{} as in \cref{obs:upgrade free fib adj}.  For
  $p \colon \tEd \to \tBd$, we can regard $\DFeB(p)$ as the pullback
  \begin{equation}
    \label{eq:free dec fib as pb of fibs}
    \begin{tikzcd}
      \DFeB(p) \ar[r] \ar[d] & \DARelax(\tBd)^{\diafib}  \ar[d, "{(\ev_{i},\ev_{1-i})}"] \\
      \tEd \times \tBd \ar[r, "p \times \id"] & \tBd \times \tBd
    \end{tikzcd}
  \end{equation}
  of decorated partial $\vepsilon$-fibrations over $\tBd$ (using
  \cref{propn:ev dec fib}). 
  Thus $\DFeB$ does indeed define a functor $\DCatITsl{\tBd} \to
  \DPFibe_{/\tBd}$.

  To see that this gives a left adjoint, it suffices to check that the
  unit and counit from the previous section also give natural
  transformations $\eta \colon \id \to \DUeB\DFeB$ and
  $\epsilon \colon \DFeB\DUeB \to \id$, which amounts to checking they
  are compatible with the decorations we now consider; the triangle
  identities will then automatically hold. 

  For the unit, this follows from observing that the degeneracy gives
  a decorated functor $\tBd \to
  \DARelax(\tBd)^{\diafib}$. On the other hand,
  \cref{lem:dec pb for dec fib} implies that for
  $p \colon \tEd \to \tBd$ a decorated partial $\vepsilon$-fibration,
  the counit at $p$
  can be identified with the decorated functor
  \[ \ev_{1-i} \colon \DARelax(\tEn)^{\diafib} \to \tEd,\]
  which is therefore a morphism of decorated partial $\ve$-fibrations;
  this completes the proof.
\end{proof}

\begin{observation}\label{obs:unit eps-equiv}
  We  can view the unit map of the adjunction of \cref{propn:free dec
    partial fib} at $p \colon \tCd \to \tBd$ as a map 
  \[
    \begin{tikzcd}
     \tCd \ar[rr, "\eta"] \ar[dr, "p"'] & & \DFeB(p) \ar[dl] \\
      & \tBd.
    \end{tikzcd}
  \]
  in $\DCatITsl{\tBd}$. Suppose $q \colon \tE \to \tB$ is a partial
  $\vepsilon$-fibration with respect to $\tBd$; then $q \colon \tEn \to \tBd$ is a
  decorated partial fibration, and moreover any decorated functor
  $\DFeB(p) \to q$ over $\tBd$ is a morphism of decorated partial
  $\vepsilon$-fibrations (\ie{} it automatically preserves (co)cartesian $1$- and
  $2$-morphisms), so that composition with $\eta$ induces an equivalence
  \[ \Map_{/\tBd}(\DFeB(p), \tEn) \isoto \Map_{/\tBd}(\tCd, \tEn).
 \]
 In other words, $\eta$ is an $\vepsilon$-equivalence over $\tBd$ with
 these decorations.
\end{observation}

With a bit more work, we can find a ``homotopy inverse'' of $\eta$
that lets us upgrade this observation to the following statement:
\begin{propn}\label{propn:unit fib univ eps-equiv}
  Let $p \colon \tEd \to \tBd$ be a decorated partial $\ve$-fibration,
  and consider the unit map of the adjunction of \cref{propn:free dec
    partial fib} at $p$ as a map
  \[
    \begin{tikzcd}
     \tEd \ar[rr, "\eta"] \ar[dr, "p"'] & & \DFeB(p) \ar[dl] \\
      & \tBd.
    \end{tikzcd}
  \]
  For any decorated functor $f \colon \tAd \to \tBd$, the pullback
  $f^{*}\eta \colon f^{*}\tEd \to f^{*}\DFeB(p)$  is an
  $\ve$-equivalence over $\tAd$.
\end{propn}

For the proof we make use of a canonical (op)lax transformation
associated to the free decorated fibration, which we spell out in the
case $\ve = (0,1)$ for simplicity:

\begin{construction}\label{constr:lax adj tr for free fib}
  Consider the decorated functor
  \[ \pi_{(0,1)} \colon [1]^{\sharp} \otimes [1]^{\flat} \to [1]^{\sharp} \times [1]^{\flat} \to
    [1]^{\sharp},\]
  which takes $(0,0)$ to $0$ and the remaining objects to $1$.

  This induces for any decorated \itcat{} $\tCd$ a functor of \itcats{}
  \[\pi_{(0,1)}^{*} \colon \DARoplax(\tCd) \to \DFUN([1]^{\sharp}
  \otimes [1]^{\flat}, \tCd)^{\oplax} \simeq
  \AR^{\oplax}(\DAR^{\oplax}(\tCd)),\] which is adjoint to a lax
  transformation
  \[ [1] \otimes \DARoplax(\tCd) \to \DARoplax(\tCd);\]
  unwinding the definition we see that this transformation goes from
  the identity to $s_{0}^{*}\ev_{1}$.
\end{construction}

\begin{lemma}\label{lem:lax tr for dec fib}
  Suppose $p \colon \tEd \to \tBd$ is a
  decorated partial $(0,1)$-fibration. Then the construction above
  gives a decorated functor
  \[
    \begin{tikzcd}
      {[1]^{\sharp}} \otimesd \DARoplax(\tEn)^{\diafib} \ar[rr] \ar[dr,
      "\ev_{1} \circ \operatorname{proj}"]
      & &  \DARoplax(\tEn)^{\diafib} \ar[dl, "\ev_{1}"] \\
      & \tEd
    \end{tikzcd}
  \]
  and so a lax natural transformation
  \[ [1]^{\sharp} \otimesd \DUeB\DFeB\DUeB(p) \to \DUeB\DFeB\DUeB(p) \]
  from $\id$ to $(\etaeB \DUeB) \circ (\DUeB \epseB)$
  via \cref{lem:dec pb for dec fib}.
\end{lemma}
\begin{proof}
  It is clear that the source and target functors
  $\id$ and $s_{0}^{*}\ev_{1} \colon \DARoplax(\tEn)^{\diafib} \to
  \DARoplax(\tEn)^{\diafib}$ are both decorated. We therefore need to
  check that
  \begin{itemize}
  \item for every object $f \colon x \to y$ of $\DARoplax(\tEn)$
    (where $f$ is a cocartesian morphism), the
    induced map $f \to s_{0}^{*}\ev_{1}(f) = \id_{y}$ is decorated,
  \item for every morphism $\alpha \colon f \to g$ in
    $\DARoplax(\tEn)$, the 2-morphism in the associated
    naturality square is decorated.
  \end{itemize}
  Unpacking the definition, we can identify the morphism in the first
  point as the square
  \[
    \begin{tikzcd}
     x \ar[r, "f"] \ar[d, "f"'] & y  \ar[d, "\id"] \\
     y \ar[r, "\id"'] & y.
    \end{tikzcd}
  \]
  This commutes, and both $f$ and $\id_{y}$ are decorated in $\tEd$,
  so this is indeed decorated. In the second point, the naturality
  square unpacks to the cube
  \[
    \begin{tikzcd}[row sep=small,column sep=small]
     x \ar[rr] \ar[dr] \ar[dd, "f"] & & x' \ar[dr] \ar[dd, "g"] \\
      & y \ar[rr,crossing over]  & & y' \ar[dd, "="] \\
     y \ar[rr] \ar[dr, "="] & & y' \ar[dr, "="] \\
      & y \ar[rr] \ar[uu,leftarrow,crossing over, "="{near end}]  & & y'
    \end{tikzcd}
  \]
  where we have not shown that the back and top faces contain the
  2-morphism $\alpha$; the other four faces commute. In particular,
  both the top and bottom faces contain a decorated 2-morphism in
  $\tEd$, so as an oplax square in $\DARoplax(\tEn)$ this contains a
  decorated 2-morphism, as required.
\end{proof}

\begin{proof}[Proof of \cref{propn:unit fib univ eps-equiv}]
  We apply \cref{lem: lax tr eps-equiv} using the counit of the
  adjunction and the lax transformation from 
  \cref{lem:lax tr for dec fib} (and its variants for other variances).
\end{proof}

\begin{remark}\label{rem:laxadj}
  Suppose $p \colon \oE \to \oB$ is a cocartesian fibration of
  \icats{}. Then our definitions above give functors
  $\eta \colon \oE \to \oE \times_{\oB} \Ar(\oB)$ and
  $\epsilon \colon \oE \times_{\oB} \Ar(\oB) \to \oE$ such that
  $\epsilon\eta \simeq \id_{\oE}$. In fact, this is part of the data
  of an adjunction $\epsilon \dashv \eta$, which leads to the
  characterization of cocartesian fibrations over $\oB$ as the
  functors $p$ such that $\eta$ has a left adjoint over $\oB$ (see
  \cite[Theorem 5.2.8]{RVelements}). We can think of \cref{constr:lax
    adj tr for free fib} as extending the definition of the unit of
  this adjunction to the \itcatl{} context. Note, however, that as it
  is only an (op)lax transformation, it does not actually exhibit an
  adjunction of \itcats{}. It would be interesting to know if the
  characterization of fibrations via adjunctions could be extended to
  \itcats{} via some notion of ``lax adjunctions'', but we will not
  pursue this here.
\end{remark}

\subsection{Free fibrations on decorated \itcats{}}
\label{sec:free fib on dec}

In this subsection we will use \cref{propn:free dec partial fib} together
with the straightening equivalence for decorated fibrations from
\cref{thm: dec straighten} to get a description of the left adjoint of
the forgetful functor from $\vepsilon$-fibrations to decorated
\itcats{} over the base. We then specialize this to obtain free
1-fibred fibrations on \emph{marked} \itcats{}, which will play an
important role in our study of colimits and Kan extensions below. Note
that we will later also be able to describe the corresponding
fibrations as certain localizations of free decorated fibration in
\S\ref{sec:loc fib}.

We start by identifying the straightening of free
decorated $\vepsilon$-fibrations, which requires some notation:
\begin{notation}\label{not:decoratedslices}
  For $b \in \tB$, we introduce the following notation for the
  decorated (op)lax slices induced from
  $\DARelax(\tBd)^{\diafib}$: 
  \begin{itemize}
  \item $\tBd_{b \upslash}$ is $\{b\} \times_{\tBd}
    \DARopl(\tBd)^{\diafib}$ via $\ev_{0}$; this is $\DFFree^{(0,1)}_{\tBd}(\{b\})$.
  \item $\tBd_{\upslash b}$ is $\{b\} \times_{\tBd}
    \DARopl(\tBd)^{\diafib}$ via $\ev_{1}$; this is $\DFFree^{(1,0)}_{\tBd}(\{b\})$.
  \item $\tBd_{b \downslash}$ is $\{b\} \times_{\tBd}
    \DARlax(\tBd)^{\diafib}$ via $\ev_{0}$; this is $\DFFree^{(0,0)}_{\tBd}(\{b\})$.
  \item $\tBd_{\downslash b}$ is $\{b\} \times_{\tBd}
    \DARlax(\tBd)^{\diafib}$ via $\ev_{1}$; this is $\DFFree^{(1,1)}_{\tBd}(\{b\})$.
  \end{itemize}
  Given a functor $p \colon \tCd \to \tBd$ and $b \in \tB$ we then denote by
  $\tCd_{b \upslash}$ the pullback $\tCd \times_{\tBd} \tBd_{b
    \upslash}$ using the functor to $\tBd$ induced by $\ev_{1}$, and
  similarly in the other variances; we can then identify the fibre at
  $b$ of the free decorated fibration $\DFeB(p)$ as
  $\tCd \times_{\tBd} \DFFree^{\checkvepsilon}_{\tBd}(\{b\})$, \ie{} as
  \begin{itemize}
  \item $\DFFree^{(1,0)}(p)_{b} \simeq \tCd_{b \upslash}$,
  \item $\DFFree^{(0,1)}(p)_{b} \simeq \tCd_{\upslash b}$,
  \item $\DFFree^{(1,1)}(p)_{b} \simeq \tCd_{b \downslash}$,
  \item $\DFFree^{(0,0)}(p)_{b} \simeq \tCd_{\downslash b}$.
  \end{itemize}
\end{notation}

\begin{propn}
  Under the straightening equivalence
  \[\Fun(\tB^{\veop}, \DCATIT) \simeq \DFib^{\vepsilon}_{/\tBss},\]
  the decorated free fibration $\DFeBsharp(p)$ of $p \colon \tCd \to \tBss$
  corresponds to the functor
  \[ b \mapsto \tCd \times_{\tBss}
    \DFFree^{\checkvepsilon}_{\tBss}(\{b\}) \simeq
    \begin{cases}
      \tCd_{b \upslash}, & \vepsilon = (1,0), \\
      \tCd_{\upslash b}, & \vepsilon = (0,1), \\
      \tCd_{b \downslash}, & \vepsilon = (1,1), \\
      \tCd_{\downslash b}, & \vepsilon = (0,0),
    \end{cases}
  \]
  where the functoriality in $b$ comes from pulling back the
  straightening of \[\ev_{1-i} \colon \DARelax(\tBss) \to \tBss.\]
\end{propn}
\begin{proof}
  This follows from the naturality of straightening and the pullback
  square of decorated fibrations \cref{eq:free dec fib as pb of fibs}.
\end{proof}

\begin{cor}\label{cor:free fib on dec}
  The left adjoint to the forgetful functor
  \[ \Fibe_{/\tB} \to \DCatITsl{\tBss}\] takes $\tCd \to \tBss$
  to the unstraightening of the functor $\tB^{\veop} \to \CATIT$
  given by
  \[ b \mapsto \Ld(\tCd \times_{\tBss}
    \DFFree^{\checkvepsilon}_{\tBss}(\{b\})) \simeq     \begin{cases}
      \Ld\tCd_{b \upslash}, & \vepsilon = (1,0), \\
      \Ld\tCd_{\upslash b}, & \vepsilon = (0,1), \\
      \Ld\tCd_{b \downslash}, & \vepsilon = (1,1), \\
      \Ld\tCd_{\downslash b}, & \vepsilon = (0,0).
    \end{cases}
 \]
\end{cor}
\begin{proof}
  The forgetful functor factors as
  \[\Fibe_{/\tB} \xto{(\blank)^{\natural}} \DFibe_{/\tB} \to
    \DCatITsl{\tBss},\] so its left adjoint factors as $\DFeBsharp$
  followed by the left adjoint of $(\blank)^{\natural}$.  \cref{cor:
    dec str comp flat} shows that $(\blank)^{\natural}$ corresponds
  under straightening to composition with $(\blank)^{\flat\flat}$, so
  its left adjoint corresponds to composition with $\Ld$.
\end{proof}

Now we specialize our results to describe free \emph{1-fibred}
$\ve$-fibrations:
\begin{defn}
  Let $\MFIBe_{/\tB}$ denote the full sub-\itcat{} of
  $\DFIBe_{/\tB}$ spanned by the decorated
  $\ve$-fibrations whose source is in the image of $(\blank)^{\sharp}
  \colon \MCATIT \to \DCATIT$; we refer to
  these as \emph{marked $\ve$-fibrations over $\tB$}
\end{defn}

\begin{observation}
  For $\ve = (i,j)$, let $\oFIB_{/\tB}^{\ve}$ denote the full sub-\itcat{} of
  $\FIBe_{/\tB}$ on the \emph{1-fibred} $\ve$-fibrations, \ie{} those
  whose fibres are \icats{}, or equivalently those where \emph{all}
  2-morphisms are $j$-cartesian. We can then regard $\oFIB_{/\tB}^{\ve}$ as a full
  subcategory of $\MCATITsl{\tB^{\sharp}}$, where $\tE \to \tB$
  corresponds to $\tE^{\natural} \to \tB^{\sharp}$ with
  $\tE^{\natural}$ marked by the $i$-cartesian 1-morphisms (since for
  1-fibred fibrations all 2-morphisms are $j$-cartesian and so they
  are always preserved). In fact, this describes $\oFIB^{\ve}_{/\tB}$
  as precisely the intersection of $\FIBe_{/\tB}$ and
  $\MCATITsl{\tB^{\sharp}}$ in $\DCATITsl{\tB^{\sharp\sharp}}$ (where
  we embed the latter via $(\blank)^{\sharp}$).
\end{observation}

\begin{notation}\label{not:free 1-fib functor}
  Let $(\tE,S)$ be a marked \itcat{}. For a functor $p \colon \tE \to
  \tB$, we let
  \[ \FrofeB(\tE,S) \colon \tB^{\veop} \to \CATI\]
  be the functor obtained by first unstraightening the free decorated
  $\ve$-fibration $\DFFree^{\ve}_{\tBss}(p)$, where we view $p$ as a map
  $(\tE,S)^{\sharp} \to \tBss$, and then composing with $\Ld$. Thus 
  $\FrofeB(\tE,S)$ is given by 
  \[ b \mapsto \Ld((\tE,S)^{\sharp} \times_{\tBss}
    \DFFree^{\checkvepsilon}_{\tBss}(\{b\})) \simeq     \begin{cases}
      \Ld\tEd_{b \upslash}, & \vepsilon = (1,0), \\
      \Ld\tEd_{\upslash b}, & \vepsilon = (0,1), \\
      \Ld\tEd_{b \downslash}, & \vepsilon = (1,1), \\
      \Ld\tEd_{\downslash b}, & \vepsilon = (0,0),
    \end{cases}
  \]
  where we can identify $\tEd_{b \upslash}$ as $(\tE_{b \upslash},
  S_{b})^{\sharp}$ with $S_{b}$ consisting of maps whose projection to
  $\tE$ lies in $S$, and similarly in the other cases. In particular, these
  decorated \itcats{} have
  \emph{all} 2-morphisms decorated, so after localizing them this
  functor indeed takes values in $\CATI$. 
  In particular, given a marking $(\tB,E)$ of $\tB$, we have the functor
  \[ \FrofeB(\tB,E) \colon \tB^{\veop} \to \CATI,\] given by
  \[ b \mapsto  \begin{cases}
      \Ld\tBd_{b \upslash}, & \vepsilon = (1,0), \\
      \Ld\tBd_{\upslash b}, & \vepsilon = (0,1), \\
      \Ld\tBd_{b \downslash}, & \vepsilon = (1,1), \\
      \Ld\tBd_{\downslash b}, & \vepsilon = (0,0),
    \end{cases}
  \]
  where $\tBd_{b \upslash}$ is decorated by the 1-morphisms that lie
  over $E$ and all 2-morphisms, and similarly in the other 3 cases.
\end{notation}

\begin{cor}\label{cor:free 1-fibred on marked}
  The fully faithful inclusion
  \[\oFIB_{/\tB}^{\ve} \hookrightarrow \MCATITsl{\tB^{\sharp}}\]
  has a left adjoint, which sends $(\tC,S) \to \tB^{\sharp}$ to the
  unstraightening of the
  functor $\FrofeB(\tC,S) \colon \tB^{\veop} \to \CATI$. In
  particular, for $F \colon \tB^{\veop} \to \CATI$ with corresponding
  $\ve$-fibration $\tE \to \tB$, we have a natural equivalence of
  \icats{}
  \[ \Nat_{\tB^{\veop},\CATI}(\FrofeB(\tC,S), F) \simeq
    \MCATITsl{\tB^{\sharp}}((\tC,S), \tE^{\natural}).\]
\end{cor}
\begin{proof}
  We need to show that the adjunction from \cref{cor:free fib on dec}
  restricts to these full subcategories; the only thing to check is
  that the left adjoint takes a marked \itcat{} to a 1-fibred
  fibration, which is clear since $\tEd \times_{\tBss}
  \DFFree^{\checkvepsilon}_{\tBss}(\{b\})$ then has \emph{all}
  2-morphisms decorated, and so applying $\Ldec{\blank}$ produces an \icat{}.
\end{proof}

\subsection{Smoothness for decorated \itcats{}}
\label{sec:smooth}

Suppose $p \colon \oE \to \oB$ is a cocartesian fibration of
\icats{}. Then $p$ is in particular exponentiable, meaning that the
pullback functor $p^{*} \colon \CatIsl{\oB} \to \CatIsl{\oE}$ has a
right adjoint $p_{*}$. Moreover, this right adjoint has the property
that if $q \colon \oD \to \oE$ is a cartesian fibration, then so is
$p_{*}(q)$ over $\oB$. Our goal in this section is to generalize these
results, as well as the closely related notion of smooth functors from
\cite[\S 4.1.2]{LurieHTT}, to the setting of decorated \itcats{}.

\begin{defn}\label{def:decconduche}
  A decorated functor $p \colon \tA^{\diamond} \to \tB^{\diamond}$ is
  said to be \emph{exponentiable} if the pullback functor
  \[
    p_{\dec}^* \colon {\DCATIT}_{/ \tB^{\diamond}} \to {\DCATIT}_{/ \tA^{\diamond}}
  \]
  admits a right adjoint $p_{\dec,*}$ which we call the \emph{decorated pushforward functor}.
\end{defn}

\begin{propn}\label{prop:decfibisexp}
  Suppose $p \colon \tA^{\diamond} \to \tD^{\sharp \sharp}$ is a decorated
   $\ve$-fibration. 
  Then $p$ is exponentiable.
\end{propn}

For the proof we use the following observation:

\begin{lemma}\label{lem:decfibexpbase}
  Let $F \colon \oI \to (\DCatIT)_{/\tD^{\sharp \sharp}}$ be a functor
  where $\oI$ is an \icat{}, and suppose that the colimit of $F$ is
  preserved by the functors $\Ud$, $(\blank)_{(1)}^{\leq}$ and
  $(\blank)_{(2)}$. Then given a decorated $\vepsilon$-fibration
  $p \colon \tA^{\diamond} \to \tD^{\sharp \sharp}$, the canonical map
  \[\colim_{\oI}(p_{\dec}^*\circ
    F) \to p_{\dec}^{*}(\colim_{\oI} F)\]
  is an equivalence.
\end{lemma}
\begin{proof}
  Let $F^{\leq 1}_{(1)}\xrightarrow{}F \xleftarrow{} F_{(2)}$ be the
  cospan associated to $F$ (here we are using the description of
  $\DCatIT$ given in \cref{eq:dcat via mcat}) and consider the
  following commutative diagram, where the front part computes
  $p_{\dec}^*(\colim_{\oI} F)$ by our assumption on this colimit:
	\[\begin{tikzcd}[column sep=0.1em, row sep=2em]
	& {\colim_{\oI}F^{\leq 1}_{(1)} } && {\colim_{\oI}F} && {\colim_{\oI}F_{(2)}} \\
	{p^{\leq 1,*}_{(1)}(\colim_{\oI}F_{(1)} )^{\leq 1}} && {p^{*}(\colim_{\oI}F )} && {p^{*}_{(2)}(\colim_{\oI}F_{(2)} )} \\
	& \tD^{\leq 1} && \tD && \tD \\
	{\tA_{(1)}^{\leq 1}} && \tA && {\tA_{(2)}}
	\arrow[from=1-2, to=1-4]
	\arrow[from=1-2, to=3-2]
	\arrow[from=1-4, to=3-4]
	\arrow[from=1-6, to=1-4]
	\arrow[from=1-6, to=3-6]
	\arrow[from=2-1, to=1-2]
	\arrow[from=2-1, to=2-3]
	\arrow[from=2-1, to=4-1]
	\arrow[from=2-3, to=1-4]
	\arrow[from=2-3, to=4-3]
	\arrow[from=2-5, to=1-6]
	\arrow[from=2-5, to=2-3]
	\arrow[from=2-5, to=4-5]
	\arrow[from=3-2, to=3-4]
	\arrow[from=3-6, to=3-4]
	\arrow["p^{\leq 1}_{(1)}",from=4-1, to=3-2]
	\arrow[from=4-1, to=4-3]
	\arrow["p",from=4-3, to=3-4]
	\arrow["p_{(2)}",from=4-5, to=3-6]
	\arrow[from=4-5, to=4-3]
\end{tikzcd}\]
Since $p$ is a decorated $\vepsilon$-fibration, the functors
$p_{(1)}^{\leq 1}$, $p$ and $p_{(2)}$ are all $\ve$-fibrations, and so
pullback along them preserve colimits by
\cref{thm:fibsexponentiable}. The front of the diagram will therefore
describe $\colim_{\oI} p^{*}F$ after localizing it to ${\DCatIT}_{/
  \tA^{\diamond}}$. However, in this case no localization is required,
so this completes the proof.
\end{proof}

\begin{lemma}\label{lem:decfibisexp}
	Let $p \colon \tA^{\diamond} \to \tD^{\sharp \sharp}$ be a decorated $\vepsilon$-fibration. Then for every decorated functor $q \colon \tX^{\diamond} \to \tA^{\diamond}$ the presheaf
	\[
		{\DCatIT}_{/\tA^{\diamond}}(p_{\dec}^*(\blank),\tX^{\diamond}) \colon {\DCatIT}_{/\tD^{\sharp \sharp}} \to \SS,
	\]
	is representable. 
\end{lemma}
\begin{proof}
	We consider the pointwise monomorphism of spaces (cf. \cref{cor:forgetfaithful}) 
	\[
		\Phi \colon \DCatITsl{\tA^{\diamond}}(p_{\dec}^*(\blank),\tX^{\diamond}) \to \CatITsl{\tA}(p^*(\blank),\tX)\simeq \CatITsl{\tD}(\blank,p_*\tX),
	\]
where the final isomorphism follows from the fact that the underlying functor of $p$ is an ordinary $\vepsilon$-fibration together with \cref{thm:fibsexponentiable}.

Let $p_{\dec,*} \tX$ be the sub-\itcat{} of $p_*\tX$ whose objects,
morphisms and 2-morphisms are given by decorated functors
$\tT^{\flat \flat }\times_{\tD^{\diamond}}\tA^{\diamond} \to
\tX^{\diamond}$ over $\tA^{\diamond}$ with $\tT\in
G=\{[0],[1],C_2\}$, respectively. To see that this definition does in fact give a
sub-\itcat{}, we need to verify that these classes of $i$-morphisms
are stable under composition in $p_*\tX$. This can be verified by
checking that $p^*_{\dec}$ preserves certain colimits of elements in
$G$. It is immediate to verify that those colimits satisfy the
hypothesis of \cref{lem:decfibexpbase}. This shows $p_{\dec,*}\tX$ is
a sub-\itcat{} of $p_* \tX$.

Now we define the decorations on $p_{\dec,*}\tX$: We define
$p_{\dec,*}\tX^{\diamond}_{(i)}$ for $i=1,2$ to be sub-\itcat{} of
$p_{*} \tX$ whose objects, morphisms and 2-morphisms are given by
decorated functors over $\tA^{\diamond}$ of the form
$\tTd \times_{\tD^{\diamond}}\tA^{\diamond} \to \tX^{\diamond}$
with
\[ \tTd \in
  \begin{cases}
    [0], [1]^{\sharp},C_{2}^{\sharp\flat}, & i = 1,\\
    [0],[1]^{\flat}, C_{2}^{\flat\sharp}, & i = 2,
  \end{cases}
\]
respectively. The same argument as before shows that
the definition above yields sub-\itcats{} as desired. Moreover, we
obtain canonical functors
	\[
		p_*\tX^{\diamond}_{(1)} \xrightarrow{}  p_*\tX \xleftarrow{}  p_*\tX^{\diamond}_{(2)} 
	\]
	We observe that since $[0]^{\flat \flat}=[0]^{\sharp \flat}=[0]^{\flat \sharp}$ both functors above induce equivalences on underlying spaces. We further observe that $[1]^{\flat \flat}=[1]^{\flat \sharp}$ and so the right-most morphism is an equivalence on underlying $\infty$-categories. Finally, we note that $C_2^{\sharp \flat}\simeq C_2^{\flat\flat}\amalg_{\partial C_2^{\flat \flat}} \partial C_2^{\sharp \flat}$ and that this colimit  again satisfies the hypothesis of \cref{lem:decfibexpbase}, which implies that the left-most map of the span above is locally full.

    To finish the proof, we note that we have a pointwise monomorphism of spaces ${\DCatIT}_{/\tD^{\sharp \sharp}}(-,p_{\dec,*}\tX) \hookrightarrow {\CatIT}_{/\tD}(-,p_*\tX)$ whose image agrees with that of $\Phi$. We conclude that $p_{\dec,*}\tX$ is the desired representing object.
\end{proof}

\begin{proof}[Proof of \cref{prop:decfibisexp}]
  By \cref{lem:decfibisexp}, the functor $p^*_{\dec}$ admits a right
  adjoint at the level of underlying \icats{}. As $p^*_{\dec}$ is
  $\CatI$-linear, we can promote $p^*_{\dec}$ to an
  $(\infty,2)$-categorical left adjoint using \cref{propn:upgrade adjn
    enr} as in the proof of \cref{thm:fibsexponentiable}.
\end{proof}

\begin{observation}\label{obs:fibre of pushforward}
  Suppose $p \colon \tAd \to \tBd$ is exponentiable. Then for
  $q \colon \tEd \to \tAd$, the fibre of $p_{\dec,*}(q)$ at
  $b \in \tB$ is $\DFUN_{/\tAd}(\tAd_{b}, \tEd)$.  Indeed, for a
  decorated \itcat{} $\tKd$, we have a natural equivalence
  \[
    \begin{split}
      \Map(\tKd, (p_{\dec,*}(q))_{b}) & \simeq \Map_{/\tBd}(\tKd
                                        \times \{b\}, p_{\dec,*}(q))
      \\
                                      & \simeq \Map_{/\tAd}(\tKd \times \tAd_{b}, \tEd) \\
       & \simeq \Map(\tKd, \DFUN_{/\tAd}(\tAd_{b}, \tEd)).
    \end{split}
  \]
  More generally, given a pullback square
  \[
    \begin{tikzcd}
     \tCd \ar[r, "\alpha"] \ar[d, "q"'] & \tAd  \ar[d, "p"] \\
     \tDd \ar[r, "\beta"'] & \tBd
    \end{tikzcd}
  \]
  where $p$ and $q$ are exponentiable, we have a Beck--Chevalley
  equivalence
  \[ \beta^{*}p_{*} \simeq q_{*}\alpha^{*},\]
  since the corresponding mate transformation of left adjoints is
  obviously an equivalence.
\end{observation}

\begin{defn}\label{def:smooth}
  A decorated functor $p \colon \tA^{\diamond} \to \tD^{\diamond}$ is
  said to be $\vepsilon$-\emph{smooth} if $p$ is exponentiable, 
  and given a commutative diagram 
    \[\begin{tikzcd}
        {\tS^{\prime\diamond}} & {\tT^{\prime\diamond}} & {\tA^{\diamond}} & {} \\
        {\tS^{\diamond}} & {\tT^{\diamond}} & {\tB^{\diamond}}
        \arrow["{\phi^{\prime}}", from=1-1, to=1-2]
        \arrow[from=1-1, to=2-1]
        \arrow[from=1-2, to=1-3]
        \arrow[from=1-2, to=2-2]
        \arrow[from=1-3, to=2-3]
        \arrow["\phi", from=2-1, to=2-2]
        \arrow[from=2-2, to=2-3]
      \end{tikzcd}
    \]
    where both squares are pullbacks and $\phi$ is an
    $\vepsilon$-cofibration (\cf{} \cref{def:epsiloncof}), then
    $\phi^{\prime}$ is again an $\vepsilon$-cofibration. 
\end{defn}

\begin{lemma}\label{lem:smooth and pushfwd}
  Suppose $p \colon \tAd \to \tBd$ is
  $\vepsilon$-smooth. Then
  \begin{enumerate}[(i)]
  \item The pullback functor \[p_{\dec}^{*} \colon \DCATITsl{\tBd} \to
    \DCATITsl{\tAd}\] preserves $\ve$-equivalences.
  \item The pushforward functor \[p_{\dec,*} \colon \DCATITsl{\tAd} \to
    \DCATITsl{\tBd}\] preserves partial $\ve$-fibrations.
  \item The adjunction $p_{\dec}^{*} \dashv p_{\dec,*}$ restricts on full
    subcategories to an adjunction
    \[ p_{\dec}^{*} \colon \PFIBe_{/\tBd} \rightleftarrows \PFIB_{/\tAd} \colon p_{\dec,*}.\]
  \end{enumerate}
\end{lemma}
\begin{proof}
  The first part is immediate from the definition of smoothness, and
  implies the second by adjunction. It follows that both of the
  functors $p^{*}_{\dec}$ and $p_{\dec,*}$ preserve partial
  $\ve$-fibrations, so the adjunction restricts as claimed in the final part.
\end{proof}

\begin{thm}\label{thm:decsmooth}
	Let $p \colon \tA^{\diamond} \to \tB^{\sharp \sharp}$ be a decorated $\checkvepsilon$-fibration. Then $p$ is $\vepsilon$-smooth.
      \end{thm}

\begin{remark}
  The cases of \cref{prop:decfibisexp} and \cref{thm:decsmooth} where
  $p$ is an $\ve$-fibration (without additional decorations) can be
  extracted from Theorem 3.2.3.15 and Proposition 3.2.5.9 of
  \cite{loubaton}, respectively. Loubaton also allows the marking on
  the target to be non-maximal, so his result also covers some cases we
  do not consider.
\end{remark}

For the proof we need the following technical input:
\begin{propn}\label{prop:decoratedescent}
	Suppose that we are given pullback diagrams of decorated \itcats{}
	\[
		\begin{tikzcd}
			\tA^{\diamond} \arrow[r] \arrow[d,"p"] & \tA^{\blacklozenge} \arrow[d,"\overline{p}"] \\
			\mathbb{O}^{n,\diamond} \arrow[r] & (\mathbb{O}^{n})^{\sharp \sharp}{,}
		\end{tikzcd}
		\enspace \enspace \enspace
		\begin{tikzcd}
			\tB^{\diamond} \arrow[r] \arrow[d,"q"] & \tB^{\blacklozenge} \arrow[d,"\overline{q}"] \\
			\mathbb{O}^{n,\dagger}_+ \arrow[r] & (\mathbb{O}^{n}_+)^{\sharp \sharp}{,}
		\end{tikzcd}
	\]
	where $\overline{p}$ and $\overline{q}$ are decorated
        $(1,0)$-fibrations and we use the notation of \cref{def:orientalhorns}. Then we have pushout squares in $\DCATIT$
	\[
		\begin{tikzcd}
			\tA^{\diamond}_{n}\times \mathbb{O}^{[0,n-1],\flat \flat } \arrow[d] \arrow[r] & \tA^{\diamond}_{n}\times \mathbb{O}^{n,\diamond} \arrow[d] \\
			\tA^{\diamond}\times_{\mathbb{O}^{n,\diamond}} \mathbb{O}^{[0,n-1],\flat \flat } \arrow[r] & \tA^{\diamond}{,}
		\end{tikzcd} \enspace \enspace \enspace 
		\begin{tikzcd}
			\tB^{\diamond}_{n}\times \mathbb{O}_+^{[0,n-1],\dagger} \arrow[d] \arrow[r] & \tB^{\diamond}_{n}\times \mathbb{O}^{n,\dagger} \arrow[d] \\
			\tB_{\diamond}\times_{\mathbb{O}_+^{n,\dagger}} \mathbb{O}^{[0,n-1],\dagger } \arrow[r] & \tB^{\diamond}{,}
		\end{tikzcd}
	\]
	where $\tA^{\diamond}_{n}$ and $\tB_{n}^{\diamond}$ denote the corresponding fibres over the object $n \in \mathbb{O}^{n}$.
\end{propn}
\begin{proof}
	Let $f_{\diamond}\colon P_n^{\diamond} \to \tA^{\diamond}$ and $g_{\diamond} \colon Q_n^{\diamond} \to \tB^{\diamond}$ denote corresponding functors out of the pushout. Invoking \cite[Corollary 3.73, Remark 3.74]{ASI23} we see that the functors $f_{\diamond}$ and $g_{\diamond}$ are equivalences on underlying \itcats{}. Therefore our question reduces to verifying that the decorations on both sides agree. First, we observe that $P_n^{\diamond}$ and $Q_{n}^{\diamond}$ contain all of the decorations of $\tA^{\diamond}$ and $\tB^{\diamond}$ that live in the fibre over a point $i \in \mathbb{O}^n$. We proceed case by case.
	\begin{itemize}
		\item Let $u \colon x \to y$ be an edge in $\tA^{\diamond}_{(1)}$ living over $n-1 \to n$. Then we can express $u=\beta \circ \alpha$ where $\beta$ is a decorated morphism in the fibre over $n-1$ and $\alpha$ is a decorated edge that factors through $P_n^{\diamond}$. We conclude that $u$ also factors through $P_n^{\diamond}$. The analogous claim for $\tB^{\diamond}$ is immediate. 
		\item Let $u,v \colon x \to y$ be morphisms in
                  $\tA^{\diamond}$ that live over the morphisms $\{i,n\}$ and $\{i,n-1,n\}$, respectively, and consider $\theta \colon u \to v$ in $\tA^{\diamond}_{(2)}$. We consider the map 
		\[
			\phi \colon \{y\}\times \mathbb{O}^2 \to \tA_{n}\times \mathbb{O}^2 \to \tA
		\]
		and denote by $\phi(0)=\hat{x}$ and by $\hat{\theta}$
                the corresponding map $C_2 \to \tA$. Then it follows
                that there exists a map $\iota \colon x \to \hat{x}$
                such that the whiskering $\iota \circ \hat{\theta}$ equals
                $\theta$. Since $\tA_{2}^{\diamond} \to
                \mathbb{O}^{n,\diamond}_{(2)}$ is a fibration, we can
                choose $\hat{\theta}$ to factor through
                $\tA^{\diamond}_{(2)}$ and so the result
                holds. \qedhere
	\end{itemize}
\end{proof}

\begin{proof}[Proof of \cref{thm:decsmooth}]
	We let $\epsilon=(0,1)$ without loss of generality. The proof now follows as a combination of \cref{prop:decfibisexp}, \cref{prop:01equivorientals} and \cref{prop:decoratedescent}.
\end{proof}

The following corollary now follows directly from \cref{thm:decsmooth}
and \cref{lem:smooth and pushfwd}.

\begin{cor}\label{cor:pushforwardisfib}
	Let $p \colon \tA^{\diamond} \to \tB^{\sharp \sharp}$ be a
        decorated $\overline{\vepsilon}$-fibration. Then the
        adjunction $p_{\dec}^{*} \dashv p_{\dec,*}$ on slices of
        $\DCATIT$ restricts on full subcategories to an adjunction
	\[
		p_{\dec}^* : \FIB^{\vepsilon}_{/\tB} \to
                \PFIB^{\vepsilon}_{/\tA^{\diamond}} : p_{\dec,*}. \qed
	\]
\end{cor}

\begin{observation}
  Since the left adjoint is $\CatIT$-linear, the adjunction of
  \itcats{} above can even be upgraded to an adjunction of
  $(\infty,3)$-categories using a variant of \cref{propn:upgrade adjn enr}.
\end{observation}

As a useful special case, we have:
\begin{cor}\label{cor:pushforwardisfib marked}
  Let $q \colon (\tA,I) \to \tB^{\sharp}$ be a functor of marked
  \itcats{} such that $q^{\sharp}$ is a decorated
  $\cve$-fibration. Then the
  adjunction $q_{\dec}^{*} \dashv q_{\dec,*}$ on slices of
  $\DCATIT$ restricts on full subcategories to an adjunction
  \[
    q_{\dec}^* : \FIB^{\ve}_{/\tB} \to \FIB^{\ve}_{/(\tA,I)} :
    q_{\dec,*}.
  \]
\end{cor}
\begin{proof}
  Since $\FIB^{\ve}_{/(\tA,I)}$ is a full subcategory of
  $\PFIB^{\ve}_{/(\tA,I)^{\sharp}}$ and the pullback functor $q_{\dec}^{*}$
  takes values in this full subcategory, this is immediate from
  \cref{cor:pushforwardisfib} applied to $q^{\sharp}$.
\end{proof}

\subsection{Pushforward of partial fibrations}
\label{sec:cofree}

Our goal in this section is to generalize \cref{cor:pushforwardisfib}
and see that pullback of fibrations along an arbitrary decorated
functor $f \colon \tCd \to \tDss$ has a right adjoint. As special
cases this will produce right Kan extensions for functors to $\CATI$,
which will be the starting point for our discussion of Kan extensions
below in \S\ref{sec:kanext}, as well as \emph{cofree} fibrations.

\begin{defn}\label{def:Rf}
  Let $f \colon \tCd \to \tDss$ be a decorated
  functor. We define a functor
  \[R_{f} \colon \PFIBe_{/\tC^{\diamond}} \to \PFIBe_{/\tDss} =
    \FIBe_{/\tD}\]
  as the composite
  \[
    \PFIBe_{/\tC^{\diamond}} \xto{\pi^{*}}  \PFIBe_{/\tQd} \xto{\phi_{*}} \PFIBe_{/\tDss}
  \]
  where
  \[ \tQd := \tCd \times_{\tDss} \DARelax(\tDss)^{\diafib},\]
  the first functor is given by taking pullbacks along the
  projection
  \[\pi \colon \tQd \to \tCd\]
  (see \S\ref{sec:free dec})
  and the second functor is the right adjoint (see
  \cref{cor:pushforwardisfib}) to the pullback functor along
  \[\phi = \DFFree^{\cve}_{\tDss}(f) \colon \tQd \to \tDss,\] which is an $\cve$-fibration.
\end{defn}

\begin{observation}\label{obs:Rf is restr of adjoint}
  By definition, $R_{f}$ is the restriction to full subcategories of
  partial fibrations of the functor
  \[ \DCATITsl{\tCd} \xto{\pi^{*}}
    \DCATITsl{\tQd} \xto{\phi_{*}}
    \DCATITsl{\tDss}.
  \]
  This has a left adjoint $\pi_{!}\phi^{*}$ given by pullback along
  $\phi$ followed by composition with $\pi$. This functor takes 
  $q \colon
  \tAd \to \tDss$ to the projection
  \[ \tCd \times_{\tDss} \DARelax(\tDss)^{\diafib} \times_{\tDss}
    \tAd \to \tCd,\]
  which is the pullback to $\tCd$ of the functor
  \[ \DFFree^{\ve}_{\tDss}(q) \colon \DARelax(\tDss)^{\diafib}
    \times_{\tDss} \tAd \to \tDss,\]
  so that $\pi_{!}\phi^{*} \simeq f^{*}\DFFree^{\ve}_{\tDss}$.
\end{observation}

\begin{observation}\label{obs:Rf fibre}
  Consider a partial $\ve$-fibration $q \colon \tEn \to \tCd$.
  By \cref{obs:fibre of pushforward}, in the situation above we can
  identify the fibres of $R_{f}(q)$ as
  \[ R_{f}(q)_{d} \simeq \DFUN_{/\tQd}(\tQd_{d}, \pi^{*}\tEn) \simeq
    \DFUN_{/\tCd}(\tQd_{d}, \tEn).\]
  Using the notation from \ref{not:decoratedslices}, we get
  \[ R_{f}(q)_{d} \simeq
    \begin{cases}
      \DFUN_{/\tCd}(\tCd_{d\upslash}, \tEn), & \ve = (0,1), \\
      \DFUN_{/\tCd}(\tCd_{\upslash d}, \tEn),  & \ve = (1,0), \\
      \DFUN_{/\tCd}(\tCd_{d \downslash}, \tEn),  & \ve = (0,0), \\
      \DFUN_{/\tCd}(\tCd_{\downslash d}, \tEn),  & \ve = (1,1),.
    \end{cases}
  \]
\end{observation}

\begin{thm}\label{thm:cofree}
  Let $f \colon \tC^{\diamond} \to \tD^{\sharp \sharp}$ be a decorated
  functor. Then there exists an adjunction of \itcats{}
  \[
    f^* \colon \FIB^{\vepsilon}_{/\tD}  \llra \PFIB^{\vepsilon}_{/\tC^{\diamond}}  \colon R_f
  \]
  where $f^*$ denotes the pullback functor.
\end{thm}

\begin{proof}
  Let $q \colon \tAd \to \tD^{\sharp \sharp}$ be an
  $\ve$-fibration. Then for every partial $\ve$-fibration
  $p \colon \tXd \to \tCd$, we have from \cref{obs:Rf is restr of adjoint} a natural
  equivalence
	\[
		\FIB^{\vepsilon}_{/\tD}(q,R_f(p)) \simeq \DCATITsl{\tC^{\diamond}}(f^{*}\DFFree^{\ve}_{\tDss}{\epsilon}(q),p).
	\]
	Moreover, by pulling back the unit map in \cref{propn:free dec
          partial fib}, we obtain a morphism $\chi_{f}
        \colon f^{*}q \to
        f^{*}\DFFree^{\ve}_{\tDss}(q)$, which is an
        $\ve$-equivalence by \cref{propn:unit fib univ eps-equiv}.
        In other words, restriction along $\chi_{f}$ yields an
        equivalence
        \[\DCATITsl{\tC^{\diamond}}(f^{*}\DFFree^{\ve}_{\tDss}(q),p)
        \xrightarrow{\simeq}
        \DCATITsl{\tC^{\diamond}}(f^*q,p) \simeq
        \PFIB^{\ve}_{/\tCd}(f^{*}q, p).\] The result now
        follows.
\end{proof}

We note two important special cases: For any functor of \itcats{} $f \colon \tA \to
\tB$, we get:
\begin{cor}
  The pullback functor
  \[ f^{*} \colon \FIB^{\ve}_{/\tB} \to \FIB^{\ve}_{/\tA}\]
  has a right adjoint $R_{f}$. \qed
\end{cor}
Together with our explicit description of this right adjoint, this
will be the starting point for our discussion of Kan extensions for
\itcats{} below
in \S\ref{sec:kanext}. The following variant will similarly lead to a
notion of \emph{partially lax} Kan extensions:
\begin{cor}\label{cor:lax kan adjn on fibs}
  Suppose $f \colon (\tC,I) \to \tD^{\sharp}$ is a functor of marked
  \itcats{}. Then the pullback functor
  \[ f^{*} \colon \FIB^{\ve}_{/\tD} \to \FIB^{\ve}_{/(\tC,I)} \]
  has a right adjoint $R_{f}$.
\end{cor}
\begin{proof}
  We can identify $\FIB^{\ve}_{/(\tC,I)}$ as the full sub-\itcat{} of
  $\PFIB^{\vepsilon}_{/(\tC,I)^{\sharp}}$ spanned by the
  $\ve$-fibrations. Since the pullback functor from
  $\FIB^{\ve}_{/\tD}$ takes values in this full subcategory, the right
  adjoint from \cref{thm:cofree} restricts to a right adjoint of
  $f^{*}$, as required.
\end{proof}

On the other hand, taking the identity as a
functor $\tB^{\flat\flat} \to \tB^{\sharp\sharp}$ we get the existence
of cofree fibrations:
\begin{cor}\label{cor:cofree}
  The forgetful functor
  \[ \FIB^{\ve}_{/\tB} \to \CATITsl{\tB} \]
  has a right adjoint. \qed
\end{cor}
More generally, for any decorated \itcat{} $\tBd$ we get a right
adjoint to the forgetful functor
$\FIB^{\ve}_{/\tB} \to \PFIB^{\ve}_{/\tBd}$. 

For later use, it will be useful to also identify the unit and counit
of the adjunction of \cref{thm:cofree} a bit more explicitly:
\begin{observation}\label{obs:unitcofree}
  Let $f \colon \tC^{\diamond} \to \tD^{\sharp \sharp}$ be a decorated
  functor and let $q \colon \tX^{\diamond} \to \tDss$ be an
  $\vepsilon$-fibration. Then the map
  $f^* \DFFree^{\ve}_{\tDss}(q) \to f^*q$, obtained by pulling back the
  counit of the adjunction in \cref{propn:free dec partial fib},
  corresponds via the adjunction of \cref{obs:Rf is restr of adjoint} to
  a map $\epsilon \colon q \to R_f(f^*q)$. Note that $\eta$ is sent
  under the equivalence
  \[
    \FIB^{\vepsilon}_{/\tD}(q,R_f(f^*q)) \xrightarrow{\simeq} \PFIB^{\vepsilon}_{/\tC^{\diamond}}(f^*q,f^*q)
  \]
to the identity on $f^*q$, which identifies $\epsilon$ as the unit of
the adjunction in \cref{thm:cofree}. Given $p \in
\FIB^{\vepsilon}_{/\tD}$, we can apply the previous discussion,
together with the fact that forgetful functor in \cref{propn:free dec
  partial fib} is fully faithful, to identify the composite
\[
  \DCATITsl{\tDss}(p,q) \xrightarrow{\eta_*} \DCATITsl{\tDss}(p,R_f(f^*q)) \xrightarrow{\simeq} \DCATITsl{\tC^{\diamond}}(f^*p,f^*q),
\]
with the pullback functor along $f$. Moreover, we can now obtain a commutative diagram

  \begin{equation}\label{eq:obsunitcofree}
    \begin{tikzcd}
    \DCATITsl{\tDss}(p,q) \arrow[r,"\eta_*"] \arrow[d,"\xi^*"] &   \DCATITsl{\tDss}(p,R_f(f^*q)) \arrow[d,"\simeq"] \\
      \DCATITsl{\tDss}(f_!f^* p,q)  \arrow[r,"\simeq"] & \DCATITsl{\tC^{\diamond}}(f^*p,f^*q).
  \end{tikzcd}
  \end{equation}
where $\xi^*$ is the unit of $f_! \dashv f^*$. 
\end{observation}

\begin{observation}\label{obs:counitcofree}
  Let $f,q$ be as before, and let $p \in \DCATITsl{\tDss}$. We consider the commutative diagram
 \begin{equation}\label{eq:obscounitcofree}
     \begin{tikzcd}
       \DCATITsl{\tC^{\diamond}}(p,f^*R_f(q)) \arrow[r] \arrow[d,"\simeq"] &  \DCATITsl{\tC^{\diamond}}(p,q) \arrow[d,"\simeq"] \\
        \DCATITsl{\tC^{\diamond}}(f^*\DFFree_{\tDss}^{\vepsilon}(f_!p),q) \arrow[r,"\psi^*"] & \DCATITsl{\tC^{\diamond}}(p,q)
    \end{tikzcd}
 \end{equation}
  where the bottom is given by precomposition along the morphism
  \[
 \psi \colon p \xrightarrow{} f^*f_!p \xrightarrow{} f^*\DFFree_{\tDss}^{\vepsilon}(f_!p)
  \]
  where the first maps is the unit of $f_! \dashv f^*$ and the second
  is the pullback of the unit of the adjunction in \cref{propn:free
    dec partial fib}. Note that this diagram is natural in $p$, so by the Yoneda lemma we obtain a map $\epsilon \colon f^*R_f(q) \to q$ which we claim can be identified with the counit of the adjunction in \cref{thm:cofree}. In order to do so, we claim that the composite
  \[
    \FIB^{\vepsilon}_{/\tD}(p,R_f(q)) \xrightarrow{} \PFIB^{\vepsilon}_{/\tC^{\diamond}}(f^*p,f^*R_f(q)) \xrightarrow{\epsilon_*} \PFIB^{\vepsilon}_{/\tC^{\diamond}}(f^*p,q)
  \]
  can be identified with the natural isomorphism in the proof of \cref{thm:cofree}. The result follows from the commutativity of the diagram
    \[\begin{tikzcd}
    {f^*p} & {f^*f_!f^*p} & {f^*\DFFree^{\vepsilon}_{/\tC^{\diamond}}(f_!f^*p)} \\
    & {f^*p} & {f^*\DFFree^{\vepsilon}_{/\tC^{\diamond}}(p)}
    \arrow[from=1-1, to=1-2]
    \arrow["\simeq"', from=1-1, to=2-2]
    \arrow[from=1-2, to=1-3]
    \arrow[from=1-2, to=2-2]
    \arrow[from=1-3, to=2-3]
    \arrow[from=2-2, to=2-3]
  \end{tikzcd}\]
  after some minor unraveling of the definitions.
\end{observation}

\subsection{Localization of fibrations}
\label{sec:loc fib}

Given a functor of \icats{} $F \colon \oC \to \MCatI$, Hinich showed
in \cite[\S 2]{HinichDK}
that we can identify the cocartesian fibration for the functor $\Lm(F)$ (where we
invert the marked morphisms) as a localization of the fibration for
the underlying functor $\Um(F) \colon \oC \to \CatI$; an alternative
proof of this very useful comparison has also been given by
Nikolaus and Scholze~\cite[Proposition
A.14]{NikolausScholze}. Our goal in this section is to prove the
following \itcatl{} extension of this result; in fact, we will see
that our results on cofree fibrations allow us to prove this by a much
easier argument than the existing ones in the \icatl{} case.

\begin{thm}\label{thm:locn of fib}
  Suppose $\tEd \to \tB^{\sharp\sharp}$ is a decorated
  $(0,1)$-fibration, corresponding to a functor $F \colon \tB \to
  \DCATIT$. If $\tEd|_{\tB^{\flat\flat}}$ denotes the pullback to
  $\tB^{\flat\flat}$ (\ie{} we keep only the decorations in $\tEd$
  that map to equivalences), then the localization
  $\Ldec{\tEd|_{\tB^{\flat\flat}}} \to \tB$ is a $(0,1)$-fibration, and
  corresponds to the functor $\Ldec{F} \colon \tB \to \CATIT$.
\end{thm}

We'll derive this as a special case of the following more general
observation:
\begin{propn}\label{propn:partial fib loc on pb of dec fib}
  Suppose $p \colon \tEd \to \tB^{\sharp\sharp}$ is a decorated
  $(0,1)$-fibration, corresponding to a functor $F \colon \tB \to
  \DCATIT$, and let $\tE' \to \tB$ be the $(0,1)$-fibration for
  $\Ldec{F}$. Then for any decorated functor $f \colon \tCd \to
  \tB^{\sharp\sharp}$, the pullback of the canonical map $\eta \colon \tEd \to
  \tE'^{\natural}$ induces an equivalence
  \[ \DFun_{/\tCd}(f^{*}\tE'^{\natural}, \tQd) \isoto
    \DFun_{/\tCd}(f^{*}\tEd, \tQd)\]
  for any partial $(0,1)$-fibration $\tQd \to \tCd$. In other words,
  $f^{*}\eta$ exhibits $f^{*}\tE'^{\natural}$ as the localization
  $L^{(0,1)}_{\tCd}(f^{*}\tEd)$ to a partial $(0,1)$-fibration.
\end{propn}
\begin{proof}
  Using the adjunction
  \[ f^{*} : \FIB^{(0,1)}_{/\tB^{\sharp\sharp}} \rightleftarrows
    \PFIB^{(0,1)}_{/\tCd} : R_{f}\]
  from \cref{thm:cofree} and (decorated) straightening, we get natural
  equivalences of \icats{}
  \[
    \begin{split}
      \DFun_{/\tCd}(f^{*}\tE'^{\natural}, \tQd)
      & \simeq
        \DFun_{/\tB^{\sharp\sharp}}(\tE'^{\natural}, R_{f}(\tQd)^{\natural})\\
      & \simeq \Nat_{\tB, \CATIT}(\Ldec{F}, Q) \\
      & \simeq \Nat_{\tB, \DCATIT}(F, Q^{\flat\flat}) \\
      & \simeq \DFIB_{/\tB^{\sharp\sharp}}^{(0,1)}(\tEd, R_{f}(\tQd)^{\natural}),
    \end{split}
  \]
  where $Q$ denotes the straightening of $R_{f}(\tQd)$.
  Now we observe that for $\tP \to \tB$ a $(0,1)$-fibration, morphisms of
  decorated fibrations $\tEd \to \tP^{\natural}$ are just morphisms of
  decorated \itcats{} over $\tB^{\sharp\sharp}$ (as any decorated
  functor to $\tP^{\natural}$ in particular preserves cocartesian
  morphisms and cartesian 2-morphisms). We can therefore use
  \cref{obs:Rf is restr of adjoint} to get an equivalence
  \[
    \begin{split}
      \DFIB_{/\tB^{\sharp\sharp}}^{(0,1)}(\tEd,
      R_{f}(\tQd)^{\natural})
      & \simeq \DFun_{/\tB^{\sharp\sharp}}(\tEd, R_{f}(\tQd)^{\natural}) \\
      & \simeq \DFun_{/\tCd}(f^{*}\DFFree^{(0,1)}_{\tB^{\sharp\sharp}}(p), \tQd). \\
    \end{split}
  \]
  Now we can use
  that the unit map
  $p \to \DFFree^{(0,1)}_{\tB^{\sharp\sharp}}(p)$ pulls back to
  a $(0,1)$-equivalence along any map by \cref{propn:unit fib univ
    eps-equiv}; since $\tQd$ is a partial $(0,1)$-fibration this gives
  an equivalence
  \[ \DFun_{/\tCd}(f^{*}\DFFree^{(0,1)}_{\tB^{\sharp\sharp}}(p), \tQd) \simeq
    \DFun_{/\tCd}(f^{*}\tEd, \tQd),\]
  as required. 
\end{proof}

\begin{proof}[Proof of \cref{thm:locn of fib}]
  Let $\tE' \to \tB$ be the $(0,1)$-fibration for $\Ldec{F}$; we will
  prove that it has the universal property of the localization
  $\Ldec{\tEd|_{\tB^{\flat\flat}}}$: Applying \cref{propn:partial fib
    loc on pb of dec fib} to the identity functor $\tB^{\flat\flat}
  \to \tB^{\sharp\sharp}$, and noting that a functor to
  $\tB^{\flat\flat}$ is a partial $(0,1)$-fibration \IFF{} it is of
  the form $\tC^{\flat\flat} \to \tB^{\flat\flat}$, we get a natural
  equivalence
  \[ \DFun_{/\tB^{\flat\flat}}(\tEd|_{\tB^{\flat\flat}},
    \tC^{\flat\flat}) \simeq
    \DFun_{/\tB^{\flat\flat}}(\tE'^{\flat\flat}, \tC^{\flat\flat}),
  \]
  since $\tE'^{\natural}|_{\tB^{\flat\flat}}$ is decorated by the cocartesian 
  morphisms and cartesian 2-morphisms that lie over equivalences,
  \ie{} by the equivalences. For an \itcat{} $\tD$ we therefore have
  natural equvalences
  \[
    \begin{split}
      \Fun(\tE', \tD) & \simeq \Fun_{/\tB}(\tE', \tD \times \tB) \\
                      & \simeq
                        \DFun_{/\tB^{\flat\flat}}(\tE'^{\flat\flat},
                        (\tD \times \tB)^{\flat\flat}) \\
      & \simeq \DFun_{/\tB^{\flat\flat}}(\tEd|_{\tB^{\flat\flat}},
        (\tD \times \tB)^{\flat\flat}) \\
                      & \simeq \DFun(\tEd|_{\tB^{\flat\flat}}, \tD^{\flat\flat}) \\
      & \simeq \Fun(\Ldec{\tEd|_{\tB^{\flat\flat}}}, \tD),
    \end{split}
  \]
  It follows that we have an equivalence
  $\Ldec{\tEd|_{\tB^{\flat\flat}}} \simeq \tE'$, as required.
\end{proof}

\begin{cor}\label{cor:explicit locn to fibs}\
  \begin{enumerate}
  \item   The left adjoint to the inclusion
  \[ \FIB^{(0,1)}_{/\tB} \hookrightarrow
    \DFIB^{(0,1)}_{/\tB^{\sharp\sharp}}\]
  sends $\tEd \to \tB^{\sharp\sharp}$ to
  $\Ldec{\tEd|_{\tB^{\flat\flat}}} \to \tB$.
\item The left adjoint $L^{(0,1)}_{\tB^{\sharp\sharp}}$ to the inclusion
  \[ \FIB^{(0,1)}_{/\tB} \hookrightarrow
    \DCATITsl{\tB^{\sharp\sharp}}\]
  sends $\tCd \to \tB^{\sharp\sharp}$ to
  $\Ldec{\DFFree^{(0,1)}_{\tB^{\sharp\sharp}}(\tCd)|_{\tB^{\flat\flat}}}$. 
\item The left adjoint to the inclusion
  \[ \oFIB^{(0,1)}_{/\tB} \hookrightarrow
    \MFIB^{(0,1)}_{/\tB^{\sharp}}\]
  sends $(\tE,S) \to \tB^{\sharp}$ to
  $\Ldec{(\tE,S)^{\sharp}|_{\tB^{\flat\flat}}} \to \tB$,
  where we localize the morphisms from $S$ and all 2-morphisms in each fibre.
\item The left adjoint $L^{(0,1);1}_{\tB^{\sharp}}$ to the inclusion
  \[ \oFIB^{(0,1)}_{/\tB} \hookrightarrow
    \MCATITsl{\tB^{\sharp}}\]
  sends $(\tC,S) \to \tB^{\sharp}$ to
  $\Ldec{\DFFree^{(0,1)}_{\tB^{\sharp\sharp}}((\tC,S)^{\sharp}|_{\tB^{\flat\flat}})}$. 
  \end{enumerate}
\end{cor}
\begin{proof}
  The first statement follows from \cref{cor: dec str comp flat}
  together with \cref{thm:locn of fib}, and in (2) the right adjoint
  factors through $\DFIB^{(0,1)}_{/\tBss}$, so that the left adjoint
  factors as the decorated free fibration from \cref{propn:free dec
    partial fib} followed by the left adjoint from (1). The last
  statements follow from restricting the first two to the marked case.
\end{proof}

\section{(Co)limits and Kan extensions in \itcats{}}\label{sec:colimits}
In this section we will apply our results on fibrations from the
previous section to develop the theory of (co)limits and Kan extensions
for \itcats{}. In \S\ref{subsec:LimitsDefinitions} we introduce one
natural notion of (co)limits in \itcats{}, the (partially) lax and
oplax (co)limits, and compare some definitions of these. We then
introduce another notion of (co)limits, the \emph{weighted}
(co)limits, in \S\ref{sec:wtlim}; we prove that these types of
(co)limits are equivalent, in the sense that any weighted (co)limit can
be reexpressed as a partially lax or oplax (co)limit, and vice
versa. In \S\ref{sec:lax colim 2cat} we then look at such (co)limits
in $\CATIT$ and $\CATI$, and show that these have simple descriptions
in terms of fibrations. Next we collect some useful results on the
existence and preservation of (co)limits in \S\ref{sec:stuff}. In
\S\ref{sec:cofinal new} we will then see that the formalism of
decorated \itcats{} and our work on free fibrations make it easy to
understand cofinal functors of marked \itcats{}. After this we turn to
Kan extensions of \itcats{} in \S\ref{sec:kanext}, and then set up a
Bousfield--Kan formula for weighted colimits in
\S\ref{sec:bousfieldkan}. Finally, we give a fibrational proof of a
functorial version of the Yoneda lemma and use this to prove that
presheaves of \icats{} gives the free cocompletion of an \itcat{} in
\S\ref{sec:free cocomp}, and then apply this to compare some
characterizations of presentable \itcats{} in \S\ref{sec:pres}.

\subsection{Lax and oplax (co)limits}\label{sec:laxlim}
\label{subsec:LimitsDefinitions}
Partially (op)lax (co)limits were first introduced for strict
2-categories by Descotte, Dubuc, and Szyld in \cite{sigmalim} under
the name \emph{$\sigma$-(co)limits}, and have previously been studied
in the \icatl{} context by Berman~\cite{BermanLax} for diagrams in
$\CATI$ indexed by an \icat{}, and in greater generality by the first
author in \cite{AbMarked} and by Gagna--Harpaz--Lanari in
\cite{GagnaHarpazLanariLaxLim}. In this section we will first define
partially (op)lax colimits in \itcats{} by a universal property in
terms of partially (op)lax transformations, and derive several alternative
characterizations; in particular, we will see that our definition is
equivalent to that of \cite{GagnaHarpazLanariLaxLim}.

\begin{notation}\label{def:constant}
  We denote by $\underline{(-)} \colon \tC \to \FUN(\tI,\tC)$ the
  diagonal functor, \ie{} the one that transposes to the projection
  $\tI \times \tC \to \tC $ onto the second factor. We call this
  functor the \emph{constant diagram} functor.
\end{notation}

\begin{definition}\label{def:elaxlimit}
  Let $(\tI, E)$ be a marked \itcat{} and $F\colon \tI \to \tC$ a
  functor of \itcats{}. We say that an object of $\tC$ is:
  \begin{itemize}
  \item the \emph{$E$-(op)lax limit} of $F$, denoted by
    $\lim^{\elaxoplax}_\tI F$, if there exists an equivalence of
    functors
    \[
      \tC(c,\lim^{\elaxoplax}_\tI F) \simeq \Nat_{\tI,\tC}^{\elaxoplax}(\underline{c},F)
    \]
    that is natural in $c$;
  \item the \emph{$E$-(op)lax colimit} of $F$, denoted by
    $\colim^{\elaxoplax}_\tI F$, if there exists an equivalence of
    functors
    \[
      \tC(\colim^{\elaxoplax}_{\tI} F, c)\simeq \Nat_{\tI,\tC}^{\elaxoplax}(F,\underline{c})
    \]
    that is natural in $c$.
  \end{itemize}
  In the minimally marked case $(\tI, E)= \tI^{\flat}$ we refer to
  $E$-(op)lax (co)limits just as \emph{(op)lax (co)limits}, denoted
  $\lim^{\plax}_\tI F$ and $\colim^{\plax}_{\tI} F$. On the other
  hand, in the maximally marked case $(\tI, E)= \tI^{\sharp}$ we refer
  to them as \emph{strong}\footnote{We will see later in \cref{ex:conical} that these
    precisely correspond to \emph{conical} colimits in the sense of
    enriched category theory.} (co)limits, denoted $\lim_{\tI} F$ and
  $\colim_{\tI} F$; here there is no distinction between lax and
  oplax, as both universal properties are equivalent to
  (co)representing the (co)presheaves
  $\Nat_{\tI,\tC}(\underline{c},F)$ and
  $\Nat_{\tI,\tC}(F,\underline{c})$.
\end{definition}

\begin{lemma}\label{lem:lax colim change variance}
  For a functor $F \colon \tI \to \tC$ and an object $x \in \tC$, the following are equivalent:
  \begin{itemize}
  \item $x$ is the $E$-lax colimit of $F$ in $\tC$.
  \item $x$ is the $E$-oplax limit of $F^{\op}$ in $\tC^{\op}$.
  \item $x$ is the $E$-oplax colimit
    of $F^{\co}$ in $\tC^{\co}$.
  \item $x$ is the $E$-lax limit 
    of $F^{\coop}$ in $\tC^{\coop}$.
  \end{itemize}
\end{lemma}
\begin{proof}
  This follows from the definition and the equivalences among lax and oplax
  functor \itcats{} that arise from the order-reversing symmetry of
  the Gray tensor product under both $(\blank)^{\op}$ and
  $(\blank)^{\co}$ (see \cite[Observation 2.2.11]{AGH24}). 
\end{proof}

\begin{observation}\label{rem:limfunctor}
  Suppose that $\tC$ admits all partially (op)lax (co)limits indexed
  by a marked \itcat{} $(\tI,E)$. Then it follows that taking
  $E$-(op)lax (co)limits give adjoints to the constant diagram functor
  $\underline{(-)} \colon \tC \to \FUN(\tI, \tC)^{\elaxoplax}$:
  \[
    \underline{(-)} \colon \tC \llra \FUN(\tI,\tC)^{\elaxoplax} \colon
    \lim^{\elaxoplax}_\tI,\]
  \[  \colim^{\elaxoplax}_\tI \colon \FUN(\tI,\tC)^{\elaxoplax} \llra \tC \colon \underline{(-)}.
  \]
  Conversely, the existence of such adjoints implies that $\tC$ has
  all partially (op)lax (co)limits indexed by $(\tI, E)$.
\end{observation}

To compare our definition of (op)lax (co)limits to that given in
\cite{GagnaHarpazLanariLaxLim}, we first need to derive a description
of the \icats{} $\Nat^{\eplax}_{\tA,\tB}(F,G)$, for which we use the
following notation:
\begin{defn}\label{def:fiblaxnat}
  Let $p \colon \tA \to \tC$, $q \colon \tB \to \tC$ be
  $(i,j)$-fibrations and let $(\tC, E)$ be a marked \itcat{}.  For
  $i = 0$ we write $\Fun_{/\tC}^{E\dcoc}(\tA,\tB)$ for the full
  subcategory of $\Fun_{/\tC}(\tA, \tB)$ spanned by functors that
  preserve the cocartesian 1-morphisms that lie over $E$ as well as
  all $j$-cartesian 2-morphisms; for $i = 1$ we define 
  $\Fun_{/\tC}^{E\dcart}(\tA,\tB)$ similarly.
\end{defn}

\begin{proposition}\label{prop:Natelaxdescr}
  Let $(\tA, E)$ be a marked \itcat{}. For functors $F, G \colon \tA \to \tB$, we have a natural equivalence
  \[ \Nat^{\elax}_{\tA,\tB}(F,G) \simeq \Fun^{(E)}_{/\tA \times \tA}(\tA, (F,G)^{*}\ARopl(\tB)),\]
  where the right-hand side denotes the \icat{} of sections that send the edges in $E$ to commutative squares in $(F,G)^{*}\ARopl(\tB)$. Dually, we have an equivalence in the $E$-oplax case 
  \[ \Nat^{\eoplax}_{\tA,\tB}(F,G)^\op \simeq \Fun^{(E)}_{/\tA \times
      \tA}(\tA, (F,G)^{*}\ARlax(\tB)).\] In particular, for an object
  $b \in \tB$ we have:
  \[ \Nat^{\elax}_{\tA,\tB}(\underline{b},G) \simeq \Fun^{E\dcoc}_{/\tA}(\tA, G^{*}\tB_{b\upslash}), \enspace \enspace \Nat^{\eoplax}_{\tA,\tB}(\underline{b},G)^\op \simeq \Fun^{E\dcoc}_{/\tA}(\tA, G^{*}\tB_{b\downslash}),\]
  \[ \Nat^{\elax}_{\tA,\tB}(F,\underline{b}) \simeq \Fun^{E\dcart}_{/\tA}(\tA, F^{*}\tB_{\upslash b}), \enspace \enspace \Nat^{\eoplax}_{\tA,\tB}(F,\underline{b})^\op \simeq \Fun^{E\dcart}_{/\tA}(\tA, F^{*}\tB_{\downslash b}) .\]  
\end{proposition}

\begin{proof}
  We will first deal with the $E$-lax case in full detail and then
  explain how to adapt the argument to derive the results in the
  $E$-oplax case. Note that we have a natural equivalence
  \[ \ARopl(\FUN(\tA,\tB)^{\lax}) \simeq \FUN(\tA, \ARopl(\tB))^{\lax}\]
  from \cite[Lemma 2.2.9]{AGH24}. Identifying $\Nat^{\lax}_{\tA,\tB}(F,G)$ with the fibre over $F,G$ on the left, we get
  \[ \Nat^{\lax}_{\tA,\tB}(F,G) \simeq \FUN_{/\tA \times \tA}(\tA, (F,G)^{*}\ARopl(\tB))^{\lax}.\]
  We claim that the right-hand side is in fact the \icat{} $\Fun_{/\tA \times \tA}(\tA, (F,G)^{*}\ARopl(\tB))$. To see this it will suffice to show that for any commutative diagram
  \[
    \begin{tikzcd}
      {[1]}\otimes \tA \arrow[rr,"f"] \arrow[dr,swap,"p_{\tA}"] && (F,G)^{*}\ARopl(\tB) \arrow[dl] \\
      & \tA \times \tA, &
    \end{tikzcd}
  \]
  the map $f$ factors through the cartesian product $[1] \times \tA$,
  where $p_{\tA}$ is given by the projection to $\tA$ and followed by the diagonal map. This follows from the fact that $p_{\tA}$ factors through $[1] \times \tA$ together with \cite[Proposition 2.4.3]{AGH24}.

  To deal with the $E$-lax case, we observe that
  $\Nat^{\elax}_{\tA,\tB}(F,G)$ is a full subcategory of
  $\Nat^{\lax}_{\tA,\tB}(F,G)$. Chasing through the equivalences above
  one sees that an $E$-lax natural transformation is identified with a
  section that sends the edges in $E$ to a commutative square in
  $(F,G)^{*}\ARopl(\tB)$.

  In the oplax case we use the equivalence $\ARlax(\tX) \simeq (\ARopl(\tX^{\op}))^{\op}$ to identify the fibre at $(x,y) \in \tX \times \tX$ as
  \[ \tX^{\op}(y,x)^{\op} \simeq \tX(x,y)^{\op}.\]
  We can therefore identify $\Nat^{\oplax}_{\tA,\tB}(F,G)^{\op}$ as a fibre in
  \[ \ARlax(\FUN(\tA,\tB)^{\oplax}) \simeq \FUN(\tA, \ARlax(\tB))^{\oplax},\]
  and proceed as in the lax case.
\end{proof}

\begin{corollary}\label{cor:laxlimituniv}
  For $F \colon \tA \to \tB$ and $b \in \tB$, we have natural equivalences
  \begin{align}
    \label{eq:natlaxbF}
    \Nat^{\elax}_{\tA,\tB}(\underline{b},F) & \simeq
    \Nat^{\elax}_{\tA,\CATI}(\underline{*}, \tB(b,F(\blank))), \\
    \label{eq:natlaxFb}
    \Nat^{\elax}_{\tA,\tB}(F,\underline{b}) & \simeq \Nat^{\eoplax}_{\tA^{\op},\CATI}(\underline{*}, \tB(F(\blank),b)), \\
    \label{eq:natoplaxbF}
    \Nat^{\eoplax}_{\tA,\tB}(\underline{b},F) & \simeq
    \Nat^{\eoplax}_{\tA,\CATI}(\underline{*}, \tB(b,F(\blank))), \\
    \label{eq:natoplaxFb}
    \Nat^{\eoplax}_{\tA,\tB}(F,\underline{b}) & \simeq \Nat^{\elax}_{\tA^{\op},\CATI}(\underline{*}, \tB(F(\blank),b)).
  \end{align}
\end{corollary}
\begin{proof}
  From \cref{prop:Natelaxdescr} we have
  $\Nat^{\elax}_{\tA,\tB}(\underline{b},F) \simeq
  \Fun^{E\dcoc}_{/\tA}(\tA, F^{*}\tB^{\oplax}_{b\upslash})$, where the
  right-hand side is a
  mapping \icat{} in $\FIB^{(0,1)}_{/(\tA,E)}$. Under straightening, the identity of $\tA$
  corresponds to the constant functor $\underline{*}$ and
  $\tB_{b\upslash}$ to $\tB(b,\blank)$ by \cite[Proposition
  4.1.8]{LurieGoodwillie}; the pullback along $F$ therefore
  corresponds to $\tB(b,F(\blank))$, as required. A totally analogous
  argument shows that equivalence \cref{eq:natlaxFb} also holds. 

  We now look at $\Nat^{\eoplax}_{\tA,\tB}(\underline{b},F)^\op \simeq
  \Fun^{E\dcoc}_{/\tA}(\tA, F^{*}\tB_{b\downslash})$ and observe that
  by taking opposite \icats{} we can produce an equivalence
  \[
    \Fun^{E\dcoc}_{/\tA}(\tA, F^{*}\tB_{b\downslash}) \simeq \Fun^{E\dcart}_{/\tA^\op}(\tA^\op, (F^{*}\tB_{b\downslash})^\op)^\op.
  \]
  Finally we see that $(F^{*}\tB_{b\downslash})^\op$ is the
  $(1,0)$-fibration that classifies the functor $\tB(b,F(\blank))$, so from \cref{thm:elaxstr} we get
  \[
    \Fun^{E\dcart}_{/\tA^\op}(\tA^\op, (F^{*}\tB_{b\downslash})^\op)^\op \simeq \Nat^{\eoplax}_{\tA,\CATI}(\underline{\ast},\tB(b,F(-)))^\op,
  \]
  and the case \cref{eq:natoplaxbF} follows. The remaining
  verification  for \cref{eq:natoplaxFb} is proved by a dual argument.
\end{proof}

\begin{corollary}\label{cor:laxlimit rep}
  For $F \colon \tA \to \tB$, the $E$-(op)lax (co)limits of $F$ are
  uniquely characterized by the following (co)representability
  properties in terms of partially (op)lax limits in $\CATI$:
  \begin{itemize}
  \item $\tB(b, \lim^{\elax}_{\tA} F) \simeq \lim^{\elax}_{\tA} \tB(b,
    F)$,
  \item $\tB(\colim^{\elax}_{\tA} F, b) \simeq
    \lim^{\eoplax}_{\tA^{\op}} \tB(F, b)$,
  \item $\tB(b, \lim^{\eoplax}_{\tA} F) \simeq \lim^{\eoplax}_{\tA} \tB(b,
    F)$,
  \item $\tB(\colim^{\eoplax}_{\tA} F, b) \simeq
    \lim^{\elax}_{\tA^{\op}} \tB(F, b)$.    
  \end{itemize}
\end{corollary}
\begin{proof}
  We prove the first statement; the rest follow similarly. Here we
  have
  \[
    \begin{split}
      \tB(b, \lim^{\elax}_{\tA} F)
      & \simeq \Nat^{\elax}_{\tA,\tB}(\underline{b}, F) \\
      & \simeq \Nat^{\elax}_{\tA,\CATI}(\underline{*}, \tB(b, F)) \\
      & \simeq \Fun(*, \lim^{\elax}_{\tA} \tB(b,F)) \\
      & \simeq \lim^{\elax}_{\tA} \tB(b,F),
    \end{split}
  \]
  where the first and third equivalences follow from the definition of
  $E$-lax limits, the second from \cref{cor:laxlimituniv}, and the
  last is obvious.
\end{proof}

Combining this with \cite[Corollary 5.1.7]{GagnaHarpazLanariLaxLim},
we get the following comparison with the definitions of Gagna--Harpaz--Lanari:
\begin{corollary}
  Let $(\tA,E)$ be a marked \itcat{}. The $E$-(op)lax (co)limits of a
  functor $F \colon \tA \to \tB$ as defined above agree with the inner
  and outer (co)limits of $F$ defined in
  \cite{GagnaHarpazLanariLaxLim} as follows:\footnote{Here the
    discrepancy in the pairing between lax/oplax and inner/outer for
    limits and colimits is due to the switch from lax to oplax in
    \cref{eq:natlaxFb}, which does not occur in \cref{eq:natlaxbF}.}
  \begin{itemize}
  \item the $E$-lax limit of $F$ is the inner limit of $F$,
  \item the $E$-oplax limit of $F$ is the outer limit of $F$,
  \item the $E$-lax colimit of $F$ is the outer colimit of $F$,
  \item the $E$-oplax colimit of $F$ is the inner colimit of $F$.\qed
  \end{itemize}
\end{corollary}

\begin{remark}
  By \cref{rep fib reformulate}, exhibiting an object $c \in \tC$ as
  the $E$-oplax colimit of $F \colon \tI \to \tC$ amounts to
  exhibiting the $(0,1)$-fibration $\tC^{\eoplax}_{F\upslash} \to \tC$
  as representable by an object over $c$ as in \cref{propn:rep cart
    fib cond}. By \cref{lem:lax slice cone} we can think of such an
  object as a functor $\tJ_{E\doplax}^{\triangleright} \to \tC$ that
  satisfies the conditions of \cref{propn:rep cart fib cond}. Such
  partially lax cones and cocones can be identified with the inner and
  outer (co)limit cones used in \cite[\S 6]{GagnaHarpazLanariLaxLim}.
\end{remark}

\subsection{Weighted (co)limits in \itcats{}}\label{sec:wtlim}
In $\oV$-enriched ($\infty$-)category theory the general notion of
(co)limits is that of $\oV$-\emph{weighted} (co)limits. In this
section we first recall the definition of $\CatI$-weighted (co)limits
in \itcats{} and give an easy proof of their description as partially
(op)lax (co)limits from \cite{GagnaHarpazLanariLaxLim} using our new
description of the latter. We then show that any partially (op)lax
(co)limit can also be expressed as an explicit weighted (co)limit. This allows
us to conclude that the theories of partially lax, partially oplax,
and weighted (co)limits are equivalent, so that they all give rise to
the same notion of (co)complete \itcats{} and (co)continuous functors
among these; we thus provide a model-independent treatment of the main
results in \cite[\S 6.2]{GagnaHarpazLanariLaxLim}. 

\begin{defn}
  For $F \colon \tA \to \tB$ and $W \colon \tA \to \CATI$, the
  \emph{$W$-weighted limit} of $F$, if it exists, is the object
  $\lim^{W}_{\tA} F$ of
  $\tB$ characterized by the universal property
  \[ \tB(b, \lim^{W}_{\tA}F) \simeq \Nat_{\tA,\CATI}(W, \tB(b,F)).\]
  Dually, for $F \colon \tA \to \tB$ and $W \colon \tA^{\op} \to
  \CATI$, the \emph{$W$-weighted colimit} of $F$, if it exists, is the
  object $\colim^{W}_{\tA} F$ characterized by the universal property
  \[ \tB(\colim^{W}_{\tA}F, b) \simeq \Nat_{\tA^{\op},\CATI}(W, \tB(F, b)).\]
\end{defn}

\begin{observation}\label{obs: weighted lim rep}
  For a functor $F \colon \tA \to \CATI$, we have
  \[
    \begin{split}
      \Fun(\oC, \lim^{W}_{\tA} F)
      & \simeq \Nat_{\tA,\CATI}(W, \Fun(\oC,F)) \\
      & \simeq \Nat_{\tA,\CATI}(W \times \oC, F) \\
      & \simeq \Fun(\oC, \Nat_{\tA, \CATI}(W, F)).
    \end{split}
  \]
  By the Yoneda Lemma we can thus describe weighted limits in $\CATI$ by a natural equivalence
  \[ \lim^{W}_{\tA} F \simeq \Nat_{\tA, \CATI}(W, F). \]
  We can therefore also characterize the weighted (co)limits
  of a functor $F \colon \tA \to \tB$ by  natural equivalences
  \[ \tB(b, \lim^{W}_{\tA}F) \simeq \lim^{W}_{\tA} \tB(b, F)\]
  for $W \colon \tA \to \CATI$, and
  \[ \tB(\colim^{W}_{\tB}F, b) \simeq \lim^{W}_{\tA^{\op}} \tB(F, b)\]
  for $W \colon \tA^{\op} \to \CATI$.
\end{observation}

\begin{example}
  Given an \icat{} $\oA$ and an object $c$ in an \itcat{} $\tC$, we can consider the
  colimit of $[0] \xto{c} \tC$ weighted by $[0] \xto{\oA} \CATI$. We
  denote this by
  \[ \oA \boxtimes c := \colim^{\oA}_{[0]} c;\]
  its universal property is that there is a natural equivalence
  \[ \tC(\oA \boxtimes c, c') \simeq \Fun(\oA, \tC(c,c')),\]
  so this is precisely the \emph{tensoring} of $c$ by
  $\oA$. Similarly, the weighted limit
  \[ c^{\oA} := \lim^{\oA}_{[0]} c\]
  is the \emph{cotensoring} of $c$ by $\oA$, satisfying
  \[ \tC(c', c^{\oA}) \simeq \Fun(\oA, \tC(c,c')).\]
\end{example}

\begin{observation}\label{obs:weighted over igpd}
  Suppose $X$ is an \igpd{}. Then for $W \colon X \to \CATI$ and $F
  \colon X \to \tC$, we have
  \[ \colim^{W}_{X} F \simeq \colim_{x \in X} W(x) \boxtimes F(x),\]
  since by definition we have
  \[
    \begin{split}
      \tC(\colim^{W}_{X} F, c) & \simeq \Nat_{X, \CATI}(W, \tC(F,c))
      \\
                               & \simeq \lim_{x \in X} \Fun(W(x),
                                 \tC(F(x),c)) \\
                               & \simeq \lim_{x \in X} \tC(W(x)
                                 \boxtimes F(x), c) \\
                               & \simeq \tC(\colim_{x \in X} W(x)
                                 \boxtimes F(x), c),
    \end{split}
  \]
  where the second equivalence holds since $\FUN(X, \tC) \simeq
  \lim_{X} \tC$.
\end{observation}

\begin{example}\label{ex:conical}
  Given a functor $F \colon \tJ \to \tC$ we can always consider the weight $\underline{*} \colon \tJ \to \CATI$ that is constant at the terminal object. In this case the universal property of the \emph{conical limit} $\lim^{\underline{*}}_{\tJ} F$ is the same as that of the strong limit $\lim_{\tJ} F$: there are natural equivalences
  \[ \tC(c,\lim_{\tJ}^{\underline{*}}F) \simeq \Nat_{\tJ,\CATI}(\underline{*}, \tC(c,F)) \simeq \Nat_{\tJ,\tC}(\underline{c}, F) \simeq \tC(c, \lim_{\tJ} F)\]
  via \cref{cor:laxlimituniv}. In the case where $F \colon \oJ \to \tC$  is a functor from an \icat{}, then we also have
  \[ \tC(c,\lim^{\underline{*}}_{\oJ}F)^{\simeq} \simeq
    \Nat_{\oJ,\tC}(\underline{c}, F)^{\simeq} \simeq
    \Nat_{\oJ,\tC^{\leq 1}}(\underline{c}, F^{\leq 1})\]
  since $\FUN(\oJ,\tC)^{\leq 1} \simeq \Fun(\oJ,\tC^{\leq 1})$. Thus in this case the conical limit $\lim^{\underline{*}}_{\oJ} F$ is also the ordinary limit of $F$ viewed as a diagram in the underlying \icat{} $\tC^{\leq 1}$. In particular, if we know that $\tC$ has a certain conical/strong limit indexed over an \icat{}, to identify it we only have consider the limit in $\tC^{\leq 1}$.
\end{example}

We now prove that we can express weighted limits as both lax and oplax
limits; this was previously shown as \cite[Proposition
6.2.3]{GagnaHarpazLanariLaxLim}.
\begin{propn}\label{propn:weighted limit as lax}
  For $W \colon \tA \to \CATI$ and $F \colon \tA \to \tB$,
  let $p \colon \tW \to \tA$ be the $(0,1)$-fibration for $W$ and $q
  \colon \tW' \to \tA^{\op}$ be the $(1,0)$-fibration. If $C$ is the 
  collection of $p$-cocartesian morphisms in $\tW$ and $C'$ that of
  $q$-cartesian morphisms in $\tW'$, then we have equivalences
  \[
    \begin{split}
      \lim^{W}_{\tA} F & \simeq \lim_{\tW}^{C\dlax} F \circ p \\
       & \simeq \lim^{C'\doplax}_{\tW'^{\op}} F \circ q^{\op}.
    \end{split}
  \] 
\end{propn}
\begin{proof}
  We first consider the lax case.  Let $b \in \tB$ and observe that we
  have natural equivalences
\[
  \begin{split}
  \tB(b, \lim^{W}_{\tA} F) & \simeq
                             \Nat_{\tA,\CATI}(W,\tB(b,F(-))) \\
                           & \simeq
                             \Fun^{\coc}_{/\tA}(\tW,F^*\tB_{b\upslash})   \\
                           & \simeq \Fun^{C\dcoc}_{/\tW}(\tW,(F\circ p)^*\tB_{b\upslash}),
  \end{split}
\]
where the second equivalence is given by straightening and the third equivalence is simply given by base change. We further see 
\[
  \Fun^{C\dcoc}_{/\tW}(\tW,(F\circ p)^*\tB_{b\upslash}) \simeq \Nat^{C\dlax}_{\tW,\CATI}(\underline{\ast},\tB(b,F\circ p(-)))\simeq \Nat^{C\dlax}_{\tW,\tB}(\underline{b},F \circ p),
\]
where the first equivalence is given by \cref{thm:elaxstr} and the
second uses \cref{cor:laxlimituniv}. Therefore, both universal
properties coincide, as desired. The oplax case is proved similarly,
using straightening for $(1,0)$-fibrations instead.
\end{proof}

Dualizing, we similarly have:
\begin{corollary}\label{cor:weightedcolimitaslax}
  For $W \colon \tA^{\op} \to \CATI$ and $F \colon \tA \to \tB$,
  let  $p \colon \tW \to \tA$ be the $(1,0)$-fibration for $W$
  and $q \colon \tW' \to \tA^{\op}$ be the $(0,1)$-fibration. If
  $C$ is the collection of $p$-cartesian morphisms in $\tW$ and $C'$
  that of $q$-cocartesian morphisms in $\tW'$, we have
  \[
    \begin{split}
      \colim^{W}_{\tA} F & \simeq \colim_{\tW}^{C\dlax} F \circ p \\
       & \simeq \colim_{\tW'^{\op}}^{C'\doplax} F \circ q^{\op}.
    \end{split}
  \]
\end{corollary}

We now want to express partially (op)lax (co)limits as weighted
(co)limits:
\begin{propn}\label{propn:lax as Kan weighted}
  Let $(\tA,E)$ be a marked \itcat{} and consider functors $F \colon
  \tA \to \tB$ and $D \colon \tB \to \tC$. Then we have equivalences
  \begin{align*}
    \lim_{\tA}^{\elax}DF & \simeq \lim_{\tB}^{\Frof^{(0,1)}_{\tB}(\tA,E)} D, & \colim_{\tA}^{\elax} DF  & \simeq \colim^{\Frof^{(1,0)}_{\tB}(\tA,E)}_{\tB} D, \\
    \lim_{\tA}^{\eoplax}DF & \simeq \lim^{\Frof^{(0,0)}_{\tB}(\tA,E)^{\op}}_{\tB} D, &  \colim_{\tA}^{\eoplax}DF &\simeq \colim^{\Frof^{(1,1)}_{\tB}(\tA,E)^{\op}}_{\tB}F,
  \end{align*}
  whenever either side exists in $\tC$, where the weights $\Frof^{\ve}_{\tB}(\tA,E)$ are as defined in \cref{not:free 1-fib functor}.
\end{propn}
\begin{proof}
  For the lax limit, we have equivalences
    \[
    \Nat^{\elax}_{\tA,\tC}(\underline{c},DF)\simeq \Nat^{\elax}_{\tA,\CATI}(\underline{\ast},\tC(c,DF(\blank)))\simeq \oFIB^{(0,1)}_{/(\tA,E)}(\tA,F^{*}D^{*}\tC_{c \upslash}).
  \]
  Here the right-hand side is equivalently the mapping \icat{}
  \[\MCATITsl{\tA^{\sharp}}((\tA,E), F^{*}D^{*}(\tC_{c
      \upslash})^{\natural}) \simeq
    \MCATITsl{\tB^{\sharp}}((\tA,E), D^{*}(\tC_{c \upslash})^{\natural})
  \] where $F^{*}D^{*}(\tC_{c
    \upslash})^{\natural}$ is the canonical marking by cocartesian
  1-morphisms. By \cref{cor:free 1-fibred on marked}, this is
  naturally equivalent to $\Nat_{\tB, \CATI}(\Frof^{(0,1)}_{\tB}(\tA,E), \tC(c,D))$, so
  that we have
    \[ \Nat^{\elax}_{\tA,\tC}(\underline{c},DF) \simeq \Nat_{\tB, \CATI}(\Frof^{(0,1)}_{\tB}(\tA,E), \tC(c,D)).\]
  By the Yoneda lemma, it follows that an object of $\tC$ represents
  the left-hand side, \ie{} is an $E$-lax limit of $DF$, \IFF{} it
  represents the right-hand side, \ie{} is a
  $\Frof^{(0,1)}_{\tB}(\tA,E)$-weighted limit of $F$. The other cases
  are proved similarly.
\end{proof}

As a special case, our identification of free fibrations gives the
following more explicit version of \cite[Corollary 6.2.11]{GagnaHarpazLanariLaxLim}:
\begin{cor}\label{prop:laxasweighted}
  Let $(\tC,E)$ be a marked \itcat{} and consider a functor
  $F\colon \tC \to \tD$. Then we have equivalences
  \begin{align*}
    \lim_{\tC}^{\elax}F & \simeq \lim^{\Frof_{\tC}^{(0,1)}(\tC,E)}_{\tC}F, &
                                                                    \colim_{\tC}^{\elax}F
    & \simeq \colim^{\Frof_{\tC}^{(1,0)}(\tC,E)}_{\tC}F, \\
    \lim_{\tC}^{\eoplax}F & \simeq
                            \lim^{\Frof_{\tC}^{(0,0)}(\tC,E)}_{\tC}F, &  \colim_{\tC}^{\eoplax}F &\simeq \colim^{\Frof_{\tC}^{(1,1)}(\tC,E)}_{\tC}F.
  \end{align*}
  whenever either side exists in $\tD$. \qed
\end{cor}

Combining \cref{prop:laxasweighted} and \cref{propn:weighted limit as
  lax}, we have shown that partially lax and oplax limits can both be expressed
as weighted limits, and vice versa, so that in a sense all three types
of limits are equivalent. In particular, this implies:
\begin{cor}\label{cor:cocomplete}
  The following are equivalent for an \itcat{} $\tC$:
  \begin{enumerate}[(1)]
  \item $\tC$ has all small partially lax (co)limits.
  \item $\tC$ has all small partially oplax (co)limits.
  \item $\tC$ has all small $\CatI$-weighted (co)limits. \qed
  \end{enumerate}
\end{cor}

\begin{defn}\label{def:conti}
  We say that $\tC$ is \emph{(co)complete} if any of the conditions in
  \cref{cor:cocomplete} hold. Given a functor $f \colon \tC \to \tD$
  of \itcats{}, we will say that $f$ is \emph{(co)continuous} if it preserves all small partially (op)lax (co)limits (and therefore all small $\CatI$-weighted (co)limits). 
\end{defn}

We note the following standard decomposition results for weighted (co)limits:
\begin{propn}\label{propn:decomp}
  Suppose we have $W \colon \tC \to \CATI$ and $F \colon \tC \to \tD$. If we can write $\tC \simeq \colim_{j \in \oJ} \tC_{j}$ for some diagram $\oJ \to \CatIT$ and write $W_{j}$, $F_{j}$ for the restrictions  of $W$ and $F$ to $\tC_{j}$, then we have
  \[ \lim^{W}_{\tC} F \simeq \lim_{\oJ^{\op}}\lim_{\tC_{j}}^{W_{j}}F_{j},\]
  or more precisely if the right-hand side exists in $\tD$ then it is the $W$-weighted limit of $F$.
\end{propn}
\begin{proof}
  Since $\FUN(\tC,\tD) \simeq \lim_{j \in \oJ^{\op}} \FUN(\tC_{j},\tD)$, we have
  \[
    \begin{split}
      \tD(d, \lim^{W}_{\tC} F) & \simeq \Nat_{\tC,\CATI}(W, \tD(d, F)) \\
                               &\simeq \lim_{j \in \oJ^{\op}} \Nat_{\tC_{j}, \CATI}(W_{j}, \tD(d,F_{j})) \\
                               & \simeq \lim_{j \in \oJ^{\op}} \tD(d, \lim^{W_{j}}_{\tC_{j}} F_{j}) \\
                               & \simeq \tD(d, \lim_{j \in \oJ^{\op}}\lim^{W_{j}}_{\tC_{j}} F_{j}),
    \end{split}
  \]
  as required.
\end{proof}

\begin{propn}\label{propn:colim over colim of weights}
  For $\Phi \colon \tJ \to \tPSh(\tC)$, $W \in \tPSh(\tJ)$ and $F \colon \tC \to \tD$, we have
  \[ \colim^{\colim^{W}_{\tJ} \Phi}_{\tC} F \simeq \colim^{W}_{\tJ}
    (\colim^{\Phi(\blank)}_{\tC} F),\]
  if the right-hand side exists.
\end{propn}
\begin{proof}
  We have
  \[
    \begin{split}
      \tD(\colim^{W}_{\tJ}(\colim^{\Phi(\blank)}_{\tC} F), d)
      & \simeq \Nat_{\tJ^{\op},\CATI}(W, \tD(\colim^{\Phi(\blank)}_{\tC} F, d)) \\
      & \simeq \Nat_{\tJ^{\op},\CATI}(W, \Nat_{\tC^{\op},\CATI}(\Phi, \tD(F,d))) \\
      & \simeq \Nat_{\tC^{\op},\CATI}(\colim^{W}_{\tJ}\Phi, \tD(F,d)) \\
      & \simeq \tD(\colim^{\colim^{W}_{\tJ} \Phi}_{\tC} F, d),
    \end{split}
  \]
  \ie{} the functors corepresented by the two objects are equivalent.
\end{proof}

By combining \cref{propn:lax as Kan weighted} and \cref{propn:colim over colim of weights} we obtain the following result on decomposition of partially (op)lax colimits.

\begin{proposition}\label{prop:diagramdecomposition}
  Let $(\tI,E)$ be a decorated \itcat{} and consider a diagram
  $d \colon \tI \to \tC$. Suppose further that we are given a
  diagram $\tau \colon \oJ \to \MCatIT$ indexed by an \icat{} $\oJ$
  such that the colimit of $\tau$ in $\MCatIT$ is $(\tI,E)$, and that $\tC$
  admits strong colimits of shape $\oJ$. Write
  $(\tI_j,E_j)= \tau(j)$ and let $d_j$ be the restriction of $d$
  along the canonical map $\tI_j \to \tI$. If the partially (op)lax
  colimits of $d$ and each of the $d_j$ exists in $\tC$, then we have
  an equivalence
  \[
  \colim_{\tI}^{\elaxoplax}d \simeq \colim_{j \in \oJ}\colim_{\tI_j}^{E_{j}\dplax}d_j.
  \]
\end{proposition}
\begin{proof}
  From \cref{propn:lax as Kan weighted} we have a natural equivalence
  \[ \colim_{\tI_j}^{E_{j}\dplax}d_j \simeq
    \colim_{\tI}^{\Frofe(\tI_{j},E_{j})} d\]
  for the appropriate value of $\ve$.
  We can therefore rewrite the right-hand side as
  \[ \colim_{\tI}^{\colim_{j \in \oJ} \Frofe(\tI_{j}, E_{j})} d. \]
  It remains to observe that we have an equivalence
  \[ \Frofe(\tI, E)) \simeq \colim_{j \in \oJ} \Frofe(\tI_{j},
    E_{j});\]
  this follows from the free 1-fibred fibration being a left adjoint.
\end{proof}

\subsection{(Co)limits of \itcats{} in terms of fibrations}
\label{sec:lax colim 2cat}
In this section we will use straightening to provide a fibrational
description of partially (op)lax
(co)limits in $\CATIT$. This extends (and recovers) the known
characterization for diagrams in $\CATI$ due to Berman \cite[Theorem
4.4]{BermanLax}. The analogous results for
$(\infty,\infty)$-categories have been proved by Loubaton; see
\cite[Examples 4.2.3.12--13]{loubaton}.

\begin{notation}
  Let $(\tA,E)$ be a marked \itcat{} and $p \colon \tE \to \tA$ be an
  $\ve$-fibration for $\ve = (i,j)$. We write
  \[\tE^{\natural}|_{E} := \tEn \times_{\tA^{\sharp\sharp}} (\tA,E)^{\sharp}\]
  for the decorated \itcat{} where we decorate only those
  $i$-cartesian morphisms that lie over $E$ together with all of the
  $j$-cartesian 2-morphisms.
\end{notation}

\begin{proposition}\label{prop:laxcolimloc}
  Let $(\tA,E)$ be a marked \itcat{} and consider a functor
  $F \colon \tc{A} \to \CATIT$ with corresponding
  $\vepsilon$-fibration
  $\tE_{\ve} \to \tc{A}^{\veop}$. The $E$-(op)lax colimits of
  $F$ can then be described as
  \[
    \colim^{\elax}_{\tc{A}}F \simeq \Ldec{\tEn_{(0,j)}|_{E}}, \quad
    \colim^{\eoplax}_{\tc{A}}F \simeq \Ldec{\tEn_{(1,j)}|_{E}}
  \]
  for both $j = 0,1$, where the localization 
  $\Ldec{\tEn_{(i,j)}|_{E}}$
  is obtained by inverting the $i$-cartesian 1-morphisms over $E$ and
  \emph{all} of the $j$-cartesian 2-morphisms in $\tE_{(i,j)}$. 
\end{proposition}
\begin{proof}
  We consider the case $\vepsilon=(0,1)$ without loss of
  generality. It follows from \cref{thm:elaxstr} that we have a
  natural equivalence of functors
  \[
     \Nat^{\elax}_{\tc{A},\CATIT}\left(F, \underline{(-)}\right)\simeq \Fun^{E\dcoc}_{/\tc{A}}(\tE_{(0,1)},(-)\times \tc{A}).
   \]
   Here the right-hand side at $\tB$ is the \icat{} of functors $\tE_{(0,1)} \to \tB \times \tA$ over $\tA$ that preserve cartesian 2-morphisms and cocartesian morphisms over $E$. In $\tB \times \tA$ these are the 2-morphisms and morphisms whose projection to $\tB$ are equivalences, so under the equivalence
   \[\Fun_{/\tA}(\tE_{(0,1)}, \tB \times \tA) \simeq \Fun(\tE_{(0,1)}, \tB)\]
   the full subcategory $\Fun^{E\dcoc}_{/\tc{A}}(\tE_{(0,1)},(-)\times \tc{A})$ corresponds to 
   $\Fun(\Ldec{\tEn_{(0,1)}|_{E}}, \tB)$. In other words, we have a natural equivalence
  \[
    \Nat^{\elax}_{\tc{A},\CATIT}\left(F, \underline{(-)}\right) \simeq \Fun(\Ldec{\tEn_{(0,1)}|_{E}},\blank),
  \]
  which precisely identifies $\Ldec{\tEn_{(0,1)}|_{E}}$ as the $E$-lax
  colimit of $F$.
\end{proof}

As a special case where we can omit the localization, we get a description of fibrations over \icats{} as
(op)lax colimits (first proved in \cite{GHN} for functors to $\CATI$):
\begin{cor}
  Suppose $\oA$ is an \icat{}. For a functor $F \colon \oA \to
  \CATIT$ with corresponding $0$-fibration $\tE_{0} \to \oA$ and
  $1$-fibration $\tE_{1} \to \oA^{\op}$, we have
  \[ \tE_{0} \simeq \colim^{\lax}_{\oA} F, \quad \tE_{1} \simeq \colim^{\oplax}_{\oA} F.\]
\end{cor}

\begin{remark}
  Note that our result does not give a similar description of fibrations over
  \itcats{}, as the equivalence of \cref{prop:laxcolimloc} always
  requires us to invert the (co)cartesian 2-morphisms; to
  avoid this presumably requires an $(\infty,3)$-categorical notion of
  (op)lax colimits.
\end{remark}

We also get a description of (op)lax colimits of \icats{}:
\begin{corollary}\label{cor:laxcoliminCat1}
  Let $(\tA,E)$ be a marked \itcat{} and consider a functor $F
  \colon\tc{A} \to \CATI$ with corresponding 1-fibred
  $\vepsilon$-fibration $\tE_{\ve} \to \tc{A}^{\veop}$. Then the $E$-(op)lax colimits of $F$ in $\CATI$ can be described as
  \[
    \colim^{\elax}_{\tc{A}}F \simeq \Ldec{\tEn_{(0,j)}|_{E}}, \quad
    \colim^{\eoplax}_{\tc{A}}F \simeq \Ldec{\tEn_{(1,j)}|_{E}}
  \]
  for $j = 0,1$, where $\Ldec{\tEn_{(i,j)}|_{E}}$ is obtained by inverting
  the $i$-cartesian 1-morphisms over $E$ together with \emph{all}
  2-morphisms (since all 2-morphisms are $j$-cartesian in a 1-fibred
  $(i,j)$-fibration).  \qed
\end{corollary}

We now turn to partially (op)lax \emph{limits}; to describe these we
introduce the following notation:
\begin{defn}\label{def:FUNdcoc}
  Let $(\tA,E)$ be a marked \itcat{} and let $\tE, \tF \to \tA$ be two
  $\ve$-fibrations over $\tA$ for $\ve = (i,j)$. For $i = 0$, we
  denote by \[\FUN^{E\dcoc}_{/\tA}(\tE, \tF) :=
    \DFUN_{/(\tA,E)^{\sharp}}(\tEn|_{E}, \tF^{\natural}|_{E})\]
  the \itcat{} characterized by the universal property
  \[
    \Fun(\tX,\FUN^{E\dcoc}_{/\tc{A}}(\tE,\tF)) \simeq \Fun^{E\dcoc}_{/\tc{A}}(\tE\times \tX,\tF),
  \]
  \ie{} the full sub-\itcat{} of $\FUN_{/\tA}(\tE, \tF)$ spanned by
  the functors that preserve cocartesian morphisms over $E$ and
  $j$-cartesian 2-morphisms. (In
  other words, $\FUN^{E\dcoc}_{/\tA}(\blank,\blank)$ denotes the
  mapping \itcat{} of the natural $(\infty,3)$-category structure
  on $\Fib^{(0,j)}_{/(\tA, E)}$.)  For $i = 1$, we define
  $\FUN^{E\dcart}_{/\tA}(\tE, \tF)$ similarly. 
\end{defn}

\begin{observation}
  If $\tF \to \tA$ is a 1-fibred $\ve$-fibration for $\ve = (0,j)$,
  then the \itcat{} $\FUN^{E\dcoc}_{/\tA}(\tE, \tF)$ is actually an \icat{}: any 2-morphism
  $\tE \times C_{2} \to \tF$ must be given at each $x \in \tE$ by a
  2-morphism in $\tF$ that maps to an equivalence in $\tA$, and is
  therefore invertible since all 2-morphisms in $\tF$ are
  $j$-cartesian.
\end{observation}

\begin{propn}\label{prop:laxlimCATIT}
  Let $(\tA,E)$ be a marked \itcat{} and consider a functor $F
  \colon\tc{A} \to \CATIT$ with corresponding  $\vepsilon$-fibration
  $\tE_{\ve} \to \tc{A}^{\veop}$. Then the $E$-(op)lax limits
  of $F$ in $\CATIT$ can be described as
  \[
    \lim^{\elax}_{\tc{A}}F \simeq
    \FUN^{E\dcoc}_{/\tc{A}^{\veop}}(\tc{A}^{\veop},\tE_{(0,j)}), \quad
 \lim^{\eoplax}_{\tc{A}}F \simeq \FUN^{E\dcart}_{/\tc{A}^{\veop}}(\tc{A}^{\veop},\tE_{(1,j)})
  \]
  for $j = 0,1$.
\end{propn}
\begin{proof}
  Without loss of generality let us assume that $\vepsilon=(0,1)$. We
  note
  that we have natural equivalences
  \[
    \Fun(\blank,\FUN^{E\dcoc}_{/\tc{A}}(\tc{A},\tE_{(0,1)}))\simeq \Fun^{E\dcoc}_{/\tc{A}}(\tc{A}\times (\blank),\tE_{(0,j)})\simeq  \Nat^{\elax}_{\tc{A},\CATIT}\left(\underline{(-)},F\right),
  \]
  where the first equivalence is formal and the second uses
  \cref{thm:elaxstr}. Thus the \itcat{}
  $\FUN^{E\dcoc}_{/\tc{A}}(\tc{A},\tE_{(0,1)})$ has the universal property of the $E$-lax limit.
\end{proof}

Over an \icat{} we get a description of (op)lax limits as
sections of fibrations (again due to \cite{GHN} for functors to $\CATI$):
\begin{cor}
  Let $\oA$ be an \icat{}. For a functor $F \colon \oA \to
  \CATIT$ with corresponding $i$-fibration $\tE_{i} \to \oA$ for $i = 0,1$, we have
  \[ \lim^{\lax}_{\oA} F \simeq \FUN_{/\oA}(\oA, \tE_{0}), \quad
    \lim^{\oplax}_{\oA} F \simeq \FUN_{/\oA^{\op}}(\oA^{\op},
    \tE_{1}). \qed\]
\end{cor}

\begin{corollary}\label{cor:laxlimcat}
  Let $(\tA,E)$ be a marked \itcat{} and consider a functor $F \colon\tc{A} \to \CATI$ with corresponding 1-fibred $\vepsilon$-fibration  $\tE_{\ve} \to \tc{A}^{\veop}$. Then the $E$-(op)lax limits of $F$ in $\CATI$ can be described as
  \[
    \lim^{\elax}_{\tc{A}}F \simeq \Fun^{E\dcoc}_{/\tc{A}^{\veop}}(\tc{A}^{\veop},\tE_{(0,j)}), \enspace \lim^{\eoplax}_{\tc{A}}F \simeq \Fun^{E\dcart}_{/\tc{A}^{\veop}}(\tc{A}^{\veop},\tE_{(1,j)})
  \]
  for $j = 0,1$. \qed 
\end{corollary}

\begin{remark}
  The fibrational description of (op)lax limits in $\CATI$ from
  \cref{cor:laxlimcat} is taken as a definition in \cite[\S
  B.6]{AMGR}, where these are called \emph{left-} and \emph{right-lax}
  limits; it thus follows from \cref{cor:laxlimcat} that these agree
  with other notions of lax limits in the literature.
\end{remark}

We also note the specialization of our results to ordinary
(co)limits in $\CATIT$ indexed over \icats{}:
\begin{corollary}
  Suppose we have a functor $F \colon \oA \to \CATIT$, where $\oA$ is
  an \icat{}, with associated
  $0$-fibration  $\tE_{0} \to \oA$ and 1-fibration $\tE_{1}
  \to \oA^{\op}$.
  \begin{enumerate}[(i)]
  \item The colimit of $F$
  can be described as
  \[ \colim_{\oA} F \simeq \Ldec{\tEn_{i}}\]
  for $i = 0,1$, where $\Ldec{\tEn_{i}}$ denotes the localization
  that inverts all (co)cartesian $1$-morphisms.
\item The limit of $F$ can be described as
  \[ \lim_{\oA} F \simeq \FUN_{/\oA}^{\coc}(\oA, \tE_{0})
    \simeq \FUN_{/\oA^{\op}}^{\cart}(\oA^{\op}, \tE_{1}),\] \ie{} as the
  \itcats{} of (co)cartesian sections of these fibrations. \qed
  \end{enumerate}
\end{corollary}

Using \cref{propn:lax tr DCATIT fib}, we can similarly identify
partially (op)lax (co)limits in $\DCATIT$:
\begin{cor}\label{cor:lax lim colim DCat2}
  Let $(\tA,E)$ be a marked \itcat{} and consider a functor
  $F \colon\tA \to \DCATIT$ with corresponding decorated
  $\ve$-fibration $\pi \colon\tEd_{\ve} \to \tc{A}^{\sharp\sharp,\veop}$.
  \begin{enumerate}
  \item The $E$-(op)lax colimits of $F$ in $\DCATIT$ can be described as the pushouts
  \[
    \begin{tikzcd}
     \tEn_{(0,j)}|_{E} \ar[r] \ar[d] & \Ld(\tEn_{(0,j)}|_{E})^{\flat\flat}  \ar[d] \\
     \tEd_{(0,j)}|_{E} \ar[r] & \colim^{\elax}_{\tA}F,
   \end{tikzcd}
   \qquad
    \begin{tikzcd}
     \tEn_{(1,j)}|_{E} \ar[r] \ar[d] & \Ld(\tEn_{(1,j)}|_{E})^{\flat\flat}  \ar[d] \\
     \tEd_{(1,j)}|_{E} \ar[r] & \colim^{\eoplax}_{\tA}F,
    \end{tikzcd}   
  \]
  where $\tEd_{\ve}|_{E}$ denotes the pullback
  $\tEd \times_{\tA^{\sharp\sharp,\veop}} (\tA^{\veop},E)^{\sharp}$. In other words, we invert the $i$-cartesian morphisms over $E$ and all $j$-cartesian 2-morphisms in $\tE$, and equip this with the decoration induced by $\tEd_{\ve}|_{E}$.
\item The $E$-(op)lax limits of $F$ in $\DCATIT$ can be described as
  \[
    \begin{split}
      \lim^{\elax}_{\tA} F & \simeq \DFUN_{/(\tA, E)^{\sharp}}^{E\dcoc}((\tA, E)^{\sharp}, \tEd_{(0,1)}|_{E}) \\
                           & \simeq \DFUN_{/(\tA^{\co}, E)^{\sharp}}^{E\dcoc}((\tA^{\co}, E)^{\sharp}, \tEd_{(0,0)}|_{E}) \\
      \lim^{\eoplax}_{\tA} F & \simeq \DFUN_{/(\tA^{\op}, E)^{\sharp}}^{E\dcart}((\tA^{\op}, E)^{\sharp}, \tEd_{(1,0)}|_{E}) \\
                           & \simeq \DFUN_{/(\tA^{\coop},E)^{\sharp}}^{E\dcart}((\tA^{\coop}, E)^{\sharp}, \tEd_{(1,1)}|_{E}),
    \end{split}
  \]
  where the right-hand side denotes the decorated full sub-\itcat{} of \[\DFUN_{/(\tA^{\veop}, E)^{\sharp}}((\tA^{\veop}, E)^{\sharp}, \tEd_{\ve}|_{E})\] spanned by the sections that take 1-morphisms in $E$ to $i$-cartesian morphisms and all 2-morphisms to $j$-cartesian 2-morphisms in $\tE_{\ve}$.
  \end{enumerate}
\end{cor}
\begin{proof}
  We prove the statement for $(0,1)$-fibrations and abbreviate
  $\tEd := \tEd_{(0,1)}$. By \cref{propn:lax tr DCATIT fib}, we have a
  natural equivalence
  \[ \Nat^{\elax}_{\tA,\DCATIT}(F, \underline{\tCd}) \simeq
    \DFun^{E\dcoc}_{/(\tA,E)^{\sharp}}(\tEd|_{E}, \tCd \times
    (\tA,E)^{\sharp}).\] Here the right-hand side is the \icat{} of
  decorated functors over $(\tA,E)^{\sharp}$ that also preserve
  cocartesian morphisms over $E$ and cartesian 2-morphisms; we can
  identify these as the decorated functors $\tEd|_{E} \to \tCd$ that take
  the latter to equivalences, which shows that this copresheaf is
  corepresented by the pushout
  $\tEd|_{E} \amalg_{\tEn|_{E}} \Ld(\tEn|_{E})^{\flat\flat}$, as required.

  To identify the $E$-lax limit of $F$, we similarly consider the natural equivalence
  \[ \Nat^{\elax}_{\tA,\DCATIT}(\underline{\tCd}, F) \simeq
    \DFun^{E\dcoc}_{/(\tA,E)^{\sharp}}(\tCd \times (\tA,E)^{\sharp}, \tEd|_{E}).\]
  We can identify this as the \icat{} of decorated functors from
  $\tCd$ into \[\DFUN_{/(\tA,E)^{\sharp}}((\tA, E)^{\sharp}, \tEd|_{E})\]
  whose value at each object takes the morphisms in $E$ to cocartesian
  morphisms in $\tE$ and all 2-morphisms to cartesian 2-morphisms in
  $\tE$, as required.
\end{proof}

Finally, by combining our description of partially (op)lax (co)limits
in $\CATIT$ with the interpretation of weighted (co)limits as
partially (op)lax ones from the previous subsection, we also get
concrete descriptions of weighted (co)limits in $\CATIT$. For limits, 
\cref{propn:weighted limit as lax} gives:
\begin{corollary}\label{cor:wtlimcatit}
  For $W \colon \tA \to \CATI$ and $F \colon \tA \to \CATIT$, let
  \begin{itemize}
  \item $\tW \to \tA$ be the $(0,1)$-fibration for $W$,
  \item $\tW' \to \tA^{\op}$ be the $(1,0)$-fibration for $W$,
  \item $\tE \to \tA$ be the $(0,1)$-fibration for $F$,
  \item $\tE' \to \tA^{\op}$ be the $(1,0)$-fibration for $F$.
  \end{itemize}
  Then we have
  \[ \lim^{W}_{\tA} F \simeq \FUN^{\coc}_{/\tA}(\tW, \tE) \simeq
    \FUN^{\cart}_{/\tA^{\op}}(\tW', \tE').   \qed \]
\end{corollary}

Dually, \cref{cor:weightedcolimitaslax} gives the following:
\begin{corollary}\label{cor:wtcolimcatit}
  For $W \colon \tA^{\op} \to \CATI$ and $F \colon \tA \to \CATIT$, let
  \begin{itemize}
  \item $\tW \to \tA$ be the $(1,0)$-fibration for $W$,
  \item $\tW' \to \tA^{\op}$ be the $(0,1)$-fibration for $W$,
  \item $\tE \to \tA$ be the $(0,1)$-fibration for $F$,
  \item $\tE' \to \tA^{\op}$ be the $(1,0)$-fibration for $F$.
  \end{itemize}
  Then we have
  \[
    \begin{split}
      \colim^{W}_{\tA} F & \simeq \Ld(\tE^{\natural} \times_{\tA^{\sharp\sharp}} (\tW,C)^{\sharp}) \\
                         & \simeq \Ld(\tE'^{\natural} \times_{\tA^{\op,\sharp\sharp}} (\tW',C')^{\sharp}).
    \end{split}
  \]
  where $C$ consists of the cartesian morphisms in $\tW$ and $C'$ of
  the cocartesian morphisms in $\tW'$.
\end{corollary}

\begin{remark}
  In the special case where $F$ and $W$ are diagrams of ordinary
  categories, this formula for $\Cat$-weighted colimits in $\Cat$
  appears in \cite{Lambert}.
\end{remark}

Using \cref{cor:lax lim colim DCat2}, we get a similar description of weighted colimits in $\DCATIT$:
\begin{corollary}\label{cor:wtcolimcatit}
  For $W \colon \tA^{\op} \to \CATI$ and $F \colon \tA \to \DCATIT$, let
  \begin{itemize}
  \item $\tW \to \tA$ be the $(1,0)$-fibration for $W$,
  \item $\tW' \to \tA^{\op}$ be the $(0,1)$-fibration for $W$,
  \item $\tEd \to \tA^{\sharp\sharp}$ be the decorated $(0,1)$-fibration for $F$,
  \item $\tE'^{\diamond} \to \tA^{\op,\sharp\sharp}$ be the decorated $(1,0)$-fibration for $F$.
  \end{itemize}
  Then the $W$-weighted colimit of $F$ is given by the pushouts
  \[
    \begin{tikzcd}
     \tEn \times_{\tA^{\sharp\sharp}} (\tW,C)^{\sharp} \ar[r] \ar[d] & \Ld(\tEn \times_{\tA^{\sharp\sharp}} (\tW,C)^{\sharp})^{\flat\flat}  \ar[d] \\
     \tEd\times_{\tA^{\sharp\sharp}} (\tW, C)^{\sharp} \ar[r] & \colim^{W}_{\tA}F,
   \end{tikzcd}\]
   \[
    \begin{tikzcd}
     \tE'^{\natural} \times_{\tA^{\op,\sharp\sharp}} (\tW',C')^{\sharp} \ar[r] \ar[d] & \Ld(\tE'^{\natural} \times_{\tA^{\op,\sharp\sharp}} (\tW',C')^{\sharp})^{\flat\flat}  \ar[d] \\
     \tE'^{\diamond}\times_{\tA^{\op,\sharp\sharp}} (\tW',C')^{\sharp} \ar[r] & \colim^{W}_{\tA}F,
   \end{tikzcd}   
   \]
  where $C$ consists of the cartesian morphisms in $\tW$ and $C'$ of
  the cocartesian morphisms in $\tW'$. \qed
\end{corollary}

\subsection{Existence and preservation of (co)limits}
\label{sec:stuff}

In this section we collect a few useful results giving criteria for partially (op)lax (co)limits to exist and be preserved.

\begin{proposition}\label{prop:adjpreservelimtis}
  Let $L \colon \tC \llra \tD \colon R$ be an adjunction of
  \itcats{}. Then the left adjoint $L$ preserves all partially (op)lax
  colimits that happen to exists in $\tC$. Dually, the right adjoint
  $R$ preserves all partially (op)lax limits that happen to exist in
  $\tD$.
\end{proposition}
\begin{proof}
  We give the proof for the left adjoint; the remaining case is formally dual. 
  Observe that by a marked variant of \cref{propn:prod pres gives 2-fun}, 
  for every marked \itcat{} $(\tI,E)$ the functor $\FUN(\tI,\blank)^{\elaxoplax}$ preserves adjunctions. Thus postcomposition with the adjunction $L \dashv R$ yields an adjunction
  \[
  L_* \colon \FUN(\tI, \tC)^{\elaxoplax} \llra \FUN(\tI, \tD)^{\elaxoplax} \colon R_*.
  \]
  Given a functor $F \colon \tI \to \tC$ admitting an $E$-(op)lax colimit, we can then produce the following chain of natural equivalences
  \[
    \begin{split}
      \Nat^{\elaxoplax}_{\tI,\tD}(L_*F,\underline{d})
      & \simeq \Nat^{\elaxoplax}_{\tI,\tC}(F,\underline{Rd}) \\
      & \simeq \tC\left(\colim^{\elaxoplax}_{\tC}F,Rd\right) \\
      & \simeq \tD\left(L\left(\colim^{\elaxoplax}_{\tC}F\right),d\right),
    \end{split}
  \]
  which precisely shows that $L(\colim^{\elaxoplax}_{\tC}F)$ has the
  universal property of the colimit $\colim^{\elaxoplax}_{\tC}L_{*}F$, as
  required.
\end{proof}

\begin{proposition}\label{prop:colimfunctorcat}
  Let $(\tI,E)$ be a marked \itcat{} and $\tC$ an \itcat{} that admits
  $E$-lax (co)limits over $\tI$.
  \begin{enumerate}[(i)]
  \item For any marked \itcat{} $(\tS, T)$, the \itcat{}
    $\FUN(\tS,\tC)^{T\doplax}$ admits $E$-lax (co)limits for all
    functors $\tI \to \FUN(\tS,\tC)^{T\doplax}$ that take the
    morphisms in $E$ to strong transformations, and these are computed pointwise in $\tC$.
  \item These (co)limits are preserved by the functor
    \[f^{*} \colon \FUN(\tS,\tC)^{T\doplax} \to \FUN(\tS',
      \tC)^{T'\doplax}\] for any functor of marked \itcats{}
    $f \colon (\tS',T') \to (\tS,T)$.
  \item Dually, if $\tC$ has $E$-oplax (co)limits over $\tI$,then
    $\FUN(\tS,\tC)^{T\dlax}$ admits partially \emph{oplax} (co)limits
    of such functors, and these are again computed pointwise and
    preserved by all restriction functors.
  \end{enumerate}
\end{proposition}

\begin{proof}
  We will only deal with lax limits, since the cases of oplax limits
  and (op)lax colimits are formally dual. Our assumptions guarantee
  the existence of an adjunction
  \[ L = \text{const} : \tC \rightleftarrows \FUN(\tI, \tC)^{\elax} :
    R =\lim^{\elax}_{\tI}.\]
  Since $\FUN(\tS,\blank)^{T\doplax}$ preserves adjunctions by \cref{propn:prod pres gives 2-fun},
  this induces an adjunction
  \[
  L_* \colon \FUN(\tS,\tC)^{T\doplax} \llra \FUN(\tS,\FUN(\tI,\tC)^{\elaxoplax})^{T\doplax} \colon R_*
\]
where the left adjoint is given by post-composition with the constant
diagram functor and the right adjoint is given by postcomposition
along the $E$-(op)lax limit functor. Let
$\FUN'(\tS,\FUN(\tI,\tC)^{\elax})^{T\doplax}$ denote the full
sub-\itcat{} on functors that take the morphisms in $T$ to strong
transformations; the constant diagram functor $L_{*}$ factors through
this, so we can restrict the adjunction to this
sub-\itcat{}. Moreover, it follows from \cref{marked lax hom adj} that
we can identify this with the \itcat{}
$\FUN'(\tI,\FUN(\tS,\tC)^{T\doplax})^{\elax}$ of functors that take
the morphisms in $E$ to strong natural transformations. We thus have
an adjunction
\[ \FUN(\tS,\tC)^{T\doplax} \leftrightarrows
  \FUN'(\tI,\FUN(\tS,\tC)^{T\doplax})^{E\dlax},\]
where the left adjoint is still the constant diagram functor. The
right adjoint, which by construction computes $E$-lax limits over
$\tI$ objectwise in $\tS$, therefore gives $E$-lax limits in 
$\FUN(\tS,\tC)^{T\doplax}$ for all functors in
$\FUN'(\tI,\FUN(\tS,\tC)^{T\doplax})^{E\dlax}$, which proves (i). Now (ii) follows since this construction is clearly natural in $(\tS,T)$.
\end{proof}

We note the following useful special case, where the marking is maximal:
\begin{cor}\label{cor:colim strong functor cat}
  Let $(\tI,E)$ be a marked \itcat{} and $\tC$ an \itcat{} that admits
  $E$-(op)lax (co)limits over $\tI$.
  \begin{enumerate}[(i)]
  \item For any \itcat{} $\tS$, the \itcat{} $\FUN(\tS, \tC)$ admits
    $E$-lax (co)limits for all functors $\tI \to \FUN(\tS, \tC)$, and
    these are computed pointwise in $\tC$.
  \item These (co)limits are preserved by the functor
    \[f^{*} \colon \FUN(\tS,\tC) \to \FUN(\tS',
      \tC)\] for any functor of \itcats{}
    $f \colon \tS' \to \tS$.
  \end{enumerate}
\end{cor}

\begin{observation}  
  Suppose $\tC$ is a (co)complete \itcat{}. Then \cref{cor:colim
    strong functor cat} implies that for any $\tS$, the
  \itcat{} $\FUN(\tS, \tC)$ is also (co)complete, with its partially
  (op)lax colimits computed pointwise in $\tC$. Moreover, taking $f$
  in \cref{prop:colimfunctorcat}(ii) to be the map
  $(\tS,T) \to \tS^{\sharp}$, we see that the inclusion
  \[ \FUN(\tS, \tC) \to \FUN(\tS, \tC)^{T\doplax}\] preserves all
  partially lax (co)limits, and so is (co)continuous (though the
  target does not necessarily admit \emph{all} (co)limits). In fact,
  we will see later in \cref{cor:delaxify} that this functor has a right or a
  left adjoint when $\tC$ is  complete or cocomplete, respectively.
\end{observation}

\begin{proposition}\label{prop:yonedacont}
  Let $\tC$ be a \itcat{}. Then the Yoneda embedding
  $h_{\tC} \colon \tC \to \tPSh(\tC)$ preserves all partially (op)lax
  limits that exist in $\tC$.
\end{proposition}
\begin{proof}
  Let $(\tI,E)$ be a marked \itcat{} and let $D \colon \tI \to \tC$ be a diagram. Since $h_{\tC}$ is fully faithful, it follows that we have natural equivalences
  \[
  \Nat_{\tC^\op,\CATI}\left(h_{\tC}(\blank),h_{\tC}(\lim^{\elaxoplax}_{\tI}D)\right) \simeq \tC\left(\blank,\lim_{\tI}^{\elaxoplax}D \right) \simeq \Nat_{\tI,\tC}^{\elaxoplax}\left(\underline{(\blank)},D\right)
  \]
  As a consequence of \cite[Corollary 2.8.13]{AGH24}, the functor
  \[
    h_{\tC,*} \colon \FUN(\tI,\tC)^{\elaxoplax} \to \FUN(\tI, \tPSh(\tC))^{\elaxoplax}
  \]
  is also fully faithful, so
  we obtain a natural equivalence
  \[
 \Nat_{\tI,\tC}^{\elaxoplax}\left(\underline{(\blank)},D\right) \simeq \Nat_{\tI,\tPSh(\tC)}^{\elaxoplax}(\underline{h_{\tC}},h_{\tC}D).
  \]
  The previous discussion together with the universal property of the
  limit implies that we have a canonical morphism
  $h_{\tC}(\lim^{\elaxoplax}_{\tI}D) \to
  \lim^{\elaxoplax}_{\tI}h_{\tC}D$. Invoking again the Yoneda lemma,
  we see that our map is pointwise an equivalence, which concludes the
  proof by the conservativity noted in \cref{obs:cons on lax
    transfos}.
\end{proof}

\begin{proposition}\label{prop:limitscommute}
  Let $\tC$ be an \itcat{} that admits partially (op)lax limits of
  shape $(\tI,E)$ and suppose that we are given another marked
  \itcat{} $(\tS,T)$ and a functor
  $F \colon \tS \to \FUN(\tI,\tC)^{\elaxoplax}$ such that the
  $T$-(op)lax limit of $F$ exists in
  $\FUN(\tI,\tC)^{\elaxoplax}$, and such that $\tC$ admits partially
  (op)lax limits indexed by $(\tS,T)$. Then we have
  an equivalence in $\tC$
  \[
    \lim^{\elaxoplax}_{\tI} G \simeq \lim^{T\dplax}_{\tS} H.
  \]
  where $G := \lim^{T\dplax}_{\tS}F$ is the limit of $F$ in
  $\FUN(\tI,\tC)^{\elaxoplax}$ and $H$ is the composite functor
  $\lim^{\elaxoplax}_{\tI} \circ F$.
\end{proposition}
\begin{remark}
  It is tempting to write this equivalence as
  \[
    \lim^{\elaxoplax}_{\tI} \lim^{T\dplax}_{\tS} F \simeq
    \lim^{T\dplax}_{\tS} \lim^{\elaxoplax}_{\tI} F,
  \]
  but this is potentially misleading as we are not claiming that the
  limit of $F$ in $\FUN(\tI,\tC)^{\elaxoplax}$ is computed objectwise
  by a limit in $\tC$. This is the case under certain conditions,
  however, as we saw above in \cref{prop:colimfunctorcat}.
\end{remark}
\begin{proof}[Proof of \cref{prop:limitscommute}]
  For $c \in \tC$, we have natural equivalences
  \[ \tC(c, \lim^{\elaxoplax}_{\tI} G) \simeq
    \Nat^{\elaxoplax}_{\tI,\tC}(\underline{c}_{\tI}, G) \simeq
    \Nat^{T\dplax}_{\tS, \FUN(\tI,\tC)^{\elaxoplax}}(\underline{\underline{c}_{\tI}}_{\tS},
    F),\]
  where we subscript the constant diagrams with their domains for clarity.
  Since $\tC$ admits $E$-(op)lax limits over $\tI$, we have an
  adjunction
  \[ L := \underline{(\blank)}_{\tI} : \tC \rightleftarrows
    \FUN(\tI,\tC)^{\elaxoplax} : \lim^{\elaxoplax}_{\tI} =: R,\] and
  since $\FUN(\tS,\blank)^{T\dplax}$ preserves adjunctions by
  \cref{propn:prod pres gives 2-fun}, this induces by composition an
  adjunction
  \[ L_{*} : \FUN(\tS,\tC)^{T\dplax} \rightleftarrows \FUN(\tS,
    \FUN(\tI,\tC)^{\elaxoplax})^{T\dplax} : R_{*}\]
  The constant diagram $\underline{\underline{c}_{\tI}}_{\tS}$ is also
  $L_{*}(\underline{c}_{\tS})$, so using this adjunction we get a
  natural equivalence
  \[ \Nat^{T\dplax}_{\tS, \FUN(\tI,\tC)^{\elaxoplax}}(\underline{\underline{c}_{\tI}}_{\tS},
    F) \simeq \Nat^{T\dplax}_{\tS,\tC}(\underline{c}_{\tS}, R_{*}F)
    \simeq \tC(c, \lim^{T\dplax}_{\tS} H),\]
  as required.
\end{proof}

\subsection{Cofinal functors}\label{sec:cofinal new}
In this section we  use our work so far to give a treatment of
cofinality for \itcats{}, giving model-independent proofs of results
that previously appeared in \cite{AS23,AbMarked,GagnaHarpazLanariLaxLim}.  
Related results on cofinality for $(\infty,\infty)$-categories also
appear in \cite[\S 4.2.3]{loubaton}.

We start by defining cofinal maps by an orthogonality property and
considering some consequences of this before we characterize
cofinality in terms of the preservation of (co)limits in
\cref{propn:cofinal colimits}; we also derive a more concrete
criterion for cofinality in \cref{thm:adaggerrefined}.

\begin{defn}
  A functor of marked \itcats{} $F \colon (\tI,S) \to (\tJ, T)$ is called
  \emph{$\ve$-cofinal} if it is left orthogonal to all 1-fibred
  $\ve$-fibrations $\tE^{\natural} \to \tB^{\sharp}$.
\end{defn}

\begin{defn}\label{defn:cofinalmarkedequiv}
  We write
  $L^{\ve;1}_{\tB^{\sharp}}$ for the left adjoint to the
  inclusion
  \[ \oFIB^{\ve}_{/\tB} \hookrightarrow \MCATITsl{\tB^{\sharp}}\]
  (which we described explicitly in \cref{cor:explicit locn to fibs}).
  We say that a morphism of marked \itcats{} $F \colon (\tI,S) \to
  (\tJ,T)$ over $\tB^{\sharp}$ is a \emph{marked $\ve$-equivalence}
  over $\tB^{\sharp}$ if
  $L^{\ve;1}_{\tB^{\sharp}}(F)$ is an equivalence.
\end{defn}

\begin{lemma}\label{lem:marked eps-equiv cond}
  For a morphism of marked \itcats{} $F \colon (\tI,S) \to
  (\tJ,T)$ over $\tB^{\sharp}$, the following are equivalent:
  \begin{enumerate}
  \item $F$ is a marked $\ve$-equivalence over $\tB^{\sharp}$.
  \item $F$ is an $\ve$-equivalence over $\tB^{\sharp\sharp}$ when
    viewed as a functor $(\tI,S)^{\sharp} \to (\tJ,T)^{\sharp}$.
  \item For $b \in \tB$, the induced functor of \icats{}
  \[ \FrofeB(\tI,S)(b) \to \FrofeB(\tJ,T)(b)\]
  is an equivalence.
  \end{enumerate}
\end{lemma}
\begin{proof}
  The first two conditions are equivalent since 
  any map from
  $(\tI,S)^{\sharp}$ to a fibration over $\tBss$ factors through its
  underlying 1-fibred fibration.
  The equivalence of the first and third conditions follows from the
  description of the straightening of
  $L^{\ve;1}_{\tB^{\sharp}}(\tI,S)$ from \cref{cor:free 1-fibred on
    marked}.
\end{proof}

\begin{propn}\label{propn:cofinal charn}
  The following are equivalent for $F \colon (\tI,S) \to (\tJ, T)$:
  \begin{enumerate}[(1)]
  \item $F$ is $\ve$-cofinal.
  \item $F$ is a marked $\ve$-equivalence over $\tJ^{\sharp}$.
  \item $F$ is an $\ve$-equivalence over $\tJ^{\sharp\sharp}$ when
    viewed as a functor $(\tI,S)^{\sharp} \to (\tJ,T)^{\sharp}$.
  \item $L^{\ve;1}_{\tJ^{\sharp}}(F) \colon
    L^{\ve;1}_{\tJ^{\sharp}}(\tI,S) \to
    L^{\ve;1}_{\tJ^{\sharp}}(\tJ,T)$ is an equivalence.
  \item For every $b \in \tJ$, the functor of \icats{}
    \[ \FrofeB(\tI,S)(b) \to \FrofeB(\tJ,T)(b)\]
    is an equivalence.
  \end{enumerate}
\end{propn}
\begin{proof}
  Since any functor $(\tJ,S) \to \tB^{\sharp}$ factors through
  $\tJ^{\sharp}$, and 1-fibred $\ve$-fibrations are closed under
  pullback, (1) holds \IFF{} for any 1-fibred $\ve$-fibration $\tE^{\natural}
  \to \tJ^{\sharp}$, there is a unique lift in any square
  \[
    \begin{tikzcd}
     (\tI,S) \ar[r] \ar[d, "F"'] & \tE^{\natural}  \ar[d] \\
     (\tJ,T) \ar[r] \ar[dashed,ur] & \tJ^{\sharp}.
    \end{tikzcd}
  \]
  This is equivalent to $F$ being a local equivalence for 1-fibred
  $\ve$-fibrations when viewed as a map in $\MCATITsl{\tB^{\sharp}}$,
  so (1) is equivalent to (2). The equivalence of this with the
  remaining conditions is then a special case of \cref{lem:marked eps-equiv cond}.
\end{proof}

\begin{observation}
  We have the following immediate properties of $\ve$-cofinal maps:
  \begin{itemize}
  \item $\ve$-cofinal functors are closed under cobase change,
    composition, retracts, and colimits.
  \item If $F$ is $\ve$-cofinal then $GF$ is $\ve$-cofinal \IFF{} $G$
    is.
  \item If $F \colon \tCd \to \tDd$ is an $\ve$-cofibration, then the
    underlying functor of marked \itcats{} is $\ve$-cofinal. In
    particular, by \cref{obs:loc e-eqce} all localizations of
    \itcats{} give $\ve$-cofinal maps for any $\ve$.
  \end{itemize}
\end{observation}

\begin{observation}\label{obs:cofinal change variance}
  The following are equivalent for a functor of marked \itcats{} $F \colon (\tI,S) \to
  (\tJ,T)$ (\cf{} \cref{obs:fib change variance}):
  \begin{itemize}
  \item $F$ is $(1,0)$-cofinal.
  \item $F^{\op}$ is $(0,0)$-cofinal.
  \item $F^{\co}$ is $(1,1)$-cofinal.
  \item $F^{\coop}$ is $(0,1)$-cofinal.
  \end{itemize}
\end{observation}

We next characterize our four notions of cofinality in terms of
partially (op)lax (co)limits:
\begin{defn}
  We say that a functor of marked \itcats{} $F \colon (\tI,S) \to
  (\tJ,T)$ is \emph{(op)lax colimit-cofinal} if for every functor $D \colon \tJ
  \to \tC$ the $T$-(op)lax colimit of $D$ exists \IFF{} the
  $S$-(op)lax colimit of $DF$ exists, and the canonical map
  \[ \colim^{S\dplax}_{\tI} DF \to \colim^{T\dplax}_{\tJ} F\] is an
  equivalence. Dually, we say that $F$ (op)lax limit-cofinal if the
  corresponding condition for partially (op)lax limits holds.
\end{defn}

\begin{propn}\label{propn:cofinal colimits}
  For a functor of marked \itcats{} $F \colon (\tI,S) \to
  (\tJ,T)$ we have
  \begin{itemize}
  \item $F$ is $(1,0)$-cofinal \IFF{} it is lax colimit-cofinal.
  \item $F$ is $(0,1)$-cofinal \IFF{} it is lax limit-cofinal.
  \item $F$ is $(1,1)$-cofinal \IFF{} it is oplax colimit-cofinal.
  \item $F$ is $(0,0)$-cofinal \IFF{} it is oplax limit-cofinal.
  \end{itemize}
\end{propn}
\begin{proof}
  We prove the first case; the other 3 follow
  from this by applying
  \cref{lem:lax colim change variance} and \cref{obs:cofinal change
    variance}. Consider first the
  representable functor $\tJ(j,\blank) \colon \tJ \to \CATI$; here
  \cref{cor:laxcoliminCat1} implies that the canonical map
  \[ \colim^{S\dlax}_{\tI}
    \tJ(j,F(\blank)) \to \colim^{T\dlax}_{\tJ} \tJ(j,\blank)\]
  is equivalent to 
  \[ \Frof^{(1,0)}_{\tJ}(\tI,S)(j) \to
    \Frof^{(1,0)}_{\tJ}(\tJ,T)(j),\]
  so if $F$ is lax final then this is an equivalence for all $j$,
  \ie{} $F$ is $(1,0)$-cofinal. On the other hand, \cref{propn:lax as
    Kan weighted} identifies the comparison map on lax colimits as the
  map
  \[ \colim_{\tJ}^{\Frof^{(1,0)}_{\tJ}(\tI,S)} D \to
    \colim_{\tJ}^{\Frof^{(1,0)}_{\tJ}(\tJ,T)} D \]
  induced by the transformation of weights $\Frof^{(1,0)}_{\tJ}(\tI,S)
  \to \Frof^{(1,0)}_{\tJ}(\tJ,T)$. Thus the map on lax colimits is an
  equivalence when this transformation is a natural equivalence,
  \ie{} when $F$ is $(1,0)$-cofinal.
\end{proof}

\begin{remark}
  Loubaton proves a version of \cref{propn:cofinal colimits} for
  $(\infty,\infty)$-categories as \cite[Theorem 4.2.3.21]{loubaton}.
\end{remark}

\begin{propn}
  Consider a diagram of marked \itcats{}
  \begin{equation}
    \label{eq:pb for cofinal smooth}
      \begin{tikzcd}
        (\tC,P) \ar[r, "f'"] \ar[d] & (\tD,Q)  \ar[d, "p'"] \ar[r] &
        (\tD, R) \ar[d, "p"] \\
        (\tI,S) \ar[r, "f"'] & (\tJ,T) \ar[r] & \tJ^{\sharp}
      \end{tikzcd}
  \end{equation}
  where both squares are pullbacks (so $Q$ consists of the morphisms
  in $R$ that lie over $T$). If $p$ is $\ve$-smooth when
  viewed as a morphism $(\tD,R)^{\sharp} \to \tJ^{\sharp\sharp}$ 
  and $f$ is $\ve$-cofinal, then $f'$ is
  also $\ve$-cofinal.
\end{propn}
\begin{proof}
  By \cref{propn:cofinal charn} $f$ is an $\ve$-equivalence over
  $\tJ^{\sharp\sharp}$ when
  viewed as a decorated functor. It therefore follows from 
  \cref{lem:smooth and pushfwd} that $f'$ is an $\ve$-equivalence over
  $(\tD,R)^{\sharp}$. It is then also an $\ve$-equivalence over
  $\tD^{\sharp\sharp}$ by \cref{obs:cofib vs equiv}, and hence
  $\ve$-cofinal as required.
\end{proof}

Specializing this using \cref{thm:decsmooth}, we get:
\begin{cor}
  Consider a diagram \cref{eq:pb for cofinal smooth} of marked
  \itcats{} as above, with both squares pullbacks. If $p$ is a marked
  $\overline{\ve}$-fibration and 
  $f$ is $\ve$-cofinal, then $f'$ is
  also $\ve$-cofinal. \qed  
\end{cor}

Our next goal is to give a more concrete characterization of
cofinality, which is based on the following properties of $\Frof^{(1,0)}_{\tJ}(\tJ,T)$:
\begin{lemma}\label{lem:Frof initial stuff}
  For a marked \itcat{} $(\tJ, T)$, we have:
  \begin{enumerate}[(i)]
  \item For every $j \in \tJ$, the \icat{}
    $\Frof^{(1,0)}_{\tJ}(\tJ,T)(j)$ has an initial object
  \item The image in $\Frof^{(1,0)}_{\tJ}(\tJ,T)(j)$ of any
    marked morphism $f \colon j \to j'$ in $T$ is initial.
  \item For every marked morphism $f \colon j \to j'$ in $T$, the
    functor $\Frof^{(1,0)}_{\tJ}(\tJ,T)(f)$ preserves initial objects.
  \end{enumerate}
\end{lemma}
\begin{proof}
  By definition, $\Frof^{(1,0)}_{\tJ}(\tJ,T)(j)$ is the \icat{}
  obtained from $\tJ_{j \upslash}$ by inverting the commuting
  triangles under $j$ given by marked morphisms in $\tJ$ as well as
  all 2-morphisms. We first consider the localization
  $\Lcat\tJ_{j \upslash}$ that inverts all 2-morphisms; we
  claim $\id_{j}$ is an initial object here. Indeed, for a morphism $f
  \colon j \to j'$ we can identify the mapping \icat{}
  $\tJ_{j \upslash}(\id_{j},f)$ as $\tJ(j,j')_{f/}$ using
  \cref{ARlax mor pb}; this
  has an initial object, and so 
  $(\Lcat\tJ_{j \upslash})(\id_{j},f) \simeq
  \|\tJ(j,j')_{f/}\|$ is contractible, which means that $\id_{j}$ is an
  initial object. Now consider the composition
  \[ \{\id_{j}\} \to \Lcat\tJ_{j \upslash}(\id_{j},f) \to
    \Frof^{(1,0)}_{\tJ}(\tJ,T)(j);
  \]
  here the first functor is limit-cofinal since $\id_{j}$ is initial,
  while the second functor is also limit-cofinal as it is a
  localization of \icats{}. The composite is therefore also
  limit-cofinal, which means that $\id_{j}$ is also an initial object
  of $\Frof^{(1,0)}_{\tJ}(\tJ,T)(j)$. This proves (i).

  If $f \colon j \to j'$ is a marked morphism, the commuting
  triangle
  \[
    \begin{tikzcd}
      {} & j \ar[dl, "="'] \ar[dr, "f"] \\
      j \ar[rr, "f"] & & j'
    \end{tikzcd}
  \]
  is sent to an equivalence in $\Frof^{(1,0)}_{\tJ}(\tJ,T)(j)$, and so
  the image of $f$ is equivalent to that of $\id_{j}$, and so must
  also be initial, giving (ii). Part (iii) follows as the image of
  $\id_{j'}$ under $\Frof^{(1,0)}_{\tJ}(\tJ,T)(f)$ is the image of
  $f$, which is initial.
\end{proof}

\begin{theorem}\label{thm:adaggerrefined}
  A functor of marked \itcats{} $f \colon (\bII,E) \to (\bJJ,T)$
  is $(1,0)$-cofinal (\ie{} lax colimit-cofinal)
  if and only if the following conditions hold:
    \begin{enumerate}[(i)]
    \item For every $j \in \bJJ$, there exists an initial object in
      $\Frof^{(1,0)}_{\tJ}(\tI,E)(j)$. 
    \item The initial object of $\Frof^{(1,0)}_{\tJ}(\tI,E)(j)$ is
      sent by $\Frof^{(1,0)}_{\tJ}(f)$ to an initial object in
      $\Frof^{(1,0)}_{\tJ}(\tJ,T)(j)$.
    \item Given an object $x \in \tI$, 
      the object of $\Frof^{(1,0)}_{\tJ}(\tI,E)(f(x))$
      that is the image of $(x,\id_{f(x)})$ is initial.
  \item
    For every marked morphism $\phi \colon j \to j'$ in $T$, the functor
      \[
        \Frof^{\ve}_{\tJ}(\tI,E)(\phi) \colon \Frof^{\ve}_{\tJ}(\tI,E)(j')
        \to \Frof^{\ve}_{\tJ}(\tI,E)(j)		 	\]
      preserves initial objects.
    \end{enumerate}
    Similarly, in the other three cases we have:
    \begin{itemize}
    \item $f$ is $(1,1)$-cofinal (oplax colimit-cofinal) if and only
      if the same conditions hold with $(1,0)$ replaced by $(1,1)$.
    \item $f$ is $\ve$-cofinal with $\ve = (0,1)$ or $(0,0)$ ((op)lax
      limit-cofinal) if and only if
      the analogous conditions hold with $(1,0)$ replaced by $\ve$ and 
      ``initial object'' replaced by ``terminal object''.
    \end{itemize}
  \end{theorem}
  \begin{proof}
    By \cref{propn:cofinal charn}, the functor $f$ is $(1,0)$-cofinal
    \IFF{} $\Frof^{\ve}_{\tJ}(f)$ is a natural equivalence. It
    therefore follows from \cref{lem:Frof initial stuff} that the
    conditions are necessary. We are left with proving that they are
    also sufficient for $\Frof^{\ve}_{\tJ}(f)$ to be an equivalence.

    Let $q \colon \tQ \to \tJ$ be the 1-fibred $(1,0)$-fibration
    corresponding to $\Frof^{(1,0)}_{\tJ}(\tI,E)$ and $\tP \to \tJ$ be
    that corresponding to $\Frof^{(1,0)}_{\tJ}(\tJ,T)$, with
    $\Phi \colon \tQ \to \tP$ the morphism of 1-fibred
    $(1,0)$-fibrations corresponding to
    $\Frof^{\ve}_{\tJ}(f)$; we also have the canonical maps
    $\eta_{\tJ} \colon \tJ \to \tP$ and $\eta_{\tI} \colon \tI \to
    \tQ$ that exhibit these as free 1-fibred $(0,1)$-fibrations over
    $\tJ$ on
    $(\tJ,T)$ and $(\tI,E)$, respectively.
    We consider the full sub-\itcat{} $\tQ_{0} \subseteq \tQ$ of
    fibrewise initial objects. Assumption (i) implies that the
    restriction $\tQ_{0} \to \tJ$ is essentially surjective; it is
    also fully faithful since for objects $x,y \in \tQ_{0}$ over $j,j'$
    in $\tJ$, the fibre of the cocartesian fibration
    $\tQ(x,y) \to \tJ(j,j')$ at $f \colon j \to j'$ is equivalent to
    $\tQ_{j}(x,f^{*}j')$, which is contractible since by assumption
    $x$ is initial in $\tQ_{j}$. Hence we have an equivalence
    $\tQ_{0} \isoto \tJ$, and its inverse gives a section
    $s \colon \tJ \to \tQ$ of $q$ that picks out the fibrewise initial
    objects. Moreover, condition (iv) implies that $s$ takes the
    marked morphisms in $T$ to cartesian morphisms. We can therefore
    extend $s$ over the free 1-fibred $(1,0)$-fibration $\tP$ to a morphism of 1-fibred
    $(1,0)$-fibrations $\bar{s} \colon \tP \to \tQ$ that straightens
    to a natural transformation
    \[ \sigma \colon \Frof^{(1,0)}_{\tJ}(\tJ,T) \to
      \Frof^{(1,0)}_{\tJ}(\tI,E).\]
    Moreover, the composite $\tJ \xto{s} \tQ \xto{\Phi} \tP$ picks out
    the
    fibrewise initial objects in $\tP$ by condition (ii), so that by
    uniqueness the composite $\Phi \circ \bar{s}$ is the
    identity, hence so is $\Frof^{\ve}_{\tJ}(f) \circ \sigma$. On the
    other hand, by the freeness of $\tQ$ the composite $\bar{s} \circ
    \Phi$ is uniquely determined by the functor
    \[ \tI \to \tQ \xto{\Phi} \tP \xto{\bar{s}} \tQ. \]
    Here the composite of the first two maps is equivalent to the
    composite
    \[ \tI \xto{f} \tJ \xto{\eta_{\tJ}} \tP,\]
    so the full composite agrees with
    \[ \tI \xto{f} \tJ \xto{s} \tQ. \]
    To identify this with $\eta_{\tI}$, consider the commutative
    triangle
    \[
      \begin{tikzcd}
       \tI \ar[rr, "\eta_{\tI}"] \ar[dr, "f"'] & & \tQ \ar[dl, "q"] \\
        & \tJ.
      \end{tikzcd}
    \]
    By assumption (iii), $\eta_{\tI}$ factors through
    $\tilde{\eta}_{\tI} \colon \tI \to \tQ_{0}$, so that we can
    identify $f$ with the composite $\tI \xto{\tilde{\eta}_{\tI}}
    \tQ_{0} \isoto \tJ$. It follows that $s \circ f$ is equivalent to
    $\eta_{\tI}$, and so $\bar{s} \circ \Phi$ must be the identity,
    and hence so is $\sigma \circ \Frof^{\ve}_{\tJ}(f)$. We have thus
    shown that $\sigma$ is an inverse of $\Frof^{\ve}_{\tJ}(f)$; hence
    the latter is an equivalence, as we wanted to prove.
\end{proof}

As a consequence, we get the following simpler sufficient (but not
necessary) conditions for cofinality of a map from the point:
\begin{cor}\label{cor:point cofinal simple cond}
  Suppose $(\tJ,T)$ is a marked \itcat{} and $j$ is an object of
  $\tJ$ for which the following conditions hold:
  \begin{itemize}
  \item for every object $j'$ of $\tJ$, there exists a marked morphism
    $j' \to j$,
  \item and every marked morphism $j' \to j$ is initial in $\tJ(j',j)$.
  \end{itemize}
  Then the functor $\{j\}^{\flat} \to (\tJ, T)$ is cofinal.
\end{cor}
\begin{proof}
  We apply \cref{thm:adaggerrefined}.  In this case
  $\Frof^{(1,0)}_{\tJ}(\{j\})$ is the functor $\tJ(\blank,j)$, so the
  two assumptions together give condition (i) as well as (iii), since
  $\id_{j}$ is always marked, and (iv), since marked morphisms are
  closed under composition. For condition (ii), we recall from
  \cref{lem:Frof initial stuff} that the initial object of
  $\Frof^{(1,0)}_{\tJ}(\tJ,T)(j')$ is the image of $\id_{j'}$, and a
  marked morphism $\phi \colon j' \to j$ gives a morphism
  $\id_{j'} \to \phi$ whose image in $\Frof^{(1,0)}_{\tJ}(\tJ,T)(j')$
  is an equivalence; thus the initial object $\phi$ in $\tJ(j',j)$ is
  indeed sent to an initial object, as required.
\end{proof}

\begin{remark}\label{rem:cofinalrepresentable}
  We saw in \cref{propn:rep cart fib cond} that if $p \colon \tE \to
  \tB$ is a 1-fibred $(1,0)$-fibration and $e$ is an object of $\tE$,
  then the functor $\{e\}^{\flat} \to (\tE,C)$, where $C$ denotes the
  cartesian morphisms, is $(1,0)$-cofinal \IFF{} the conditions of
  \cref{cor:point cofinal simple cond} are satisfied (and both are
  equivalent to $e$ exhibiting $p$ as a representable fibration). Thus
  in this particular case these a priori stronger conditions are
  actually necessary for cofinality.
\end{remark}

\subsection{Kan extensions}\label{sec:kanext}
In this section we apply our results on pushforward of fibrations from
\S\ref{sec:cofree} to construct Kan extensions of \itcats{}. This
includes not only ordinary Kan extensions, which can also be
constructed using enriched \icat{} theory \cite{heineweighted}, but
also \emph{partially lax} Kan extensions, which originally appeared in
\cite{Abellan2023} in a model-categorical approach.\footnote{Loubaton
  also defines the notion of partially lax Kan extension for \iicats{}
in \cite[\S 4.2.4]{loubaton}, but does not seem to prove they exist
except in the non-lax case (his Corollary 4.2.4.9).}

We start by
considering the case of right Kan extensions in $\CATIT$, which
follows immediately from our previous results via straightening:
\begin{propn}\label{prop:kanbasecase}
  Let $f \colon (\tC, I) \to \tD^{\sharp}$ be a functor of
  marked \itcats{}. Then there exists an adjunction of \itcats{}
  \[
    f^* \colon \FUN(\tD,\CATIT) \llra \Fun(\tC,\CATIT)^{I\dplax} \colon f_*^{I\dplax}
  \]
  where $f^*$ is the obvious restriction functor. For $F \colon
  \tC \to \CATIT$, we can describe $f_{*}^{I\dplax}F$ by the
  pointwise formula     
\[  \begin{split}
    f_*^{I\dlax}F(d) & \simeq
                       \lim^{I_{d}\dlax}_{\tC_{d\upslash}}\left(\tC_{d
                       \upslash} \to \tC \xrightarrow{F} \CATIT\right)
    \\
    f_*^{I\doplax}F(d) & \simeq
                       \lim^{I_{d}\doplax}_{\tC_{d\downslash}}\left(\tC_{d
                       \upslash} \to \tC \xrightarrow{F} \CATIT\right),                       
  \end{split}
\]
where $I_{d}$ denotes the morphisms that project over $\tC$ to
morphisms in $I$ and to commuting triangles in $\tD$. 
\end{propn}
\begin{proof}
  The existence of the adjunction follows by combining \cref{cor:lax kan adjn on fibs} with the
  straightening equivalence of \cref{thm:elaxstr}: in the lax case by
  straightening the adjunction for $f$ in
  the case $\ve = (0,1)$ or that for $f^{\co}$ in the case $\ve =
  (0,0)$, and in the oplax case by straightening the
  adjunction for $f^{\coop}$ with $\ve = (1,1)$ or $f^{\op}$ with $\ve
  = (1,0)$. To identify the values of $f_*^{I\dlax}F$ we then combine
  \cref{obs:Rf fibre} with the fibrational description of $I$-(op)lax
  limits in $\CATIT$ from \cref{prop:laxlimCATIT}.
\end{proof}
We call the functor $f_*^{I\dplax}$ the \emph{$I$-(op)lax right Kan
  extension} functor along $f$. Our goal is now to use the Yoneda
embedding to extend these Kan extensions to more general targets than
$\CATIT$, for which we start with the following observations:

\begin{observation}\label{obs:squarelaxkan}
  Let $f \colon (\tC, I) \to \tD^{\sharp}$ be a marked functor. Then
  for every \itcat{} $\tA$, restriction along $f$ yields a commutative diagram
	\[
		\begin{tikzcd}
			\FUN(\tD,\tA) \arrow[d] \arrow[r,"f^*"] & \FUN(\tC,\tA)^{I\dplax} \arrow[d] \\
			\FUN( \tA^\op \times \tD,\CATI) \arrow[r,"(f\times \id)^*"] & \FUN(\tA^\op \times \tC,\CATI)^{\overline{I}\dplax} 
		\end{tikzcd}
	\]
	where the vertical functors are induced by the Yoneda
        embedding and $\overline{I}$ contains all pairs of morphisms
        whose component in $\tC$ lies in $I$.

        From \cite[Corollary 2.8.5]{AGH24} we see that the vertical functors are fully faithful with essential image given by those functors $F \colon \tX \times \tA^\op \to \CATI$,  with $\tX \in \{\tC,\tD\}$, such that $F(-,x)$ is representable for every $x \in \tX$.
\end{observation}

\begin{observation}\label{obs:embedding}
  Let $(\tA,I)$ be a marked \itcat{} and let $J$ denote the marking of
  $(\tA,I) \times \tB^{\sharp}$. Applying \cref{marked lax hom adj} to
  $(\tA,I)$, $\tB^{\sharp}$ and $\CATIT^{\sharp}$ we get an
  equivalence
  \[
    \FUN(\tB\times \tA,\CATIT)^{J\dlax}\simeq \FUN(\tA,\FUN(\tB,\CATIT))^{I\dlax},
  \]
  as well as a fully faithful functor
  \[
    \FUN(\tA,\FUN(\tB,\CATIT))^{I\dlax} \to \FUN(\tB,\FUN(\tA,\CATIT)^{I\dlax}),
  \] 
  with essential image given by those functors that send every
  morphism in $\tB$ to a strong natural transformation, or
  equivalently factor through the cartesian product.
\end{observation}

\begin{lemma}\label{lem:pointwiselimit}
  Given a marked \itcat{} $(\tA,I)$ and an \itcat{} $\tB$, let
  $\pi_{\tB} \colon (\tA,I) \times \tB^{\sharp} \to \tB^{\sharp}$
  be the projection and write $J$ for the marking of $(\tA,I) \times \tB^{\sharp}$. Then the right adjoint
  \[ \pi_{\tB,*} \colon \FIB^{\ve}_{/(\tA \times \tB, J)} \to
    \FIB^{\ve}_{/\tB}\]
  to pullback along $\pi_{\tB}$ from \cref{cor:pushforwardisfib marked} can
  be identified under straightening with the composite
  \[
    \begin{split}
      \FUN((\tA \times \tB)^{\veop}, \CATIT)^{J\dplax} & \hookrightarrow
      \FUN(\tB^{\veop},\FUN(\tA^{\veop},\CATIT)^{I\dplax}) \\ & \to
      \FUN(\tB^{\veop}, \CATIT)
    \end{split}
  \] where the first functor is the
  inclusion from \cref{obs:embedding} and the second is given by
  composing with
  \[\lim^{I\dplax}_{\tA^{\veop}} \colon
    \FUN(\tA^{\veop},\CATIT)^{I\dplax} \to \CATIT.\]
\end{lemma}
\begin{proof}
  We deal with the case $\vepsilon=(0,1)$. Let $p$ and $q$ be the
  $(0,1)$-fibrations for functors $F \colon \tB \to \CATIT$ and
  $G \colon \tA \times \tB \to \CATIT$, respectively. We then have a
  natural identification of mapping \icats{}
  \[
    \FIB^{(0,1)}_{/\tB}(p, \pi_{\tB,*}(q)) \simeq
    \FIB^{(0,1)}_{/(\tA \times \tB, J)}(p \times \tA, q)
    \simeq
    \Nat^{J\dlax}_{\tA \times \tB, \CATIT}(F\circ \pi_{\tA}, G),
  \]
  using the straightening equivalence \cref{thm:elaxstr}, where
  $\pi_{\tA}$ denotes the projection to $\tA$. By \cref{obs:embedding} we have a natural identification
  \[
\Nat^{J\dlax}_{\tA \times \tB, \CATIT}(F \circ \pi_{\tA}, G) \simeq \Nat_{\tB,\FUN(\tA,\CATIT)^{I\dlax}}(F', G')
\]
where $G'(b)(\blank) \simeq G(\blank,b)$ and similarly for $F'$, so
that we can identify $F'$ as the functor $b \mapsto \underline{F(b)}$
given by composing with the constant diagram functor over $\tA$. This
functor in $F$ then
has a right adjoint, given by taking $I$-lax limits over $\tA$
pointwise, so that we have a natural equivalence
\[ \FIB^{(0,1)}_{/\tB}(p, \pi_{\tB,*}(q)) \simeq \Nat_{\tB,\CATIT}(F,
  \lim_{\tA}^{I\dlax}G).\]
Thus the straightening of $\pi_{\tB,*}(q)$ is the functor
$\lim_{\tA}^{I\dlax}G$, as required.
\end{proof}

\begin{thm}\label{thm:laxKan}
  Let $f \colon (\tC,I) \to \tD^{\sharp}$ be a marked functor and let $\tB$ be an \itcat{}. Then the following holds:
  \begin{itemize}
  \item If $\tB$ admits $I_{d}$-lax limits over $\tC_{d \upslash}$
    for all $d$, then we have an adjunction
    \[
      f^* \colon \FUN(\tD,\tB) \llra \Fun(\tC,\tB)^{I\dlax} \colon f_*^{I\dlax}
    \]
    where $f_{*}^{I\dlax}F(d) \simeq \lim^{I_{d}\dlax}_{\tC_{d
        \upslash}} F$.
  \item If $\tB$ admits $I_{d}$-lax colimits over $\tC_{\upslash
      d}$ for all $d$, then we have an adjunction 
    \[
      f_!^{I\dlax} \colon \Fun(\tC,\tB)^{I\dlax} \llra \FUN(\tD,\tB) \colon f^*
    \]
    where $f_{!}^{I\dlax}F(d) \simeq
    \colim^{I_{d}\dlax}_{\tC_{\upslash d}} F$.
  \item If $\tB$ admits $I_{d}$-oplax limits over $\tC_{d \downslash}$
    for all $d$, then we have an adjunction
    \[
      f^* \colon \FUN(\tD,\tB) \llra \Fun(\tC,\tB)^{I\doplax} \colon f_*^{I\doplax}
    \]
    where $f_{*}^{I\doplax}F(d) \simeq \lim^{I_{d}\doplax}_{\tC_{d
        \downslash}} F$.
  \item If $\tB$ admits $I_{d}$-oplax colimits over $\tC_{\downslash
      d}$ for all $d$, then we have an adjunction 
    \[
      f_!^{I\doplax} \colon \Fun(\tC,\tB)^{I\doplax} \llra \FUN(\tD,\tB) \colon f^*
    \]
    where $f_{!}^{I\doplax}F(d) \simeq
    \colim^{I_{d}\doplax}_{\tC_{\downslash d}} F$.
  \end{itemize}
  When they exist, we refer to the functors $f_*^{I\dplax}$ and
  $f_{!}^{I\dplax}$ as the $I$-(op)lax right and left Kan extension
  functors along $f$, respectively.
\end{thm}
\begin{proof}
  Applying the involutions $\co$ and $\op$ we reduce to verifying the
  first claim. This amounts to
  showing that the commutative square in \cref{obs:squarelaxkan} is
  horizontally adjointable. Note that the bottom horizontal morphism
  admits a right adjoint by \cref{prop:kanbasecase}, so we may reduce
  again to showing that this right adjoint restricts appropriately. In
  terms of fibrations, we thus want to show that the functor
  \[ R_{g} \colon \FIB^{(0,1)}_{/(\tB^{\op} \times \tC,
      J)} \to \FIB^{(0,1)}_{/\tB^{\op} \times \tD},\]
  where $g \colon (\tB^{\op} \times \tC,
      J) \to (\tB^{\op} \times \tD)^{\sharp}$ is the marked functor
      $(\tB^{\op})^{\sharp} \times f$,
  has the following property: given a $(0,1)$-fibration $p \colon \tX
  \to \tB^{\op}\times \tC$ such that for every $c \in \tC$, the
  pullback $\tE_{c} \to \tB^{\op}$ straightens to a representable
  presheaf on $\tB$, then the fibration $R_{g}(p) \colon \tY \to
  \tB^{\op} \times \tD$ has the analogous property for every $d \in
  \tD$.

  We thus want to identify the pullback $i^{*}R_{g}(p)$ for $i \colon
  \tB^{\op} \times \{d\} \to \tB^{\op} \times \tD$. Recall that by
  definition, $R_{g}$ is the composite $\phi_{*}\pi^{*}$ for functors
  \[ (\tB^{\op} \times \tC,J)^{\sharp} \xleftarrow{\pi} \tQd \xrightarrow{\phi}
    (\tB^{\op} \times \tD)^{\sharp\sharp}, \] using the notation from
  \cref{def:Rf}. If we consider the pullback
  \[
    \begin{tikzcd}
     \tQd_{d} \ar[r, "j"] \ar[d, "\phi_{d}"'] & \tQd  \ar[d, "\phi"] \\
     \tB^{\op,\sharp\sharp} \times \{d\} \ar[r, "i"'] & (\tB^{\op}
     \times \tD)^{\sharp\sharp},
    \end{tikzcd}
  \]
  then \cref{obs:fibre of pushforward} gives a base change equivalence
  $i^{*}\phi_{*} \simeq \phi_{d,*}j^{*}$, so that we have
  \[ i^{*}R_{g} \simeq \phi_{d,*}r^{*}\] where $r$ is the composite
  $\tQd_{d} \xto{j} \tQd \xto{\pi} (\tB^{\op} \times
  \tC,J)^{\sharp}$. Unpacking the notation, we see that $\tQd_{d}$ is
  $\tCd_{d \upslash} \times \tB^{\op,\sharp\sharp}$ with $r$ being $s
  \times \id_{\tB^{\op}}$ where $s$ is the canonical functor $\tCd_{d
    \upslash} \to (\tC,I)^{\sharp}$ and $\phi_{d}$ is the projection
  to $\tB^{\op}$. The desired property therefore follows from
  \cref{lem:pointwiselimit} and our assumption on $\tB$.
\end{proof}

In the maximally marked case, this specializes to give ordinary Kan extensions of \itcats{}:
\begin{cor}\label{cor:strong kan ext}
  For a functor of \itcats{} $f \colon \tC \to \tD$ and an \itcat{}
  $\tB$, we have:
  \begin{itemize}
  \item If $\tB$ admits $\tD(d,f(\blank))$-weighted limits over $\tC$
    for all $d$, then we have an adjunction
    \[
      f^* \colon \FUN(\tD,\tB) \llra \Fun(\tC,\tB) \colon f_*
    \]
    where $f_{*}F(d) \simeq \lim_{\tC}^{\tD(d,f(\blank))} F$.
  \item If $\tB$ admits $\tD(F(\blank),d)$-weighted colimits
    over $\tC$ for all $d$, then we have an adjunction 
    \[
      f_! \colon \Fun(\tC,\tB) \llra \FUN(\tD,\tB) \colon f^*
    \]
    where $f_{!}F(d) \simeq \colim_{\tC}^{\tD(F(\blank),d)} F$.
  \end{itemize}
\end{cor}
\begin{proof}
  We prove the first case. From the maximally marked case of
  \cref{thm:laxKan}, we see that the right adjoint exists and is given
  either by taking partially lax limits over $\tC_{d \upslash}$ or by
  partially oplax limits over $\tC_{d \downslash}$, in both cases
  marked by the cartesian morphisms over $\tC$, if these exist. Here
  $\tC_{d \upslash} \to \tC$ is the $(0,1)$-fibration for the functor
  $\tD(d,f(\blank))$ while $\tC_{d \downslash} \to \tC$ is the
  $(0,0)$-fibration for $\tD(d,f(\blank))^{\op}$, so that
  $(\tC_{d \downslash})^{\op} \to \tC^{\op}$ is the $(1,0)$-fibration
  for $\tD(d,f(\blank))$. It then follows from \cref{propn:weighted
    limit as lax} that both of these compute the
  $\tD(d,f(\blank))$-weighted limit of $F$.
\end{proof}

At the other extreme, taking $f$ to be the identity gives the following special case:
\begin{cor}\label{cor:delaxify}
  Let $\tB$ be an \itcat{} and $(\tC,I)$ a marked \itcat{}. Then we have:
  \begin{itemize}
  \item If $\tB$ admits $I_{c}$-lax limits over $\tC_{c \upslash}$
    for all $c \in \tC$, then we have an adjunction
    \[
      \id^* \colon \FUN(\tC,\tB) \llra \Fun(\tC,\tB)^{I\dplax} \colon \id_*^{I\dlax}
    \]
    where $\id_{*}^{I\dlax}F(c) \simeq \lim^{I_{d}\dlax}_{\tC_{c
        \upslash}} F$.
  \item If $\tB$ admits $I_{c}$-lax colimits over $\tC_{\upslash
      c}$ for all $c \in \tC$, then we have an adjunction 
    \[
      \id_!^{I\dlax} \colon \Fun(\tC,\tB)^{I\dlax} \llra \FUN(\tC,\tB) \colon \id^*
    \]
    where $\id_{!}^{I\dlax}F(c) \simeq
    \colim^{I_{c}\dlax}_{\tC_{\upslash c}} F$.
  \item If $\tB$ admits $I_{c}$-oplax limits over $\tC_{c \downslash}$
    for all $c \in \tC$, then we have an adjunction
    \[
      \id^* \colon \FUN(\tC,\tB) \llra \Fun(\tC,\tB)^{I\dplax} \colon \id_*^{I\doplax}
    \]
    where $\id_{*}^{I\doplax}F(c) \simeq \lim^{I_{c}\doplax}_{\tC_{c
        \downslash}} F$.
  \item If $\tB$ admits $I_{c}$-oplax colimits over $\tC_{\downslash
      c}$ for all $c \in \tC$, then we have an adjunction 
    \[
      \id_!^{I\doplax} \colon \Fun(\tC,\tB)^{I\doplax} \llra \FUN(\tC,\tB) \colon \id^*
    \]
    where $\id_{!}^{I\doplax}F(c) \simeq
    \colim^{I_{c}\doplax}_{\tC_{\downslash c}} F$. \qed
  \end{itemize}
\end{cor}

\begin{observation}\label{rem:LKEff}
  Let $f \colon (\tC,I) \to \tD^{\sharp}$ be a marked functor and
  assume that $\tB$ has $I_{d}$-lax limits over $\tC_{d
    \upslash}$. We make the following claims:
    \begin{itemize}
      \item One can identify the unit of the adjunction $f^* \dashv f_*^{I\dlax}$ at an object $c$ with the
  canonical map on $I\dlax$ limits induced via restriction along the
  marked functor
  $(\tC_{d\upslash},I_{d}) \to \tD^{\sharp}_{d\upslash}$.
\item One can identify the counit of the adjunction
  $f^* \dashv f_*^{I\dlax}$ at an object $c$ with the canonical map on
  $I\dlax$ limits induced via restriction along the marked functor
  $(\tC_{f(c)\upslash},I_{f(c)}) \to
  \tD^{\sharp}_{f(c)\upslash}$. 
\end{itemize}
  To see this, recall that we showed in \cref{thm:laxKan} that the
  commutative square in \cref{obs:squarelaxkan} is horizontally
  adjointable. In particular, this allows us to reduce to
  $\tB=\CATI$. The first claim from \cref{obs:unitcofree} by setting
  $p=\DFFree^{\vepsilon}_{\tDss}(t_d)$ for $t_d \colon [0] \to \tDss$
  selecting the object $d$. The second claim follows similarly from
  \cref{obs:counitcofree}, by setting $p$ to be a decorated functor
  with source the terminal category.

  In particular, if $f$ is fully faithful, we
  have that $\tC_{f(c)\upslash} \simeq \tC_{c \upslash}$; since the
  inclusion of $\id_{c}$ is $(0,1)$-cofinal  by
  \cref{propn:rep cart fib cond}, this characterization of the
  adjunction counit allows us to conclude that
  $f_*^{I\dlax}$ is a fully faithful functor.
\end{observation}

\begin{propn}\label{thm:laxcones}
	Let $(\tI,E)$ be a marked
        \itcat{} and let $\tC$ be an \itcat{} that admits partially
        (op)lax colimits of shape $(\tI,E)$. Then the left
        Kan extension along $\iota \colon \tI^{\sharp \sharp} \hra
        (\tI_{\elaxoplax}^{\colimcone})^{\sharp \sharp}$ exists, is
        fully faithful and has as essential image the full
        sub-$(\infty,2)$-category on $E$-(op)lax colimit cones.  
\end{propn}
\begin{proof}
  Since the inclusion $\iota$ is fully faithful by
  \cref{prop:inclusionconeff}, the left Kan extension $\iota_{!}$ will
  be fully faithful if it exists by \cref{rem:LKEff}. To show that
  $\iota_!$ does exist, we will apply \cref{cor:strong kan ext} and verify that the corresponding weighted colimits exist. We start by considering the pullback diagram
	\[
		\begin{tikzcd}
			\tI(j)=\left(\tI^{\colimcone}_{\elaxoplax}\right)_{ \upslash j}\times_{\tI^{\colimcone}_{\elaxoplax}}\tI  \arrow[d,"p"] \arrow[r,"\iota"] & \left(\tI^{\colimcone}_{\elaxoplax}\right)_{ \upslash j} \arrow[d,"\pi"] \\
			\tI \arrow[r,"j"] &\tI^{\colimcone}_{\elaxoplax},
		\end{tikzcd}
	\]
	and note that $\tI(j)$ can be identified with $\tI_{
          \upslash j}$ whenever $j\neq *$,
        where $*$ denotes the cone
        point. In particular this weighted colimit exists and is
        given by evaluating our functor at the object $j$. If
        $j= \ast$, then \cref{rem:conestraightening} tells us that
        $\tI(*) \to \tI$ is the $(1,0)$-fibration
        $\Frof^{(1,0)}_{\tI}(\tI,E)$. We conclude that the
        corresponding weighted colimit exists by our assumptions after
        invoking \cref{propn:lax as Kan weighted}.

	To finish the proof, we must show that a functor $\hat{F} \colon \bII_{\elaxoplax}^{\colimcone} \to \bCC$ defines an $\elaxoplax$ colimit cone for $F\simeq \iota^* \hat{F}$ if and only if $\hat{F} \simeq \iota_{!} F$. From \cref{rem:LKEff} we know that the counit transformation can be identified with the map between colimits
  \[
    \colim^{\natural\text{-}\laxoplax}_{\bII(*)}\hat{F}\circ \pi \circ \iota  \to \colim^{\natural\text{-}\laxoplax}_{\left(\bII^{\colimcone}_{\elaxoplax}\right)_{ \upslash *}} \hat{F} \circ \pi \simeq \hat{F}(\ast)
  \]
  where the last equivalence comes from identifying the colimit over the slice with evaluation at the object corresponding to the identity map, exactly as in the case of $\tI(j)$ for $j \neq *$. 

  Next, we observe that the unit of the adjunction in \cref{cor:free
    1-fibred on marked} gives a section $s \colon (\bII,E) \to
  \bII(\ast)$ of $p$. This
  adjunction immediately implies that $s$ is a marked
  $(1,0)$-equivalence over $\tI^{\sharp}$, \ie{} $(1,0)$-cofinal.
  Combining this with the equivalence $\hat{F} \circ j \circ  p  \simeq \hat{F}\circ \pi \circ \iota$,  we can identify the counit transformation with the map
  \[
    \colim^{\elaxoplax}_{(\bII,E)} \hat{F}\circ j \xrightarrow{} \hat{F}(*).
  \]
  By construction, this map is induced, via the universal property of
  the $\elaxoplax$-colimit, by the cone determined by $\hat{F}$. We
  conclude that this map is an equivalence precisely when $\hat{F}$ is
  a colimit cone.
\end{proof}

\subsection{A Bousfield--Kan formula for weighted colimits}
\label{sec:bousfieldkan}
Our goal in this section is to use our results on Kan extensions to
obtain a Bousfield--Kan type (or ``bar construction'') description of
weighted colimits in \itcats{}; for a general enrichment this is also
done by Heine~\cite[Theorem 3.44]{heineweighted} using very different
methods. The starting point for this is a description of mapping
\icats{} in functor \itcats{}, which we deduce from the monadicity
theorem via the following observations:

\begin{propn}
  Suppose $\tD$ is a cocomplete \itcat{} and $f \colon \tC' \to \tC$
  is an essentially surjective functor of small \itcats{}. Then the
  adjunction
  \[ f_{!} : \Fun(\tC', \tD) \rightleftarrows \Fun(\tC, \tD) : f^{*}\]
  is monadic, and $f^{*}$ preserves all colimits.
\end{propn}
\begin{proof}
  It follows from \cref{lem:ess surj restr cons} that $f^{*}$ is
  conservative. Moreover, it preserves all colimits by
  \cref{cor:colim strong functor cat}. It therefore
  follows from the monadicity
  theorem \cite[Theorem 4.7.3.5]{LurieHA} that the adjunction is monadic.
\end{proof}

 As a special case, we have:
 \begin{cor}\label{cor:monadiccore}
   Suppose $\tD$ is a cocomplete \itcat{} and $\tC$ is a small
   \itcat{}. Then the adjunction
   \[ i_{!} : \Fun(\tC^{\simeq},\tD) \rightleftarrows \Fun(\tC,
     \tD) : i^{*}\]
   is monadic, where $i$ is the core inclusion $\tC^{\simeq} \to \tC$. \qed
 \end{cor}

 \begin{propn}\label{propn:nat end}
   For functors of \itcats{} $F,G \colon \tC \to \tD$, the \icat{} of natural
   transformations from $F$ to $G$ can be expressed as a limit of the
   form
   \[ \Nat_{\tC,\tD}(F,G) \simeq \lim_{[n] \in \simp}
     \lim_{(c_{0},\dots,c_{n}) \in (\tC^{\simeq})^{\times
         n+1}}\Fun(\tC(c_{0},\dots,c_{n}), \tD(F(c_{0}), G(c_{n}))\]
   where \[\tC(c_{0},\dots,c_{n}) := \tC(c_{0},c_{1}) \times
     \tC(c_{1},c_{2}) \times \cdots \times \tC(c_{n-1},c_{n}).\]
 \end{propn}
 \begin{proof}
   By embedding $\tD$ in (large) presheaves if necessary (which gives
   a fully faithful functor on functor \itcats{} by \cite[Corollary
   2.7.14]{AGH24}), we may without loss of generality assume that
   $\tC$ is small and $\tD$ is cocomplete (with respect to some
   universe).  Since the adjunction $i_{!} \dashv i^{*}$ for
   $i \colon \tC^{\simeq} \to \tC$ is then monadic by
   \cref{cor:monadiccore}, the object $F$ has a simplicial free
   resolution as a colimit
  \[ F \simeq \colim_{[n] \in \simp^{\op}} i_{!}T^{n}i^{*}F\] in
  $\Fun(\tC, \tD)$, where $T := i^{*}i_{!}$ (\eg{} by the proof of
  \cite[Lemma 4.7.3.13]{LurieHA}). Since
  $\FUN(\tC, \tD)$ is a cocomplete \itcat{} by \cref{cor:colim strong
    functor cat}, this is also an \itcatl{} colimit, so that we get an
  equivalence
  \[
    \begin{split}
      \Nat_{\tC,\tD}(F, G) & \simeq \lim_{[n] \in \simp}
                             \Nat_{\tC,\tD}(i_{!}T^{n}i^{*}F, G)
      \\
       & \simeq \lim_{[n] \in \simp} \Nat_{\tC^{\simeq},
         \tD}(T^{n}i^{*}F, i^{*}G) \\
      & \simeq \lim_{[n] \in \simp} \lim_{c \in \tC^{\simeq}}
        \tD((T^{n}i^{*}F)(c), G(c)),
    \end{split}
  \]
  where the first equivalence uses that the left Kan extension $i_{!}$
  is an \itcatl{} left adjoint and the last equivalence uses that $\tC^{\simeq}$ is an \igpd{},
  so that we have an equivalence $\FUN(\tC^{\simeq}, \tD) \simeq
  \lim_{\tC^{\simeq}} \tD$. 

  Now \cref{obs:weighted over igpd} and the weighted colimit formula
  for left Kan extensions from \cref{cor:strong kan ext} imply that
  for $W \colon \tC^{\simeq} \to \tD$, we have
  \[ (i_{!}W)(c) \simeq \colim^{\tC(\blank,c)}_{\tC^{\simeq}} W
    \simeq \colim_{x \in \tC^{\simeq}} W(x) \boxtimes \tC(x,c),\]
  where $\boxtimes$ denotes the tensoring of $\tD$ over \icats{},
  so that
  \[
    \begin{split}
      (T^{n}W)(c) & \simeq \colim_{x \in \tC^{\simeq}} T^{n-1}W(x) \boxtimes
                    \tC(x,c) \\
       & \simeq \colim_{(x_{1},\ldots,x_{n}) \in
         (\tC^{\simeq})^{\times n}} W(x_{1}) \boxtimes \tC(x_{1},\dots,x_{n},c).
    \end{split}
  \]
  Taking this colimit out and rewriting the resulting limit, we then get
  \[
    \begin{split}
      \Nat_{\tC,\tD}(F, G) & \simeq \lim_{[n] \in \simp}
                             \lim_{c \in \tC^{\simeq}} \tD(\colim_{(x_{1},\ldots,x_{n}) \in
         (\tC^{\simeq})^{\times n}} F(x_{1}) \boxtimes
                             \tC(x_{1},\dots,x_{n},c), G(c))
      \\
      & \simeq \lim_{[n] \in \simp}
                             \lim_{(x_{0},\ldots,x_{n}) \in
        (\tC^{\simeq})^{\times (n+1)}} \tD(F(x_{0}) \boxtimes
        \tC(x_{0},\dots,x_{n}), G(x_{n})) \\
       & \simeq \lim_{[n] \in \simp}
                             \lim_{(x_{0},\ldots,x_{n}) \in
        (\tC^{\simeq})^{\times (n+1)}} \Fun(\tC(x_{0},\dots,x_{n}), \tD(F(x_{0}), G(x_{n}))),
    \end{split}
  \]
  as required.
\end{proof}

\begin{cor}\label{cor:cotensordecomp}
  For a functor of \itcats{} $F \colon \tC \to \tD$, where $\tC$ is
  small, we have:
  \begin{enumerate}[(i)]
  \item For any weight $W \colon \tC^{\op} \to \CATI$, the weighted
    colimit $\colim^{W}_{\tC}F$ exists and is computed as
  \[  \colim_{[n] \in \Dop} \colim_{(c_{0},
      \dots, c_{n}) \in (\tC^{\simeq})^{\times (n+1)}} F(c_{0})
    \boxtimes \left(\tC(c_{0},\dots,c_{n}) \times W(c_{n})\right)\]
  if this exists in $\tD$.
\item   For any weight $W
  \colon \tC \to \CATI$, the weighted limit
  $\lim^{W}_{\tC}F$ exists and is computed as
  \[  \lim_{[n] \in \simp} \lim_{(c_{0},
      \dots, c_{n}) \in (\tC^{\simeq})^{\times (n+1)}}
    F(c_{n})^{W(c_{0}) \times \tC(c_{0},\dots,c_{n})}\]
  if this exists in $\tD$.
  \end{enumerate}
\end{cor}
\begin{proof}
  We prove the colimit case; the limit case is proved similarly, or
  follows by applying $(\blank)^{\op}$. For $d \in \tD$, we have
  \[
    \begin{split}
      \tD(\colim^{W}_{\tC} F, d) & \simeq \Nat_{\tC^{\op}, \CATI}(W,
                                   \tD(F, d)) \\
       & \simeq \lim_{[n] \in \simp} \lim_{(c_{0},\ldots,c_{n}) \in
        (\tC^{\simeq})^{\times (n+1)}} \Fun(W(c_{0}) \times
         \tC^{\op}(c_{0},\dots,c_{n}), \tD(F(c_{n}), d)) \\
      & \simeq \lim_{[n] \in \simp} \lim_{(c_{0},\ldots,c_{n}) \in
        (\tC^{\simeq})^{\times (n+1)}} \tD(F(c_{0}) \boxtimes
        \left(\tC(c_{0},\dots,c_{n}) \times W(c_{n})\right), d) \\
                                 & \simeq \tD(\colim_{[n] \in \Dop} \colim_{(c_{0},
      \dots, c_{n}) \in (\tC^{\simeq})^{\times (n+1)}} F(c_{0})
    \boxtimes \left(\tC(c_{0},\dots,c_{n}) \times W(c_{n}), d)\right),
    \end{split}
  \]
  as required, where in the third equivalence we have reordered the
  list of objects to pass from $\tC^{\op}$ to $\tC$.
\end{proof}

\begin{cor}\
  \begin{enumerate}[(i)]
  \item An \itcat{} $\tC$ is cocomplete \IFF{} $\tC$ is tensored over
    $\CatI$ and $\tC^{\leq 1}$ is a cocomplete \icat{}.
  \item An \itcat{} $\tC$ is complete \IFF{} $\tC$ is cotensored over
    $\CatI$ and $\tC^{\leq 1}$ is a complete \icat{}.   \qed
  \end{enumerate}
\end{cor}

\subsection{Free cocompletion}
\label{sec:free cocomp}
Our goal in this section is to identify $\tPSh(\tC)$ as the free
cocompletion of a small \itcat{} $\tC$. This is a special case of
Heine's very general construction of free cocompletions of enriched
\icats{} under a class of colimits in \cite[\S 3.11]{heineweighted};
the analogue for $(\infty,\infty)$-categories has also been proved by
Loubaton \cite[Corollary 4.2.4.8]{loubaton}. For our proof we need a
functorial version of the Yoneda lemma, which we will derive from our
results on fibrations using the following construction:
\begin{construction}
  Given a functor $\Phi \colon \tA^{\op} \times \tB \to \CATI$, we can
  view it as a functor $\Phi' \colon \tA^{\op} \to \FUN(\tB, \CATI)
  \simeq \oFIB_{/\tB}^{(0,1)}$. Identifying the latter with a full
  sub-\itcat{} of $\MCATITsl{\tB^{\sharp}}$, we can straighten this to
  a morphism of marked (1,0)-fibrations
  \[
    \begin{tikzcd}
     (\tE,I) \ar[rr, "{(p,q)}"] \ar[dr, "p"'] & & \tA^{\sharp} \times \tB^{\sharp} \ar[dl] \\
      & \tA^{\sharp},
    \end{tikzcd}
  \]
  meaning that $p$-cartesian 1-morphisms and $p$-cocartesian
  2-morphisms in $\tE$ map to equivalences in $\tB$.
  Given a 1-morphism $\phi \colon x \to y$ in $\tE$ over $(f\colon a \to a', g
  \colon b \to b')$ in $\tA \times \tB$, we can factor $\phi$ as
  \[ x \xto{\phi'} f^{*}y \xto{\bar{f}} y\]
  where $\bar{f}$ is $p$-cartesian over $f$ and so lies over $(f,
  \id_{b'})$; then $\phi'$ lies over $(\id_{a},g)$ and so factors as
  $x \to g_{!}x \to f^{*}y$ through a $q_{a}$-cocartesian morphism in
  $\tE_{a}$; the morphism $\phi$ lies in the marking $I$ precisely if
  the resulting morphism $g_{!}x \to f^{*}y$ is an equivalence, \ie{}
  if $\phi$ factors as a $q_{a}$-cocartesian morphism followed by a
  $p$-cartesian morphism. We refer to
  \[ (p,q) \colon (\tE,I) \to \tA^{\sharp} \times \tB^{\sharp}\] as
  the \emph{marked $(1,0)$-bifibration} associated to $\Phi$. If we
  instead view $\Phi$ as a functor
  $\tB \to \FUN(\tA^{\op}, \CATI) \simeq \oFIB^{(1,0)}_{/\tA}$ and
  straighten this over $\tB$, we similarly obtain the \emph{marked
    $(0,1)$-bifibration}\footnote{In fact, we expect this to be the
    same as the marked $(1,0)$-bifibration for $\Phi$, just with the
    order of $\tA$ and $\tB$ reversed; this would follow from an
    \itcatl{} version of the uniqueness of bivariant straightening
    proved in \cite{TwoVariable}, but we will not pursue this here.}
  \[ (q',p') \colon (\tE', I') \to \tB^{\sharp} \times \tA^{\sharp}\]
   for $\Phi$.
 \end{construction}

\begin{propn}\label{propn:push pull bifib}
  Given a functor $\Phi
  \colon \tA^{\op} \times \tB \to \CATI$, let 
  \[ (p,q) \colon (\tE,I) \to \tA^{\sharp} \times \tB^{\sharp},
    \qquad (q',p') \colon (\tE',I') \to \tB^{\sharp} \times \tA^{\sharp}\]
  be the marked $(1,0)$-bifibration and marked $(0,1)$-bifibration
  corresponding to $\Phi$, respectively.
  \begin{enumerate}[(i)]
  \item For a functor $G \colon \tB
    \to \CATI$ with corresponding 1-fibred $(0,1)$-fibration $\tQ \to
    \tB$, the functor $\Nat_{\tB,\CATI}(\Phi, G) \colon \tA \to \CATI$ is
    classified by the 1-fibred $(0,1)$-fibration $p_{*}q^{*}\tQ$ over
    $\tA$.
  \item For a functor $H \colon \tA^{\op} \to \CATI$ with corresponding
    1-fibred $(1,0)$-fibration $\tP \to \tA$, the functor
    $\Nat_{\tA^{\op},\CATI}(\Phi, G) \colon \tB^{\op} \to \CATI$ is
    classified by the 1-fibred $(1,0)$-fibration $q'_{*}p'^{*}\tP$ over
    $\tB$.
  \end{enumerate}
\end{propn}
\begin{proof}
  We prove the first case. Let $\Phi'$ be the functor $\tA^{\op} \to \oFIB^{(0,1)}_{/\tB}$
  associated to $\Phi$. Then we have a natural equivalence
  \[ \Nat_{\tB, \CATI}(\Phi, G) \simeq \oFIB^{(0,1)}_{/\tB}(\Phi', \tQ)
    \simeq \DCATITsl{\tB^{\sharp\sharp}}(\Phi'^{\natural}, \tQ^{\natural}). \]
  Given $W \colon \tA \to \CATI$, we get a natural pullback
  square
  \[
    \begin{tikzcd}
     \Nat_{\tA,\CATI}(W, \Nat_{\tB,\CatI}(\Phi,G)) \ar[r]
     \ar[d] & \Nat_{\tA,\CATI}(W, \DFun(\Phi'^{\natural},
     \tQ^{\natural}))  \ar[d] \\
     * \ar[r] & \Nat_{\tA,\CATI}(W,
     \DFun(\Phi'^{\natural}, \tB^{\sharp\sharp})),
    \end{tikzcd}
  \]
  where we regard $\Phi'$ as a functor to $\DCATIT$. This allows us to
  identify the top left \icat{} as
  \[\DFun_{/\tB^{\sharp\sharp}}(\colim^{W}_{\tA^{\op}}\Phi'^{\natural},
    \tQ^{\natural})\]
  with the weighted colimit taken in $\DCATIT$. Here $p \colon (\tE,I)^{\sharp}
  \to \tA^{\sharp\sharp}$ is the decorated $(1,0)$-fibration for
  $\Phi'$,
  so if $\pi \colon \tW \to \tA$ is the $(0,1)$-fibration for $W$, then
  \cref{cor:wtcolimcatit} identifies this weighted colimit as the
  localization of $(\tE,I)^{\sharp} \times_{\tA^{\sharp\sharp}}
  \tW^{\natural}$ where we invert
  \begin{itemize}
  \item the 1-morphisms that project to a $p$-cartesian morphism in
    $\tE$ and a $\pi$-cocartesian morphism in $\tW$,
  \item the 2-morphisms that project to a $p$-cocartesian 2-morphism
    in $\tE$ (as all 2-morphisms  in $\tW$ are $\pi$-cartesian).
  \end{itemize}
  We claim that in our case we can ignore this localization: given a
  commutative triangle
  \[
    \begin{tikzcd}
      (\tE,I)^{\sharp} \times_{\tA^{\sharp\sharp}}
  \tW^{\natural} \ar[rr, "\alpha"] \ar[dr] & & \tQ^{\natural} \ar[dl, "\pi"] \\
      & \tB^{\sharp\sharp}
    \end{tikzcd}
  \]
  we see that any 1-morphism in the source that projects to a
  $p$-cartesian morphism in $\tE$ is decorated, and so is sent to a
  $\pi$-cocartesian morphism in $\tQ^{\natural}$ --- but it also
  projects to an equivalence in $\tB$ via $q$, so its image is
  $\pi$-cocartesian over an equivalence and so must itself be an
  equivalence. Similarly, the image of any 2-morphism that projects to
  a $p$-cocartesian morphism in $\tE$ must be invertible in $\tQ$. It
  follows that we have natural equivalences
  \[
    \begin{split}
      \Nat_{\tA,\CATI}(W, \Nat_{\tB,\CatI}(\Phi,G))
      & \simeq \DFun_{/\tB^{\sharp\sharp}}((\tE,I)^{\sharp} \times_{\tA^{\sharp\sharp}}
        \tW^{\natural}, \tQ^{\natural}) \\
      & \simeq \oFIB^{(0,1)}_{/\tB}(q_{!}p^{*}\tW^{\natural},
        \tQ^{\natural}) \\
      & \simeq \oFIB^{(0,1)}_{/\tB}(\tW^{\natural},
        p_{*}q^{*}\tQ^{\natural}),
    \end{split}
  \]
  as required.
\end{proof}

As a special case, we can identify the fibration for a functor of the form $\Fun(F(\blank), \oC)$:
\begin{corollary}\label{cor:fib for Fun}
  Consider a functor $F \colon \tB \to \CATI$ and let
  \begin{itemize}
  \item $p \colon \tE \to \tB$ be the $(0,1)$-fibration for $F$,
  \item $q \colon \tE' \to \tB^{\op}$ be the $(1,0)$-fibration for $F$.
  \end{itemize}
  Then for any \icat{} $\oC$, the functor $\Fun(F(\blank), \oC) \colon \tB^{\op} \to \CATI$ is classified by
  \begin{itemize}
  \item the $(1,0)$-fibration $p_{*}(\oC \times \tE)$,
  \item the $(0,1)$-fibration $q_{*}(\oC \times \tE')$. \qed
  \end{itemize}
\end{corollary}

Combined with our work on right Kan extensions, we now obtain the desired
functorial version of the Yoneda Lemma:
\begin{notation}
  For an \itcat{} $\tC$, Let $h_{\tC} \colon \tC \to \tPSh(\tC)$ be
  the functor obtained by straightening \[(\ev_{0},\ev_{1}) \colon \ARopl(\tC) \to \tC \times
  \tC\] in the second variable to a functor
  \[ \tC \to \oFIB^{(1,0)}_{/\tC}\simeq \tPSh(\tC).\]
\end{notation}
\begin{cor}\label{cor:functorial yoneda}
  We have a natural equivalence
  \[ \Nat_{\tC^{\op},\CATI}(h_{\tC}, \blank) \simeq \id\]
  of functors $\tPSh(\tC) \to \tPSh(\tC)$.
\end{cor}
\begin{proof}
  From \cref{propn:push pull bifib}, we know that
  $\Nat_{\tC^{\op},\CATI}(h_{\tC}, \blank)$ can be identified in terms
  of fibrations as
  \[ \ev_{0,*}\ev_{1}^{*} \colon \oFIB^{(1,0)}_{/\tC} \to
    \oFIB^{(1,0)}_{/\tC}.\]
  But \cref{thm:cofree} identifies the latter as the right adjoint to
  pullback along $\id_{\tC}$.
\end{proof}

\begin{remark}
  A similarly functorial version of the Yoneda lemma for \iicats{}
  appears as \cite[Theorem 4.2.1.20]{loubaton}.
\end{remark}

\begin{cor}[``Co-Yoneda Lemma'']
  For an \itcat{} $\tC$ we have a natural equivalence
  \[ \colim^{\Phi}_{\tC} h_{\tC} \simeq \Phi \]
  for $\Phi \in \tPSh(\tC)$.
\end{cor}
\begin{proof}
  By definition, this weighted colimit should satisfy
  \[ \tPSh(\tC)(\colim^{\Phi}_{\tC} h_{\tC}, \Psi) \simeq \Nat_{\tC^{\op},
      \CATI}(\Phi, \Nat_{\tC^{\op}, \CATI}(h_{\tC}, \Psi)),\]
  and from \cref{cor:functorial yoneda} the right-hand side is
  naturally equivalent to the \icat{} $\Nat_{\tC^{\op}, \CATI}(\Phi, \Psi)$.
\end{proof}

With these results in hand, we can now turn to the free cocompletion
of a small \itcat{}.

\begin{propn}
  Suppose $\tC$ is a small \itcat{} and $\tD$ is a cocomplete
  \itcat{}.
  \begin{enumerate}[(i)]
  \item For any functor $F \colon \tC \to \tD$, the left Kan extension
    $h_{\tC,!}F \colon \tPSh(\tC) \to \tD$ exists and is cocontinuous.
  \item For any cocontinuous functor $G \colon \tPSh(\tC) \to \tD$,
    the counit map $h_{\tC,!}h_{\tC}^{*}G \to G$ is an equivalence.
  \end{enumerate}
\end{propn}
\begin{proof}
  By \cref{cor:strong kan ext}, the Kan extension $h_{\tC,!}F$ exists
  if $\tD$ admits colimits weighted by
  $\tPSh(\tC)(h_{\tC}(\blank), \phi)$ for all $\phi$ in $\tPSh(\tC)$;
  by \cref{cor:functorial yoneda} this presheaf is naturally
  equivalent to $\phi$, so this exists by our assumption that $\tC$ is
  small and $\tD$ is cocomplete. To see that $h_{\tC,!}F$ is
  cocontinuous, we consider $\Phi \colon \tJ \to \tPSh(\tC)$ and $W
  \in \tPSh(\tJ)$ and use \cref{propn:colim over colim of weights} to compute
  \[ h_{\tC,!}F(\colim^{W}_{\tJ} \Phi) \simeq
    \colim_{\tC}^{\colim^{W}_{\tJ} \Phi} F \simeq \colim_{\tJ}^{W}
    \colim_{\tC}^{\Phi} F \simeq \colim_{\tJ}^{W} h_{\tC,!}F(\Phi).\]
  Finally, if $G$ preserves colimits, then the counit map
  $h_{\tC,!}h_{\tC}^{*}G \to G$ at $\phi \in \PSh(\tC)$ is the
  canonical map
  \[ \colim^{\phi}_{\tC} G(h_{\tC}) \to G(\colim^{\phi}_{\tC}
    h_{\tC}),\]
  which is an equivalence since $G$ is cocontinuous.
\end{proof}

\begin{cor}
  If $\tC$ is a small \itcat{} and $\tD$ is a cocomplete
  \itcat{}, then the restriction functor
  \[ h_{\tC}^{*} \colon \FUN(\tPSh(\tC), \tD) \to \FUN(\tC, \tD)\]
  has a fully faithful left adjoint $h_{\tC,!}$ with image the
  cocontinuous functors. \qed
\end{cor}

\subsection{Presentable \itcats{}}
\label{sec:pres}

In this section we consider a rather simple-minded definition of
presentable \itcats{}, and show that this admits a number of other
characterizations generalizing those for presentable \icats{}:

\begin{defn}
  An \itcat{} $\tC$ is \emph{presentable} if $\tC$ is cocomplete and
  $\tC^{\leq 1}$ is presentable, or equivalently accessible.
\end{defn}

We state the comparison first and then introduce the terminology
needed to unpack it:
\begin{thm}\label{thm:pres char}
  The following are equivalent for an \itcat{} $\tC$:
  \begin{enumerate}[(1)]
  \item $\tC$ is presentable.
  \item $\tC$ is cocomplete and locally small, and there is a regular
    cardinal $\kappa$ and a small full sub-\itcat{} $\tC_{0} \subseteq
    \tC$ consisting of 2-$\kappa$-compact objects, such that every
    object in $\tC$ is the conical colimit of a diagram in $\tC$
    indexed by a $\kappa$-filtered \icat{}.
  \item There is a small \itcat{} $\tJ$ and a fully faithful functor 
    $\tC \hookrightarrow \tPSh(\tJ)$ that admits a left adjoint and whose image is closed under
    conical colimits over $\kappa$-filtered \icats{} for some regular
    cardinal $\kappa$.
  \item There is a small \itcat{} $\tJ$, a small set $S$ of 1-morphisms in $\tPSh(\tJ)$ and an
    equivalence $\tC \simeq \Loc_{S}(\tPSh(\tJ))$.
  \end{enumerate}
\end{thm}

\begin{defn}
  An object $c \in \tC$ is \emph{2-$\kappa$-compact} for a regular
  cardinal $\kappa$ if
  $\tC(c,\blank)$ preserves conical colimits over $\kappa$-filtered \icats{}.
\end{defn}

\begin{observation}\label{obs: 2-compact via tensor}
  If $\tC$ admits conical colimits over $\kappa$-filtered \icats{} and
  tensors by $[1]$, then an object $c \in \tC$ is 2-$\kappa$-compact
  \IFF{} $c$ and $[1] \boxtimes c$ are both $\kappa$-compact in
  $\tC^{\leq 1}$. (In fact, it suffices that $[1] boxtimes c$ is
  2-$\kappa$-compact as $c$ is a retract of this.)
\end{observation}

\begin{defn}
  For $s \colon x \to y$ and $c$ in an \itcat{} $\tC$, we say that $c$
  is \emph{2-local} with respect to $s$ if the functor
  \[ s^{*} \colon \tC(y,c) \to \tC(x,c)\]
  is an equivalence. If $S$ is a set of 1-morphisms in $\tC$, we write
  $\Loc_{S}(\tC)$ for the full sub-\itcat{} of objects that are
  2-local with respect to all elements of $S$.
\end{defn}

\begin{observation}\label{obs:2-local local}
  Suppose $\tC$ is an \itcat{} that admits tensors by $[1]$. Then an
  object $c$ is 2-local with respect to $s$ \IFF{} $c$ is local in
  $\tC^{\leq 1}$ with
  respect to both $s$ and $s \boxtimes [1]$. Similarly, if $\tC$
  admits cotensors by $[1]$ then $c$ is 2-local with respect to $s$
  \IFF{} $c$ and $c^{[1]}$ are local with respect to $s$ in $\tC^{\leq
  1}$.
\end{observation}

Now we consider some easy consequences of standard results on
presentable \icats{}.

\begin{observation}
  Suppose $\tC$ is a presentable \itcat{}. Then $\tC$ is also
  cotensored over $\CatI$, as the presheaves $\Map(\oK, \tC(\blank,
  c))$ preserve limits and so are representable. Since $\tC^{\leq 1}$
  is complete, it follows that $\tC$ is also a complete \itcat{}.
\end{observation}

\begin{observation}\label{obs:pres loc small}
  If $\tC$ is a presentable \itcat{}, then for $c,c' \in \tC$ we have
  that
  \[ \Map(\oK, \tC(c,c')) \simeq \tC(\oK \boxtimes c, c')^{\simeq}\]
  is a small \igpd{} for all small \icats{} $\oK$. It follows that
  $\tC(c,c')$ is a small \icat{}, and so $\tC$ is locally small.
\end{observation}

Combining the adjoint functor theorem of \cite[Theorem 5.5.2.9]{LurieHTT}
with \cref{propn:upgrade adjoint}, we get:
\begin{propn}\label{cor:adjoint functor thm}
  Suppose $\tC$ and $\tD$ are locally small \itcats{} with $\tC$
  presentable and consider a functor $F
  \colon \tC \to \tD$.
  \begin{enumerate}
  \item $F$ is a left adjoint \IFF{} it is cocontinuous.
  \item If $\tD$ is also presentable, then $F$ is a right adjoint \IFF{} it is continuous and $F^{\leq
      1}$ is accessible.\qed
  \end{enumerate}
\end{propn}

\begin{cor}
  Suppose $\tC$ is a presentable \itcat{}.
  \begin{enumerate}
  \item A presheaf $F \colon \tC^{\op} \to \CATI$ is representable
    \IFF{} $F$ is continuous.
  \item A copresheaf $F \colon \tC \to \CATI$ is corepresentable
    \IFF{} $F$ is continuous and $F^{\leq 1}$ is accessible.
  \end{enumerate}
\end{cor}
\begin{proof}
  In the second case, the condition is equivalent to $F$ having a left
  adjoint $L \colon \CATI \to \tC$, which is the unique cocontinuous functor taking
  the point to $x = L(*)$. Then $x$ corepresents $F$ as we get
  \[ \tC(x, \blank) \simeq \Fun(*, F(\blank)) \simeq F(\blank).\]
  The first case is similar, using that $F^{\op}$ has a right adjoint.
\end{proof}

\begin{lemma}
  Suppose $\tC$ is a presentable \itcat{}. Then so is $\FUN(\tK, \tC)$
  for any small \itcat{} $\tK$. In particular, $\tPSh(\tK) :=
  \FUN(\tK^{\op}, \CATI)$ is always presentable.
\end{lemma}
\begin{proof}
  First suppose $\tK$ is an \icat{}. Then $\FUN(\tK, \tC)^{\leq 1}
  \simeq \Fun(\tK, \tC^{\leq 1})$, which is presentable. The general
  case then follows from \cref{cor:monadiccore} together with
  \cite[Corollary 4.2.3.7]{LurieHTT} or \cite[Corollary 6.8]{HenryMeadows}.
\end{proof}

We next make some simple observations about 2-compact and 2-local
objects that will lead us to the proof of \cref{thm:pres char}.

\begin{lemma}\label{lem:compact is 2-compact}
  Suppose $\tC$ is a presentable \itcat{}. Then there exists a regular
  cardinal $\kappa$ such that $\tC^{\leq 1}$ is $\kappa$-accessible
  and an object $c \in \tC$ is 2-$\kappa$-compact \IFF{} it is a
  $\kappa$-compact object of $\tC^{\leq 1}$.
\end{lemma}
\begin{proof}
  By \cref{obs: 2-compact via tensor} it is enough to show there
  exists a $\kappa$ such that \[[1] \boxtimes \blank \colon \tC^{\leq
    1} \to \tC^{\leq 1}\] preserves $\kappa$-compact objects and
  $\tC^{\leq 1}$ is $\kappa$-accessible. But this is a cocontinuous
  functor between presentable \icats{}, so such a $\kappa$ must exist:
  by the adjoint functor theorem \cite[Theorem 5.5.2.9]{LurieHTT} this
  functor has a right adjoint, which must be $\kappa$-accessible for
  some $\kappa$ such that $\tC^{\leq 1}$ is $\kappa$-accessible by
  (the proof of) \cite[Lemma 5.4.7.7]{LurieHTT}; the left adjoint then
  preserves $\kappa$-compact objects, as required.
\end{proof}

\begin{lemma}\label{lem:local closed limit}
  For any set of maps $S$ in an \itcat{}, the full sub-\itcat{}
  $\Loc_{S}(\tC)$ is closed under weighted limits in $\tC$.
\end{lemma}
\begin{proof}
  Given $W \colon \tK \to \CATI$ and $F \colon \tK \to \Loc_{S}\tC$,
  consider the commutative square
  \[
    \begin{tikzcd}
     \tC(y,\lim^{W}_{\tK}F) \ar[r, "s^{*}"] \ar[d, "\sim"'] & \tC(x,\lim^{W}_{\tK}F)  \ar[d, "\sim"] \\
     \Nat_{\tK,\CATI}(W, \tC(y,F)) \ar[r, "{\tC(s,F)_{*}}"'] & \Nat_{\tK,\CATI}(W, \tC(x,F)).
    \end{tikzcd}
  \]
  Here the bottom horizontal map is an equivalence since the natural
  transformation $\tC(s,F)$ is pointwise invertible and so an
  equivalence in $\FUN(\tK, \CATI)$.
\end{proof}

\begin{lemma}\label{lem: acc locn 2-pres}
  Suppose $\tC$ is a presentable \itcat{} and
  $\tC' \hookrightarrow \tC$ is a full sub-\itcat{} that is closed
  under conical colimits over $\kappa$-filtered \icats{} for some
  $\kappa$. If the inclusion has a left adjoint, then $\tC'$ is a
  presentable \itcat{}.
\end{lemma}
\begin{proof}
  It is clear that $\tC'$ is cocomplete, since colimits can be
  computed by applying the left adjoint of the inclusion to the
  colimit in $\tC$. Moreover, $\tC'^{\leq 1}$ is presentable by
  \cite[Proposition 5.5.4.2]{LurieHTT} since it is a reflective
  subcategory of the presentable \icat{} $\tC^{\leq 1}$, with the
  inclusion preserving $\kappa$-filtered colimits.
\end{proof}

\begin{lemma}\label{lem:acc locn is 2-local}
  Suppose $\tC$ is a presentable \itcat{} and $\tC_{0}$ is a full
  sub-\itcat{}. Then the following are equivalent:
  \begin{enumerate}[(1)]
  \item $\tC_{0}$ is closed under conical colimits over
    $\kappa$-filtered \icats{} for some regular cardinal $\kappa$ and
    the inclusion $\tC_{0} \hookrightarrow \tC$ has a left adjoint.
  \item There exists a small set $S$ of morphisms in $\tC$ such that
    $\tC_{0} \simeq \Loc_{S}\tC$.
  \end{enumerate}
\end{lemma}
\begin{proof}
  Given (1), we know from \cite[Proposition 5.5.4.2]{LurieHTT} that
  there exists a small set $S$ of morphisms in $\tC$ such that
  $\tC_{0}^{\leq 1} \simeq \Loc_{S}(\tC^{\leq 1})$. But we also know
  that $\tC_{0}$ is closed under cotensors in $\tC$ (since the
  inclusion is a right adjoint and so preserves limits by
  \cref{prop:adjpreservelimtis}), so if $x \in \tC^{\leq 1}$ is
  $S$-local then so is $x^{[1]}$, hence $x$ is also $S$-2-local and so
  $\tC_{0}$ contains precisely the $S$-2-local objects, as
  required. Conversely, given (2) we know from \cref{obs:2-local
    local} that $\tC_{0}^{\leq 1} \simeq \Loc_{S'} \tC^{\leq 1}$ where
  $S'$ consists of $S$ and $S \boxtimes [1]$. Hence the inclusion
  $\tC_{0}^{\leq 1} \to \tC^{\leq 1}$ has a left adjoint by
  \cite[Proposition 5.5.4.15]{LurieHTT} and therefore also a left
  adjoint at the $(\infty,2)$-level by \cref{propn:upgrade adjoint}
  since $\Loc_{S}\tC$ is closed under cotensors with $[1]$ by
  \cref{lem:local closed limit}.
\end{proof}

\begin{proof}[Proof of \cref{thm:pres char}]
  Suppose first that $\tC$ is presentable. Then $\tC$ is locally small
  by   \cref{obs:pres loc small}, and the rest of condition (2)
  follows from \cite[Proposition 5.4.2.2]{LurieHTT} together with
  \cref{lem:compact is 2-compact}. Conversely, condition (2) implies
  that $\tC$ is presentable by applying \cite[Proposition
  5.4.2.2]{LurieHTT} again.
  
  The equivalence of the last two conditions is a special case of
  \cref{lem:acc locn is 2-local}, while condition (3) implies
  that $\tC$ is presentable by \cref{lem: acc locn 2-pres}. It remains
  to show that if $\tC$ is presentable then condition (3) holds.
  
  Choose $\kappa$ as in
  \cref{lem:compact is 2-compact} and consider the inclusion
  $i \colon \tC^{\kappa} \hookrightarrow \tC$. Since $\tC^{\kappa}$ is
  small and $\tC$ is cocomplete, we have a Yoneda extension
  \[ i' = h_{\tC^{\kappa},!}i \colon \tPSh(\tC^{\kappa}) \to \tC,\]
  which is cocontinuous, and so has a right adjoint \[R \colon \tC \to \tPSh(\tC^{\kappa})\] by
  \cref{propn:upgrade adjoint}. We claim that $R$ preserves conical
  colimits over $\kappa$-filtered \icats{} and is fully faithful,
  which will complete the proof. To see the former, we note that
  colimits in presheaves are computed pointwise, so given a diagram
  $\phi \colon \oJ \to \tC$ with $\oJ$ a $\kappa$-filtered \icat{} it
  suffices to check that for $x \in \tC^{\kappa}$,
  \[ \colim_{\oJ} \tPSh(\tC^{\kappa})(h_{\tC^{\kappa}}x, R(\phi))
    \simeq \tPSh(\tC^{\kappa})(h_{\tC^{\kappa}}x, \colim_{\oJ} R(\phi)) \to
    \tPSh(\tC^{\kappa})(h_{\tC^{\kappa}}x, R(\colim_{\oJ} \phi))\]
  is an equivalence. Via the adjunction and the equivalence between
  $i' \circ h_{\tC^{\kappa}}$ and the inclusion $i$, we can identify
  this as
  \[ \colim_{\oJ} \tC(x, \phi) \to \tC(x, \colim_{\oJ} \phi).\] This
  is indeed an equivalence since $x$ is 2-$\kappa$-compact. To see
  that $R$ is fully faithful we can equivalently prove that the counit
  map $i'R(c) \to c$ is an equivalence for all $c \in \tC$. Since $R$
  and $i'$ preserve $\kappa$-filtered colimits and any $c$ is the
  $\kappa$-filtered colimit of a diagram in $\tC^{\kappa}$, we may
  reduce to the case where $c$ is 2-$\kappa$-compact. In this case we
  know $c \simeq i'h_{\tC^{\kappa}}(c)$, so applying the 2-of-3
  property to the triangle equivalence for $i'$ we see that it is
  enough to show that the unit map $h_{\tC^{\kappa}}c \to Rc$ is an equivalence,
  or equivalently that the map
  \[ \tPSh(\tC^{\kappa})(\phi, h_{\tC^{\kappa}}(c)) \to \tC(i'\phi,
    i(c))\]
  is an equivalence for all presheaves $\phi$; we can detect this on
  representable presheaves, where 
  we are asking for the map
  \[ \tC^{\kappa}(x, c) \to \tC(i(x),i(c))\]
  to be an equivalence, which is clear.
\end{proof}

\newcommand{\etalchar}[1]{$^{#1}$}


\begin{thebibliography}{LMGR{\etalchar{+}}24}

\bibitem[Abe23]{Abellan2023}
Fernando Abellán.
\newblock Comparing lax functors of $(\infty,2)$-categories, 2023,
  \href{http://arxiv.org/abs/2311.12746}{{\ttfamily arXiv:2311.12746}}.

\bibitem[AG22]{AbMarked}
Fernando Abellán~García.
\newblock Marked colimits and higher cofinality.
\newblock {\em J. Homotopy Relat. Struct.}, 17(1):1--22, 2022.

\bibitem[AGH25]{AGH24}
Fernando Abell\'an, Andrea Gagna, and Rune Haugseng.
\newblock Straightening for lax transformations and adjunctions of
  {$(\infty,2)$}-categories.
\newblock {\em Selecta Math. (N.S.)}, 31(4):Paper No. 85, 81, 2025,
  \href{http://arxiv.org/abs/2404.03971}{{\ttfamily arXiv:2404.03971}}.

\bibitem[AM24]{2Topoi}
Fernando Abellán and Louis Martini.
\newblock $(\infty,2)$-topoi and descent, 2024,
  \href{http://arxiv.org/abs/2410.02014}{{\ttfamily arXiv:2410.02014}}.

\bibitem[AMGR17]{AyalaMazelGeeRozenblyumEnr}
David Ayala, Aaron Mazel-Gee, and Nick Rozenblyum.
\newblock Factorization homology of enriched $\infty$-categories, 2017,
  \href{http://arxiv.org/abs/1710.06414}{{\ttfamily arXiv:1710.06414}}.

\bibitem[AMGR24]{AMGR}
David Ayala, Aaron Mazel-Gee, and Nick Rozenblyum.
\newblock Stratified noncommutative geometry.
\newblock {\em Mem. Amer. Math. Soc.}, 297(1485):iii+260, 2024,
  \href{http://arxiv.org/abs/1910.14602}{{\ttfamily arXiv:1910.14602}}.

\bibitem[AS23a]{ASI23}
Fernando Abellán and Walker~H. Stern.
\newblock 2-cartesian fibrations {II}: A {G}rothendieck construction for
  $\infty$-bicategories.
\newblock {\em J. Inst. Math. Jussieu (to appear)}, 2023,
  \href{http://arxiv.org/abs/2201.09589}{{\ttfamily arXiv:2201.09589}}.

\bibitem[AS23b]{AS23}
Fernando Abellán and Walker~H. Stern.
\newblock On cofinal functors of $\infty$-bicategories, 2023,
  \href{http://arxiv.org/abs/2304.07028}{{\ttfamily arXiv:2304.07028}}.

\bibitem[Ber24]{BermanLax}
John~D. Berman.
\newblock On lax limits in {$\infty$}-categories.
\newblock {\em Proc. Amer. Math. Soc.}, 152(12):5055--5066, 2024,
  \href{http://arxiv.org/abs/2006.10851}{{\ttfamily arXiv:2006.10851}}.

\bibitem[CM23]{CMGray}
Timothy Campion and Yuki Maehara.
\newblock A model-independent gray tensor product for $(\infty,2)$-categories,
  2023,  \href{http://arxiv.org/abs/2304.05965}{{\ttfamily arXiv:2304.05965}}.

\bibitem[DDS18]{sigmalim}
M.E. Descotte, E.J. Dubuc, and M.~Szyld.
\newblock Sigma limits in 2-categories and flat pseudofunctors.
\newblock {\em Adv. Math.}, 333:266--313, 2018.

\bibitem[GH15]{GHenriched}
David Gepner and Rune Haugseng.
\newblock Enriched {$\infty$}-categories via non-symmetric {$\infty$}-operads.
\newblock {\em Adv. Math.}, 279:575--716, 2015.

\bibitem[GHL21]{GHLGray}
Andrea Gagna, Yonatan Harpaz, and Edoardo Lanari.
\newblock Gray tensor products and lax functors of {$(\infty,2)$}-categories.
\newblock {\em Adv. Math.}, 391:Paper No. 107986, 32, 2021,
  \href{http://arxiv.org/abs/2006.14495}{{\ttfamily arXiv:2006.14495}}.

\bibitem[GHL24]{GHLFib}
Andrea Gagna, Yonatan Harpaz, and Edoardo Lanari.
\newblock Cartesian fibrations of {$(\infty,2)$}-categories.
\newblock {\em Algebr. Geom. Topol.}, 24(9):4731--4778, 2024,
  \href{http://arxiv.org/abs/2107.12356}{{\ttfamily arXiv:2107.12356}}.

\bibitem[GHL25]{GagnaHarpazLanariLaxLim}
Andrea Gagna, Yonatan Harpaz, and Edoardo Lanari.
\newblock Marked limits in $(\infty,2)$-categories.
\newblock {\em Doc. Math.}, published online, 2025.

\bibitem[GHN17]{GHN}
David Gepner, Rune Haugseng, and Thomas Nikolaus.
\newblock Lax colimits and free fibrations in $\infty$-categories.
\newblock {\em Doc. Math.}, 22:225--1266, 2017.

\bibitem[Hau15]{HaugsengRect}
Rune Haugseng.
\newblock Rectification of enriched {$\infty$}-categories.
\newblock {\em Algebr. Geom. Topol.}, 15(4):1931--1982, 2015.

\bibitem[Hei23]{HeineEquiv}
Hadrian Heine.
\newblock An equivalence between enriched {$\infty$}-categories and
  {$\infty$}-categories with weak action.
\newblock {\em Adv. Math.}, 417:Paper No. 108941, 140, 2023.

\bibitem[Hei24]{heineweighted}
Hadrian Heine.
\newblock The higher algebra of weighted colimits, 2024,
  \href{http://arxiv.org/abs/2406.08925}{{\ttfamily arXiv:2406.08925}}.

\bibitem[HHLN23]{TwoVariable}
Rune Haugseng, Fabian Hebestreit, Sil Linskens, and Joost Nuiten.
\newblock Lax monoidal adjunctions, two-variable fibrations and the calculus of
  mates.
\newblock {\em Proc. Lond. Math. Soc. (3)}, 127(4):889--957, 2023,
  \href{http://arxiv.org/abs/2011.08808}{{\ttfamily arXiv:2011.08808}}.

\bibitem[Hin16]{HinichDK}
Vladimir Hinich.
\newblock Dwyer-{K}an localization revisited.
\newblock {\em Homology Homotopy Appl.}, 18(1):27--48, 2016.

\bibitem[Hin20]{HinichYoneda}
Vladimir Hinich.
\newblock Yoneda lemma for enriched {$\infty$}-categories.
\newblock {\em Adv. Math.}, 367:107129, 119, 2020.

\bibitem[HM25]{HenryMeadows}
Simon Henry and Nicholas~J. Meadows.
\newblock Higher theories and monads.
\newblock {\em High. Struct.}, 9(1):227--268, 2025.

\bibitem[Lam17]{Lambert}
Michael Lambert.
\newblock Computing weighted colimits, 2017,
  \href{http://arxiv.org/abs/1711.05903}{{\ttfamily arXiv:1711.05903}}.

\bibitem[LMGR{\etalchar{+}}24]{Soergel}
Yu~Leon Liu, Aaron Mazel-Gee, David Reutter, Catharina Stroppel, and Paul
  Wedrich.
\newblock A braided monoidal $(\infty,2)$-category of {S}oergel bimodules,
  2024,  \href{http://arxiv.org/abs/2401.02956}{{\ttfamily arXiv:2401.02956}}.

\bibitem[Lou24]{loubaton}
F\'elix Loubaton.
\newblock Categorical theory of $(\infty,\omega)$-categories, 2024,
  \href{http://arxiv.org/abs/2406.05425}{{\ttfamily arXiv:2406.05425}}.

\bibitem[Lur09a]{LurieHTT}
Jacob Lurie.
\newblock {\em Higher topos theory}, volume 170 of {\em Annals of Mathematics
  Studies}.
\newblock Princeton University Press, Princeton, NJ, 2009.

\bibitem[Lur09b]{LurieGoodwillie}
Jacob Lurie.
\newblock $(\infty,2)$-categories and the {G}oodwillie calculus {I}, 2009,
  \href{http://arxiv.org/abs/0905.0462}{{\ttfamily arXiv:0905.0462}}.

\bibitem[Lur17]{LurieHA}
Jacob Lurie.
\newblock Higher algebra, 2017.
\newblock \url{https://www.math.ias.edu/~lurie/papers/HA.pdf}.

\bibitem[MGS24]{secondaryK}
Aaron Mazel-Gee and Reuben Stern.
\newblock A universal characterization of noncommutative motives and secondary
  algebraic {K}-theory.
\newblock {\em Ann. K-Theory}, 9:369--445, 2024.

\bibitem[NS18]{NikolausScholze}
Thomas Nikolaus and Peter Scholze.
\newblock On topological cyclic homology.
\newblock {\em Acta Math.}, 221(2):203--409, 2018.

\bibitem[Nui24]{Nuiten}
Joost Nuiten.
\newblock On straightening for {S}egal spaces.
\newblock {\em Compos. Math.}, 160(3):586–656, 2024.

\bibitem[RV22]{RVelements}
Emily Riehl and Dominic Verity.
\newblock {\em Elements of {$\infty$}-category theory}, volume 194 of {\em
  Cambridge Studies in Advanced Mathematics}.
\newblock Cambridge University Press, Cambridge, 2022.

\bibitem[Sha21]{Shah}
Jay Shah.
\newblock Parametrized higher category theory {I}{I}: Universal constructions,
  2021,  \href{http://arxiv.org/abs/2109.11954}{{\ttfamily arXiv:2109.11954}}.

\bibitem[Ste20]{StefanichPres}
Germ\'{a}n Stefanich.
\newblock Presentable $(\infty,n)$-categories, 2020,
  \href{http://arxiv.org/abs/2011.03035}{{\ttfamily arXiv:2011.03035}}.

\bibitem[Str87]{StreetOriental}
Ross Street.
\newblock The algebra of oriented simplexes.
\newblock {\em J. Pure Appl. Algebra}, 49(3):283--335, 1987.

\end{thebibliography}
\end{document}